\documentclass[11pt,reqno]{amsart} 

\usepackage[utf8]{inputenc}

\usepackage[utf8]{inputenc}
\usepackage[a4paper]{geometry}

\usepackage{amsmath}
\usepackage{amsfonts}
\usepackage{amssymb}
\usepackage{amsthm}
\usepackage{paralist}
\usepackage{natbib}
\usepackage{color}
\usepackage[colorlinks=true,linkcolor=blue,citecolor=blue,pdfborder={0 0 0}]{hyperref}

\usepackage{dsfont}
\usepackage{bm}

\usepackage[linesnumbered,ruled]{algorithm2e}

\usepackage{color}

\usepackage{dsfont}
\usepackage{bm}

\usepackage{graphicx}
\usepackage{enumerate}
\usepackage{epstopdf}
\usepackage{diagbox}

\newtheorem{theorem}{Theorem}[section]
\newtheorem{proposition}[theorem]{Proposition}
\newtheorem{lemma}[theorem]{Lemma}

\theoremstyle{definition}

\newtheorem{condition}[theorem]{Condition}

\theoremstyle{remark}
\newtheorem{remark}[theorem]{Remark}
\newtheorem{example}[theorem]{Example}

\numberwithin{equation}{section}

\newcommand{\dto}{\rightsquigarrow}
\newcommand{\pto}{\xrightarrow{p}}

\newcommand{\R}{\mathbb{R}}

\newcommand{\Z}{\mathbb{Z}}
\newcommand{\N}{\mathbb{N}}
\newcommand\Prob{\mathbb{P}}    

\newcommand{\FF}{\mathbb{F}}
\newcommand{\TT}{\mathbb{T}}

\newcommand{\Bc}{\mathcal{B}}

\newcommand{\Nc}{\mathcal{N}}
\newcommand{\Fc}{\mathcal{F}}
\newcommand{\Jc}{\mathcal{J}}

\newcommand{\scs}{\scriptscriptstyle}

\newcommand{\eps}{\varepsilon}

\newcommand{\ind}{\operatorname{\bf{1}}}
\newcommand{\diff}{{\,\mathrm{d}}}

\newcommand{\bn}{b_n} 
\newcommand{\kn}{k_n} 
\newcommand{\tailp}{e_n} 
\newcommand{\tailplim}{{e}} 
\newcommand{\matr}[1]{
\begin{pmatrix}
    #1
\end{pmatrix}
}

\newcommand{\ip}[1]{\lfloor #1 \rfloor}

\newcommand{\Fin}{F^{\leftarrow}}

\newcommand{\Exp}{\operatorname{E}}
\newcommand{\Var}{\operatorname{Var}}
\newcommand{\Cov}{\operatorname{Cov}}

\newcommand{\Be}{\operatorname{B}}
\newcommand{\No}{\operatorname{N}}
\newcommand{\slb}{\operatorname{sl}}
\newcommand{\djb}{\operatorname{dj}}

\newcommand{\even}{\operatorname{even}}
\newcommand{\odd}{\operatorname{odd}}

\begin{document}

\title[Maximum Likelihood Estimation of the Extremal Index]{Weak convergence of a pseudo maximum likelihood estimator for the extremal index}

\author{Betina Berghaus}
\address{Ruhr-Universit\"at Bochum, Fakult\"at f\"ur Mathematik, Universit\"atsstr.\ 150, 44780 Bochum, Germany}
\email{betina.berghaus@rub.de}

\author{Axel B\"ucher}
\address{Ruhr-Universit\"at Bochum, Fakult\"at f\"ur Mathematik, Universit\"atsstr.\ 150, 44780 Bochum, Germany}
\email{axel.buecher@rub.de}

\date{\today}

\begin{abstract}
The extremes of a stationary time series typically occur in clusters. A primary measure for this phenomenon is the extremal index, representing  the reciprocal of the expected cluster size. Both a disjoint and a sliding blocks estimator for the extremal index are analyzed in detail. In contrast to many competitors, the estimators only depend on the choice of one parameter sequence.  
We derive an asymptotic expansion, prove asymptotic normality
and show consistency of an estimator for the asymptotic variance. Explicit calculations in certain models and a finite-sample Monte Carlo simulation study  reveal that the sliding blocks estimator outperforms other blocks estimators, and that 
it is competitive to runs- and inter-exceedance estimators in  various models. The methods are applied to a variety of financial time series.
\medskip


\end{abstract}

\maketitle

\date{\today}


\section{Introduction}

An adequate description of the extremal behavior of a time series is important in many applications, such as in hydrology, finance or actuarial science (see, e.g., Section~1.3 in the monograph \citealp{BeiGoeSegTeu04}). The extremal behavior can be characterized by the tail of the marginal law of the time series and by the serial dependence; that is, by the tendency that extremal observations tend to occur in clusters. A primary measure of extremal serial dependence is given by the extremal index $\theta\in[0,1]$, which can be interpreted as being equal to the reciprocal of the mean cluster size.  The underlying theory was worked out in \cite{Lea83, LeaLinRoo83, Obr87, HsiHusLea88, LeaRoo88}.

Estimating the extremal index based on a finite stretch from the time series has been extensively studied in the literature. Common approaches are based on the blocks method, the runs method and the inter-exceedance times method (see \citealp{BeiGoeSegTeu04}, Section~10.3.4, for an overview). The first two methods usually depend on two parameters to be chosen by the statistician: a threshold sequence and a cluster identification scheme parameter (such as a block length). In contrast, inter-exceedance type-estimators are attractive since  they only depend on a threshold sequence. Some references are \cite{Hsi93, SmiWei94, WeiNov98, FerSeg03, Suv07,Rob09, RobSegFer09,SuvDav10}, among others. 
The present paper is on a blocks estimator (and a slightly modified version) due to \cite{Nor15}, which, remarkably, only depends on a cluster identification parameter. This makes the estimator practically appealing in comparison to other blocks methods.

In many papers on estimating the extremal index, either no asymptotic theory is given (such as in \citealp{Suv07,Nor15}), or the asymptotic theory is incomplete in the sense that theory is developed for a non-random threshold sequence, while in practice a random sequence must be used (as, e.g., in \citealp{WeiNov98,RobSegFer09}). As pointed out in the latter paper, ``the mathematical treatment of such random threshold sequences requires complicated empirical process theory''. In the present paper, the mathematical treatment is comprehensive, working out all the arguments needed from empirical process theory.

Let us proceed by motivating and defining the estimator: throughout, $X_1, X_2, \dots$ denotes a stationary sequence of real-valued random variables with stationary cumulative distribution function (cdf) $F$.  The sequence is assumed to have an extremal index $\theta \in (0,1]$: for any $\tau>0$, there exists a sequence $u_n=u_n(\tau)$ such that $\lim_{n \to \infty} n \bar F(u_n) = \tau$ and such that
\begin{align*} 
\lim_{n\to\infty} \Prob(M_{1:n} \le u_n) = e^{-\theta \tau}.
\end{align*}
Here, $\bar F=1-F$ and $M_{1:n} = \max(X_1, \dots, X_n)$.

For simplicity, we assume that $F$ is continuous (c.f.\ Remark~\ref{rem:cont} below) and define a sequence of  standard uniform random variables by $U_s = F(X_s)$. For $x\in(0,1)$, let $u_n=\Fin(1-x/n)$
 and $u_n'=\Fin(e^{-x/n})$, 
where $F^{\leftarrow}$ denotes the generalized, left-continuous inverse of the cdf $F$. 
Then, $n \bar F(u_n) =  x$ and $n \bar F(u_n') = n(1-e^{-x/n}) \to x$  as $n\to \infty$, whence
\begin{align} \label{eq:exp2}
\Prob( n(1-N_{1:n}) \ge x) 
&= 
\Prob( M_{1:n} \le u_n)   
\to 
 e^{-\theta x}, \\
\Prob( - n\log(N_{1:n}) \ge x)  \label{eq:exp3}
&= 
\Prob( M_{1:n} \le u_n')   
\to 
e^{-\theta x},
\end{align}
where $N_{1:n} = F(M_{1:n}) = \max\{U_1, \dots, U_n\}$. In other words, both $Y_{1:n} = -n\log(N_{1:n})$ and $Z_{1:n}=n(1-N_{1:n})$ asymptotically follow an exponential distribution with parameter~$\theta$. The result concerning $Y_{1:n}$ inspired \cite{Nor15} to estimate $\theta$ by the maximum likelihood estimator for the exponential distribution, based on a sample of estimated block maxima.  

More precisely, suppose that we observe a stretch of length $n$ from the time series $(X_s)_{s\ge 1}$. Divide the sample into $\kn$ blocks of length $\bn$, and for simplicity assume that $n=\bn \kn$ (otherwise, the final block would consist of less than $\bn$ observations and should be omitted). For $i=1, \dots, \kn$, let
\[
M_{ni} = M_{((i-1)\bn + 1) : (i\bn)} =  \max\{ X_{(i-1)\bn + 1}, \dots, X_{i\bn} \} 
\]
denote the maximum over the $X_s$ from the $i$th block. Also, let $N_{ni} = F(M_{ni}) =  \max\{ U_{(i-1)\bn + 1}, \dots, U_{i\bn} \}$ and $Y_{ni} = -\bn \log(N_{ni})$. If $\bn$ is sufficiently large, then, by~\eqref{eq:exp3}, the (unobservable) random variables $Y_{n1}, \dots, Y_{nk}$ 
form an approximate sample from the Exponential$(\theta)$-distribution. Moreover, as common when working with block maxima of a time series, they may be considered as asymptotically independent, which prompted \cite{Nor15} to estimate $\theta$ by the maximum-likelihood estimator for the Exponen\-tial$(\theta)$ distribution:
\[
\tilde \theta_n^{\No} = \Big( \frac{1}{\kn} \sum_{i=1}^{\kn} Y_{ni} \Big)^{-1}.
\]
Note that $\tilde \theta_n^{\No}$ should not be considered an estimator, as it is based on the unknown cdf~$F$. Subsequently, we call $\tilde \theta_n^{\No}$ an oracle for $\theta$. 

In practice, the $U_s$ are not observable, whence they need to be replaced by their observable counterparts giving rise to the definitions
\[
\hat N_{ni} =  \hat F_n(M_{ni}) \quad  \text{and} \quad \hat Y_{ni} = -\bn\log(\hat N_{ni}), 
\]
where $\hat F_n(x) = n^{-1} \sum_{s=1}^n \ind(X_s \le x)$ denotes the empirical cdf of $X_1, \dots, X_n$.
We obtain, up to a bias correction discussed below,  Northrop's estimator
\begin{align} \label{eq:djbN}
\hat \theta_n^{\No} = \hat \theta_n^{\No,\djb} = \Big( \frac{1}{\kn} \sum_{i=1}^{\kn} \hat Y_{ni} \Big)^{-1}.
\end{align}
In \cite{Nor15}, no asymptotic theory on $\hat \theta_n^{\No}$ given. 
While deriving the asymptotic distribution of the oracle $\tilde \theta_n^{\No}$ may appear tractable (see also \citealp{Rob09}: essentially, a central limit theorem for rowwise dependent triangular arrays is to be shown, followed by an argument using the delta method), asymptotic theory on the estimator $\hat \theta_n^{\No}$ is substantially more difficult due to the additional serial dependence induced by the rank transformation (which on top of that operates between blocks instead of within blocks).

A central contribution of the present paper is the derivation of the asymptotic distribution of $\hat \theta_n^{\No}$.  It will further turn out that the impact of the rank transformation is non-negligible, resulting in different asymptotic variances of $\hat \theta_n^{\No}$ and the corresponding oracle~$\tilde \theta_n^{\No}$. 
For that purpose, it will be convenient to consider the following (mathematically  simpler) variant of Northrop's estimator,
\begin{align} \label{eq:djb}
\hat \theta_n^{\Be} = \hat \theta_n^{\Be,\djb} = \Big( \frac{1}{\kn} \sum_{i=1}^{\kn} \hat Z_{ni} \Big)^{-1}, \qquad  \hat Z_{ni} = \bn(1-\hat N_{ni}).
\end{align}
This estimator can either be motivated following the above lines, but using \eqref{eq:exp2} rather than \eqref{eq:exp3} as a starting point, or by consulting \cite{Rob09} and writing 
\begin{align} \label{eq:robeq}
\frac{1}{\hat \theta_n^{\Be, \djb}}= \int_0^\infty \hat p_n^{(\tau)}(0)\mathrm{d}\tau,
\end{align}
with $\hat p_n^{(\tau)}(0)$ denoting Robert's estimator for $e^{-\theta \tau}$ (page 275 in \citealp{Rob09}, with `$>$' replaced by `$\ge$' in his definition of $\hat N^{\scs (\tau)}_{\scs r_n,j}$).
We will show below (Theorem~\ref{theo:Northrop}) that $\hat \theta_n^{\Be}$ and $\hat \theta_n^{\No}$ are in fact asymptotically equivalent.  
We also present asymptotic theory  for modifications of $\hat \theta_n^{\No}$ and $\hat \theta_n^{\Be}$ based on sliding block maxima, which is the second main contribution of the paper. Finally, the asymptotic expansions for $\hat \theta_n^{\Be}$ suggest estimators for the asymptotic variance of $\hat \theta_n^{\No}$ and $\hat \theta_n^{\Be}$ (and its sliding blocks variants); proving  their consistency is the third main contribution.

The remaining parts of this paper are organized as follows: in Section~\ref{sec:pre}, we present mathematical preliminaries needed to formulate and derive the asymptotic distributions of the estimators for $\theta$. Asymptotic equivalence, consistency and asymptotic normality is then shown in Section~\ref{sec:main}. Estimators of the asymptotic variance are handled in Section~\ref{sec:boot}.  In Section~\ref{sec:bias}, we propose a simple device to reduce the bias of the estimator and relate it to the ad-hoc approach in \cite{Nor15}. Examples are worked out in detail in Section~\ref{sec:examples}, while finite-sample results and a case study are presented in Sections~\ref{sec:sim} and~\ref{sec:app}, respectively.
Sections~\ref{sec:proofs} and~\ref{sec:proofs2} contain a sequence of auxiliary lemmas needed for the proof of the main results. Their proofs, as well as additional proofs are postponed to the supplementary material (Appendices~\ref{sec:disjoint}, \ref{sec:sliding},~\ref{sec:nor} and \ref{sec:addproofs}). The supplementary material also contains additional simulation results (Appendix~\ref{sec:addsim}).

\section{Mathematical preliminaries} \label{sec:pre}

The serial dependence of the time series $(X_s)_s$ will be controlled via mixing coefficients. For two sigma-fields $\Fc_1, \Fc_2$ on a probability space $(\Omega, \Fc, \Prob)$, let
\begin{align*}
\alpha({\Fc}_1,{\Fc}_2)
&=
\sup_{A \in {\Fc}_1, B\in {\Fc}_2} |\Prob(A\cap B)-\Prob(A)\Prob(B)| .
\end{align*}
In time series extremes, one usually imposes assumptions on the decay of the mixing coefficients between sigma-fields generated by $\{X_i \ind(X_s > \Fin(1-\eps_n)): s\le  \ell\}$ and  $\{X_s \ind(X_s > \Fin(1-\eps_n)): s \ge \ell+k \}$, where $\eps_n \to 0$ is some sequence reflecting the fact that only the dependence in the tail needs to be restricted (see, e.g., \citealp{Roo09}). For our purposes, we need slightly more to control even the dependence between the smallest of all block maxima (see also Condition~\ref{cond:cond}\eqref{item:blockdiv} below). More precisely, for $-\infty \le p < q \le \infty$ and $\eps\in(0,1]$, let $\Bc_{p:q}^\eps$  denote the sigma algebra generated by $U_s^\eps:=U_s \ind(U_s > 1- \eps)$ with $s\in \{p, \dots, q\}$ and define, for $\ell\ge 1$,
\[
\alpha_{\eps}(\ell) = \sup_{k \in \N} \alpha(\Bc_{1:k}^\eps, \Bc_{k+n:\infty}^\eps).
\] 
Note that the coefficients are increasing in $\eps$, whence they are bounded by the standard alpha-mixing coefficients of the sequence $U_s$, which can be retrieved for $\eps=1$.   In Condition~\ref{cond:cond}\eqref{item:rate2} below, we will impose a condition on the decay of the mixing coefficients for small values of $\eps$.

The extremes of a time series may be conveniently described by the point process of normalized exceedances. The latter is defined, 
for a Borel set $A\subset E:= (0,1]$ and a number $x\in[0,\infty)$, by
\[
N_n^{(x)}(A) = \sum_{s=1}^n \ind(s/n \in A, U_s > 1-x/n).
\]
Note that $N_{n}^{(x)}(E)=0$ iff $N_{1:n} \le 1-x/n$; the probability of that event converging to $e^{-\theta x}$ under the assumption of the existence of extremal index~$\theta$.

Fix $m\ge 1$ and $x_1> \dots > x_m>0$. For $1\le p < q \le n$, let
$\Fc_{p:q,n}^{(x_1, \dots, x_m)}$ denote the sigma-algebra generated by the events $\{U_i > 1- x_j/n\}$ for $p\le i \le q$ and  $1 \le j \le m$. For $1\le \ell \le n$, define
\begin{multline*}
\alpha_{n,\ell}(x_1, \dots, x_m)=\sup\{ |\Prob(A\cap B) - \Prob(A) \Prob(B)| : \\
A \in \Fc_{1:s,n}^{(x_1, \dots, x_m)}, B \in \Fc_{s+\ell:n,n}^{(x_1, \dots, x_m)}, 1 \le s \le n-\ell\}.
\end{multline*}
The condition $\Delta_n(\{u_n(x_j)\}_{1\le j \le m})$ is said to hold if there exists a sequence $(\ell_n)_n$ with $\ell_n=o(n)$ such that $\alpha_{n,\ell_n}(x_1, \dots, x_m) =o(1)$ as $n\to\infty$. 
A sequence $(q_n)_n$ with $q_n=o(n)$ is said to be $\Delta_n(\{u_n(x_j)\}_{1\le j \le m})$-separating if there exists a sequence $(\ell_n)_n$ with $\ell_n=o(q_n)$ such that $nq_n^{-1} \alpha_{n,\ell_n}(x_1, \dots,x_m) =o(1)$ as $n\to\infty$. 
If $\Delta_n(\{u_n(x_j)\}_{1\le j \le m})$ is met, then such a sequence always exists, simply take $q_n=\ip{\max\{ n\alpha_{n,\ell_n}^{\scs 1/2}, (n\ell_n)^{\scs 1/2}\}}.$ 
 
 By Theorems 4.1 and 4.2 in \cite{HsiHusLea88},
if the extremal index exists and the $\Delta(u_n(x))$-condition is met ($m=1$), then a necessary and sufficient condition for weak convergence of $N_n^{\scriptscriptstyle(x)}$ is convergence of the conditional distribution of $N_{n}^{\scriptscriptstyle(x)}(B_n)$ with $B_n=(0,q_n/n]$ given that there is at least one exceedance of $1-x/n$ in $\{1, \dots, q_n\}$ to a probability distribution $\pi$ on $\N$, that is,
\[
\lim_{n\to\infty} \Prob(N_n^{(x)} (B_n) = j \mid N_{n}^{(x)}(B_n)>0) = \pi(j)  \qquad \forall \, j\ge 1,
\]
where $q_n$ is some $\Delta(u_n(x))$-separating sequence. Moreover, in that case, the convergence in the last display holds for any $\Delta(u_n(x))$-separating sequence $q_n$. If the $\Delta(u_n(x))$-condition holds for any $x>0$, then $\pi$ does not depend on $x$ (\citealp{HsiHusLea88}, Theorem~5.1). 
 
A multivariate version of the latter results is stated in \cite{Per94}, see also the summary in \cite{Rob09}, page 278, and the thesis \cite{Hsi84}. Suppose that the extremal index exists and that  the $\Delta(u_n(x_1), u_n(x_2))$-condition is met for any $x_1\ge x_2\ge0, x_1 \ne0$. Moreover assume that there exists a family of probability measures $\{\pi_2^{\scs(\sigma)}: \sigma\in [0,1]\}$ on $\Jc = \{(i,j): i \ge j \ge 0, i \ge 1\}$ such that, for all $(i,j) \in \Jc$,
\[
\lim_{n\to\infty} \Prob(N_n^{(x_1)} (B_n) = i, N_n^{(x_2)} (B_n) = j \mid N_{n}^{(x_1)}(B_n) >0) = \pi_2^{(x_2/x_1)}(i,j),
\]
where $q_n$ is some $\Delta(u_n(x_1), u_n(x_2))$-separating sequence.  In that case, the two-level point process $\bm N_n^{\scriptscriptstyle(x_1,x_2)}=(N_n^{\scriptscriptstyle(x_1)}, N_{n}^{\scriptscriptstyle(x_2)})$ converges in distribution to a point process with characterizing Laplace transform explicitly stated in \cite{Rob09} on top of page 278. Note that
\[
\pi_2^{(1)}(i,j)=\pi(i) \ind(i=j), \qquad \pi_2^{(0)}(i,j) = \pi(i) \ind(j=0).
\]

The following set of conditions will be imposed to establish asymptotic normality of the estimators.

\begin{condition} \label{cond:cond}~
\begin{enumerate}[(i)]
\item\label{item:generalpoint} \textbf{Extremal index and the point process of exceedances.} The extremal index $\theta\in (0,1]$ exists and the above assumptions guaranteeing convergence of the one- and two-level point process of exceedances are satisfied.

\item\label{item:momn} \textbf{Moment assumption on the point process.} There exists $\delta>0$ such that, for any $\ell>0$, there exists a constant $C_\ell'$ such that
\[
\Exp[|N_n^{(x_1)}(E) - N_n^{(x_2)}(E)|^{2+\delta}] \le C_\ell'(x_1 - x_2) \qquad \forall\, \ell \ge x_1 \ge x_2 \ge 0, n \in \N.
\]

\item\label{item:rate2} \textbf{Asymptotic independence in the big-block/small-block heuristics.}
There exists $c_2\in(0,1)$ and $C_2>0$ such that  
\[
\alpha_{c_2}(\ell) \le C_2 \ell^{-\eta}
\]
for some $\eta\ge 3(2+\delta)/(\delta-\mu)>3$ with $0< \mu<\delta \wedge(1/2)$ and with $\delta>0$ from Condition~\eqref{item:momn}.
The block size $\bn\to\infty$ is chosen in such a way that 
\begin{align} \label{eq:rate1}
\kn=o(\bn^{2}) , \qquad n\to\infty,
\end{align}
and such
that there exists a sequence $\ell_n \to \infty$ (to be thought of as the length of small blocks which are to be clipped-of at the end of each block of size $\bn$) satisfying
$
\ell_n = o(\bn^{2/(2+\delta)})$ and  
$\kn \alpha_{c_2} (\ell_n) =o(1); 
$
all convergences being for $n\to\infty$.

\item\label{item:varbound} \textbf{Bound on the variance of the empirical process.}
There exist some constants $c_1\in(0,1), C_1>0$ such that, for all $y\in(0,c_1)$ and all $n \in \N$,
\[
\Var\Big\{\sum_{s=1}^{n} \ind(U_s > 1- y)  \Big\} \le C_1(n y + n^2y^2).
\]

\item\label{item:blockdiv} \textbf{All standardized block maxima of size $\bn/2$ converge to $1$.} For all $c\in(0,1)$, we have
\[
\lim_{n \to \infty} \Prob\left( \textstyle \min_{i=1}^{2\kn} N_{ni}' \le c\right)  = 0,
\] 
where $N_{ni}' = \max\{ U_s: s \in [ (i-1) \bn/2 + 1, \dots, i \bn/2 ]\}$, for $i=1, \dots , 2\kn$, denote consecutive standardized block maxima of (approximate) size $\bn/2$.

\item\label{item:moment} \textbf{Existence of moments of maxima.} With $\delta>0$ from Condition~\eqref{item:momn}, we have
\[
\limsup_{n\to\infty} \Exp[Z_{1:n}^{2+\delta}] < \infty.
\]

\item\label{item:bias} \textbf{Bias.} 
As $n\to\infty$,
\[
\Exp[Z_{1:\bn}] = \theta^{-1} + o(\kn^{-1/2}).
\]
\end{enumerate}

\end{condition}

Assumptions \eqref{item:generalpoint}--\eqref{item:rate2} are suitable adaptations of Conditions (C1) and (C2) in \cite{Rob09}; in fact, they can be seen to imply the latter. Among other things, these conditions are needed to apply his central result, Theorem 4.1, on the weak convergence of the tail empirical process on $[0,\infty)$.
Note that the assumptions are satisfied for solutions of stochastic difference equations, see Example 3.1 in \cite{Rob09}.
The Assumption in \eqref{eq:rate1} is a growth condition that is needed in the proof of Lemma~\ref{lem:boundsupp}. As argued in \cite{RobSegFer09}, it is actually a weak requirement, as in many time series models it is a necessary condition for the bias condition in \eqref{item:bias} to be true (see Section~\ref{sec:examples} below). Finally, a positive extremal index can be guaranteed by assuming that
\begin{align} \label{eq:positive}
\lim_{m \to \infty} \limsup_{n \to \infty} \Prob  ( N_{m:\bn} > 1-\tfrac{x}{n} \mid U_1 \geq 1- \tfrac{x}{n} ) =0
\end{align}
for any $x>0$, see \cite{BeiGoeSegTeu04}, formula (10.8). We will additionally need this assumption for the calculation of the asymptotic variance of the estimators.

In a slightly different form concerning only the tail, Assumption~\eqref{item:varbound} has also been made in Condition (C3) in \cite{Dre00} for proving weak convergence of the tail empirical process. In comparison to there, the extra factor $n^2y^2$ allows for additional flexibility, in that it allows for $O(n^2)$-non-negligible covariances, as long as their contribution is at most $y^2$. 
In Section~\ref{sec:examples}, we show that 
the assumption holds for solutions of stochastic difference equations, such as the squared ARCH-model, and for max-autoregressive models.

Recall that $N_{ni}^{\bn}$ is approximately Beta$(\theta,1)$-distributed. As a consequence, every standardized block maximum $N_{ni}$ must  converge to $1$ as the sample size grows to infinity. Still, out of the sample of $\kn$ block maxima, the smallest one could possibly be smaller than one, especially when the number of blocks is large. Assumption \eqref{item:blockdiv} prevents this from happening; note that a similar assumption has also been made in \cite{BucSeg15}, Condition 3.2. Imposing the assumption even for block maxima $N_{ni}'$ of size $\bn/2$ guarantees that also the minimum over all big sub-block maxima (needed in the proof for the disjoint blocks estimator) and the minimum over all sliding block maxima of size $\bn$ (needed in the proof for the sliding blocks estimator) converges to $1$.

Assumption~\eqref{item:moment} is needed to deduce uniform integrability of the sequence $Z_{1:\bn}^2$. It implies convergence of the variance of $Z_{1:\bn}$ to that of an exponential distribution with parameter $\theta$.
Finally,~\eqref{item:bias} requires the approximation of the first moment of $Z_{1:\bn}$ by that of an exponential distribution to be sufficiently accurate.

\section{Main results} 
\label{sec:main} 

In this section we prove consistency and asymptotic normality of the disjoint blocks estimators $\hat \theta_n^{\No,\djb}$ and $\hat \theta_n^{\Be,\djb}$ defined in \eqref{eq:djbN} and \eqref{eq:djb}, respectively, as well as of variants which are based on sliding blocks and which we will denote by $\hat \theta_n^{\No,\slb}$ and $\hat \theta_n^{\Be,\slb}$, respectively. We begin by defining the latter estimators.

Divide the sample into $n-\bn+1$ blocks of length $\bn$, i.e., for $t = 1, \dots ,n-\bn+1$, let
\[
M_{nt}^{\slb} = M_{t:(t+\bn-1)} = \max\{ X_{t}, \dots , X_{t+\bn-1}\}.
\]
Analogously  to the notation used in the definition of the estimators for disjoint blocks, we will write $N_{nt}^{\slb} = F (M_{nt}^{\slb})$, $Z_{nt}^{\slb} = \bn (1- N_{nt}^{\slb})$ and  $Y_{nt}^{\slb} = -\bn \log(N_{nt}^{\slb})$ and define their empirical counterparts $\hat N_{nt}^{\slb} = \hat{ F}_ n(M_{nt}^{\slb}), \hat Z_{nt}^{\slb} = \bn (1- \hat N_{nt}^{\slb})$ and $\hat Y_{nt}^{\slb} = -\bn \log(\hat N_{nt}^{\slb})$, where $\hat F_n$ is the empirical cdf of $X_1, \dots, X_n$. Just as for the disjoint blocks estimators, the (pseudo-)obser\-vations $\hat Z_{nt}^{\slb}$ and $\hat Y_{nt}^{\slb}$ are approximately exponentially distributed with mean $\theta^{-1}$, which suggests to estimate $\theta$ by the reciprocal of their empirical mean:
\[
\hat{\theta}_n^{\Be,\slb} = \Big(  \frac{1}{n-\bn+1} \sum_{t=1}^{n-\bn+1} \hat Z_{nt}^{\slb} \Big)^{-1}, \qquad
\hat{\theta}_n^{\No,\slb} = \Big(  \frac{1}{n-\bn+1} \sum_{t=1}^{n-\bn+1} \hat Y_{nt}^{\slb} \Big)^{-1}.
\]
Up to a bias correction discussed below, $\hat \theta_n^{\No, \slb}$ is the sliding blocks estimator proposed in \cite{Nor15}.
Note that, for both estimators, no data has to be discarded if $\bn$ is not a divisor of the sample size $n$. 

The first central result is on first order asymptotic equivalence between the proposed estimators, proven in Section~\ref{sec:nor} in the supplementary material.

\begin{theorem}\label{theo:Northrop}
Suppose that
Condition~\ref{cond:cond} and \eqref{eq:positive} is 
met. Then, as $n\to\infty$,
\[
\sqrt{\kn}(\hat \theta_n^{\Be,\djb} - \hat \theta_n^{\No,\djb})=o_\Prob(1) 
\quad \text{ and } \quad
\sqrt{\kn}( \hat \theta_n^{\Be, \slb} - \hat \theta_n^{\No,\slb})=o_\Prob(1).
\]
\end{theorem}

As a consequence of this theorem, we may concentrate on the mathematically simpler estimators $\hat{\theta}_n^{\Be,\djb}$ and $\hat{\theta}_n^{\Be,\slb}$ in the following asymptotic analysis. We will shortly write $\hat{\theta}_n^{\djb}$ and $\hat{\theta}_n^{\slb}$, respectively. Note that, while $\hat{\theta}_n^{\slb}$ is based on a substantially larger number of blocks than the disjoint blocks estimator, the blocks are heavily correlated. The following theorem is the central result of this paper and shows that both estimators are consistent and converge at the same rate to a normal distribution. The disjoint blocks estimator has a larger asymptotic variance than the sliding blocks estimator (see also \citealp{RobSegFer09}).

\begin{theorem} \label{theo:main}
Suppose that Condition~\ref{cond:cond} and \eqref{eq:positive} is met.  Then
\[
\sqrt{\kn} (\hat \theta_n^{\djb} - \theta) \dto \Nc(0, \theta^4 \sigma_{\djb}^2) 
\quad \text{ and } \quad 
\sqrt {\kn} (\hat \theta_n^{\slb} - \theta) \dto \Nc(0, \theta^4 \sigma_{\slb}^2),
\]
where 
\begin{align*}
\sigma^2_{\djb} 
&=
4  \int_0^1\frac{ \Exp[\zeta_{1}^{\scs(\sigma)} \zeta_{2}^{\scs(\sigma)}] }{(1+\sigma)^{3}}\, \diff \sigma  + 4 \theta^{-1} \int_0^1 \frac{ \Exp[\zeta_{1}^{\scs(\sigma)} \ind(\zeta_{2}^{\scs(\sigma)}=0)] }{(1+\sigma)^{3}} \, \diff \sigma - \theta^{-2}, \\
\sigma^2_{\slb} 
&=4  \int_0^1\frac{ \Exp[\zeta_{1}^{\scs(\sigma)} \zeta_{2}^{\scs(\sigma)}] }{(1+\sigma)^{3}}\, \diff \sigma  + 4 \theta^{-1} \int_0^1 \frac{ \Exp[\zeta_{1}^{\scs(\sigma)} \ind(\zeta_{2}^{\scs(\sigma)}=0)] }{(1+\sigma)^{3}} \, \diff \sigma - \frac{4-4\log(2)}{\theta^2},
\end{align*}
with $(\zeta_{1}^{\scs(\sigma)}, \zeta_{2}^{\scs(\sigma)})\sim \pi_2^{\scs(\sigma)}$.  
In particular,
$\sigma^2_{\djb} = \sigma^2_{\slb} + \{3-4\log(2)\} / \theta^2 \approx \sigma^2_{\slb} + 0.2274/\theta^2$.
\end{theorem}

It is interesting to note that the asymptotic variance of the disjoint blocks estimator is substantially more complicated than if one would naively treat the $\hat Z_{ni}$ (or the $\hat Y_{ni}$) as an iid sample from the exponential distribution with parameter~$\theta$ (as is done in \citealp{Nor15}; the variance would then simply be $\theta^{2}$). A heuristic explanation can be found in Remark~\ref{rem:proofidea} below. A formal proof is given at the end of this section, with several auxiliary lemmas postponed to Section~\ref{sec:proofs} (for the disjoint blocks estimator) and to Section~\ref{sec:proofs2} (for the sliding blocks estimator). Explicit calculations are possible for instance for a max-autoregressive process, see Section~\ref{subsec:armax}, or for the iid case.

\begin{example} 
If the time series is serially independent,  a simple calculation shows that $\pi(i)=\ind(i=1)$ and $\pi_2^{\scs(\sigma)}(i,j)= (1-\sigma) \ind(i=1,j=0) + \sigma \ind(i=1,j=1)$. This implies
\[
\theta=1, \quad \Exp[\zeta_{1}^{\scs(\sigma)} \zeta_{2}^{\scs(\sigma)}] = \sigma , \quad \Exp[\zeta_{1}^{\scs(\sigma)} \ind(\zeta_{2}^{\scs(\sigma)}=0)]= 1- \sigma
\]
and therefore $\theta^4\sigma^2_{\djb}=1/2$ and $\theta^4 \sigma^2_{\slb}\approx 0.2726$. It is worthwhile to mention that these values are smaller than the variances of any of the disjoint and sliding blocks estimators considered in \cite{RobSegFer09}, respectively. 
Moreover, note that asymptotic variance of the oracle $\tilde \theta_n^{\djb}$ is equal to $\theta^2=1$, which is twice as large as when the marginal cdf is estimated.
Finally, it can be seen that the same formulas are valid whenever $\theta =1$: the fact that $\theta^{-1} \geq \sum_{i=1}^\infty i\pi(i)$ implies that $\pi(1) =1$. By \eqref{eq:xi2} in the supplementary material, we then obtain $\pi_2^{\scs(\sigma)}= (1-\sigma) \ind(i=1,j=0) + \sigma \ind(i=1,j=1)$. 
\end{example}

\begin{remark}[Main idea for the proof] \label{rem:proofidea}
Define $Z_{ni}=\bn(1-N_{ni})$ and
\begin{align} \label{eq:tntnhat}
 \hat T_n^{\djb} &= \textstyle  \frac{1}{\kn} \sum_{i=1}^{\kn} \hat Z_{ni}, \qquad&   T_n^{\djb} &= \textstyle  \frac{1}{\kn} \sum_{i=1}^{\kn} Z_{ni}, \\
\hat T_n^{\slb} &= \textstyle \frac{1}{n-\bn+1} \sum_{t=1}^{n-\bn+1}  \hat Z_{nt}^{\slb}, \qquad &  T_n^{\slb} &=\textstyle  \frac{1}{n-\bn+1} \sum_{t=1}^{n-\bn+1}  Z_{nt}^{\slb}. \label{eq:tntnhatsl}
\end{align}
In the following, we only consider the disjoint blocks estimator, the argumentation for the sliding blocks estimator is similar. For the ease of notation, we will skip the upper index and just write $\hat T_n$ instead of $\hat T_n^{\djb}$, etc. 
Asympotic normality of $\hat \theta_n$  may be deduced from the delta method and weak convergence of $\sqrt k(\hat T_n - \theta^{-1})$. 
The roadmap to handle  the latter is as follows: decompose
\begin{align} \label{eq:central}
\sqrt{\kn} (\hat T_n - \theta^{-1})  =  \sqrt{\kn} (\hat T_n - T_n) + \sqrt{\kn} (T_n - \theta^{-1}).
\end{align}
Using a big-block/small-block type argument,  the asymptotics of the second summand on the right-hand side can be deduced from a central limit theorem for rowwise independent triangular arrays. Depending on the choice of the block sizes, an asymptotic  bias term may appear, which we control by Condition~\ref{cond:cond}\eqref{item:bias}. The first summand is more involved, and also contributes to the limiting distribution: 
first, 
for $x \ge 0$, let
\begin{align} \label{eq:tailp}
\tailp(x) = \frac{1}{\sqrt {\kn} }  \sum_{s=1}^n \{  \ind (U_s > 1-x/\bn) - x/\bn \} 
\end{align}
denote the tail empirical process of $X_1, \dots, X_n$ and let
\begin{align} \label{eq:hkn}
\hat H_{\kn}(x) = \frac{1}{\kn} \sum_{i=1}^{\kn} \ind (Z_{ni} \le x) 
\end{align}
be the empirical distribution function of $Z_{n1}, \dots, Z_{n\kn}$.
Then
\begin{align}\label{eq:long2}
\sqrt{\kn} (\hat T_n - T_n) 
&= 
\frac{\bn}{\sqrt{\kn}} \sum_{i=1}^{\kn} (N_{ni} - \hat N_{ni})  \\
&=
\frac{\bn}{n \sqrt{\kn}}  \sum_{i=1}^{\kn}   \sum_{s=1}^n \{N_{ni} -  \ind(U_s \le N_{ni}) \}  \nonumber \\
&=
\frac{1}{\kn^{3/2}} \sum_{i=1}^{\kn}  \sum_{s=1}^n \{ \ind(U_s > 1- Z_{ni} / \bn) - Z_{ni} / \bn \}  \nonumber \\
&=
\frac1{\kn} \sum_{i=1}^{\kn}  \tailp(Z_{ni}) 
=
\int_0^{\max_{i=1}^{\kn} Z_{ni}} \tailp(x) \diff \hat H_{\kn}(x).  \nonumber
\end{align}
Since $Z_{ni}$ is approximately exponentially distributed with parameter $\theta$, one may expect that  $\hat H_{\kn}(x)$ converges to $H(x) = 1- \exp(-\theta x)$  in probability, for $n \to \infty$ and for any $x\ge 0$. Moreover, on an appropriate domain, $\tailp \dto \tailplim$ for some Gaussian process $\tailplim$ \citep{Dre00, Dre02, Roo09,Rob09,DreRoo10}, whence a candidate limit for the expression on the left-hand side of the previous display is given by
$
\int_0^\infty \tailplim(x) \theta  e^{-\theta x}\diff x.
$
The latter distribution is normal, and joint convergence of both terms on the right-hand side of \eqref{eq:central} will finally allow for the derivation of the asymptotic distribution of $\hat \theta_n$.
These heuristic arguments have to be made rigorous. 
\end{remark}

\begin{remark}[Disjoint blocks: alternative proof]
As pointed out by a referee, the asymptotic distribution of the disjoint blocks estimator may alternatively be derived by completely relying on results in \cite{Rob09}. The idea  is as follows. First, recall \eqref{eq:robeq},
where $\hat p _n^{\scs (\tau)}(0)= \kn^{-1} \sum_{j =1}^{\kn} \ind(\hat Z_{ni} > \tau)$ for $\tau>0$. 
Since $\hat F_n(x) < p$ if and only if $x<\hat F_n^\leftarrow(p)$, this expression coincides with the definition of $\hat p _n^{\scs (\tau)}(0)$ used in \cite{Rob09}, middle of page 275, up to a `$>$'-sign replaced by a `$\ge$'-sign in his definition of $\hat N_{r_n,j}^{\scs(\tau)}$.
Hence, by Theo\-rem~4.2 in that reference, assuming the latter replacement to be asymptotically negligible, we have
$
\sqrt{\kn} \big\{ \hat p_n^{\scs (\cdot)}(0)- p^{\scs (\cdot)}(0)\big \} \dto \hat e_0(\cdot)
$
in some appropriate metric space, where $\hat e_0$ is a Gaussian process. The continuous mapping theorem implies
\[
\textstyle \sqrt{\kn }\big \{ \int_0^\cdot p_n^{(\tau)}(0) \mathrm{d}\tau - \int_0^\cdot p^{(\tau)}(0) \mathrm{d}\tau\big \}  \dto \int_0^\cdot \hat e_0(\tau) \mathrm{d}\tau,
\]
again on some appropriate metric space. Some tedious, but straightforward calculations show that the random variable $\lim_{t \to \infty}\int_0^t \hat e_0(\tau) \mathrm d{\tau}$ has the same law as the limit that we obtained with the approach stated in Remark \ref{rem:proofidea}. 
We do not give any further details on this approach as it is limited to the case of disjoint blocks.
\end{remark}

\begin{remark}[On continuity of $F$] \label{rem:cont} In the introduction, we assumed for simplicity that $F$ is continuous. 
Some thoughts reveal that the main limit relations motivating the estimators, that is \eqref{eq:exp2} and \eqref{eq:exp3}, continue to hold under the weaker assumption that
\begin{align*}
\lim_{x\to x_F} \frac{1-F(x-)}{1-F(x)} = 1,
\end{align*}
where $x_F$ denotes the right endpoint of the support of $F$. By Theorem~1.7.13 in \cite{LeaLinRoo83}, this condition is also necessary for the extremal index to exist.
However, the proofs of our theoretical results do not easily generalize to this weaker assumption, the reason being that we heavily rely on the asymptotic equivalence of $e_n$ in  \eqref{eq:tailp} and $\bar e_n$ on page~281 in \cite{Rob09} (to apply his Theorem 4.1 on weak convergence of $\bar e_n$) and on centredness of $e_n$ on $[0,\eps\bn]$ (to show negligibility of certain terms in Lemma \ref{lem:boundsupp} and \ref{lem:boundsupp2}). A further discussion is beyond the scope of this~paper.
\end{remark}

\begin{proof}[Proof of Theorem~\ref{theo:main} (Disjoint blocks)]
Write $\hat T_n=\hat T_n^{\djb}$ and $T_n=T_n^{\djb}$. Recall the definitions of $\tailp$ and $\hat H_{\kn}$ in \eqref{eq:tailp} and \eqref{eq:hkn}, respectively. For $\ell \in \N$, let
\begin{multline*}
D_n 
= 
\int_0^{\hat m}  
\tailp(x) \diff \hat H_{\kn}(x),  \quad
D_{n,\ell} 
= 
 \int_0^{ \ell} 
\tailp(x) \diff \hat H_{\kn}(x),  \\
D_\ell 
=
 \int_0^{\ell} \tailplim(x) \theta e^{-\theta x} \diff x,
\end{multline*}
where $\hat m= \max Z_{ni}$.
Also, let $G_n = \sqrt{\kn} (T_n -  \Exp T_n)$ and let $G$ be defined as in Lemma~\ref{lem:fidirob}. Suppose we have shown that \smallskip
\begin{compactenum}
\item[(i)] 
For all $\delta>0$: $\lim_{\ell \to \infty} \limsup_{n\to\infty} \Prob( |D_{n,\ell} - D_n | > \delta) = 0$;  \smallskip
\item[(ii)]
For all $\ell \in \N$: $D_{n,\ell} + G_n \dto D_{\ell} + G$ as $n \to \infty$; \smallskip
\item[(iii)] $D_{\ell}  + G \dto D + G\sim \Nc(0, \sigma^2_{\djb})$ as $\ell \to \infty$. \smallskip
\end{compactenum}
It then follows from \eqref{eq:long2} and Wichura's theorem (\citealp{Bil79}, Theorem 25.5) that
\[
\sqrt n ( \hat T_n - \Exp T_n) = D_n + G_n \dto \Nc(0, \sigma^2_{\djb}), \qquad n \to \infty.
\]
By Condition~\ref{cond:cond}\eqref{item:bias}, we obtain that $\sqrt{\kn} (\hat T_n - \theta^{-1}) \dto \Nc(0,\sigma^2_{\djb}).$ The theorem then follows from the delta-method.

The assertion in (i) is proved in Lemma~\ref{lem:boundsupp}.
The assertion in (ii) is proved in Lemma~\ref{lem:weakdnlgn} (it is a consequence of the continuous mapping theorem and Lemmas~\ref{lem:appleb} and~\ref{lem:tailpweak}),
The assertion in (iii) follows from the fact that $D_\ell + G$ is normally distributed with variance $\sigma_\ell^2$ as specified in Lemma~\ref{lem:weakdnlgn}, and the fact that by Lemma \ref{lem:convvar} $\sigma_\ell^2 \to \sigma^2_{\djb}$ for $\ell \to \infty$. 
\end{proof}

\begin{proof}[Proof of Theorem~\ref{theo:main} (Sliding blocks)]
Let $\hat H_{\kn}^{\slb}$ denote the empirical distribution function of the $Z_{nt}^{\slb}$,  $\hat H_{\kn}^{\slb}(x) = \frac{1}{n-\bn+1} \sum_{t=1}^{n-\bn+1} \ind( Z_{nt}^{\slb} \leq x)$, and let 
\begin{multline*}
D_n^{\slb} 
= 
\int_0^{\hat m^{\slb}}  
\tailp(x) \diff \hat H_{\kn}^{\slb}(x),  \quad
D_{n,\ell}^{\slb}
= 
 \int_0^{ \ell} 
\tailp(x) \diff \hat H_{\kn}^{\slb}(x),  \\
D_\ell^{\slb} 
=
 \int_0^{\ell} \tailplim(x) \theta e^{-\theta x} \diff x,
\end{multline*}
where $\hat m^{\slb}=\max_t Z_{nt}^{\slb}$.
With this notation the proof  follows along the same lines as  for the disjoint blocks, with Lemma \ref{lem:boundsupp}, \ref{lem:appleb} and \ref{lem:fidirob} replaced by Lemma \ref{lem:boundsupp2}, \ref{lem:appleb2} and \ref{lem:fidirob2}, respectively.
\end{proof}

\section{Variance estimation} \label{sec:boot}

For statistical inference on $\theta$,  estimators for the asymptotic variance formulas in Theorem~\ref{theo:main} are needed. Unfortunately, the formulas themselves are too complicated to base such estimators on a simple plug-in principle. Rather than that, we rely on an asymptotic expansion of the disjoint blocks estimator resulting from a careful inspection of the proofs. Note that, since $\sigma^2_{\djb} = \sigma^2_{\slb} - \{3-4\log(2)\} / \theta^2$, an estimator for the variance of the disjoint blocks estimator can immediately be transferred into one for the sliding blocks estimator.   This is particularly useful since a straightforward extension of our proposed estimator for $\sigma^2_{\djb}$  to the sliding blocks estimator is not possible and would require the choice of an additional tuning parameter.

The proof of Theorem~\ref{theo:main}, in particular the central decomposition in \eqref{eq:central} and the calculations in \eqref{eq:long2}, allows to write
$\TT_n^{\djb} = \sqrt{\kn} (\hat T_n^{\djb} - \theta^{-1})$ as
\begin{align*}
\frac{1}{\sqrt {\kn}} \sum_{j=1}^{\kn} ( Z_{nj} - \theta^{-1} ) + \int_0^\infty e_n(x)\, \diff H(x)   + o_\Prob(1) 
=
\frac{1}{\sqrt {\kn}} \sum_{j=1}^{\kn} B_{nj} + o_\Prob(1),
\end{align*}
where 
\[
\textstyle B_{nj }=Z_{nj} - \theta^{-1}  +  \int_0^\infty  \sum_{s\in I_j} \big\{  \ind(U_s > 1- \tfrac{x}{\bn}) - \tfrac{x}{\bn} \big\} \diff H(x)
\]
and where $I_j=\{(j-1)\kn + 1, \dots, j\kn\}$ denotes the $j$th block of indices.   
The proof of Theorem~\ref{theo:main} shows that $B_{n1}, \dots, B_{n\kn}$ are asymptotically independent (big block/small block heuristics) and centred, and that their empirical mean multiplied by $\sqrt{\kn}$ converges to a centred normal distribution with variance  $\sigma_{\scs \djb}^{\scs2}$. Hence, their second empirical moment should be a consistent estimator for $\sigma_{\scs \djb}^{\scs2}$. As the sample $B_{n1}, \dots, B_{n\kn}$ depends on unknown quantities, we must replace these objects by empirical counterparts, leading us to define
\begin{align*}
\hat B_{nj} 
&= 
\textstyle\hat Z_{nj} - \hat T_n + \sum_{s\in I_j} \frac{1}{\kn} \sum_{i=1}^{\kn} \{  \ind(\hat U_s > 1- \tfrac{\hat Z_{ni}}{\bn}) - \tfrac{\hat Z_{ni}}{\bn} \}  \\
&= 
\textstyle\hat Z_{nj}  + \sum_{s\in I_j} \frac{1}{\kn} \sum_{i=1}^{\kn}  \ind(\hat U_s > 1- \tfrac{\hat Z_{ni}}{\bn}) - 2\cdot \hat T_n^{\djb},
\end{align*} 
where $\hat U_s= \hat F_n(X_s)$.
The following proposition shows that 
\[
\hat \sigma_{\djb}^2=  \frac{1}{\kn} \sum_{j=1}^{\kn} \hat B_{nj}^2, \qquad \hat \sigma_{\slb}^2=\hat \sigma_{\djb}^2- \{ 3-4\log(2)\} (\hat \theta_{n}^{\slb})^{-2}.
\]
are in fact consistent estimators for $\sigma_{\djb}^2$ and $\sigma_{\slb}^2$, respectively, provided that moments of order slightly larger than $4$ exist. To simplify the proofs, we assume beta-mixing of the times series, since it allows for stronger coupling results than alpha-mixing. We also impose a further growth condition on the block size, which allows for a further simplification within the proof (which is given in in the supplementary material).

\begin{proposition}[Consistency of variance estimators] \label{prop:boot}
Additionally to the assumptions imposed in Condition~\ref{cond:cond} suppose that $\bn=o(\kn^2)$ for $n\to\infty$ (hence, $\bn^{\scs 1/2} \ll \kn \ll \bn^2$), that Condition~\ref{cond:cond}\eqref{item:rate2} is met with the alpha-mixing coefficient $\alpha_{c_2}(\ell)$ replaced by the beta-mixing coefficient $\beta_1(\ell)$ (see the proof for a precise definition) and that Condition~\ref{cond:cond}\eqref{item:momn} and \eqref{item:moment} are met with $\delta>2$.
Then, as $n\to\infty,$
\[
\hat \sigma_{\djb}^2\pto \sigma_{\djb}^2 \quad \text{ and } \quad
\hat \sigma_{\slb}^2\pto \sigma_{\slb}^2. 
\]
\end{proposition}

\section{Bias reduction} \label{sec:bias}

While the previous sections were concerned with the $O(1/\sqrt{\kn})$-asymptotics, we will now have a heuristic look at the $O(1/\kn)$-asymptotics, in particular in terms of expectations. As a result, we will obtain a bias reduction scheme.
Let $(\hat T_n, T_n, \sigma^2) \in \{ (\hat T_n^{\djb},  T_n^{\djb}, \sigma_{\djb}^2), (\hat T_n^{\slb},T_n^{\slb}, \sigma_{\slb}^2)\}$ denote any of the quantities defined in \eqref{eq:tntnhat}, \eqref{eq:tntnhatsl} or Theorem~\ref{theo:main}. 
A Taylor expansion allows to  write 
\begin{align*}
\hat T_n^{-1} - \theta &= - \theta^2 (\hat T_n - T_n) -\theta^2 (T_n - \theta^{-1}) + \theta^3 (\hat T_n - \theta^{-1})^2 + O_\Prob(\kn^{-3/2}) \\
& \equiv a_{n1}+ a_{n2} + a_{n3} + O_\Prob(\kn^{-3/2}).
\end{align*}
Let $\mu_{nj}=\Exp[a_{nj}]$. The second component $\mu_{n2}$ is inherent to the time series $(X_s)_{s\in\N}$ itself. In many examples, it can be seen to be of the order $O(\bn^{-1})$, see for instance Section~\ref{sec:examples} or similar calculations made in \cite[Section 6]{RobSegFer09}.  Since $\sqrt{\kn}(\hat T_n - \theta^{-1}) \dto \Nc(0,\sigma^2)$, it seems plausible that the third component $\mu_{n3}$ satisfies $\mu_{n3} = \kn^{-1} \theta^3 \sigma^2 + o(\kn^{-1})$, though we will not give a precise proof. 
Finally, consider the first component $\mu_{n1}$, which  is essentially due to the use of the empirical distribution function in the definition of the estimator. 
The following lemma gives a first-order asymptotic expansion, which turns out to be the same for the disjoint and sliding blocks estimator.
\begin{lemma}\label{lem:bias}
Additionally to the conditions of Theorem \ref{theo:main} suppose that Condition~\ref{cond:cond}\eqref{item:rate2} is met with $c_2=1$. Then
\[
\lim_{n \to \infty} \kn \Exp [\hat T_n - T_n]  
=
- \frac{1}{ \theta}.
\]
where $(\hat T_n, T_n) \in \{ (\hat T_n^{\djb},  T_n^{\djb}), (\hat T_n^{\slb},T_n^{\slb})\}$ as defined in \eqref{eq:tntnhat} and  \eqref{eq:tntnhatsl}.
\end{lemma} 

The proof is given in Section~\ref{sec:addproofs}. As a consequence, we have $\mu_{n1}=  \kn^{-1}\theta+ o(\kn^{-1})$.
Now, plugging-in $\hat \theta_n$ and $\hat \sigma_n^2$ as a consistent estimator for $\theta$ and $\sigma^2$, we can estimate  $\mu_{n1}$  and $\mu_{n3}$ by 
$\hat {\mu}_{n1} = \kn^{-1} \hat \theta_n$ and $\hat \mu_{n3} = \kn^{-1} \hat \theta_n^3 \hat \sigma_n^2$, respectively.
Subtracting these expression from $\hat \theta_n$, we obtain the bias-reduced estimator
\[
\hat \theta_{n,bc} =\hat \theta_{n} - \kn^{-1} \hat \theta_n- \kn^{-1} \hat \theta_n^3 \hat \sigma_n^2.
\]
The $O(1/\sqrt{\kn})$-asymptotics will not be affected, but $\hat \theta_{n,bc}$ shows a better finite-sample performance and is therefore used in Section~\ref{sec:sim}.

Note that if we are additionally willing to assume that $\kn \mu_{n2}=\kn\Exp [T_n - \theta^{-1}] = \kn \Exp[Z_{1:\bn} - \theta^{-1}] = o(1)$ as $n\to\infty$ (cf.\ Condition~\ref{cond:cond}\eqref{item:bias}), we obtain that 
$\mu_{n1}$ and $\mu_{n3}$ are in fact the dominating bias-components. In common models, the assumption $\kn\Exp [T_n - \theta^{-1}] = o(1)$ is satisfied as soon as $\kn / \bn= o(1)$ (see Section~\ref{sec:examples}). In comparison to the assumption $\kn / \bn^2= o(1)$ in Condition~\ref{cond:cond}\eqref{item:rate2} larger block sizes are required. Similar assumptions have also been made for the bias reductions in \cite{RobSegFer09}.

Finally, note that the bias reduction based on $\hat \mu_{n1}$ can actually be alternatively motivated by the fact that $\hat \theta_{n}^{\djb} - \kn^{-1} \hat \theta_n^{\djb}$ is equal to $(\kn^{-1} \sum_{i=1}^{\kn} \tilde Z_{ni})^{-1}$, where $\tilde Z_{ni} = \bn(1-\hat F_{n,-i}(M_{ni}))$ with $\hat F_{n,-i}$ being the empirical cdf of $(X_s)_{s\notin I_i}$. The idea of using $\hat F_{n,-i}$ rather than $\hat F_n$ has been used in \cite{Nor15} as a bias reduction scheme.

\section{Examples} \label{sec:examples}

Two examples are worked out in this section. For the max-autoregressive processes, considered in Section~\ref{subsec:armax}, explicit calculations for the asymptotic variance formulas in Theorem~\ref{theo:main} are possible. These allow for a theoretical comparison with the blocks estimators from \cite{Rob09} and \cite{RobSegFer09}. All assumptions imposed in Condition~\ref{cond:cond} are shown to hold. In Section~\ref{subsec:stochdiff}, we consider solutions of stochastic difference equations such as ARCH-processes. Complementing results from \cite{Rob09} we show that Condition~\ref{cond:cond}\eqref{item:varbound} is satisfied.

\subsection{Max-autoregressive processes} \label{subsec:armax}
Consider the max-autoregressive process of order one, ARMAX(1) in short, defined by the recursion
\[
X_s = \max\{ \alpha X_{s-1}, (1-\alpha) Z_s \}, \qquad s \in \Z,
\]
where $\alpha \in [0,1)$ and where $(Z_s)_s$ denotes an i.i.d.\ sequence of standard Fr\'echet random variables. A stationary solution of this recursion is given by
\[
X_s = \max_{j \ge 0} (1-\alpha) \alpha^j Z_{s-j},
\]
which shows that the stationary distribution is standard Fr\'echet as well.
The sequence has extremal index $\theta = 1- \alpha$ and its cluster size distribution is geometric, i.e., $\pi(j) = \alpha^{j-1}(1-\alpha)$ for $j \ge 1$ (see, e.g., Chapter 10 in \citealp{BeiGoeSegTeu04}). Moreover, it follows from Proposition 5.3.7 in \cite{Hsi84} and some simple calculations that
\begin{align*}
\pi_2^{(\sigma)}(j_1, j_2) &= \alpha^{j_2-1} \Big\{ (\sigma-\alpha^{j_1-j_2+1} ) \ind( \alpha^{j_1-j_2+1} < \sigma \le \alpha^{j_1-j_2})  \\
&\hspace{3cm} + (\alpha^{j_1-j_2}-\alpha \sigma)  \ind( \alpha^{j_1-j_2} < \sigma \le \alpha^{j_1-j_2-1}) \Big\} \\
&=
\alpha^{j_2-1} \Big\{ (\sigma-\alpha^{z+1} ) \ind( j_1 = j_2 + z ) \\
&\hspace{3cm}+ (\alpha^{z+1}-\alpha \sigma)  \ind(  j_1 = j_2 + z + 1) \Big\}
\end{align*}
for $j_1 \ge j_2 >0$, where $z= \ip{\log \sigma / \log \alpha} \in \N_0$. The formula in Proposition 5.3.7 in \cite{Hsi84} is wrong for $j_2=0$, but can be corrected to
\begin{align*}
\pi_2^{(\sigma)}(j_1, 0) 
= 
(1-\alpha) \alpha^{j_1-1} \ind(  j_1 \le z) + (\alpha^z - \sigma) \ind( j_1=1+z)
\end{align*}
for $j_1 \ge 1$.
Based on these formulas, some straightforward calculations yield
\[
\Exp[\zeta_1^{\scs(\sigma)} \zeta_2^{\scs(\sigma}] = \frac{\alpha^{z+1}+\sigma\{1+z(1-\alpha)\}}{(1-\alpha)^2}
\]
and 
\[
\Exp[\zeta_1^{\scs(\sigma)}  \ind(\zeta_2^{\scs(\sigma)}=0)] = \frac{1-\alpha^{z+1}}{1-\alpha} -\sigma(z+1).
\]
Note that, for $\alpha \to 0$, we obtain $ \Exp[\zeta_{1}^{\scs(\sigma)} \zeta_{2}^{\scs(\sigma)}] \to \sigma$ and $\Exp[\zeta_{1}^{\scs(\sigma)} \ind(\zeta_{2}^{\scs(\sigma)}=0)] \to 1- \sigma$, which corresponds to the iid scenario. The latter two displays imply 
\[
\Exp[\zeta_1^{\scs(\sigma)} \zeta_2^{\scs(\sigma}] + \theta^{-1}\Exp[\zeta_1^{\scs(\sigma)}  \ind(\zeta_2^{\scs(\sigma)}=0)]  = \frac{1+\alpha \sigma}{(1-\alpha)^2}
\]
and hence
\[
\sigma_{\djb}^2 = \frac{1+\alpha}{2(1-\alpha)^2}, \qquad  \sigma_{\slb}^2 = \frac{8 \log 2 - 5+\alpha}{2(1-\alpha)^2}.
\]
Since $\theta=1-\alpha$, the asymptotic variances of  $\sqrt{\kn}(\hat \theta_n/\theta-1)$ simply reduce to the affine linear functions $(1+\alpha)/2$ and $(8 \log 2 - 5+\alpha)/2$ for the disjoint and the sliding blocks estimator, respectively.  These functions can be compared with the asymptotic variance formulas in \cite[Formula 5.1]{RobSegFer09} and in \cite[Page 285, variance of $\hat \theta_{\scs 1,n}^{\scs (\tau)}$]{Rob09}. Note that the variance of $\hat \theta_{\scs 1,n}^{\scs (\tau)}$ in \cite{Rob09} is exactly the same as the one of the disjoint blocks estimator in \cite{RobSegFer09}. The asymptotic variance formulas depend on an additional parameter $\tau>0$ to be chosen by the statistician. Assuming we would have access to the optimal value (which can be calculated numerically, but must be estimated in practice), we obtain the variance curves depicted in Figure~\ref{fig:asyvar}. We observe that, for the ARMAX-model, the PML-estimators analyzed in this paper have a smaller asymptotic variance than the (theoretically optimal) estimators in \cite{RobSegFer09} and \cite{Rob09}.

\begin{figure}
\vspace{-.7cm}
\begin{center}
\includegraphics[width=0.49\textwidth]{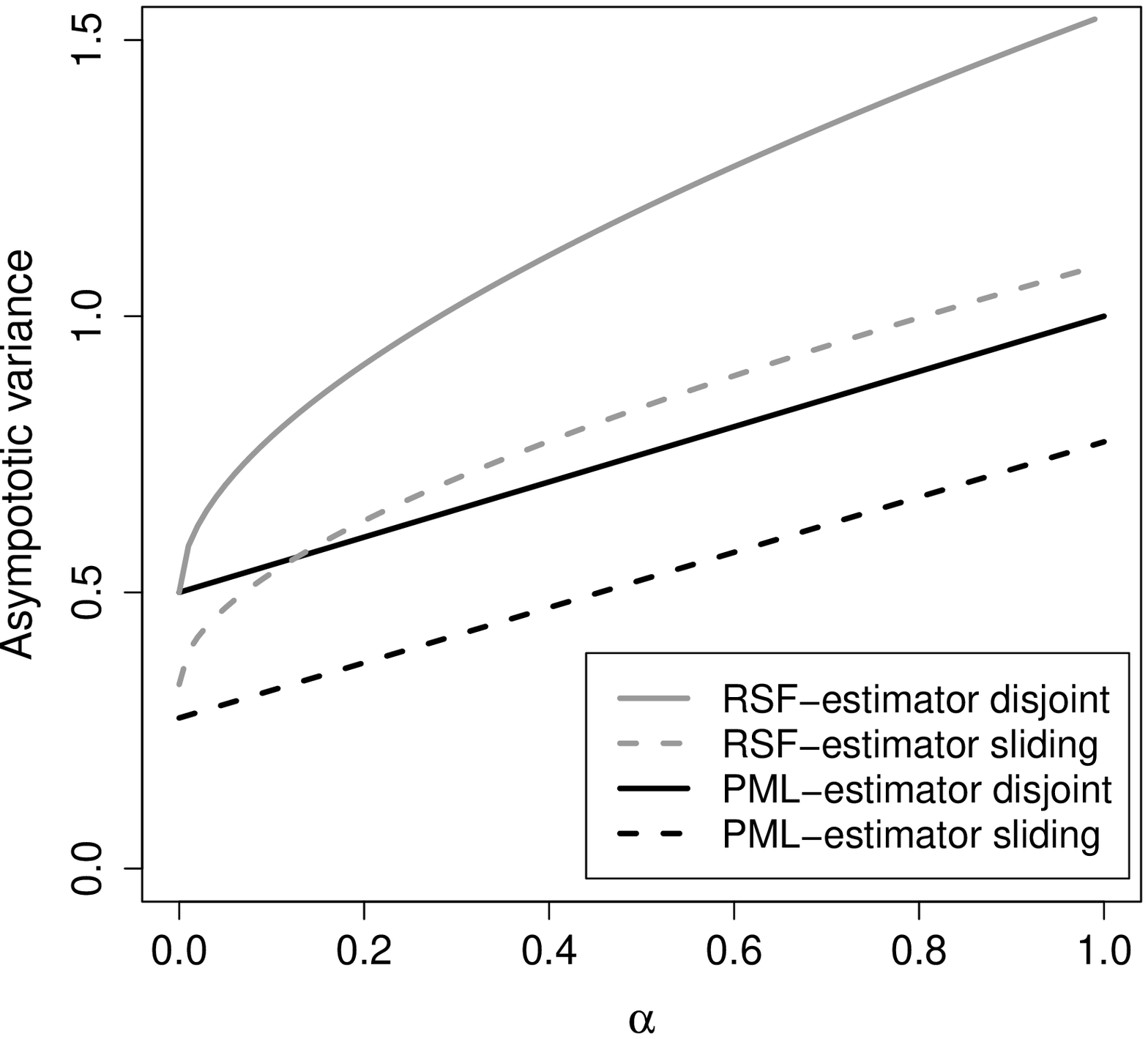}
\end{center}
\vspace{-.7cm}
\caption{\label{fig:asyvar}  Asymptotic variances of $\sqrt{\kn}(\hat \theta_n/\theta - 1)$ within the ARMAX($\alpha)$-Model for the sliding and disjoint blocks estimators analyzed in this paper (PML) and in \cite{RobSegFer09} (RSF). }
\vspace{-.7cm}
\end{figure}

Regarding the additional assumptions in Condition~\ref{cond:cond}, some tedious calculations show that Condition~\ref{cond:cond}\eqref{item:momn} is satisfied for $\delta=1$. $(X_s)_{s\in\Z}$ can further be shown to be a geometrically ergodic Markov chain, see Formula (3.5) in \cite{Bra05}. As a consequence of Theorem 3.7 in that reference,   $(X_s)_{s\in\Z}$ is geometrically $\beta$-mixing, whence Condition~\ref{cond:cond}\eqref{item:rate2} is satisfied (and also the condition on beta-mixing imposed in Proposition~\ref{prop:boot}). It can be further be shown that, with $U_s = \exp(-1/X_s)$, we have
$
\Var \{ \sum_{s=1}^n \ind(U_s > 1-y) \} \le n y \{ 1+2\alpha/(1-\alpha) \}
$
for all $y\in(0,1)$, that is, Condition~\ref{cond:cond}\eqref{item:varbound} is met. 
Moreover, a simple calculation shows that $\Prob( \min_{i=1}^{2\kn} N_{ni}' \le c)  
\le 2\kn \Prob(N_{n1}' \le c) 
= O(\kn c^{(1-\alpha)\bn/2}) = o(1),$ provided that $\kn=o(\bn^2)$.  Hence, Condition~\ref{cond:cond}\eqref{item:blockdiv}. Based on an explicit calculation of the distribution of $Z_{n1}$, it can also be seen that Condition~\ref{cond:cond}\eqref{item:moment} is satisfied for any $\delta >0$, and that $\Exp [Z_{1:\bn}] - \theta^{-1} = O(\bn^{-1})$. The latter implies that Condition~\ref{cond:cond}\eqref{item:bias} is satisfied if $ \kn = o(\bn^2) $, see \eqref{eq:rate1}. It can easily be seen that \eqref{eq:positive} is met.

\subsection{Stochastic Difference Equations}
\label{subsec:stochdiff}

Consider the equation 
\begin{align} \label{eq:diffeq}
X_s = A_s X_{s-1} + B_s, \qquad s \in \N,
\end{align}
where $(A_s, B_s)_s$ are i.i.d.\ $[0,\infty)^2$-valued random vectors. If $A_s = \alpha_1 Z_s^2$ and $B_s=\alpha_0 Z_s^2$ for some $\alpha_0,\alpha_1 >0$  and some i.i.d.\ real-valued sequence $(Z_s)_s$, the above equation defines the popular (squared) ARCH(1)-time series model. For simplicity, we assume that the distribution of $(A_1,B_1)$ is absolutely continuous. 

The existence of a stationary solution of \eqref{eq:diffeq} as well as the tail behavior of the stationary distribution $F$ of $X_s$ has been studied in  \cite{Kes73}, Theorem 5. More precisely, consider the condition
\begin{enumerate}
\item[(S)]
There exists some $\kappa>0$ such that 
\[
\hspace{-.6cm}\Exp \log  A_1 < 0, \quad \Exp[A_1^\kappa] = 1, \quad \Exp[A_1^\kappa \max(\log A_1, 0)]< \infty, \quad \Exp[B_1^\kappa] \in (0,\infty).
\]
\end{enumerate}
Under this assumption, there exists a unique stationary solution of \eqref{eq:diffeq} and the cdf~$F$ of $X_s$ satisfies $1-F(x) \sim cx^{-\kappa}$ as $x\to \infty$ for some constant $c>0$. Moreover, $F$ is continuous (\citealp{Ver79}, Theorem~3.2) and, in particular, in  the max-domain of attraction of $G_{1/\kappa}$, the generalized extreme value distribution with extreme-value index~$1/\kappa$.  

Explicit calculations for the (two-level) cluster size distribution have been carried out in \cite[Example 4.2]{Per94}. Unfortunately, the formulas are complicated and do not allow for simple expressions of the asymptotic variances in Theorem~\ref{theo:main}. 

Slight adaptations of Assumptions \eqref{item:generalpoint}--\eqref{item:rate2} of Condition~\ref{cond:cond} have been checked in \cite[Example 3.1]{Rob09}. We complement those results by showing that also \eqref{item:varbound} is satisfied.  The result is inspired by Section 4 in \cite{Dre00} and is in fact a modification of Lemma 4.1 in that paper to the present needs. Its proof is given in Section~\ref{sec:addproofs} in the supplement material.

\begin{lemma} \label{lem:archbound}
Suppose that Condition~(S) is met and let $(X_s)_s$ denote a stationary solution of \eqref{eq:diffeq}. Then Condition~\ref{cond:cond}\eqref{item:varbound} is met. \end{lemma}

\section{Finite-sample performance}
\label{sec:sim}

A simulation study is performed to illustrate the finite-sample performance of the proposed estimators and methods.
Results are presented for four time series models:
\begin{itemize}
\item 
The \textbf{ARMAX-model} from Section~\ref{subsec:armax}: 
\[
X_s = \max\{ \alpha X_{s-1}, (1-\alpha) Z_s \}, \qquad s \in \Z,
\]
where $\alpha \in [0,1)$ and where $(Z_s)_s$ is an i.i.d.\ sequence of standard Fr\'echet random variables. We consider $\alpha=0,0.25,0.5,0.75$ resulting in $\theta=1,0.75,0.5,0.25$.
\item 
The \textbf{squared ARCH-model} from Section~\ref{subsec:stochdiff}:
\[
X_s = (2 \times 10^{-5} + \lambda X_{s-1}) Z_s^2,  \qquad s \in \Z,
\]
where $\lambda\in(0,1)$  and where $(Z_s)_s$ denotes an i.i.d.\ sequence of standard normal random variables. We consider $\lambda=0.1, 0.5, 0.9, 0.99$ which implies $\theta =0.997, 0.727, 0.460, 0.422$, respectively (Table 3.1 in \citealp{DehResRooVri89}).
\item 
The \textbf{ARCH-model}:
\[
X_s = (2 \times 10^{-5} + \lambda X_{s-1}^2) ^{1/2}Z_s,  \qquad s \in \Z,
\]
where $\lambda\in(0,1)$  and where $(Z_s)_s$ denotes an i.i.d.\ sequence of standard normal random variables. We consider $\lambda=0.1, 0.5, 0.7, 0.99$ which implies $\theta =0.999, 0.835, 0.721, 0.571$, respectively (Table 3.2 in \citealp{DehResRooVri89}).

\item 
The \textbf{Markovian Copula-model} \citep{DarNguOls92}:
\[
X_s =\Fin(U_s), \quad (U_s, U_{s-1}) \sim C_\vartheta, \qquad s \in \Z.
\]
Here, $\Fin$ is the left-continuous quantile function of some arbitrary continuous cdf $F$, 
$(U_s)_s$ is a stationary Markovian time series of order~1 and $C_\vartheta$ denotes the Survival Clayton Copula with parameter $\vartheta>0$. For this model, 
$
\theta = \Prob( \max_{t \ge 1} \prod_{s=1}^t A_s \le U),
$
where $U,A_1,A_2, \dots$ are independent, $U$ is standard uniform and $A_s$ has cdf $H_\vartheta(s) = 1- (1+s^{\vartheta})^{-(1+1/\vartheta)}$, $s \ge 0$, see \cite{Per94} or \cite{BeiGoeSegTeu04}, Section 10.4.2.
We consider choices $\vartheta=0.23, 0.41, 0.68, 1.06, 1.90$  such that (approximately) $\theta = 0.2, 0.4, 0.6, 0.8, 0.95$ and fix $F$ as the standard uniform cdf (the results are independent of this choice, as the estimators are rank-based). Algorithm 2 in \cite{RemPapSou12} allows to simulate from this model. 
\end{itemize}
 
Additional simulation results for the AR-model and the doubly stochastic process from \cite{SmiWei94} turned out to be quite similar to the ARMAX-model and are not presented for the sake of brevity.  In all scenarios under consideration, the sample size is fixed to $n=8,192=2^{13}$ and the block size $\bn$ for the blocks estimators is chosen from the set $2^{2}, 2^3, \dots, 2^{9}$.

\subsection{Comparison with other estimators for the extremal index}
We present results for six different estimators: the bias-reduced sliding blocks estimator $\hat \theta_n^{\Be}$, the sliding blocks estimator from \cite{Nor15} (i.e., $\hat \theta_n^{\No}$, but with $\hat F_n$ replaced by $\hat F_{n,-i}$ in the $i$th block), the bias-reduced sliding blocks estimator from \cite{RobSegFer09} (with a data-driven choice of the threshold as outlined in Section 7.1 of that paper), the integrated version of the blocks estimator from \cite{Rob09}, the intervals estimator from \cite{FerSeg03} and the ML-estimator from \cite{Suv07}. Results for other versions of these estimators (e.g., the disjoint blocks versions or the versions based on a fixed threshold) are not presented as their performance was dominated by the above versions in almost all scenarios under consideration.
The parameters $\sigma$ and $\phi$ for the Robert-estimator (last display on page 276 of \citealp{Rob09}) are chosen as $\sigma=0.7$ and $\phi=1.3$. The intervals estimator and the Süveges-estimator require the choice of a threshold $u$, which we choose as the $1-1/\bn$ empirical quantile of the observed data. All estimators  are constrained to the interval $[0,1]$, except for Table~\ref{tab:minmse} where we also report results for the unconstrained versions.

In Figure~\ref{fig:mse1} (ARCH), as well as in Section~\ref{sec:addsim} of the supplement material (ARMAX, squared ARCH and Markovian Copula), we depict the mean-squared error $\Exp[(\hat \theta-\theta)^2]$ as a function of the block size parameter $b$, estimated on the basis of $N=10,000$ simulation runs.
For most models and estimators, the MSE-curves are U-shaped, representing the usual bias-variance tradeoff in extreme value theory (an exception being the Süveges-estimator within the ARCH-model for $\theta=0.571$, a possible reason being its high bias due to fact that his central assumption $D^{(2)}$ is not satisfied in this model).  Explicit pictures of the squared bias and variance  can be found in Section~\ref{sec:addsim} of the supplement. For the blocks estimators considered in this paper, the bias is decreasing in $b$ (the asymptotics for the exponential distribution kick in), while the variance is increasing (the convergence rate of the estimators being $1/\sqrt{k_n}$).  In terms of the bias, $\hat \theta_n^{\No}$ is clearly superior to $\hat \theta_n^{\Be}$ for small block sizes. 

The minimal values of the curves in Figure~\ref{fig:mse1} are of particular interest, and are summarized in Table~\ref{tab:minmse}. We observe that the sliding blocks estimators $\hat \theta_{n}^{\Be}$ and $\hat\theta_{n}^{\No}$ outperform the other two blocks estimators in most scenarios. For the ARMAX-model, this is in agreement with the theoretical findings presented in Figure~\ref{fig:asyvar}. Comparing $\hat\theta_{n}^{\Be}$ and $\hat\theta_{n}^{\No}$, we see that $\hat\theta_{n}^{\No}$ seems to be preferable in most scenarios.
In general, there is no clear best estimator in terms of the MSE: $\hat\theta_n^{\No}$ wins six times, the S\"uveges-estimator six times, $\hat\theta_n^{\Be}$ four times, and the intervals estimator is best in one scenario.

\begin{figure}[p!]
\vspace{-.3cm}
\begin{center}
\includegraphics[width=0.46\textwidth]{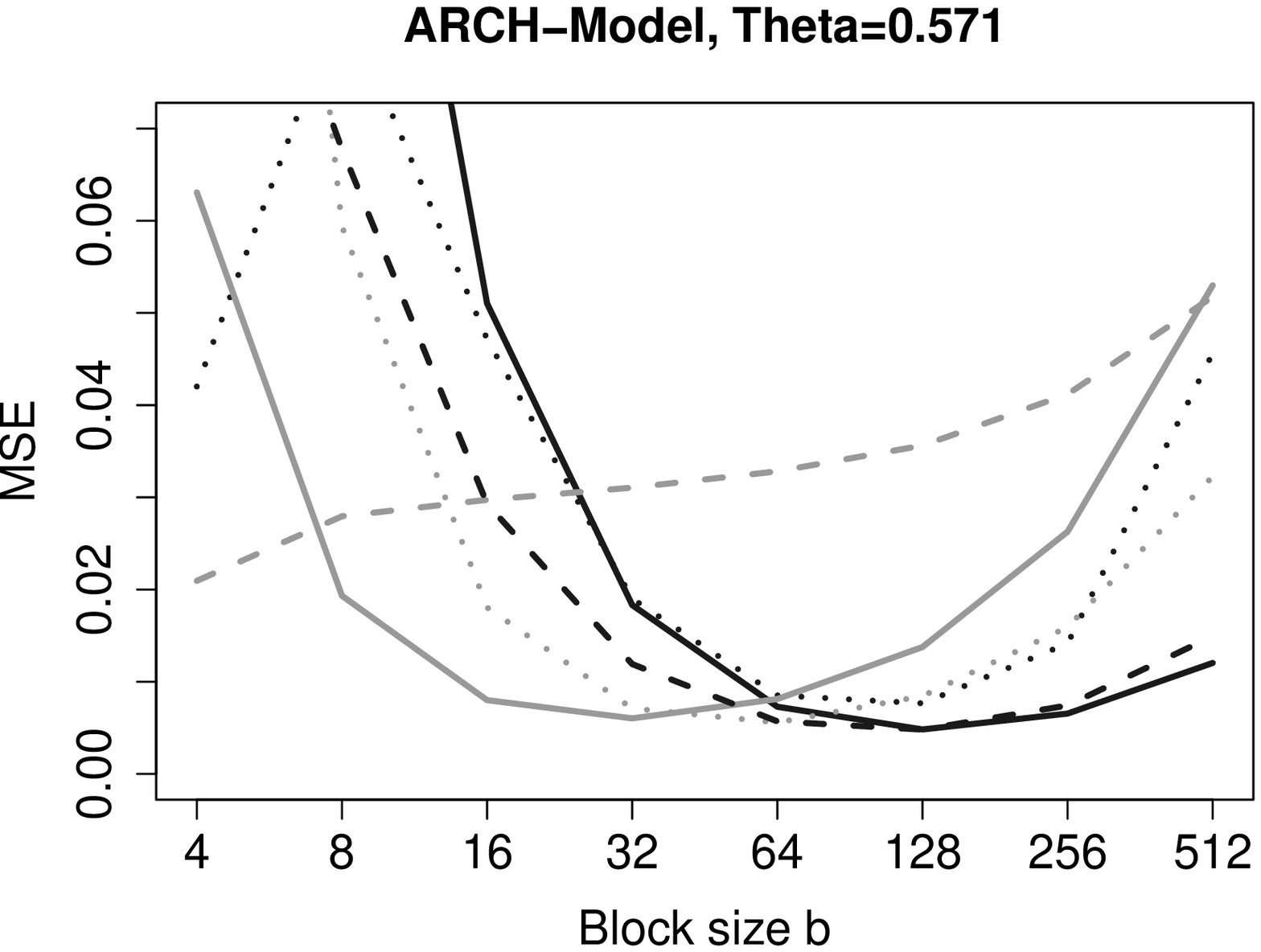}
\hspace{-.3cm}
\includegraphics[width=0.46\textwidth]{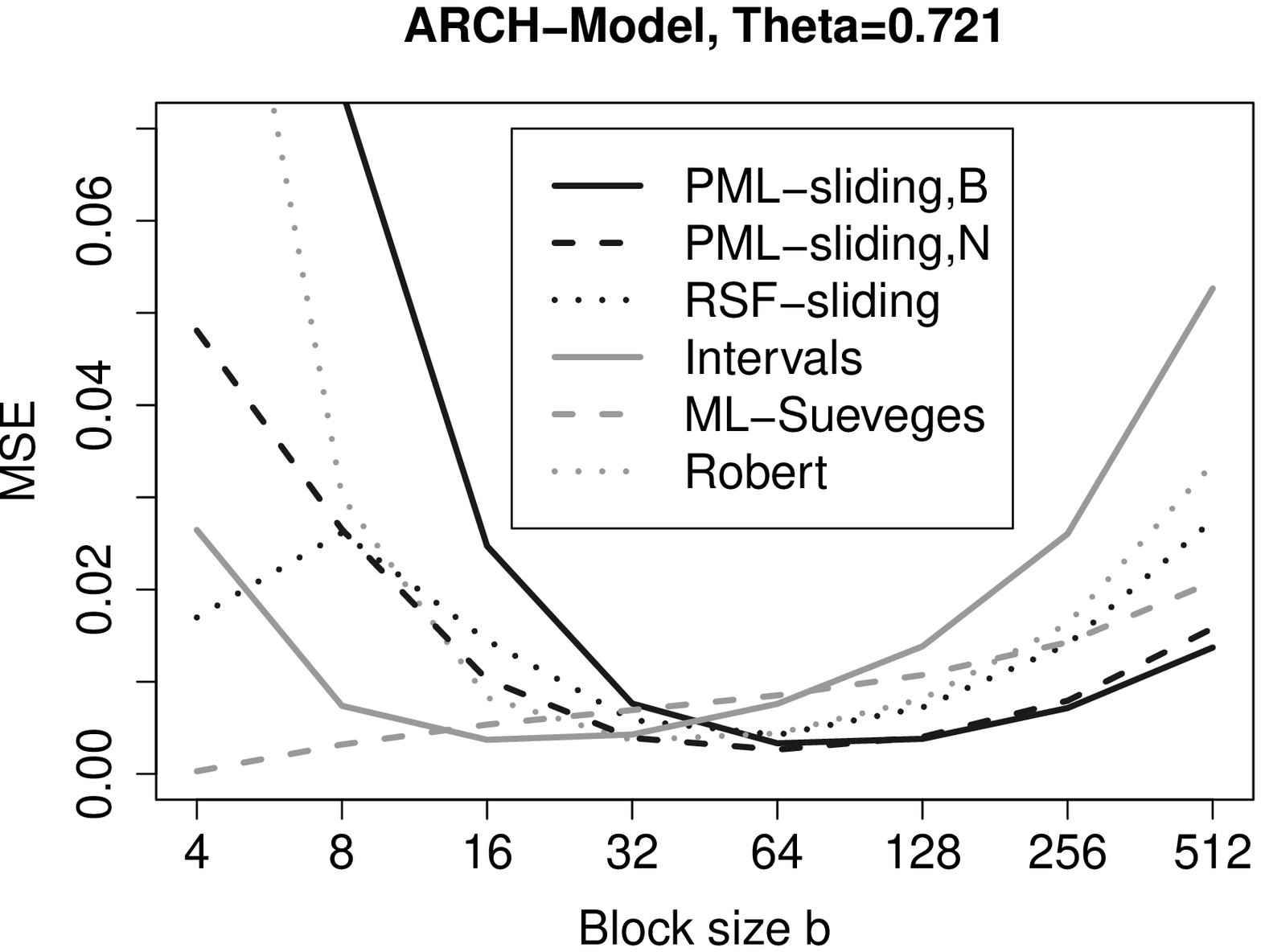}
\vspace{-.2cm}

\includegraphics[width=0.46\textwidth]{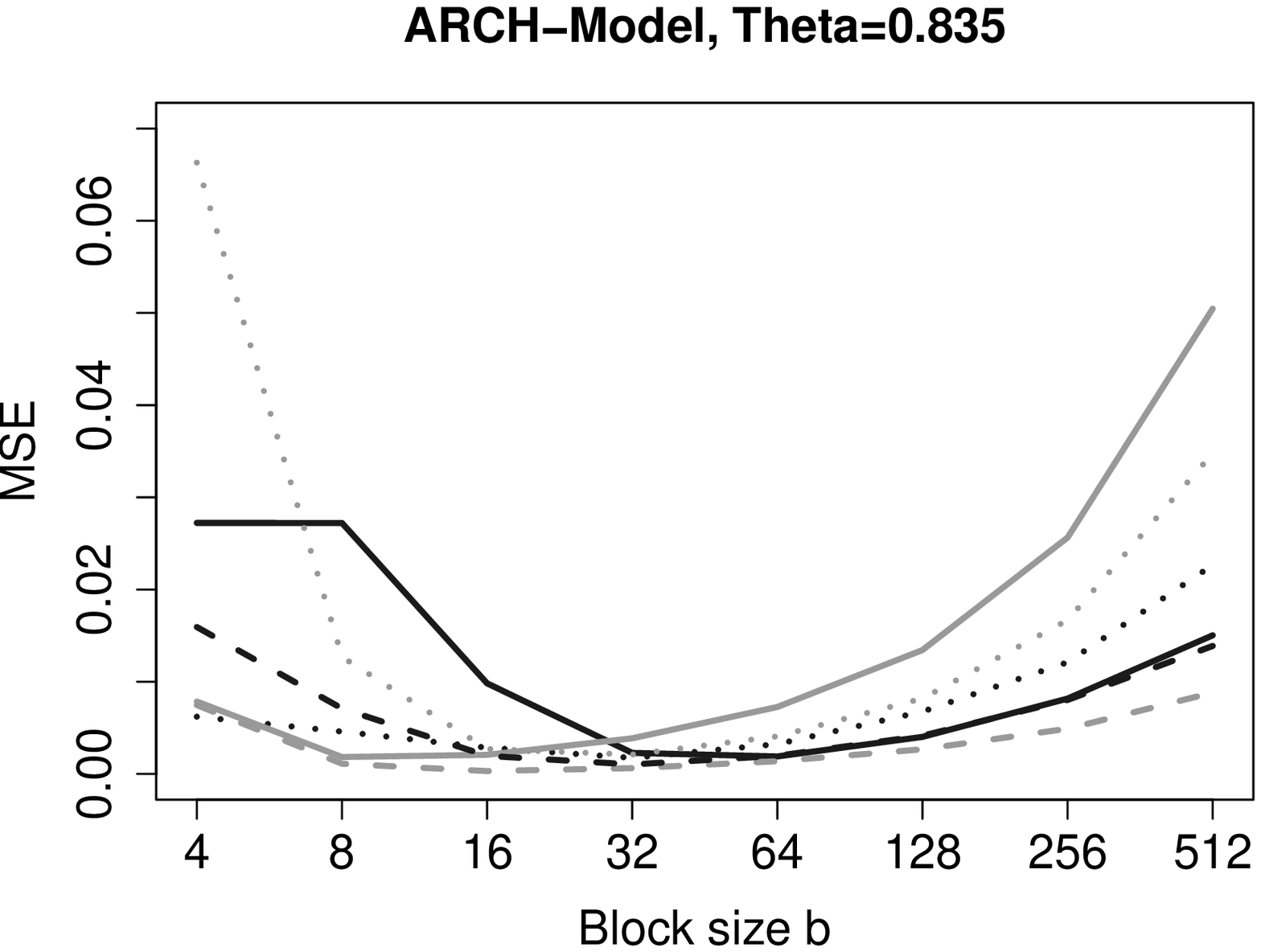}
\hspace{-.3cm}
\includegraphics[width=0.46\textwidth]{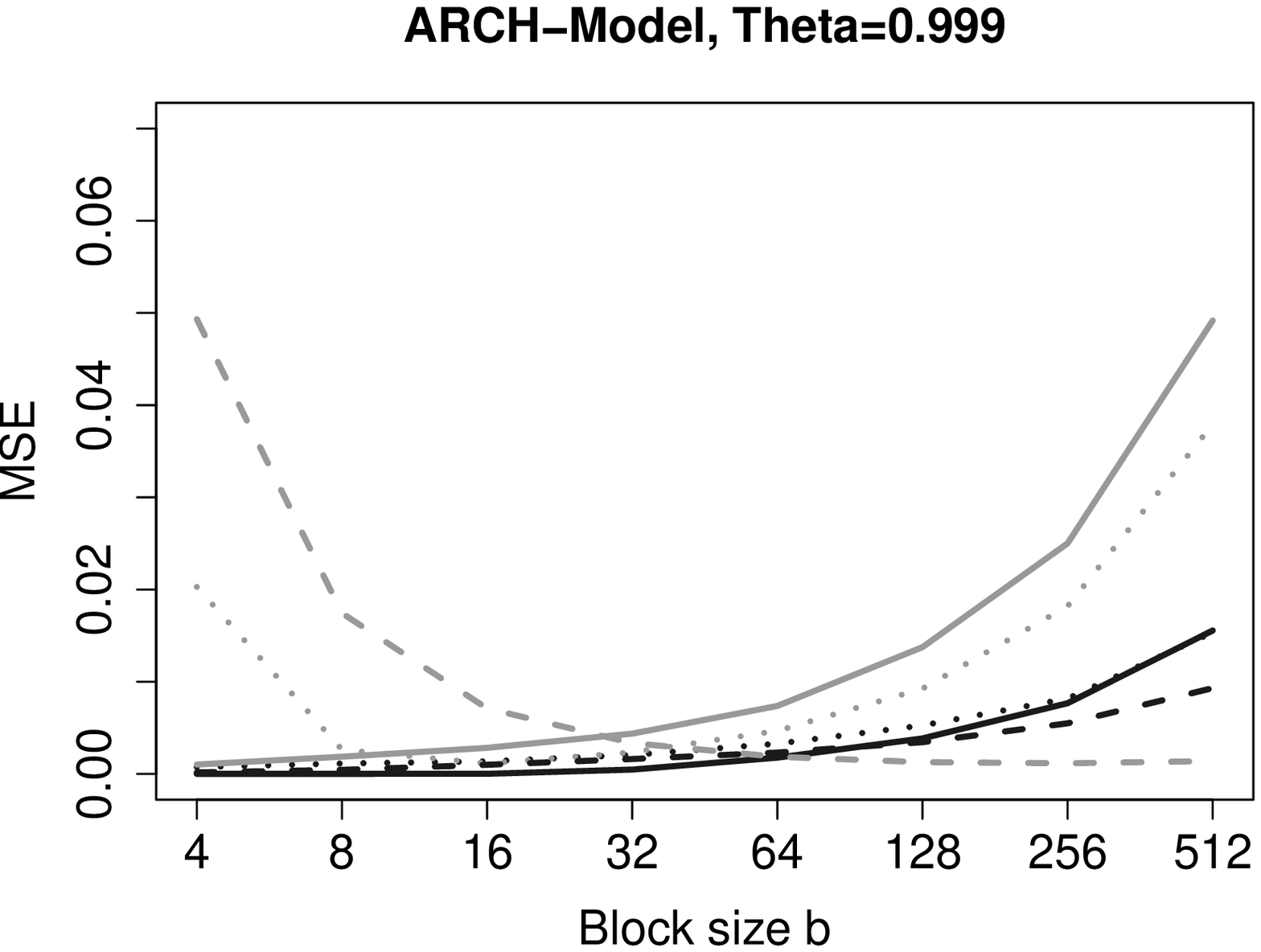}
\end{center}
\vspace{-.5cm}
\caption{\label{fig:mse1}  Mean squared error for the estimation of $\theta$ within the ARCH-model for four values of $\theta\in\{0.571,0.721,0.835,0.999\}$. 
}
\vspace*{\floatsep}

\begin{center}
\includegraphics[width=0.46\textwidth]{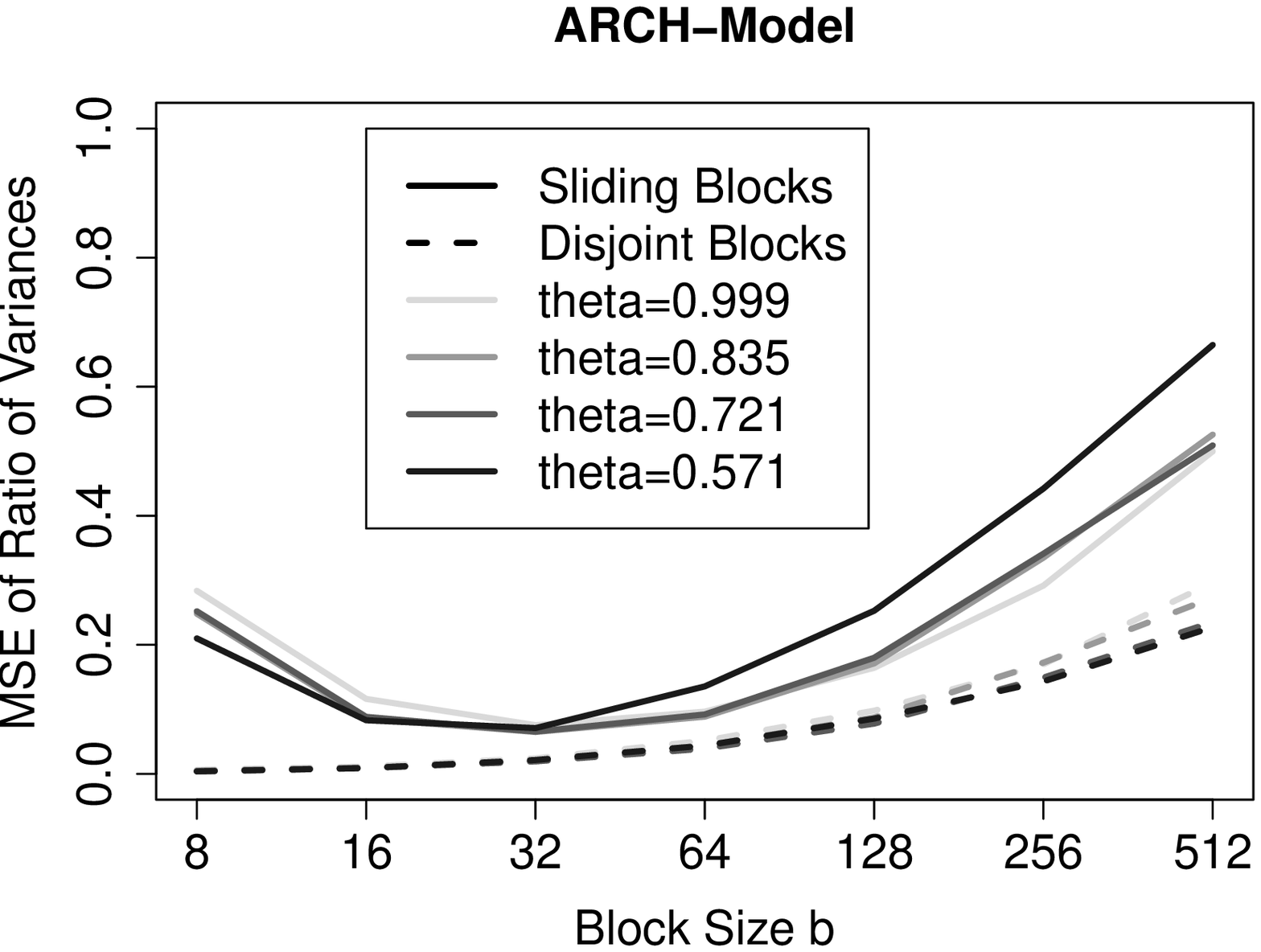}
\hspace{-.3cm}
\includegraphics[width=0.46\textwidth]{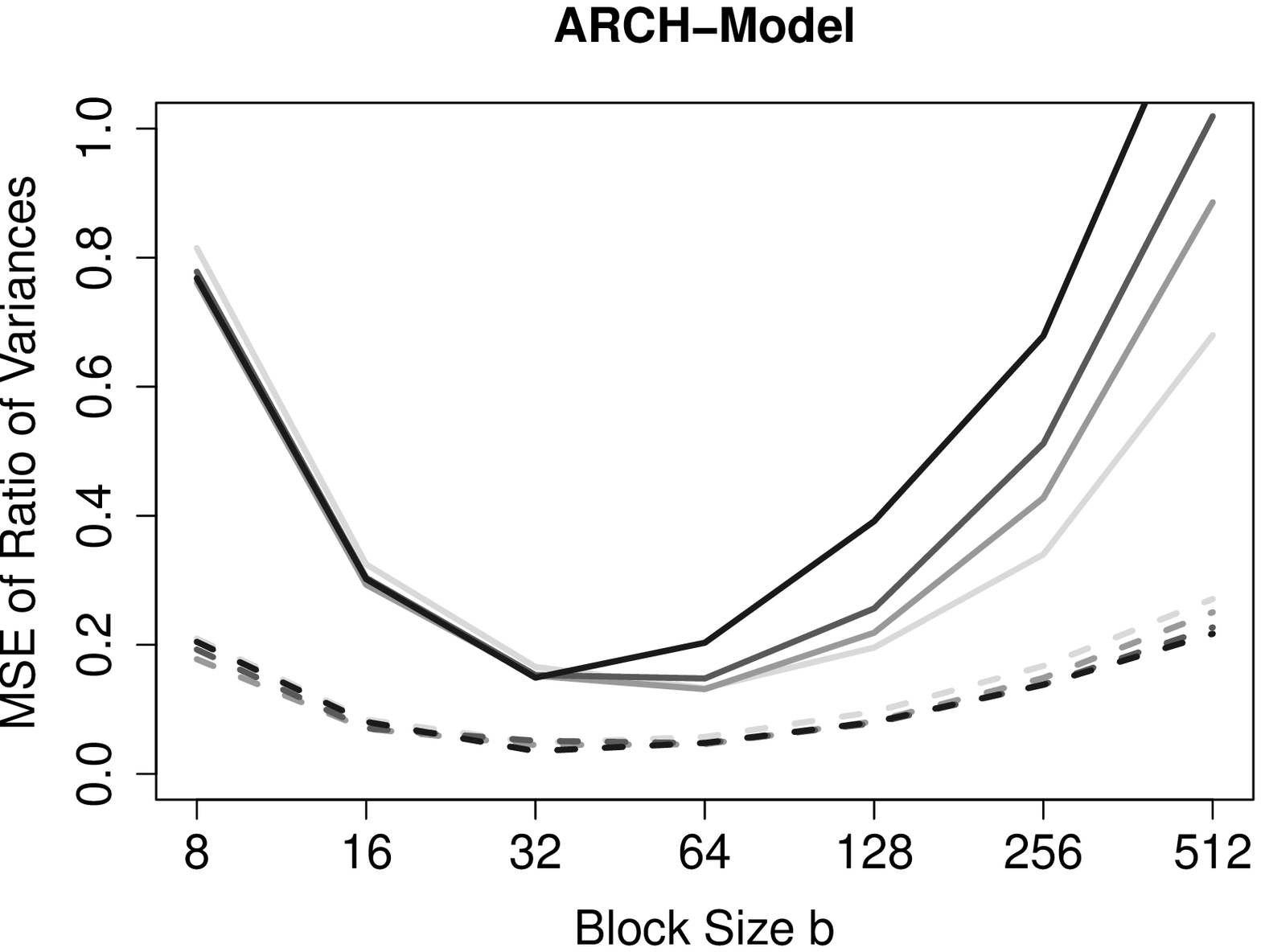}
\vspace{-.5cm}

\includegraphics[width=0.46\textwidth]{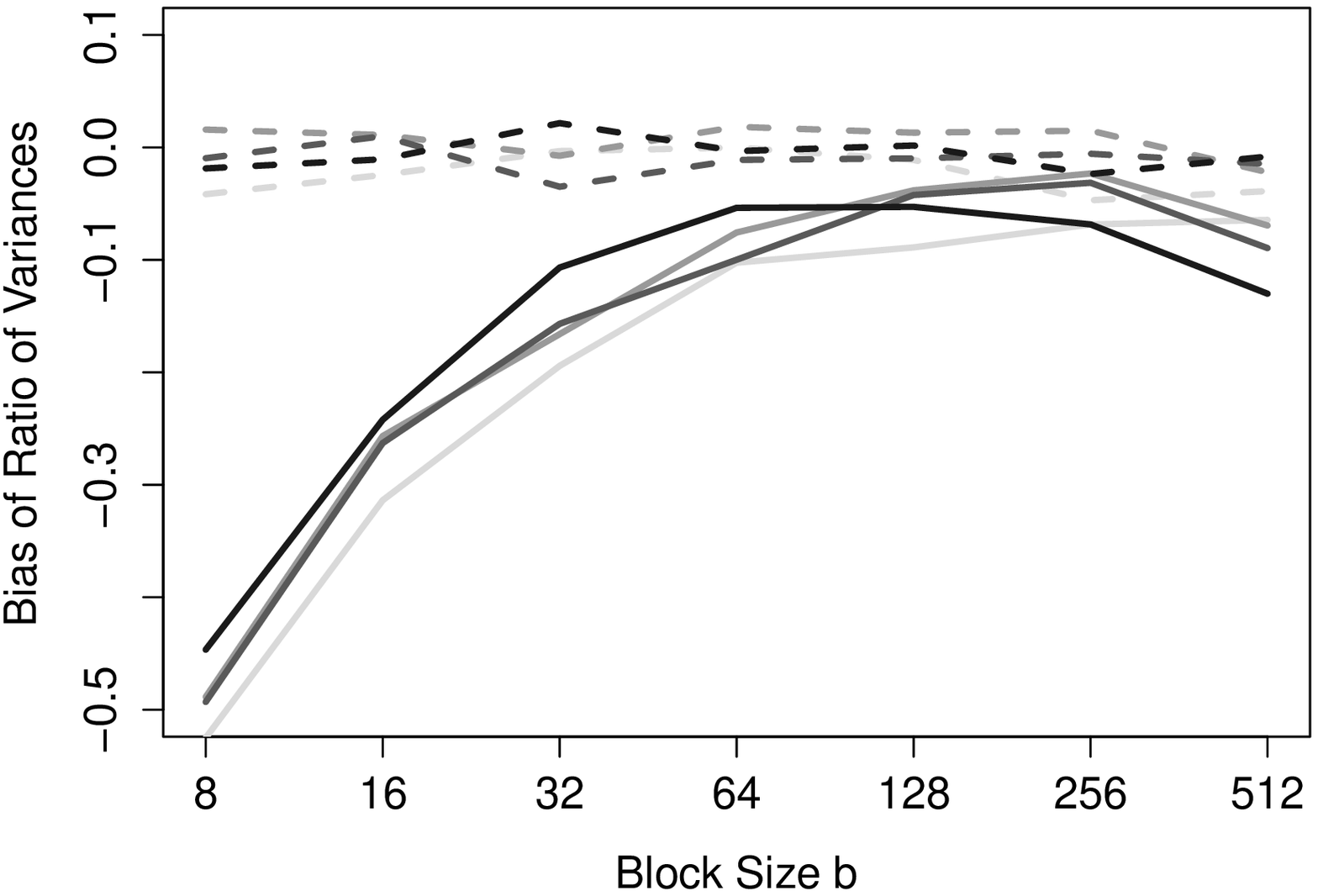}
\hspace{-.3cm}
\includegraphics[width=0.46\textwidth]{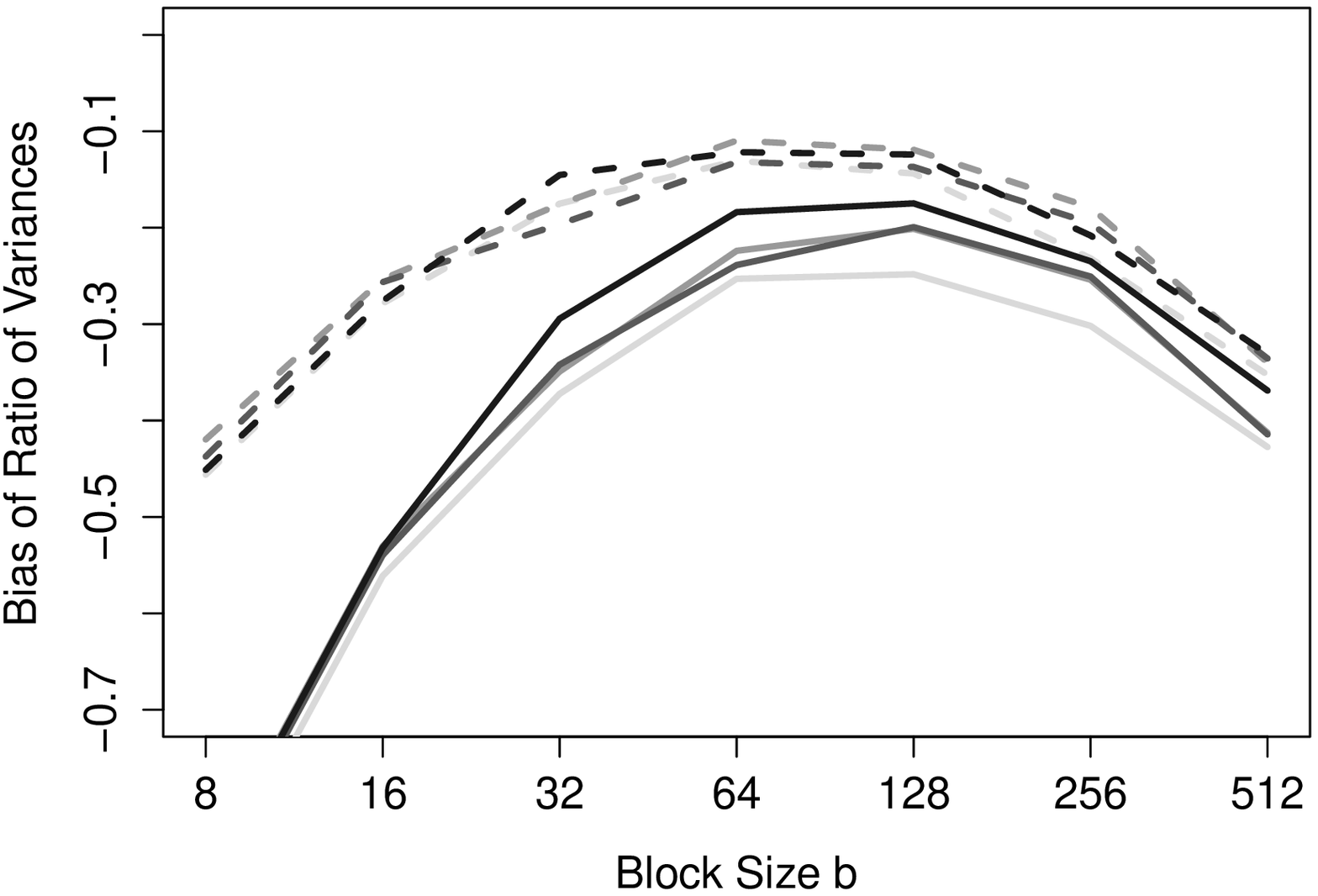}
\end{center}
\vspace{-.5cm} 
\caption{\label{fig:msevariance1}  Mean squared error $\Exp[(\hat \tau^2/ \Var (\hat \theta_n) - 1 )^2]$ and bias  $\Exp[\hat \tau^2/ \Var (\hat \theta_n) ]- 1 $ within the ARCH-model for the unconstrained estimators  $\hat \theta_n^{\Be}$ (left) and $\hat \theta_n^{\No}$ (right).
}
\end{figure}

\begin{table}[t!]
\centering
{\footnotesize
\begin{tabular}{lllllll}
  \hline
 $\theta$ & $\hat\theta_n^{\Be, \slb}$ & $\hat\theta_n^{\No, \slb}$ & RSF-sliding & Intervals & ML-S\"uveges & Robert \\ 
  \hline
0.25 & 0.91 & 0.51 & 1.35 & 0.53 &\bf  0.22 & 1.77 \\ 
  0.50 & 1.58 & 0.78 & 2.24 & 0.99 & \bf 0.63 & 2.07 \\ 
  0.75 & 2.03 & \bf  0.67 & 2.34 & 1.17 & 0.96 & 2.31 \\ 
  1.00 & \bf 0.00 (1.78) & 0.05 (0.11) & 0.10 (0.12) & 0.88 & 0.11 & 2.22 \\ 
  \hline
0.422 & 3.18 & 2.86 & 4.85 & \bf 2.53 & 3.19 & 4.00 \\ 
  0.460 & 3.53 & 2.98 & 5.45 & 2.71 & \bf  1.92 & 4.26 \\ 
  0.727 & 1.07 & \bf  0.46 & 1.46 & 1.08 & 1.44 & 1.19 \\ 
  0.997 & \bf  0.01 (0.50) & 1.56 & 1.31 (1.33) & 5.34 & 2.19 & 0.65 \\ 
  \hline
0.571 & 4.82 & \bf 4.81 & 7.65 & 6.02 & 20.94 & 5.58 \\ 
  0.721 & 3.32 & 2.63 & 4.22 & 3.70 & \bf 0.28 & 3.65 \\ 
  0.835 & 1.89 & 1.02 & 1.74 & 1.83 & \bf 0.31 & 2.09 \\ 
  0.999 &\bf 0.00 (0.98) & 0.16 (0.17)& 0.73 (0.76)& 1.01 & 1.15 & 1.13 \\ 
  \hline
    0.20 & 0.63 &\bf   0.52 & 1.72 & 0.63 & 15.14 & 1.56 \\ 
  0.40& 0.99 & \bf  0.68 & 1.61 & 0.79 & 3.80 & 1.29 \\ 
  0.60 & 1.65 & 0.92 & 1.72 & 4.77 &\bf   0.43 & 1.65 \\ 
  0.80 & 0.97 &\bf   0.18 & 0.72 & 13.00 & 2.53 & 0.63 \\ 
0.95 & \bf  0.82 (0.94) & 4.60 & 2.87 & 12.05 (12.50) & 4.32 & 1.65 (2.06) \\ 
   \hline
\end{tabular}
}

\caption{Minimal mean squared error multiplied with $10^3$ for the ARMAX-model (top 4 rows), the squared ARCH-model (upper middle 4 rows), the  ARCH-model (lower middle 4 rows) and the Markovian copula model (bottom 5 rows). The estimator with the (row-wise) smallest MSE is in boldface. Values in brackets refer to the unconstrained estimator. } \label{tab:minmse}

\vspace{-.5cm}

\end{table}

\subsection{Estimation of the asymptotic variance and coverage of confidence bands} We consider the ARMAX-, squared ARCH-, and ARCH-model as described above. We are interested in the performance of 
\[
\hat \tau^2_{\djb}= (\hat \theta_n^{\djb})^4 \hat \sigma_{\djb}^2 \quad \text{ and } \quad
\hat \tau^2_{\slb}= (\hat \theta_n^{\slb})^4 \hat \sigma_{\slb}^2
\]
as estimators for the variances of $\sqrt{\kn} \hat \theta_n^{x,\djb}$ and $\sqrt{\kn} \hat \theta_n^{x,\slb}$, respectively, where $x\in\{\Be, \No\}$. Results can be found in Figure~\ref{fig:msevariance1}  (as well as in Figures~\ref{fig:msevariance2} and~\ref{fig:msevariance3} of the supplement), where we depict the curves 
\[
\bn \mapsto 
\Exp\Big[ \Big(\frac{\hat \tau^2(\bn)}{\Var(\sqrt{\kn} \hat \theta_n(\bn))}-1\Big)^2 \Big],  \qquad 
\bn \mapsto 
\Exp\Big[ \frac{\hat \tau^2(\bn)}{\Var(\sqrt{\kn} \hat \theta_n(\bn))}-1 \Big], 
\] 
$(\hat \tau^2, \hat \theta_n) \in \{ (\hat \tau_{\scs \djb}^2, \hat \theta_n^{\scs \Be \djb}),(\hat \tau_{\scs \slb}^2, \hat \theta_n^{\scs\Be, \slb})
, (\hat \tau_{\scs \djb}^2, \hat \theta_n^{\scs \No, \djb}),(\hat \tau_{\scs \slb}^2, \hat \theta_n^{\scs \No,\slb})\}$, estimated on the basis of 10,000 simulation runs. 
Here, $\Var(\sqrt{\kn} \hat \theta_n(\bn))$ is approximated by the empirical variance of $\sqrt{\kn} \hat \theta_n(\bn)$ over additional 10,000 simulations. 
Qualitatively, we observe a similar behaviour as for the estimation of $\theta$ depicted in Figure~\ref{fig:mse1}: the curves are U-shaped and possess a minimum at some intermediate values of $\bn$. 
Due to the fact that estimator $\hat \tau^2_{\slb}$ is based on an additional estimation step (which is potentially biased, if $\bn$ is small), the approximation works better for the disjoint blocks estimator. Also, the approximation is far better for $\hat\theta_n^{\Be}$ than for $\hat \theta_{n}^{\No}$ (in particular for the bias), which may be explained by the fact that $\hat \tau^2_{\slb}$ is based on an explicit expansion for $\hat\theta_n^{\Be}$. In particular, the fact that the bias of $\hat\theta_n^{\No}$ is eventually increasing for larger block sizes may be explained by the $1/\sqrt{k_n}$-approximation of $\hat\theta_n^{\No}$ by $\hat\theta_n^{\Be}$ (Theorem~\ref{theo:Northrop}).

We are also interested in the coverage probabilities of the confidence sets
\[
{\rm CI}_{1-\alpha} = [ \hat \theta_n - \kn^{-1/2} \hat \tau u_{1-\alpha/2},  \hat \theta_n + \kn^{-1/2} \hat \tau u_{1-\alpha/2} ]
\]
for $\theta$,
where $u_{1-\alpha/2}$ denotes the $(1-\alpha/2)$-quantile of the standard normal distribution. Empirical coverage probabilities for $1-\alpha=0.95$ based on $N=10,000$ simulation runs are presented in Tables~\ref{tab:coverage1} ($\hat \theta_n^{\Be}$-versions) and \ref{tab:coverage2}  ($\hat \theta_n^{\No}$-versions), with coverage probabilities above $0.9$  in boldface. Since the variance approximation is worse for $\hat \theta_{n}^{\No}$, the coverage probabilities are worse as well.
Moreover, it can be seen that the probabilities strongly depend on the block size $\bn$, with, for $\hat \theta_n^{\Be}$, at least one reasonable choice  for every model, usually close to the MSE-minimal choice in Figure~\ref{fig:mse1} (and Figure \ref{fig:mse2} in the supplement). The larger width of the confidence sets for the disjoint blocks estimator (not presented here; it is due to the larger variance) results in a slightly better performance compared to the sliding blocks estimator.

 \begin{table}[t]
\centering
{\footnotesize
\begin{tabular}{rl ||  cccc || cccc }
\hline
\multicolumn{2}{c||}{} & \multicolumn{4}{c||}{ARMAX-model} & \multicolumn{4}{c}{ARCH-model} \\
  \hline
 &$\bn / \theta$& 0.25 & 0.5 & 0.75 & 1 & 0.571 & 0.721 & 0.835 & 0.999 \\ 
  \hline
disjoint &  16 & 0 & 0 & 0.13 & \bf 1.0&0.00 & 0.00 & 0.04 &\bf 1.00 \\ 
  &32 & 0.03 & 0.63 & 0.85 & \bf 0.99& 0.01 & 0.42 & 0.87 &\bf 0.97 \\ 
  &64 & 0.80 & \bf 0.93 &\bf  0.95 & \bf 0.98 &0.68 &\bf 0.91 &\bf 0.94 &\bf 0.93 \\ 
  &128 & \bf 0.94 & \bf 0.94 & \bf 0.94 & \bf 0.95  &\bf 0.93 &\bf 0.94 &\bf 0.92 &\bf 0.91\\ 
  &256 & \bf 0.93 &\bf  0.92 & \bf 0.91 &\bf  0.92&\bf  0.93 &\bf 0.92 &\bf 0.90 & 0.89  \\ 
  &512 &\bf  0.91 & \bf0.90 & 0.88 & 0.87&\bf 0.90 & 0.88 & 0.86 & 0.84 \\ 
\hline
  sliding &   16 & 0 & 0 & 0.02 & \bf 1.00& 0.00 & 0.00 & 0.00 &\bf 1.00 \\ 
  &32 & 0.01 & 0.46 & 0.75 &\bf 1.00 & 0.00 & 0.20 & 0.76 &\bf 0.95\\ 
  &64 & 0.71 &\bf  0.90 & \bf 0.93&\bf  0.96&0.53 & 0.86 &\bf 0.92 & 0.89 \\ 
  &128 & \bf0.92 &\bf 0.93 & \bf0.92 & \bf0.92&0.89 &\bf 0.92 & 0.88 & 0.85  \\ 
  &256 & \bf0.91 & 0.89 & 0.87 & 0.86&\bf 0.90 & 0.88 & 0.84 & 0.81 \\ 
  &512 & 0.88 & 0.85 & 0.81 & 0.76 &0.85 & 0.81 & 0.77 & 0.73  \\
   \hline
\end{tabular}
}


\caption{Empirical coverage probabilities of $95\%$-confidence bands of the constrained estimators $\hat \theta_n^{\Be}$. Values above $90\%$ are in boldface.} \label{tab:coverage1}
\vspace{-.2cm}
\end{table}

 \begin{table}[t]
\centering
{\footnotesize
\begin{tabular}{rl ||  cccc || cccc }
\hline
\multicolumn{2}{c||}{} & \multicolumn{4}{c||}{ARMAX-model} & \multicolumn{4}{c}{ARCH-model} \\
  \hline
 &$\bn / \theta$& 0.25 & 0.5 & 0.75 & 1 & 0.422 & 0.46 & 0.727 & 0.997 \\ 
  \hline
disjoint &    16 & 0.00 & 0.47 & 0.84 & \bf 0.95& 0.00 & 0.01 & 0.62 & 0.82  \\ 
  &32 & 0.50 & 0.87 &\bf  0.92 &\bf  0.96 & 0.07 & 0.64 &\bf 0.91 & 0.88  \\ 
  &64 &  0.88 &\bf  0.93 & \bf 0.93 &\bf  0.97&  0.69 & \bf 0.90 &\bf 0.92 &\bf 0.92\\ 
  &128 &\bf  0.91 & \bf 0.92 & \bf 0.92 &\bf  0.96 &0.89 &\bf 0.92 &\bf 0.92 &\bf 0.94\\ 
  &256 &\bf 0.90 & \bf 0.90 & \bf 0.91 & \bf 0.96&\bf 0.90 &\bf 0.91 &\bf 0.94 &\bf 0.94  \\ 
  &512 & 0.87 & 0.87 &\bf  0.91 & \bf 0.94 & 0.86 & 0.89 &\bf 0.92 &\bf 0.93\\ 
\hline
  sliding &   16&0.00 & 0.22 & 0.70 &\bf 0.91& 0.00 & 0.00 & 0.32 & 0.62\\ 
  &32 &0.31 & 0.80 & 0.88 &\bf 0.94 & 0.01 & 0.42 & 0.87 & 0.77 \\ 
  &64 &0.82 & 0.89 & \bf0.90 &\bf 0.94& 0.54 & 0.84 & 0.88 & 0.85\\ 
  &128 & 0.87 & 0.88 & 0.88 &\bf 0.92& 0.83 & 0.88 & 0.86 & 0.84 \\ 
  &256&0.85 & 0.84 & 0.81 & 0.86& 0.83 & 0.83 & 0.80 & 0.76\\ 
  &512 &0.75 & 0.72 & 0.69 & 0.76&0.73 & 0.69 & 0.67 & 0.62\\ 
   \hline
\end{tabular}
}


\caption{Empirical coverage probabilities of $95\%$-confidence bands of the constrained Northrop estimators $\hat \theta_n^{\No}$. Values above $90\%$ are in boldface.} \label{tab:coverage2}
\vspace{-.5cm}

\end{table}

\section{Case study} \label{sec:app}

The use of the PML-estimators and the corresponding confidence sets is illustrated on negative daily log returns of a variety of financial market indices and prices including equity (e.g., S\&P 500 Composite, MSCI World), commodities (e.g., TOPIX Oil \& Coal, Gold Bullion LBM, Raw Sugar) and U.S. treasury bonds between 04 January 1990 and 30 December 2015 ($n=6,780$ observations for each index). Clusters of large negative returns can be financially damaging and are hence of interest for risk management. 

\begin{figure}[t]
\begin{center}
\includegraphics[width=0.48\textwidth]{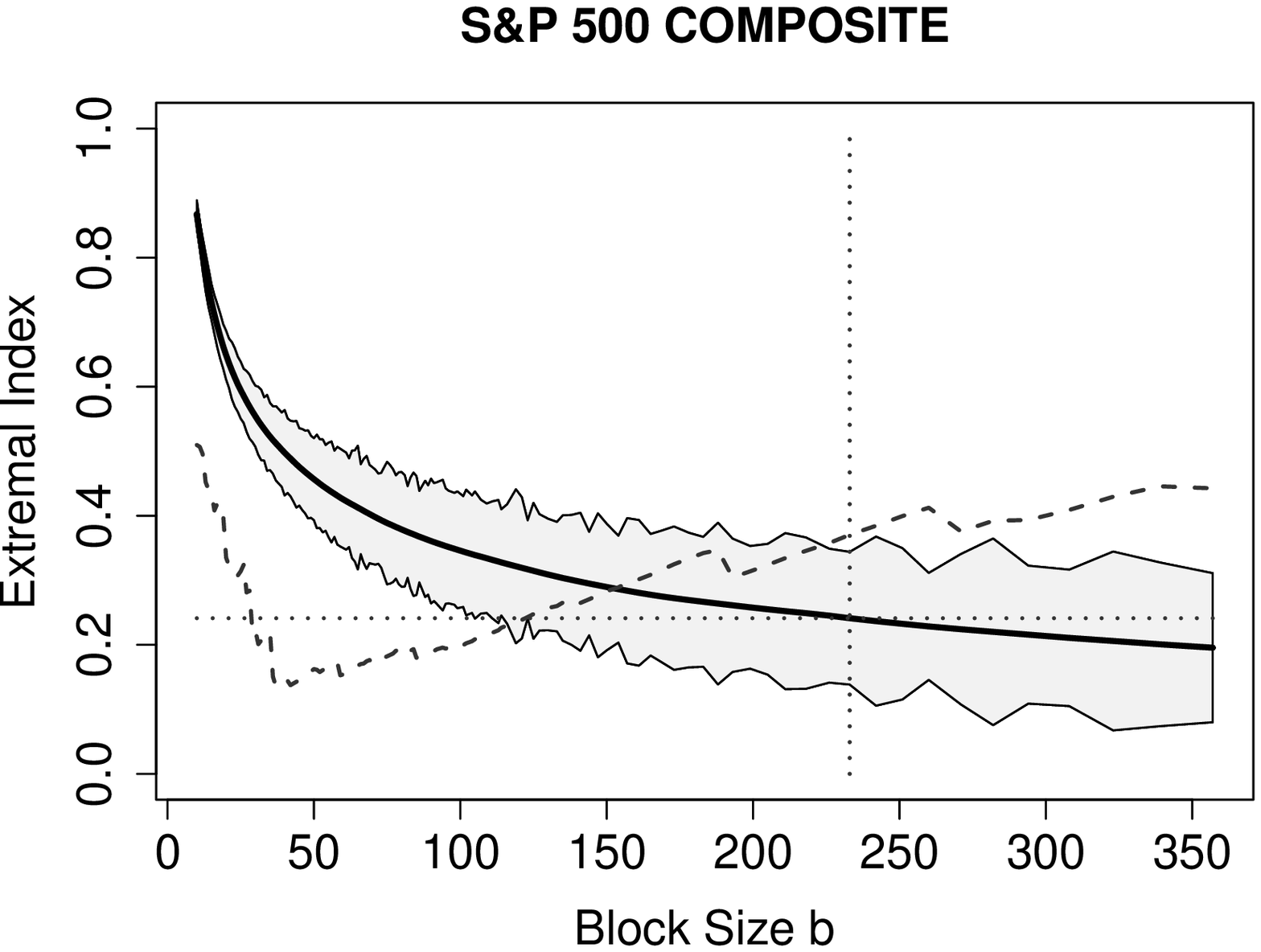}
\hspace{-.3cm}
\includegraphics[width=0.48\textwidth]{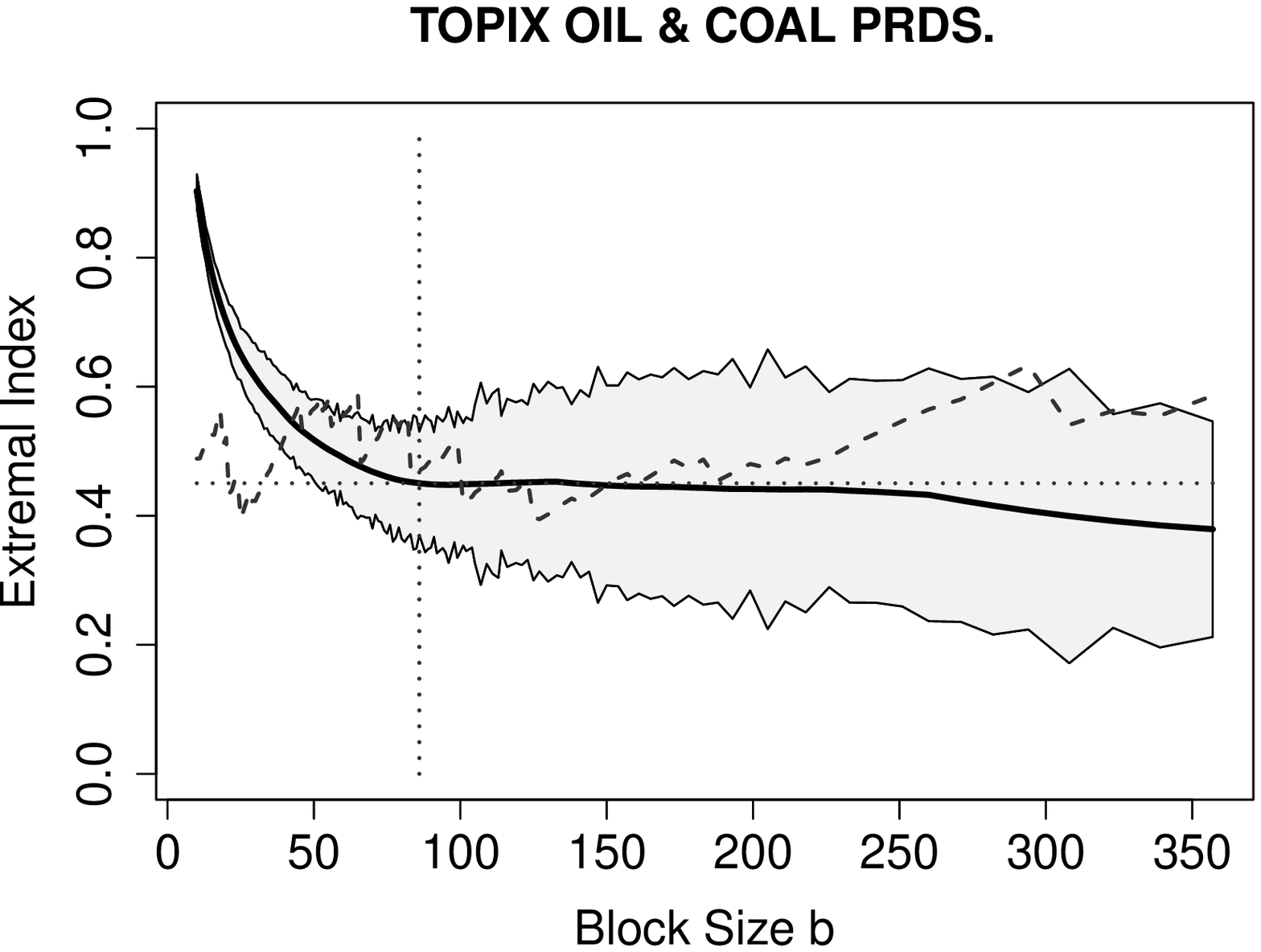}
\vspace{-.2cm}

\includegraphics[width=0.48\textwidth]{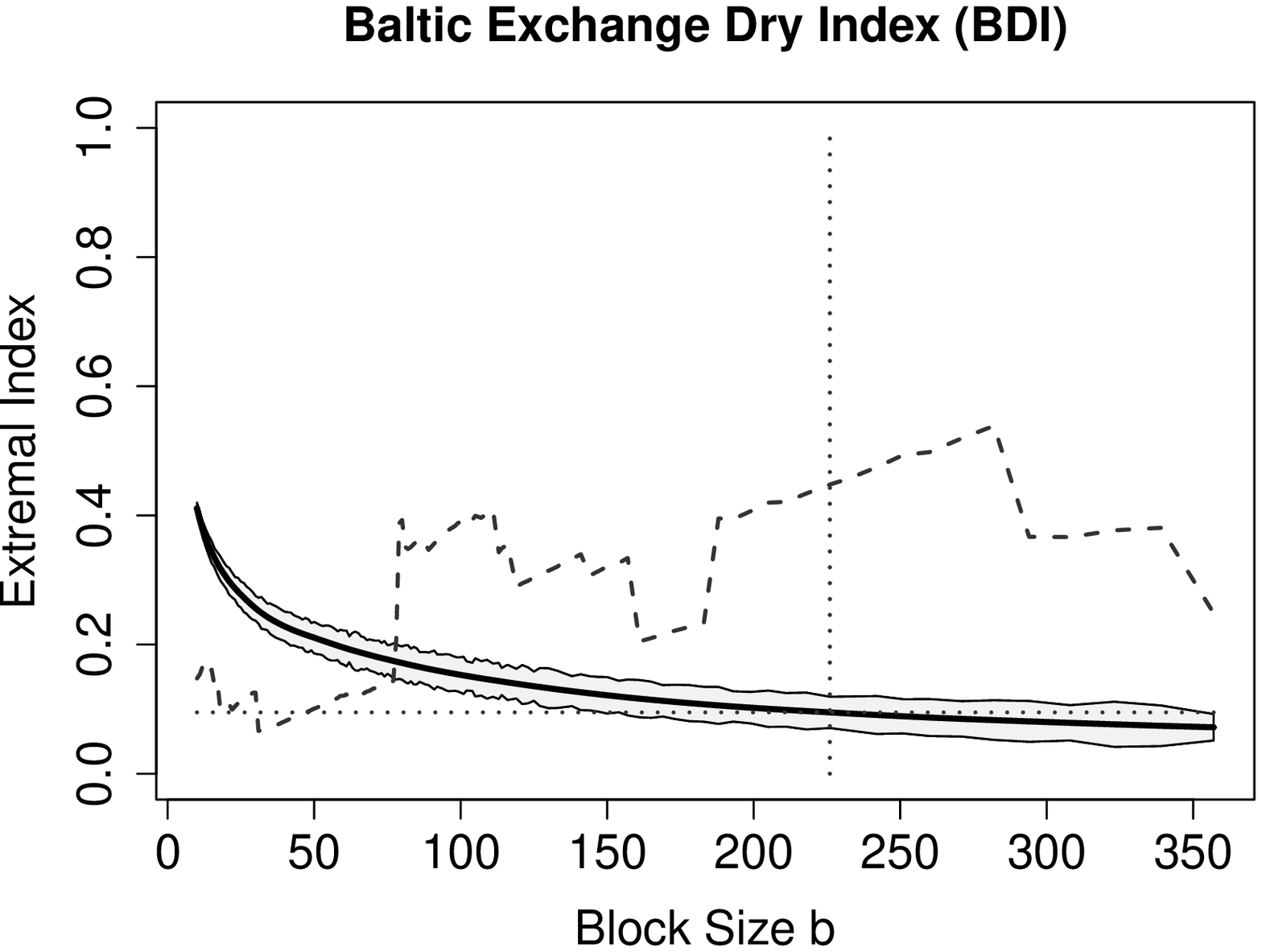}
\hspace{-.3cm}
\includegraphics[width=0.48\textwidth]{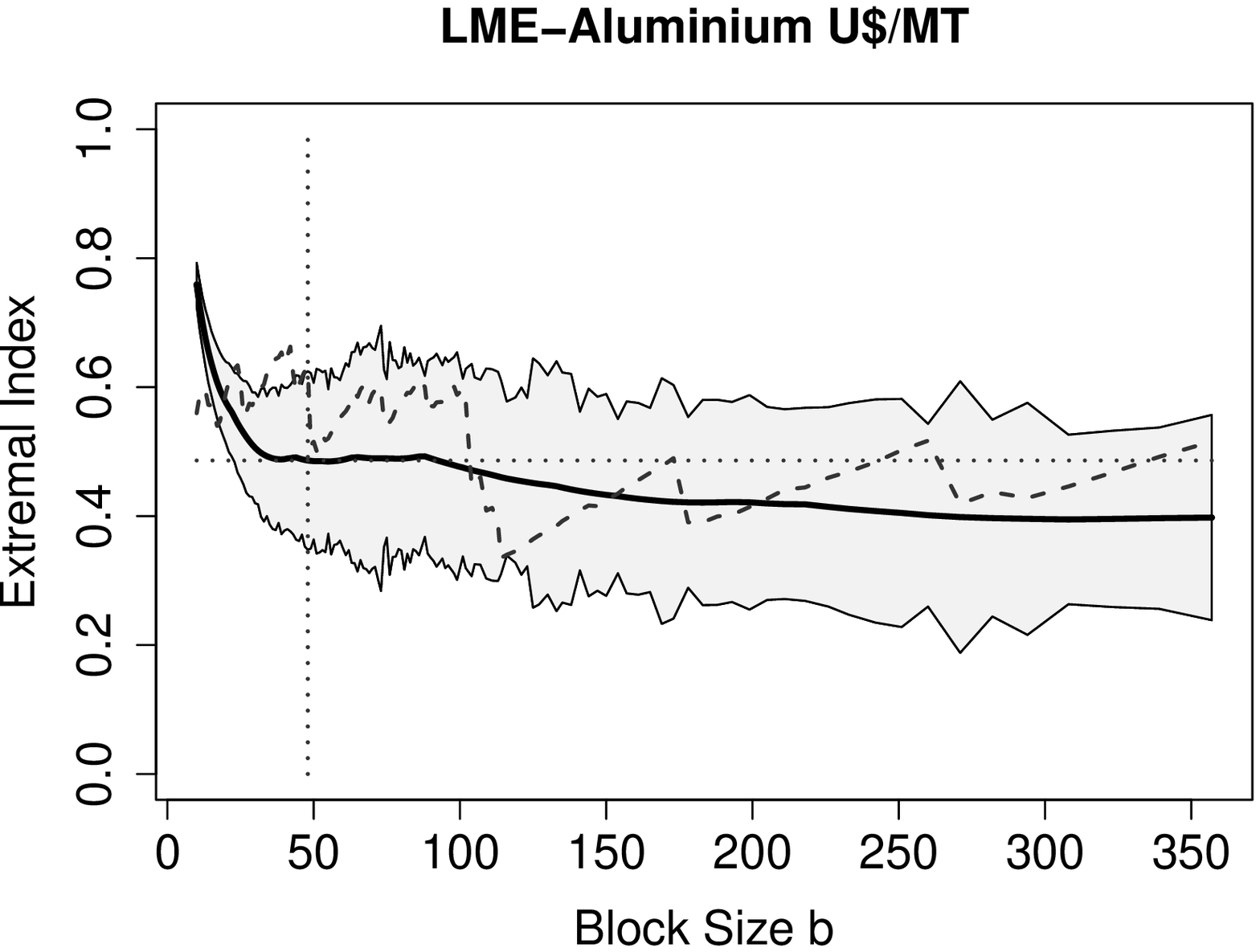}
\end{center}
\vspace{-.3cm}
\caption{\label{fig:app}  Extremal index estimates for four financial time series as a function of the block size. The solid line is the bias--reduced sliding blocks estimate, the shaded region is the pointwise $95\%$-confidence band. The dashed line is the intervals estimator. The dotted lines correspond to the selected block length based on visual inspection of the graphs.
}
\vspace{-.5cm}
\end{figure}

In Figure~\ref{fig:app}, we depict estimates of the extremal index for four typical time series as a function of the block length parameter, ranging from $b=10$ to $b=357$. The solid curves correspond to the bias corrected sliding blocks estimator $\hat \theta_n^{\Be, \slb}$, alongside  with a 95\%-confidence band based on the variance estimator from Section~\ref{sec:boot} and the normal approximation. Interestingly, the curves appear to be quite smooth, which is a typical and nice property of the sliding blocks estimator. For comparison, the (far rougher) dashed lines correspond to the intervals estimator from \cite{FerSeg03}. As highlighted by many other authors, there is no simple optimal solution for the choice of the \textit{best} block length parameter and a unique estimate for the extremal index. The dotted lines in Figure~\ref{fig:app} correspond to  case-by-case visual choices, trying to capture plateaus in the respective plots.

For the ease of comparison, this procedure has been repeated for all 20 time series under consideration (despite the fact that the entire curves provide a more detailed picture of the extremal dependence). In Table~\ref{tab:app}, we state the resulting estimates of the extremal index and the width of the corresponding confidence intervals. Interestingly, the estimates of the extremal index lie around 0.3 for most of the equity indexes (S\&P 500 Composite, MSCI World, etc.), while they are around  0.45 for many of the commodity prices (Coffee, Cotton, Aluminium). The smallest value of 0.12 is attained for the Baltic Exchange Dry Index, an index measuring  the price of moving the major raw materials by sea and usually regarded as an efficient economic indicator of future economic growth and production.

\begin{table}[t]
\centering
{\footnotesize
\begin{tabular}{l c c}
  \hline
 Index / Prices & Extremal Index & Width of C-Interval \\ 
  \hline
  Raw Sugar Cents/lb & 0.54 & 0.17 \\ 
  Coffee-Brazilian Cents/lb & 0.49 & 0.13 \\ 
  LME-Aluminium U\$/MT & 0.49 & 0.14 \\ 
  Palladium U\$/Troy Ounce & 0.46 & 0.11 \\
  TOPIX OIL \& COAL PRDS. & 0.45 & 0.08 \\  
  US T-Bill 10 YEAR & 0.44 & 0.12 \\ 
  Cotton Cents/lb & 0.42 & 0.12 \\ 
  S\&P GSCI Precious Metal & 0.42 & 0.12 \\ 
  MSCI WORLD EX US & 0.36 & 0.11 \\ 
  Crude Oil-Brent Cur. Month & 0.35 & 0.10 \\ 
  Gold Bullion LBM & 0.33 & 0.10 \\ 
  RUSSELL 2000 & 0.31 & 0.09 \\ 
  S\&P GSCI Commodity Total Return & 0.30 & 0.09 \\ 
  S\&P 500 COMPOSITE & 0.29 & 0.10 \\ 
  LMEX Index & 0.27 & 0.10 \\
  G12-DS Banks & 0.26 & 0.09 \\ 
  G7-DS Banks & 0.26 & 0.10 \\ 
  EU-DS Banks & 0.26 & 0.08 \\ 
  S\&P500 BANKS & 0.22 & 0.08 \\ 
  Baltic Exchange Dry Index (BDI) & 0.12 & 0.02 \\ 
   \hline
\end{tabular}
}
\smallskip

\caption{Sliding Blocks Estimates of the extremal index and width of corresponding confidence intervals for negative daily log returns of 20  financial market indices and prices. } \label{tab:app}

\vspace{-.5cm}

\end{table}

\section{Auxiliary Lemmas for proving Theorem~\ref{theo:main} (disjoint blocks)} \label{sec:proofs}

\begin{lemma}[Approximation by an integral with bounded support] \label{lem:boundsupp}
Under Condition~\ref{cond:cond}, for all $\delta>0$,
\[
\lim_{\ell \to \infty} \limsup_{n\to\infty} \Prob( |D_{n,\ell} - D_n | > \delta) = 0.
\]
\end{lemma}

\begin{lemma}[Approximation by a Lebesgue integral] \label{lem:appleb}
Suppose that Condition~\ref{cond:cond} is met. Then, as $n\to\infty$,
\[
D_{n,\ell} = D_{n,\ell}' + o_\Prob(1), \qquad \text{where}\quad D_{n,\ell}' = \int_0^\ell \tailp(x) \theta e^{-\theta x} \, \diff x.
\]
\end{lemma}

\begin{lemma}[Joint convergence of fidis]\label{lem:fidirob}
Under Condition~\ref{cond:cond}, for any $x_1, \dots, x_m \in [0,\infty)$, as $n\to\infty$,
\[
\Big(\tailp(x_1), \dots,  \tailp(x_m) ,  G_n  \Big)' 
\dto
\Big(\tailplim(x_1), \dots,  \tailplim(x_m) ,  G  \Big)', 
\]
the random vector on the right-hand side  being $\Nc_{m+1} \left( \bm 0, \bm \Sigma^{\djb}(x_1, \dots, x_m) \right)$-distributed with
\[
\bm \Sigma^{\djb}(x_1, \dots, x_m)  
=
\matr{
r(x_1,x_1) & \dots & r(x_1, x_m) & h(x_1) \\ 
\vdots & \ddots & \vdots & \vdots \\ 
r(x_m,x_1) & \dots & r(x_m,x_m) & h(x_m) \\
h(x_1) & \dots & h(x_m) & \theta^{-2}
} .
\]
Here, $r(0,0)=h(0)=0$ and, for $x \ge y \ge 0$ with $x\ne 0$,
\begin{align*}
r(x,y) &= \theta x \sum_{i=1}^\infty \sum_{j=0}^i ij \pi_2^{(y/x)}(i,j), \qquad
h(x) =    \int_0^x  \sum_{i=1}^\infty i p_2^{(x,y)}(i,0)\, \diff y - x/\theta,
\end{align*}
where, for $i\ge j\ge 0, i\ge 1$,
\[
p_2^{(x,y)}(i,j) = \Prob\big\{ \bm N_E^{(x,y)}=(i,j)\big\}, \quad  \bm N_E^{(x,y)} = \sum_{i=1}^{\eta} (\zeta_{i1}^{(y/x)}, \zeta_{i2}^{(y/x)})
\]
with $\eta \sim \text{Poisson}(\theta x)$ independent of iid random vectors 
$(\zeta_{i1}^{\scs(y/x)}, \zeta_{i2}^{\scs(y/x)})\sim \pi_2^{\scs(y/x)}, i\in\N$.
\end{lemma}

\begin{lemma} \label{lem:tailpweak}
Under Condition~\ref{cond:cond}, as $n\to\infty$,
\[
\Big\{\big(\tailp(x),  G_n  \big)' \Big\}_{x\in[0,\infty)}
\dto
\Big\{ \big(\tailplim(x),G\big)'\Big\}_{x \in [0,\infty)}
\]
in $D([0,\infty)) \times \R$, where $(\tailplim ,G)'$ is a centered Gaussian process with continuous sample paths and covariance functional as specified in Lemma~\ref{lem:fidirob}. Here, $D([0,\infty))$ is equipped with the metric $d(f,g) =  \int_0^\infty e^{-t} [d_t(f, g) \wedge 1]\diff t$ where $d_t$ denotes the $J_1$-metric applied to the restrictions of $f$ and $g$ to $[0,t]$.
\end{lemma}

\begin{lemma} \label{lem:weakdnlgn}
Under Condition~\ref{cond:cond}, for any $\ell\in\N$,
\[
D_{n,\ell} + G_n \dto \Nc(0, \sigma^2_\ell),
\]
as $n\to\infty$, where
\[
\sigma^2_\ell = \theta^2 \int_0^\ell\int_0^\ell r(x,y) e^{-\theta (x+y)}\, \diff x \diff y + 2 \theta \int_0^\ell h(x) e^{-\theta x} \, \diff x + \theta^{-2}
\]
\end{lemma}

\begin{lemma}\label{lem:convvar}
Under Condition~\ref{cond:cond}, as $\ell \to \infty$,
\[
\sigma_\ell^2 \to \sigma_{\djb}^2,
\]
where $\sigma_\ell^2$ and $\sigma_{\djb}^2$ are defined in Lemma \ref{lem:weakdnlgn} and Theorem \ref{theo:main}, respectively.
\end{lemma}

\section{Auxiliary Lemmas for proving Theorem~\ref{theo:main} (sliding blocks)} \label{sec:proofs2}

\begin{lemma}[Approximation by an integral with bounded support -- sliding blocks] \label{lem:boundsupp2}
Under Condition~\ref{cond:cond}, for all $\delta>0$,
\[
\lim_{\ell \to \infty} \limsup_{n\to\infty} \Prob( |D_{n,\ell}^{\slb} - D_ n^{\slb} | > \delta) = 0.
\]
\end{lemma}

\begin{lemma}[Approximation by a Lebesgue integral -- sliding blocks] \label{lem:appleb2}
Suppose Condition~\ref{cond:cond} is met. Then, as $n\to\infty$,
\[
D_{n,\ell}^{\slb} = D_{n,\ell}'^{\slb} + o_\Prob(1), \qquad \text{where}\quad D_{n,\ell}'^{\slb} = \int_0^\ell \tailp(x) \theta e^{-\theta x} \, \diff x.
\]
\end{lemma}

\begin{lemma}[Joint convergence of fidis -- sliding blocks]\label{lem:fidirob2} Let
\[
G_n^{\slb} =  \sqrt{\kn} (T_n^{\slb} -  \Exp T_n^{\slb}), \qquad 
T_n^{\slb} = \frac{1}{n-\bn+1} \sum_{t=1}^{n-\bn+1}  Z_{nt}^{\slb}.
\]
Under Condition \ref{cond:cond}, for any $x_1, \dots, x_m \in [0,\infty)$, as $n \to \infty$,
\[
\Big(\tailp(x_1), \dots,  \tailp(x_m) ,  G_n^{\slb}  \Big)' 
\dto
\Big(\tailplim(x_1), \dots,  \tailplim(x_m) ,  G^{\slb}  \Big)', 
\]
the random vector on the right-hand side  being $\Nc_{m+1} \left( \bm 0, \bm \Sigma^{\slb}(x_1, \dots, x_m) \right)$-distributed with
\[
\bm \Sigma^{\slb}(x_1, \dots, x_m)  
=
\matr{
r(x_1,x_1) & \dots & r(x_1, x_m) & h(x_1) \\ 
\vdots & \ddots & \vdots & \vdots \\ 
r(x_m,x_1) & \dots & r(x_m,x_m) & h(x_m) \\
h(x_1) & \dots & h(x_m) & \frac{2(\log(4)-1)}{\theta^2}
} 
\]
where $r$ and $h$ are defined in Lemma \ref{lem:fidirob}. 
\end{lemma}

\section*{Acknowledgments}
The authors would like to thank two anonymous referees and an Associate Editor for their constructive comments on an earlier version of this manuscript.
Moreover, they would like to thank Johan Segers (for providing the \texttt{R}-implementations of several estimators for the extremal index), Gregor Weiß (for providing the financial market data) and Daniel Ullmann and Peter Posch for fruitful discussions. 

This research has been supported by the Collaborative Research Center ``Statistical modeling of nonlinear dynamic processes'' (SFB 823) of the German Research Foundation, which is gratefully acknowledged. Parts of this paper were written when A.\ B\"ucher was a visiting professor at TU Dortmund University.

\bibliographystyle{chicago}
\bibliography{biblio}

\clearpage

\begin{center}
%
{\bfseries SUPPLEMENTARY MATERIAL ON  \\ [0mm] ``WEAK CONVERGENCE OF A PSEUDO MAXIMUM LIKELIHOOD  \\ [0mm] ESTIMATOR FOR THE EXTREMAL INDEX''}
\vspace{.5cm}

{\small BETINA BERGHAUS AND AXEL BÜCHER}

\end{center}

\begin{abstract}
Appendices~\ref{sec:disjoint} and \ref{sec:sliding} contain the proofs of the auxiliary lemmas in Section~\ref{sec:proofs} and \ref{sec:proofs2} from the main paper, respectively.  The proof of Theorem~\ref{theo:Northrop} is given in Appendix~\ref{sec:nor}, and additional results from the main paper are proven in Appendix~\ref{sec:addproofs}. Finally, additional simulation results are presented in Appendix~\ref{sec:addsim}.
\end{abstract}

\appendix

Throughout this supplement, $C$ and $C'$ denote generic constants whose values may change from line to line. The notation $o, o_\Prob, O, O_\Prob$ always refers to $n\to\infty$, if not  mentioned otherwise.  

\section{Remaining steps for the proof of Theorem~3.2 -- disjoint blocks} \label{sec:disjoint}

\begin{proof}[Proof of Lemma~\ref{lem:boundsupp}]
For some  $\eps\in(0,c_1 \wedge c_2)$, let $A_n=A_n(\eps)$ denote the event $\{\min_{i=1}^{\kn} N_{ni} > 1-\eps/2\} 
= \{ \max_{i=1}^{\kn} Z_{ni} < \eps \bn/2 \}$.  
By Condition~\ref{cond:cond}\eqref{item:blockdiv}, we have $\Prob(A_n) \to1$ as $n\to \infty$. 
We may write
\begin{align*}
D_n - D_{n,\ell} 
&=R_{n,\ell} \ind_{A_n} + o_\Prob(1)
\end{align*}
as $n \to \infty$, where, with  $I_{j}= \{(j-1)\bn +1, \dots, j\bn\}$ for $j=1, \dots, \kn$ (and $I_{j}=\varnothing$ else),
\begin{align*}
\textstyle R_{n,\ell} = \kn^{-3/2} \sum_{i=1}^{\kn} \sum_{j=1}^{\kn} \sum_{s\in I_{j}}f(U_s, Z_{ni}) g_{n,\ell}(Z_{ni})
\end{align*}
and 
\[
f(U_s, Z_{ni}) =  \ind (U_s > 1-\tfrac{Z_{ni}}{\bn}) - \tfrac{Z_{ni}}{\bn}, \qquad g_{n,\ell}(Z_{ni}) = \ind(\bn \eps/2 > Z_{ni}  \ge \ell).  
\]
Now, decompose $R_{n,\ell} = R_{n,\ell,0} + R_{n,\ell,1} + R_{n,\ell,-1}+R_{n,\ell,2}$ according to whether the second sum over $j$ is such that $j=i, j=i+1, j=i-1$ or $|j-i|\ge 2$, respectively. 
It suffices to show that $R_{n,\ell,0} \ind_{A_n}=o_\Prob(1)$ and $R_{n,\ell,\pm 1} \ind_{A_n}=o_\Prob(1)$ as $n\to \infty$, and that 
\begin{align} \label{eq:rn2}
\lim_{\ell \to \infty} \limsup_{n\to\infty} \Prob( | R_{n,\ell,2} \ind_{A_n} |> \delta) = 0
\end{align}
for all $\delta>0$.

First, since 
$
R_{n,\ell,0}=  \kn^{-3/2} \sum_{i=1}^{\kn} Z_{ni} \cdot g_{n,\ell}(Z_{ni}),
$
we obtain that
$\Exp|R_{n,\ell,0}| \le \kn^{\scs -1/2} \Exp|Z_{ni} | = o(1)$ as $n\to\infty$ by Condition~\ref{cond:cond}\eqref{item:moment}. 

Second,  we can write $R_{n,\ell,1}= \bar R_{n,\ell,1} - R_{n,\ell,0} = \bar R_{n,\ell,1} - o_\Prob(1)$, where
\[
\bar R_{n,\ell,1} = \kn^{-3/2} \sum_{i=1}^{\kn-1}  \sum_{s\in I_{i+1}} \ind(U_s > 1- \tfrac{Z_{ni}}{\bn}) g_{n,\ell}(Z_{ni}) 
\]
whence it suffices  to show that $\bar R_{n,\ell,1}\ind_{A_n}=o_\Prob(1)$.
For that purpose, define
\begin{align} \label{eq:useps}
U_s^\eps=U_s \ind(U_s>1-\eps)
, \qquad Z_{ni}^{\eps/2}=\bn(1-N_{ni}^{\eps/2}) = \bn (1- \max_{s\in I_i} U_s^{\eps/2}).
\end{align}
Note that $Z_{ni}^{\scs \eps/2}$ is $\Bc_{\scs \{(i-1)\bn+1\}:i\bn}^{\scs \eps/2}$ measurable, whence the mixing coefficients become available.
On the event $A_n$, we have $\bar R_{n,\ell,1} = \bar R_{n,\ell,1}^\eps$, where  $\bar R_{n,\ell,1}^\eps$ is defined exactly as $\bar R_{n,\ell,1}$, but with $U_s$ and $Z_{ni}$ replaced by $U_s^\eps$ and $Z_{ni}^{\scs \eps/2}$, respectively. 
By stationarity, we obtain
\[
\Exp|\bar R_{n,\ell,1}^\eps| = (\kn-1) \kn^{- 3/2} \sum_{s=1}^{\bn} \Exp\Big[\ind\big(U_{\bn+s}^\eps > 1- \tfrac{Z_{n1}^{\eps/2}}{\bn}\big) g_{n,\ell}(Z_{n1}^{\eps/2})\Big].
\]
Recall Theorem 3 in \cite{Bra83} (coupling for strongly mixing random variables): if $X$ and $Y$ are two random variables in some Borel space $S$ and $\mathbb{R}$, respectively, if $U$ is uniform on $[0,1]$ and independent of $(X,Y)$ and  if $q>0$  and $\gamma >0$ are such that $q \leq \Vert Y\Vert_\gamma=(\Exp|Y|^\gamma)^{1/\gamma}$, then there exists measurable function $f$ such that $Y^*=f(X,Y,U)$ has  the same distribution as $Y$, is independent of $X$ and satisfies 
\begin{align} \label{eq:bradley}
\Prob(\vert Y - Y^*\vert \ge q) \le  18 (\Vert Y \Vert_\gamma / q )^{\gamma/(2 \gamma+1)}\alpha(\sigma(X), \sigma(Y))^{2\gamma/(2\gamma+1)}.
\end{align}
Apply this theorem with $X=U_{\bn +s}^\eps$, $Y=Z_{n1}^{\eps/2}$, $\gamma = 2+ \delta$ and $q=q_n = \Vert Z_{n1}^{\eps/2} \Vert_{2+ \delta} $ to obtain that
\begin{multline*}
\Exp|\bar R_{n,\ell,1}^\eps| \le \kn^{-1/2} \sum_{s=1}^{\bn}\Big\{ \Exp[\ind(U_{\bn +s}^{\eps} > 1- \tfrac{Z_{n1}^{\eps/2*}+q_n}{\bn})] \\
+ 18  \cdot \alpha(\sigma(U_{\bn +s}^\eps), \sigma(Z_{n1}^{\eps/2})) ^{\tfrac{4+2\delta}{5+2\delta}}\Big\}
\end{multline*}
where $Z_{n1}^{\scs \eps/2*}$ is independent of $X$ and has the same distribution as $Z_{n1}^{\scs \eps/2}$. Note that $\alpha(\sigma(U_{\bn +s}^\eps), \sigma(Z_{n1}^{\scs \eps/2}))  \leq \alpha_{c_2}(s)$. Since $U_s^\eps\le U_s$, it follows that
\[
\Exp|\bar R_{n,\ell,1}^\eps| 
\le 
\kn^{-1/2} \bigg\{  \Exp[ Z_{n1}^{\eps/2*} ] +q_n +18 \times \sum_{s=1}^{\bn} \alpha_{c_2}(s) ^{\tfrac{4+2\delta}{5+2\delta}}\bigg\},
\]
which converges to $0$ by Conditions~\ref{cond:cond}\eqref{item:rate2} and \eqref{item:moment}.  To conclude, $R_{n,\ell,1}\ind_{A_n}=o_\Prob(1)$.

The sum $R_{n,\ell,-1}$ can be treated analogously so that it remains to show \eqref{eq:rn2}. Decompose
$R_{n,\ell,2}=\bar S_{n,\ell,1}+\bar S_{n,\ell,2}$ where
\begin{align*}
\bar S_{n,\ell,1}
&= 
\kn^{-3/2} \sum_{i=3}^{\kn}  \sum_{j=1}^{i-2} \sum_{s\in I_{j}} f(U_s, Z_{ni}) g_{n,\ell}(Z_{ni})
\end{align*}
and where $\bar S_{n,\ell,2}$ is defined analogously with the second sum ranging from $i+2$ to $\kn$. We will only treat $\bar S_{n,\ell,1}$ in the following, as $\bar S_{n,\ell,2}$ can be treated analogously. 
Recall \eqref{eq:useps} and note that, on the event $A_n$, we have $f(U_s, Z_{ni}) g_{n,\ell}(Z_{ni}) = f(U_s^\eps, Z_{ni}^{\scs \eps/2}) g_{n,\ell}(Z_{ni}^{\scs \eps/2})$.
Therefore, again on the event $A_n$,
\begin{align*}
\bar S_{n,\ell,1}
&= 
\kn^{-3/2} \sum_{i=3}^{\kn}  \sum_{j=1}^{i-2} \sum_{s\in I_{j}} f(U_s^\eps, Z_{ni}^{\eps/2}) g_{n,\ell}(Z_{ni}^{\eps/2})\\
&=
\frac{1}{\kn} \sum_{i=3}^{\kn}  e_{1:i-2}(Z_{ni}^{\eps/2}) g_{n,\ell}(Z_{ni}^{\eps/2}) 
=: 
\bar S_{n, \ell, 1}^\eps,
\end{align*}
where, for $p,q\in\{1, \dots, \kn\}$, $p<q$, and $x\ge 0$,
\[
e_{n, p:q} (x) =  \frac{1}{\sqrt{\kn}} \sum_{i=p}^q \sum_{s \in I_i} \{ \ind(U_s^\eps > 1-x/\bn) - x/\bn\}.
\]
We will show that \eqref{eq:rn2} is met with $R_{n,\ell,2}\ind_{A_n}$ replaced by $\bar S_{n, \ell, 1}^\eps$, and for that purpose we consider the first central moment of $\bar S_{n, \ell, 1}^\eps$.

Note that $|e_{1:j}(x) \ind(x \ge \ell) |\le j\bn/\sqrt \kn$ and that, for all $x,y\ge 0$ with $y-q \le x \le y+q$ for some $q>0$, we have
\[
|e_{1:j}(x) | \le |e_{1:j}(y+q)| \vee |e_{1:j}((y-q)\vee 0)| + 2 q  \sqrt{\kn},
\]
as can be shown by a case-by-case study and monotonicity arguments. The previous two inequalities, together with \eqref{eq:bradley} with $X=(U_1^\eps, \dots, U^\eps_{(i-1)\bn})$, $Y=Z_{ni}^{\scs \eps/2}$, $\gamma = 2+ \delta$ and $q =q_n = \Vert Z_{n1}^{\scs\eps/2} \Vert_{2+ \delta} /\sqrt{\kn }$,  imply that $\Exp[|\bar S_{n,\ell,1}^\eps|] $ is bounded above by
\begin{align*}
&\frac{1}{\kn} \sum_{i=3}^{\kn} \Exp\bigg[ \big\{  |e_{1:i-2}(Z_{ni}^{\scs \eps/2*} + q_n)| 
+ |e_{1:i-2}((Z_{ni}^{\scs \eps/2*} - q_n)\vee 0 )| + 2  \Vert Z_{n1}^{\scs \eps/2} \Vert_{2+ \delta}   \big\} \\ 
& \hspace{.5cm} \times \ind(\tfrac{\bn\eps}{2}+q_n>Z_{ni}^{\scs \eps/2*} \ge \ell - q_n)  \bigg]  
+ \frac{1}{\kn} 18\big ( \sqrt {\kn} \big ) ^{\tfrac{2+\delta}{5+2\delta}}\sum_{i=3}^{\kn} \frac{i \bn}{\sqrt {\kn}} \alpha_\eps(\bn)^{\tfrac{4+2\delta}{5+2\delta}},
\end{align*}
where $Z_{ni}^{\scs \eps/2*}$ is independent of $(U_1^\eps, \dots, U^\eps_{(i-1)\bn})$ and has the same distribution as $Z_{ni}^{\scs \eps/2}$.
The second sum on the right-hand side is of the order  (note that $\eta>3$)
\begin{align*}
O(\bn \kn^{1/2 + \tfrac{2+\delta}{10+4\delta}}\alpha_{c_2} (\bn)^{\tfrac{4+2\delta}{5+2\delta}})
&=
O( \kn^{ \tfrac{7+3\delta}{10+4\delta}}  \bn^{1-\eta {\tfrac{4+2\delta}{5+2\delta}}})  \\
&=
O( (\kn / \bn^{2})^{ \tfrac{7+3\delta}{10+4\delta}}  \bn^{-\tfrac{\delta}{5+2\delta}}) 
\end{align*} 
which converges to $0$ by Condition~\eqref{eq:rate1}. 

Since $\| Z_{n1}^{\scs \eps/2}\|_{2+\delta} \Prob(Z_{n1}^{\scs \eps/2} \ge \ell - q_n)$ converges to $0$ for $n\to \infty$ followed by $\ell \to \infty$,  it remains to consider the sums over 
\[
\Exp\left[ |e_{1:i-2}(Z_{ni}^{\scs \eps/2*} \pm q_n) \ind(\tfrac{\bn\eps}{2}+q_n>Z_{ni}^{\scs \eps/2*} \ge \ell - q_n) \right].
\]  
We only treat the sum involving the plus-sign. After conditioning on $Z_{ni}^{\scs \eps/2*}$ we are left with bounding $\Exp |e_{1:i-2}(z)|$ for $z\in[\ell,\eps\bn]$ (note that $\tfrac{\bn\eps}{2}+q_n>Z_{ni}^{\scs \eps/2*}$ implies that $Z_{ni}^{\scs \eps/2*}+ q_n \le \bn \eps/2 + 2q_n \le \bn \eps$ for sufficiently large $n$).
Decompose  $e_{1:i-2}=e_{1:i-2}^{\even} + e_{1:i-2}^{\odd}$ where $e_{1:i-2}^{\even}$ and $e_{1:i-2}^{\odd}$ denote the sum over the even and the odd blocks, respectively. It suffices to treat both sums separately, and we give the details for the sum over the even blocks.  
Let 
\[
V_j =V_j(z) =  \sum_{s\in I_{2j}} \{ \ind(U_s^{\eps} > 1-z/\bn) - z/\bn \}, 
\]
such that
$
e_{1:i-2}^{\even}(z) =\kn^{-\scs 1/2} \sum_{j=1}^{\scs \ip{i/2}-1}V_j.
$
Note that $\alpha(\sigma(V_j),\sigma(V_{j+1})) \le \alpha_{c_2}(\bn)$.
Repeatedly applying the coupling construction from \eqref{eq:bradley} above (with $\gamma=2$, $V_1^*=V_1$ and, in the $j$th step, $X=(V_1^*, \dots, V_j^*)$ and $Y=V_{j+1}$), together with Theorem 5.1 in \cite{Bra05}, we can inductively construct an iid sequence $(V_j^{*})_{j\ge1}$ such that $V_j^*$ has the same distribution as $V_j$ for any $j$ and such that
\[
\Prob(|V_j-V_j^*|\ge q_n') \le 18\cdot  \kn^{1/5} \alpha_{c_2}(\bn)^{4/5}, 
\]
where $q_n'=\| V_j \|_2 / \sqrt{\kn}$. Note that, since $z \le \eps\bn$, we have $\| V_j \|_2 \le C \sqrt{z+z^2}$ by Condition~\ref{cond:cond}\eqref{item:varbound}. Now
\begin{align*}
\Exp |e_{1:i-2}^{\even}(z)| 
\le 
\kn^{-1/2} \Exp\Big| \textstyle \sum_{j=1}^{ \ip{i/2}-1}V_j^*\Big| + i \kn^{-1/2} \Exp|V_j-V_j^*|. 
\end{align*}
Since $V_j^*$ is a centered iid sequence, we have the bound 
\[
 \Exp\Big| \textstyle \sum_{j=1}^{ \ip{i/2}-1}V_j^*\Big|
\le \bigg\{ \Var\Big( \textstyle \sum_{j=1}^{ \ip{i/2}-1}V_j^* \Big) \bigg\}^{1/2} 
\le
i^{1/2}\, \| V_j \|_2.
\]
By the Cauchy-Schwarz-inequality, we further have
\begin{align*}
 \Exp|V_j-V_j^*|  
 &\le 
q_n' + \Exp|V_j-V_j^*| \ind(|V_j-V_j^*|\ge q_n')   \\
&\le 
q_n' + 2 \| V_j \|_2\,\sqrt{18}\, \kn^{1/10} \alpha_{c_2}(\bn)^{2/5}.
\end{align*}
As a consequence, 
\[
\Exp |e_{1:i-2}^{\even}(z)| 
\le
\big\{ \sqrt{i/\kn}  + i \kn^{-1} + 9 \cdot i \kn^{-2/5} \alpha_{c_2}(\bn)^{2/5}  \big\} \| V_j\|_2
\]
for any $z\in[0,\eps \bn]$, where $\|V_j\|_2\le C \sqrt{z+z^2} \le C( 1 + z)$ by Condition~\ref{cond:cond}\eqref{item:varbound}. A similar bound for the sum over the odd blocks finally implies that
\begin{multline*}
\Exp\left[ |e_{1:i-2}(Z_{ni}^{\scs \eps/2*} + q_n) \ind(\tfrac{\bn\eps}{2}+q_n>Z_{ni}^{\scs \eps/2*} \ge \ell - q_n) \right]   \\
\le
C \big\{ \sqrt{i/\kn}  + i \kn^{-1} + 9 \cdot i \kn^{-2/5} \alpha_{c_2}(\bn)^{2/5}  \big\} \\
\times \Exp\left[ (1+ Z_{n1}^{\scs \eps/2*} + q_n) \ind( Z_{n1}^{\scs \eps/2*} \ge \ell -q_n) \right] 
\end{multline*}
after conditioning on $Z_{ni}^{\scs \eps/2*}$. Note that  the limes superior for $n\to \infty$ of the moment on the right-hand side can be made arbitrary small by increasing $\ell$. 
To finalize the treatment of $\Exp[|\bar S_{n,\ell,1}^\eps|]$ we are hence left with bounding the expression
\begin{align*}
\frac{1}{\kn} \sum_{i=3}^{\kn}  \big\{ \sqrt{i/\kn}  + i \kn^{-1} + 9 \cdot i \kn^{-2/5} \alpha_{c_2}(\bn)^{2/5} \big\} 
\le C + C' \cdot \kn^{3/5}\alpha_{c_2}(\bn)^{2/5}.
\end{align*}
Since $\alpha_{c_2}(\bn)^{2/5}=O(\bn^{-2\eta/5}) = O(\bn^{-6/5})$, we obtain that $\kn^{3/5}\alpha_{c_2}(\bn)^{2/5} = O((\kn/\bn^2)^{3/5})$, which converges to zero under the assumption that $\kn / \bn^2=o(1)$. 
\end{proof}

\begin{proof}[Proof of Lemma~\ref{lem:appleb}]  
Recall that $H(x)=1-\exp(-\theta x)$. We have to show that
\[
\int_0^\ell \tailp(x) \diff(\hat H_{\kn} - H)(x) = o_\Prob(1), \qquad n\to\infty,
\]
which follows from Lemma~C.8 in \cite{BerBuc17}, provided we can show that 
\[
\sup_{x\in[0,\ell]} | \hat H_{\kn}(x) - H(x) | = o_\Prob(1), \qquad n\to\infty.
\]
The last display in turn follows from pointwise convergence (in probability) of $\hat H_{\kn}$ to $H$ by a standard Gilvenko-Cantelli-type argument. For the pointwise convergence, note that $\Exp[\hat H_{\kn}(x)] = H_{\kn}(x) :=\Prob(Z_{n1} \le x) \to H(x)$ by \eqref{eq:exp2}. By similar arguments as in the proof of Proposition 3.1 in \cite{RobSegFer09} (but under slightly different assumptions) it can be shown that 
\begin{align*} 
\lim_{n\to\infty}\kn \Var\{ \hat H_{\kn}(x) \}  = e^{-\theta x}(1-e^{-\theta x}).
\end{align*}
This implies pointwise convergence in probability and hence the Lemma. 
\end{proof}

\begin{proof}[Proof of Lemma~\ref{lem:fidirob}]  
Note that weak convergence of the first $m$ components of the vector follows from Theorem~4.1 in \cite{Rob09}. Regarding joint convergence with the $(m+1)$st component, we only consider the case $m=1$ and set $x_1=x$; the general case can be treated analogously.

Recall the definition of $\ell_n$ in Condition~\ref{cond:cond}\eqref{item:rate2}. Decompose blocks $I_i = I_i^{+} \cup I_i^-$, where
\[
I_i^+ = \{(i-1)\bn + 1, \dots, i\bn - \ell_n\}, \qquad I_i^-=\{ i\bn -\ell_n + 1, \dots, i\bn\}.
\]
and let
\begin{align*}
\tailp^+(x) &= \kn^{-1/2} \sum_{i=1}^{\kn} \sum_{s \in I_i^+} \{ \ind(U_s > 1-x/\bn) - x/\bn\} \\
G_n^+ &= \kn^{-1/2} \sum_{i=1}^{\kn} Z_{ni}^+ - \Exp[Z_{ni}^+], \qquad Z_{ni}^+ = \bn(1-\max_{s\in I_i^+} U_s).
\end{align*}
As a consequence of Lemma 6.6 in \cite{Rob09}, $\tailp^-(x) = \tailp(x) - \tailp^+(x) = o_\Prob(1)$. Let us show the same for $G_n$. Denote $G_n^-=G_n-G_{n}^{+}$ and $Z_{ni}^- = Z_{ni}-Z_{ni}^+$. 
For $\eps\in(0,c_1\wedge c_2)$, let $A_n^+=\{\min_{i=1}^{\kn} N_{ni}^+ > 1-\eps\}$ and note that $\Prob(A_n^+) \to 1$ by Condition~\ref{cond:cond}\eqref{item:blockdiv}. It then suffices to show that $G_n^- \ind_{A_n^+} = o_\Prob(1)$.  We can write $G_n^- \ind_{A_n^+} = \tilde G_n^- \ind_{A_n^+} = \tilde G_n^- + o_\Prob(1)$, where
\[
\tilde G_n^- = \kn^{-1/2} \sum_{i=1}^{\kn}\{  Z_{ni}^- - \Exp[Z_{ni}^-]\} \ind(N_{ni}^+ > 1-\eps)  
\]
Now, $N_{ni}^+>1-\eps$ implies that $Z_{ni}^- = Z_{ni}^{\eps-}$, where the latter variable is defined in terms of the $U_i^\eps$ instead of the $U_i$. Hence, $\tilde G_n^- =  \kn^{\scs -1/2} \sum_{i=1}^{\kn} S_{ni}^\eps$, where
\[
S_{ni}^\eps = \{ Z_{ni}^{\eps-} - \Exp[Z_{ni}^-]\} \ind(N_{ni}^{\eps +} > 1-\eps)  
\]
is $\Bc_{\{(i-1)\bn+1\}:(i\bn)}^\eps$-measurable.  As a consequence, by stationarity 
\begin{align} \label{eq:b1}
\Var(\tilde G_n^- ) & = \Var(S_{n1}^{\eps} ) + \frac{2}{\kn} \sum_{i=1}^{\kn} (\kn - i) \Cov( S_{n1}^{\eps}  , S_{n,1+i}^{\eps} ) \nonumber \\
&\le 3 \Var(S_{n1}^{\eps} ) + \frac{2}{\kn} \sum_{i=2}^{\kn} (\kn - i) \Cov( S_{n1}^{\eps}  , S_{n,1+i}^{\eps} )
\end{align}
Let us first show that $\Var(S_{n1}^\eps)=o(1)$ as $n\to\infty$, which would follow, if we show that, for any $p\in(2,2+\delta)$, $|Z_{n1}^{\eps-} | \le  |Z_{n1}^{-}| \to 0$ in $L_p$ (the inequality follows by studying the cases $N_{ni}^+> 1-\eps$ and $\le 1-\eps$).  Since $\ell_n=o(\bn)$ we have, for any $y>0$,
\begin{align} \label{eq:zn0}
\Prob(Z_{n1}^{-}  \ne 0)  
&= 
\Prob\Big( \max_{s\in I_1} U_s > \max_{s\in I_1^+} U_s \Big) \\
&\le 
\Prob\Big( \max_{s=1}^{\bn - \ell_n} U_s \le 1-y/\bn \Big) + \Prob\Big( \max_{s=1}^{\ell_n} U_s> 1- y/\bn\Big) \nonumber \\
&\le \Prob \Big( Z_{1:\bn - \ell_n} \ge y (\bn-\ell_n)/\bn \Big) + \ell_n y/\bn  \nonumber  \\
&\to \exp(-\theta y), \nonumber
\end{align}
which can be made arbitrary small by increasing $y$. Hence, $Z_{n1}^- = o_\Prob(1)$. Since $\Exp|Z_{n1}^- |^p \le C \Exp | Z_{1:\bn - \ell_n} |^p < \infty$ for any $p\in(2,2+\delta)$ by Condition~\ref{cond:cond}\eqref{item:moment}, we can conclude that $Z_{n1}^-  \to 0$ in $L_p$.

It remains to treat the sum over the covariances on the right-hand side of \eqref{eq:b1}. By Lemma 3.11 in \cite{DehPhi02} (which is a slightly more general version of Lemma~6.3 in \citealp{Rob09}), for any $p\in(2,2+\delta)$,
\[
|\Cov(S_{n1}^\eps , S_{n,1+i}^{\eps} )| \le 10 (\Exp|S_{n1}^\eps |^p)^{2/p} 
\alpha_{c_2}( (i-1) \bn ) ^{1-2/p}
\]
(note that $S_{ni}^\eps$ is $\Bc_{\scs ( i\bn -\bn +1):(i\bn)}^\eps$-measurable). 
Now, for $i\ge 2$,  $\alpha_{c_2}((i-1)\bn ) \le \alpha_{c_2}(i-1) \le C (i-1)^{-\eta}$ by monotonicity of $\alpha_{c_2}(\ell)$. 
The sum over the covariances in \eqref{eq:b1} can thus be bounded by a multiple of 
\[
(\Exp|S_{n1}^\eps|^p)^{2/p} \sum_{i=2}^{\kn} \alpha_{c_2}((i-1)\bn)^{1-2/p}
\le
(\Exp|S_{n1}^\eps|^p)^{2/p} \sum_{i=1}^{\infty} i^{-\eta(1-2/p)}.
\]
The series converges and the moment converges to $0$ by arguments as given above.

Now, since $(\tailp^-(x),G_n^-)=o_\Prob(1)$ and $\Prob(A_n^+)\to1$,  it suffices to show that $(\tailp^+(x),G_n^+)\ind_{A_n^+} $ converges weakly to the claimed normal distribution. This in turn follows from the Cram\'er-Wold device, provided we show that for any $\lambda_1, \lambda_2\in\R$
\[
(\lambda_1 \tailp^+(x) + \lambda_2 G_n^+ )\ind_{A_n^+} 
\dto
\lambda_1 \tailplim(x) + \lambda_2 G.
\]
The left-hand side can be written as $(\kn^{-1/2} \sum_{i=1}^{\kn} \tilde f_{i,n} ) \ind_{A_n^+} = \kn^{-1/2} \sum_{i=1}^{\kn} \tilde f_{i,n} + o_\Prob(1)$, 
where $\tilde f_{i,n} = f_{i,n} \ind(Z_{ni}^+<\eps\bn)$ and
\[
f_{i,n} = \lambda_1 \sum\nolimits_{s \in I_{i}^+} \{ \ind(U_s>1-x/\bn) - x/\bn\} + \lambda_2 (Z_{ni}^+- \Exp[Z_{ni}^+]).
\]
Note that $\tilde f_{i,n}$ is $\Bc_{\scs \{ (i-1)\bn+1\}:\{i\bn-\ell_n\}}^\eps$-measurable. A standard argument based on characteristic functions (see, e.g., the proof of Lemma~6.7 in \citealp{Rob09}) shows that the weak limit of  $\kn^{\scs -1/2} \sum_{i=1}^{\kn} \tilde f_{i,n}$ is the same as if the $(\tilde f_{i,n})_{i=1,\dots, \kn}$ were considered as iid.
Now, 
\[
\frac{\sum_{i=1}^{\kn}  \Exp[|\tilde f_{i,n}|^{p}] }{ \big(\sum_{i=1}^{\kn}  \Exp[|\tilde f_{i,n}|^{2}] \big)^{p/2}}
=
\kn^{1-p/2} \frac{  \Exp[|\tilde f_{i,n}|^{p}] }{\Big( \Exp[|\tilde f_{i,n}|^{2}] \big)^{p/2}}. 
\]
By Minkowski's inequality, for any $p\in(2,2+\delta)$, $\sup_{n} \Exp[|\tilde f_{1,n}|^{p}] < \infty$ by Condition~\ref{cond:cond}\eqref{item:moment} and \eqref{item:momn}. 
As a consequence, provided $\lim_{n\to\infty} \Exp[\tilde f_{1,n}^2]$ exists, Ljapunov's condition is satisfied (\citealp{Bil79}, Theorem 27.3) and $\kn^{\scs-1/2} \sum_{i=1}^{\kn} \tilde f_{i,n}$ converges to a normal distribution with variance equal to $\lim_{n\to\infty} \Exp[\tilde f_{1,n}^2].$ 

The latter limit is equal to $\lim_{n\to\infty} \Exp[f_{1,n}^2]$, whence it remains to be shown that
\[
\lim_{n\to\infty} \Exp[f_{1,n}^2] = \lambda_1^2 r(x,x) +  2 \lambda_1 \lambda_2 h(x) + \lambda_2^2 /\theta^2,
\]
which in turn follows, 
observing the expressions for the limiting covariances $r(x,x)$ in Theorem~4.1 in \cite{Rob09}, from
\begin{align*}
&\textstyle \lim_{n\to\infty} \Cov \big\{ \sum_{s\in I_{1}^+} \ind(U_s>1-x/\bn), \bn(1-\max_{s \in I_{1}^+} U_s)  \big\} = h(x), \\
&\textstyle \lim_{n\to\infty} \Var\big\{  \bn(1-\max_{s \in I_{1}^+}U_s)  \big\} = \theta^{-2} .
\end{align*}
Repeating arguments from above, we may replace the set $I_{1}^+$ by $I_1$ in the preceding display, whence it is in fact sufficient to show that
\[
\lim_{n\to\infty} \Cov ( N_{n}^{(x)}(E),Z_{1:n}) = h(x), \qquad
\lim_{n\to\infty} \Var( Z_{1:n}) = \theta^{-2}. 
\]
By an application of Theorem~2.20 in \cite{Van98}, the second assertion follows directly from $Z_{1:n} \dto \exp(\theta)$ and Condition~\ref{cond:cond}\eqref{item:moment}. For the first convergence, abbreviate $N_n^{\scs (x)}=N_n^{\scs(x)}(E)$ and note that
\[
\Prob(N_{n}^{(x)}=i, Z_{1:n}>y ) = \Prob(N_n^{(y)} = 0, N_n^{(x)} = i) 
\to 
\begin{cases}
p_2^{(x,y)}(i,0) & x \ge y \ge 0 \\
0 & y > x \ge 0,
\end{cases}
\]
see \cite{Per94, Rob09}, that is, $(N_n^{(x)},Z_{1:n})$ converges jointly. By uniform integrability, we may deduce that
\begin{align*}
\Exp[N_n^{(x)}Z_{1:n}] 
&=
\sum_{i=1}^\infty i  \int_0^\infty \Prob(Z_{1:n}>y, N_{n}^{(x)}=i) \, \diff y  
\to
\sum_{i=1}^\infty i \int_0^x p_2^{(x,y)}(i,0) \, \diff y.
\end{align*}
The lemma finally follows from $\Exp[Z_{1:n}] \to \theta^{-1}$ and $\Exp[N_n^{\scs(x)}] \to x$.
\end{proof}

\begin{proof}[Proof of Lemma~\ref{lem:tailpweak}] 
This follows from a slight extension of Theorem~4.1 in \cite{Rob09}, with $\sigma=0$ in his notation. Indeed, a careful look at his proof shows that one may set $\sigma=0$ everywhere (whenever the last coordinate of his vector of processes $E_{m,n}$ is concerned).
\end{proof}

\begin{proof}[Proof of Lemma~\ref{lem:weakdnlgn}] 
As a consequence of Lemma~\ref{lem:appleb}, Lemma~\ref{lem:tailpweak} and the continuous mapping theorem, we have
\[
D_{n,\ell} + G_n 
= 
\theta \int_0^\ell \tailp(x) \, e^{-\theta x} \, \diff x + G_n + o_\Prob(1) 
\dto 
\theta \int_0^\ell \tailplim(x)  \, e^{-\theta x} \, \diff x + G.
\]
The right-hand side is normally distributed with variance $\sigma^2_\ell$.
\end{proof}

\begin{proof}[Proof of Lemma~\ref{lem:convvar}] 
Since
\begin{align} \label{eq:siginf}
\lim_{\ell \to \infty} \sigma_\ell^2 &= \sigma_\infty^2\\ &=  \theta^2 \int_0^\infty \int_0^\infty r(x,y) e^{-\theta (x+y)}\, \diff x \diff y + 2 \theta \int_0^\infty h(x) e^{-\theta x} \, \diff x + \theta^{-2},\nonumber
\end{align}
we only have to show, that $\sigma_\infty^2 = \sigma_{\djb}^2$.
First of all, note that, for $x> y $,
\[
r (x,y)= \theta x \sum_{i=1}^\infty \sum_{j=0}^i ij \pi_2^{\scs(y/x)}(i,j) = \theta x \Exp [\zeta_1^{ \scs (y/x)} \zeta_2^{\scs (y/x)}],
\]
where $(\zeta_1^{\scs (y/x)}, \zeta_2^{\scs (y/x)}) \sim \pi_2^{\scs(y/x)} $. Using this representation and substituting $\sigma =\frac{y}{x}$ we obtain
\begin{align*}
 & \theta^2 \int_0^\infty \int_0^\infty r(x,y) e^{-\theta (x+y)}\, \diff x \diff y   \\
& \hspace{1cm}=
 2 \theta^2 \int_0^\infty \int _0^x \theta x \Exp [\zeta_1^{\scs (y/x)} \zeta_2^{\scs (y/x)} ] e^{-\theta (x+y)}\, \diff x \diff y \\
& \hspace{1cm}=
 2 \theta^3 \int_0^1 \Exp [ \zeta_1^{\scs (\sigma)} \zeta_2^{\scs (\sigma)}] \int _0 ^\infty x^2 e^{-\theta (1+\sigma)x} \, \diff x \, \diff \sigma
=
4 \int_0^1 \frac{\Exp [ \zeta_1^{\scs(\sigma)} \zeta_2^{\scs (\sigma)}]}{(1+ \sigma)^3} \, \diff \sigma,
\end{align*}
which is exactly the first summand in $\sigma_{\djb}^2$.

Consider the second integral in $\sigma_\infty^2$. By the definition of $p_2^{\scs (x,y)} $ in Lemma \ref{lem:fidirob} we have
\begin{align*}
\sum_{i=1}^\infty i p_2^{(x,y)}(i,0) = \Exp \Big [ \sum_{j =1}^\eta \zeta_{j1}^{\scs (y/x) } \ind \Big ( \sum_{j =1}^\eta \zeta_{j2}^{\scs (y/x) }=0\Big ) \Big ],
\end{align*}
where $\eta \sim \text{Poisson}(\theta x)$ is independent of iid random vectors 
$(\zeta_{i1}^{\scs(y/x)}, \zeta_{i2}^{\scs(y/x)})\sim \pi_2^{\scs(y/x)}, i\in\N$. With the identity $\Prob (\zeta_{12}^{\scs (\sigma)} =0 ) = 1- \sigma$, which we will show later, the latter expectation can further be rewritten as
\begin{align} \label{eq:condexp}
&\hspace{-.4cm} \sum_{k =1}^\infty \Exp  \Big [ \sum_{j =1}^k \zeta_{j1}^{\scs (y/x) } \ind \Big ( \sum_{j =1}^k \zeta_{j2}^{\scs (y/x) }=0\Big ) \Big ] \Prob(\eta = k) 
\nonumber \\
&=
 \sum_{k =1}^\infty  k \Exp  \big[\zeta_{11}^{\scs (y/x) } \ind  (  \zeta_{12}^{\scs (y/x) }=0 ) \big]  \Prob (\zeta_2^{\scs(y/x)} =0)^{k-1}\Prob(\eta = k) 
 \nonumber\\
&=  
\sum_{k =1}^\infty  k \Exp  \big[\zeta_{11}^{\scs (y/x) } \ind  (  \zeta_{12}^{\scs (y/x) }=0 ) \big] (1-y/x)^{k-1} \frac{(\theta x)^k}{k !}  e^{- \theta x} 
\nonumber\\
&= 
\Exp  \big[\zeta_{11}^{\scs (y/x) } \ind  (  \zeta_{12}^{\scs (y/x) }=0 ) \big] \theta x e^{- \theta y }.
\end{align}
Hence, substituting $\sigma=y/x$,
\begin{align} \label{eq:hdjb}
h(x) = \theta x^2 \int _0^1 \Exp \big[ \zeta_1^{\scs (\sigma)} \ind (\zeta_2^{\scs (\sigma)} =0 )\big]  e^{- \theta \sigma x } \, \diff \sigma - \frac{x}{\theta}
 \end{align}
and therefore
\[
 2 \theta \int_0^\infty h(x) e^{-\theta x} \, \diff x +\theta^{-2} = 4 \theta^{-1} \int_0^1 \frac{\Exp [ \zeta_1^{\scs (\sigma)} \ind (\zeta_2^{\scs (\sigma)} =0 )] }{(1+ \sigma)^3} \, \diff \sigma - \theta^{-2},
\]
which corresponds to the remaining summands in $\sigma_{\djb}^2$.

It remains to be shown that  
\begin{align}
\Prob (\zeta_{12}^{\scs (\sigma)} =0 ) = 1- \sigma. \label{eq:xi2}
\end{align}
By the definition of $\pi_2^{(\sigma)}$ in Section~\ref{sec:pre}, we have
\begin{align*}
\Prob(\zeta_2^{\scs(\sigma)} =0 ) 
&= 1- \Prob(\zeta_2^{\scs(\sigma)} > 0 )  \\
&=
1-  \lim_{n \to \infty }  \Prob (N_n^{\scs ( \sigma x)}(B_n)>0 \vert N_n^{\scs ( x)}(B_n) >0 ) \\
&=
1-  \lim_{n \to \infty }  \Prob (N_{1:q_n}>1-\tfrac{\sigma x }{n} \vert N_{1:q_n}>1-\tfrac{x }{n} ) \\
&=
1-  \lim_{n \to \infty }  \Prob \Big  (\tfrac{N_{1:q_n}-(1-\tfrac{x}{n})}{\tfrac{x}{n}}>1-\sigma \big \vert N_{1:q_n}>1-\tfrac{x }{n} \Big ).
\end{align*}
Finally, by \eqref{eq:positive}, we can use identity (10.21) in \cite{BeiGoeSegTeu04}, which is an implication of  Theorem 3.1 in \cite{Seg05}, to deduce that, as $n \to \infty$,
\begin{multline*}
 \Prob \Big  (\tfrac{N_{1:q_n}-(1-\tfrac{x}{n})}{\tfrac{x}{n}}>1-\sigma \big \vert N_{1:q_n}>1-\tfrac{x }{n} \Big )\\
=
 \Prob \Big  (\tfrac{U_1-(1-\tfrac{x}{n})}{\tfrac{x}{n}}>1-\sigma \big \vert U_1>1-\tfrac{x }{n} \Big ) + o(1),
\end{multline*}
which converges to $\sigma$ as asserted.
\end{proof}

\section{Remaining steps for the proof of Theorem~3.2 -- sliding blocks} \label{sec:sliding}

\begin{proof}[Proof of Lemma~\ref{lem:boundsupp2}]  
The proof is similar to the proof of Lemma \ref{lem:boundsupp}, whence we only give a sketch proof. For some $\eps\in(c_1,c_2)$ let $A_n'=A_n'(\eps)$  denote the event $\{ \min_{t=1}^{n-\bn+1} N_{nt} >  1- \eps \}$. Note that $\Prob(A_n') \to 0$ by Condition~\ref{cond:cond}\eqref{item:blockdiv}. Recalling the definition of $f$ from the beginning of the proof of Lemma \ref{lem:boundsupp}, we may then write
$
D_{n}^{\slb} - D_ {n,\ell}^{\slb} = R_{n,\ell}^{\slb} \ind_{A_n'} + o_\Prob(1),
$
where
\[
R_{n,\ell}^{\slb}
= \kn^{-3/2} \sum_{i=1}^{\kn-1} \sum_{j=1}^{\kn} \sum_{s \in I_j} \bn^{-1 } \sum_{t \in I_i} f(U_s,Z_{nt}^{\slb})\ind(Z_{nt}^{\slb}\geq \ell).
\]
Now, decompose $R_{n,\ell}^{\slb} = R_{n,\ell,2}^{\slb} +R_{n,\ell,3}^{\slb}$ according to whether the second sum over $j$ is such that $|j-i|\le2$ or $|j-i|\ge 3$, respectively. 
Similar as in the proof of Lemma~\ref{lem:boundsupp}, it can be shown that $R_{n,\ell,2}^{\slb}\ind_{A_n'}=o_\Prob(1)$  and that $\lim_{\ell \to \infty} \limsup_{n\to\infty} \Prob( | R_{n,\ell,3}^{\slb} \ind_{A_n'} |> \delta) = 0$. 
\end{proof}

\begin{proof}[Proof of Lemma~\ref{lem:appleb2}]  
As in the proof of Lemma \ref{lem:appleb} the result follows if we can show that $\Var \{ \hat H_{\kn}^{\slb}(x) \} = o(1)$ for any $x \in [0,\ell]$. This in turn follows from similar arguments as in the proof of Proposition 3.1 in \cite{RobSegFer09}.
\end{proof}

\begin{proof}[Proof of Lemma~\ref{lem:fidirob2}]  
For notational convenience, we will only show the joint weak convergence of $(e_n(x), G_n^{\slb})$ for some fixed $x>0$; the general case can be shown analogously.  Let $A_n'=\{\min_{t=1}^{n-\bn +1} N_{nt}^{\slb} > 1-\eps\}$, where $\eps\in(0,c_1\wedge c_2)$ and note that $\Prob(A_n')\to1$ as $n\to\infty$. Due to the Cramér-Wold device it suffices to prove that, for any $\lambda_1,\lambda_2 \in \mathbb{R}$,
\[
\{ \lambda_1 e_n(x) + \lambda_2 G_n^{\slb}\} \ind_{A_n'} \dto \lambda_1 e(x) + \lambda_2 G^{\slb}.
\]
We may write
\begin{align*}
&\lambda_1 e_n(x) + \lambda_2 G_n^{\slb} \\
&\quad=
 \tfrac{\lambda_1}{ \kn^{1/2}} \sum_{ s=1}^n  \{ \ind(U_s > 1-\tfrac{x}{\bn}) - \tfrac{x}{\bn} \} 
 +\tfrac{ \lambda_2 \kn^{1/2}}{n-\bn+1} \sum_{s=1}^{n- \bn+1}  \{ Z_{ns}^{\slb} - \Exp[Z_{n1}^{\slb}]\} \\
&\quad= 
\sum_{j=1}^{\kn-1} \sum_{s\in I_j} \Big [   \tfrac{\lambda_1}{ \kn^{1/2}}  \{ \ind(U_s > 1-\tfrac{x}{\bn}) - \tfrac{x}{\bn} \}
+ \tfrac{ \lambda_2 \kn^{1/2}}{n-\bn+1}  \{ Z_{ns}^{\slb} - \Exp[Z_{n1}^{\slb}] \} \Big ]+ o_\Prob(1),
\end{align*}
where the $o_\Prob$ is due to omitting summands from the last block.
Choose some integer sequence $\kn^*<\kn$ such that $\kn^* \to \infty $ and $\kn^* =o(\kn^{\scs \delta/\{2(1+\delta)\}})$ as $n \to \infty$, where $\delta$  is defined in Condition \ref{cond:cond}\eqref{item:momn}. Moreover, set $q_n^* = \lfloor \kn/(\kn^*+2) \rfloor$. For $j =1, \dots , q_n^*$, define 
\[
J_j^+ = \textstyle\bigcup_{i = (j-1)(k^*_n +2) +1}^{j(k^*_n+2)-2}I_i 
\qquad \text{ and }\qquad 
J_j^- = I_{j(k^*_n +2)-1} \cup I_{j(k^*_n+2)},
\]
i.e., we combine $\kn^*$ consecutive $I_i$-blocks in one big block $J_j^{\scs+}$ of size $\kn^*\bn$ and each of the big blocks is separated by a small block $J_j^-$ of size $2\bn$, formed by merging two consecutive $I_i$-blocks. With this notation we obtain  
\begin{align*} 
\lambda_1e_n(x) + \lambda_2 G_n^{\slb}
 =
 H_n^+ + H_n^- + o_\Prob(1), \qquad H_n^\pm =  \frac{1}{\sqrt {q_n^*}} \sum_{j=1}^{q_n^*} S_{nj}^\pm,
\end{align*}
where, for $j=1,\dots , q_n^*$,
\begin{multline*}
S_{nj}^{\pm} =\sqrt {\frac{q_n^*}{\kn}}  \sum\nolimits_{s \in J_j^{\pm}}\Big[  \lambda_1 \{ \ind(U_s > 1-\tfrac{x}{\bn}) - \tfrac{x}{\bn} \} \\
+ \frac{\lambda_2 n}{n-\bn+1}  \frac1{\bn}  \{ Z_{ns}^{\slb} - \Exp[Z_{n1}^{\slb}] \} \Big] .
\end{multline*}

First, we will show that $H_n^- \ind_{A_n'}=o_\Prob(1)$. 
As in the proof of Lemma~\ref{lem:fidirob} we have $H_n^- \ind_{A_n'}= \tilde H_n^- \ind_{A_n'}  + o_\Prob(1)=  \tilde H_n^- + o_\Prob(1)$, where $\tilde H_n^-$ is defined exactly as $H_n^-$, but  with $S_{nj}^-$ replaced by
\begin{multline*}
S_{nj}^{\eps-} =\sqrt {\frac{q_n^*}{\kn}}  \sum\nolimits_{s \in J_j^{\pm}}\Big[  \lambda_1 \{ \ind(U_s^\eps > 1-\tfrac{x}{\bn}) - \tfrac{x}{\bn} \} \\
+ \frac{\lambda_2 n}{n-\bn+1}  \frac1{\bn}  \{ Z_{ns}^{\eps, \slb} - \Exp[Z_{n1}^{\slb}] \} \Big],
\end{multline*}
with $Z_{ns}^{\eps, \slb} = \bn (1-\max_{s=t}^{t+\bn-1} U_s^\eps)$. By an inequality similar to \eqref{eq:b1} and the argumentation subsequent to that inequality, it suffices to show that $\|S_{n1}^{\eps-}\|_p=o(1)$ for some $p\in(2,2+\delta)$ and that 
$
\sum_{j=2}^{q_n^*} | \Cov( S_{nj}^{\eps-}  , S_{n,1+j}^{\eps-} ) | =o(1).
$
The first assertion follows from
\begin{align*}
\Vert S_{1j}^{\eps-} \Vert_{p} 
&\le 2\sqrt {\frac{q_n^*}{\kn}}  \Big \{ \lambda_1 \Vert N_{\bn \kn^*}^{(x)}(E)\Vert_{p}+ \lambda_2 \Vert Z_{n1}^{\eps,\slb} - \Exp[Z_{n1}^{\slb}] \Vert_{p} \Big  \} \\
&= O(1/\sqrt{\kn^*})=o(1),
\end{align*}
by Condition \ref{cond:cond}\eqref{item:momn} and \eqref{item:moment} and the definition of $q_n^*$. For the second assertion,
note that $S_{nj}^{\eps-}$ is $\Bc_{\{ (j \kn^* + 2j -2)\bn +1 \}:\{ j(\kn^*+2)\bn\}}^\eps$-measurable, whence
\[
| \Cov( S_{nj}^{\eps-}  , S_{n,1+j}^{\eps-} ) |  \le 10\Vert S_{n1}^{\eps-} \Vert_{p}^2.
\alpha_{c_2}( j \kn^* \bn ) ^{1-2/p}
\]
By Condition \ref{cond:cond}\eqref{item:rate2} the sum $\sum_{j=2}^{q_n^*} \alpha_{c_2}( j \kn^* \bn ) ^{1-2/p}$ converges to 0, which implies the assertion.

It remains to be shown $H_n^+\ind_{A_n'}$ converges to a normal distribution with the claimed covariance. As in the proof of Lemma~\ref{lem:fidirob}, we can write
\[
H_n^+\ind_{A_n'} = \frac1{\sqrt{q_n^*}} \sum_{j=1}^{q_n^*} \tilde S_{nj}^{+} + o_\Prob(1), \qquad \tilde S_{nj}^{+} = S_{nj}^{+} \ind(\max\nolimits_{s\in J_j^{+}} Z_{ns}^{\slb}< \eps\bn).
\]
For $i\ne j$, the observations $\tilde S_{nj}^{+}$ and $\tilde S_{ni}^{+}$ are separated by at least one block of size $\bn$ and measurable with respect to the $\Bc_{\cdot:\cdot}^\eps$-sigma fields. Further, by Condition~\ref{cond:cond}\eqref{item:rate2}, $q_n^*\alpha_{c_2} (\bn) \le \kn \alpha_{c_2}(\bn) =o(1)$. A standard argument for the characteristic function then shows that the weak limit of $(q_n^*)^{-1/2}\sum_{j=1}^{q_n^*} \tilde S_{nj}^{\scs+}$ is the same  as if the sample $(\tilde S_{nj}^{\scs+})_{j=1, \dots, q_n^*}$ was independent, which we will assume subsequently. By arguments as before, we can then pass back to an independent sample $(S_{nj}^{\scs+})_{j=1, \dots, q_n^*}$, and weak convergence follows from the classical central limit theorem for rowwise iid triangular arrays.

By Condition~\ref{cond:cond}\eqref{item:momn} and  \eqref{item:moment} and Minkowski's inequality, we have that $\Exp[|S_{nj}^+|^{2+\delta}] = O({\kn^*}^{(2+\delta)/2})$. 
Hence,
\begin{align*}
\frac{\sum_{j=1}^{q_n^*} \Exp[|S_{nj}^+|^{2+\delta}] }{ \big(\sum_{j=1}^{q_n^*}   \Exp[||S_{nj}^+|^2] \big)^{\frac{2+\delta}{2}}}
&=
{q_n^*} ^{-\delta /2} \frac{  \Exp[|S_{nj}^+|^{2+\delta}]  }{\Exp[|S_{nj}^+|^2] ^{\frac{2+\delta}{2}}}  \\
&=
O(\kn^{-\delta/2}{\kn^*}^{1+\delta}) 0 =o(\kn^{-\delta/2+\delta/2}) = o(1), 
\end{align*}
by the definition of $\kn^*$, provided that $\lim_{n\to\infty} \Exp[(S_{n1}^+)^2]$ exists (which we will show below).
 Therefore, Ljapunov's condition is satisfied and $\lambda_1e_n(x) + \lambda_2 G_n^{\slb}$ converges weakly to a normal distribution with variance $\lim_{n\to\infty} \Exp[(S_{n1}^+)^2]$. Hence, it remains to be shown that 
\[ 
\lim_{n\to\infty} \Exp[(S_{n1}^+)^2]= \lambda_1^2 r(x,x) + 2\lambda_1\lambda_2 h(x) +\lambda_2^2  \tfrac{2(\log(4)-1)}{\theta^2}
\]
and this in turn follows from the proof of Theorem 4.1 in \cite{Rob09} (for the first summand in the latter display) and Lemma \ref{lemma:covEZ}, \ref{lemma:covZZ} and \ref{lem:convvar2} below (note that,  with $n^*=\kn^* \bn$, we can write $S_{n1}^+=\lambda_1 e_{n^*} + \lambda_2 G_{n^*}^{\slb}+o_\Prob(1)$  and that all assumptions in Condition~\ref{cond:cond} are satisfied if $n$ and $\kn$ are replaced by $n^*$ and $\kn^*$).
\end{proof}

\begin{lemma}\label{lemma:covEZ}
Suppose Conditions \ref{cond:cond}\eqref{item:momn}, \eqref{item:rate2} and \eqref{item:moment} are met. Then, for any $x \in [0,\infty)$, as $n \to \infty$, 
\[
\Cov(e_n(x) , G_n^{\slb}) \to h_ {sl}(x),
\]
where $h_{\slb}(0)=0$ and, for $x\ne 0$,
\begin{multline*}
h_{\slb}(x) 
=
 \frac2\theta \bigg [ \sum_{i=1}^\infty i \int_0^1 \bigg\{  \theta \int_0^x \sum_{l = 0}^i    p^{(\xi x)}(l) p_2^{((1-\xi)x,(1-\xi)y)}(i-l,0) e^{- \theta \xi y} \, \diff y \\ 
 +  p^{(\xi x)}(i) e^{-\theta x} \bigg \}\, \diff \xi  - x \bigg ], 
\end{multline*}
where $p_2$ is defined in Lemma \ref{lem:fidirob} and where, for $x>0$,
\[
p^{(x)}(i) = \Prob\big( N_E^{(x)}=i\big), \quad  N_E^{(x)} = \sum_{i=1}^{\eta} \xi_{i}
\]
with $\eta \sim \text{Poisson}(\theta x)$ independent of iid random variables
$\xi_{i} \sim \pi , i\in\N$.
\end{lemma}

\begin{proof}[Proof of Lemma~\ref{lemma:covEZ}]   For the sake of a clear exposition, we will assume that both $U_s$ and $Z_{nt}^{\slb}$ are measurable with respect to the $\Bc_{\cdot:\cdot}^\eps$-sigma fields; the general case follows by multiplication with suitable indicator functions as in the previous proofs. Introduce the notation $A_{j} = \sum_{s\in I_j} \ind(U_s > 1-x/\bn)$ and $B_j= \sum_{s\in I_j}  Z_{nt}^{\slb}$. We can write
\begin{multline*}
\Cov(e_n(x), G_n^{\slb})  
=
\frac{1}{n-\bn+1} \sum_{i=1}^{\kn}  \sum_{j=1}^{\kn-1} \Cov (A_i, B_j ) \\ 
+ \frac{1}{n-\bn+1} \sum_{i=1}^{\kn}  \Cov(A_i, Z_{n,n-\bn+1}^{\slb}).
\end{multline*}
The second sum on the right hand-side is negligible, since both $\|A_j\|_2 = \|N_{\bn}^{(x)}(E) \|_2 =O(1)$ and $\|Z_{n,n-\bn+1}^{\slb}\|_2 = O(1)$ by Condition \ref{cond:cond}\eqref{item:momn} and \eqref{item:moment}.  Regarding the first sum, by stationarity, we can write
\begin{align*}
& \frac{1}{n} \sum_{i=1}^{\kn}  \sum_{j=1}^{\kn-1} \Cov (A_i, B_j )  \\
 = &\,
\frac{1}{n} \sum_{i=1}^{\kn-1}  \sum_{j=1}^{\kn-1} \Cov (A_i, B_j ) + O(\bn/n) \\
= &\,
\frac{\kn-1}{n} \Cov(A_1,B_1)  
+\sum_{h=2}^{\kn-1}  \frac{\kn-h}{n}  \big\{ \Cov (A_1, B_{h} ) + \Cov(A_{h}, B_1) \big\} + o(1).
\end{align*}
Split  the right-hand side according to whether $\Cov(A_i,B_j)$ is such that either $i-j \in\{0,1\}$, or $i-j \in\{-1,2\}$ or $i-j \in\{-\kn+2, \dots, \kn-2\} \setminus\{-1,0,1,2\}$. Up to negligible terms, this allows to write the right-hand side of the previous display as $R_{n1} + R_{n2} +R_{n3}$, where $R_{n1} = \bn^{-1} \Cov (A_2, B_{1}+B_2 )$, $R_{n2} =  \bn^{-1} \Cov (A_3, B_{1}+B_4 )$ and 
\begin{align*}
R_{n3} &= \sum_{h=3}^{\kn-1}  \frac{\kn-h}{n}   \Cov (A_1, B_{h} ) + \sum_{h=4}^{\kn-1}  \frac{\kn-h}{n}   \Cov(A_{h}, B_1).
\end{align*}

Both sums in $R_{n3}$ converge to $0$: first, $\|A_j\|_{2+\delta}=O(1)$ and $\|B_j \|_{2+\delta} =O(\bn)$. Second, the variables defining $A_1$ and $B_h$ are at least $(h-1)\bn$-observations apart, while the variables defining $A_1$ and $B_h$ are at least $(h-2)\bn$-observations apart. As a consequence, by Lemma~3.11 in \cite{DehPhi02},
\[
|R_{n3}| \le C \sum_{h=1}^{\kn} \alpha_{c_2}^{\delta/(2+\delta)} (h \bn ) \le C \bn^{-\eta} \sum_{h=1}^{\infty} h^{-\eta} =o(1).
\]

The term $R_{n2}$ is also negligible: we have
\begin{align*}
\bn^{-1}\Cov (A_3, B_{4} )
&=
\bn^{-1}\Cov (A_1, B_{2} ) \\
&=
\bn^{-1} \sum_{t=\bn+1}^{2\bn} \Cov\{ \textstyle \sum_{s=1}^{\bn} \ind(U_s > 1-x/\bn), Z_{nt}^{\slb}\}. 
\end{align*}
The covariance on the right-hand side can be bounded by a multiple of $\alpha_{c_2}(t-\bn)^{\delta/(2+\delta)}$.
The remaining sum over the mixing-coefficients converges, such that $\bn^{-1}\Cov (A_3, B_{4} )=O(\bn^{-1})$. The covariance $\bn^{-1}\Cov (A_3, B_{1} )$ can be treated similarly.

It remains to be shown that 
\[
R_{n1} =  \frac{1}{\bn} \Cov (A_2, B_{1}+B_2 )= \frac{1}{\bn} \sum_{t=1}^{2\bn} \Cov \Big \{ \sum_{s \in I_2} \ind(U_s > 1- \tfrac{x}{\bn}) , Z_{nt}^{\slb} \Big \}
\]
converges to $h_{\slb}(x )$. To this end, define functions $f_n,g_n:[0,1] \to \mathbb{R}$ by
\begin{align*}
f_n(\xi) 
&= 
\sum_{t=1}^{\bn} \Exp\Big [ \sum_{s\in I_2} \ind(U_s > 1- \tfrac{x}{\bn})  Z_{nt}^{\slb} \Big ] \ind \{ \xi \in[ \tfrac{t-1}{\bn}, \tfrac{t}{\bn} ) \}, \\
g_n(\xi) 
&= 
\sum_{t=\bn+1}^{2\bn} \Exp\Big [ \sum_{s \in I_2} \ind(U_s > 1- \tfrac{x}{\bn})  Z_{nt}^{\slb} \Big ] \ind \{ \xi \in[ \tfrac{t-\bn-1}{\bn}, \tfrac{t-\bn}{\bn} ) \}.
\end{align*}
With this notation, we obtain 
\[
\Cov(e_n(x), G_n^{\slb}) = \int_0^1 \{ f_n(\xi)  + g_n(\xi) \} \, \mathrm{d} \xi - 2x \Exp[Z_{n1}^{\slb}] + o(1).
\]
By uniform integrability of $Z_{n1}^{\slb}$ we have  $\Exp[Z_{n1}^{\slb}] \to \theta ^{-1}$, as $n \to \infty$ . Furthermore, for any $n$, $f_n$ and $g_n$ are uniformly bounded by $\Vert \sum_{s \in I_1} \ind(U_s > 1- \tfrac{x}{\bn})  \Vert_2 \times \Vert  Z_{n1}^{\slb} \Vert_{2}$, which again is uniformly bounded in $n$ by Condition \ref{cond:cond}\eqref{item:momn} and \eqref{item:moment}, i.e.,  $\sup_n( \Vert f_n\Vert_\infty +\Vert g_n\Vert_\infty)< \infty$. Hence, by dominated convergence, the lemma follows if we show that, for any $\xi \in (0,1)$,
\begin{multline} \label{eq:limitmean}
\lim_{n \to \infty} f_n(1-\xi) = \lim_{n \to \infty} g_n (\xi)\\
 =   \sum_{i=1}^\infty i  \int_0^x \sum_{l = 0}^i   p^{(\xi x)}(l) p_2^{((1-\xi)x,(1-\xi)y)}(i-l,0) e^{- \theta \xi y} \, \diff y + \theta^{-1} p^{(\xi x)}(i) e^{-\theta x}.
\end{multline}

We only do this for $g_n$, as $f_n$ can be treated similarly. 
Fix  $\xi \in (0,1)$ and note that
\begin{align*}
g_n(\xi) &= \Exp\Big[\sum_{s \in I_2} \ind(U_s > 1- \tfrac{x}{\bn} )Z_{n,( \lfloor (1+\xi ) \bn \rfloor+1)}^{\slb}\Big].
\end{align*}
Let us first show joint weak convergence of the two variables inside this expectation, and for that purpose consider 
\begin{align*} \nonumber
F_n(i,y) :=&\ \Prob \Big (\textstyle \sum_{s =\bn+1}^{2\bn} \ind(U_s > 1- \tfrac{x}{\bn}) = i , Z_{n , (\lfloor (1+\xi ) \bn \rfloor+1)}^{\slb} \geq y  \Big ) \\
=&\ \Prob \Big  (\textstyle\sum_{ s=\bn +1 }^{\scs 2\bn} \ind (U_s> 1-\tfrac{x}{\bn} )=i ,\\ 
& \hspace{4.2cm} \textstyle \sum_{ s =  \lfloor (1+\xi ) \bn \rfloor+1}^{  \lfloor (1+\xi ) \bn \rfloor+\bn} \ind(U_s> 1- \tfrac{y}{\bn})=0\Big )   
\end{align*}
For $y\in(0,x]$, we can write
$
F_n(i,y) =  \sum_{l =0}^i A_n(l,i),
$
where
\begin{align*}
A_n({l,i})= &\Prob\Big  (\textstyle\sum_{\scs s= \bn+1}^{\scs \lfloor (1+\xi ) \bn \rfloor }\ind(U_s > 1- \tfrac{x}{\bn}) = l  , \\
 &\hspace{1.2cm} \textstyle  \sum_{\scs s = \lfloor (1+\xi ) \bn \rfloor+1}^{\scs 2\bn}\ind(U_s> 1- \tfrac{x}{\bn}) =i- l, \nonumber\\
&\hspace{2.4cm} \textstyle \sum_{\scs s = \lfloor (1+\xi ) \bn \rfloor+1}^{\scs 2\bn}\ind(U_s > 1- \tfrac{y}{\bn}) =0,  \\ 
&\hspace{3.6cm} \textstyle\sum_{s =2\bn+1}^{  \lfloor (2+\xi ) \bn \rfloor}\ind(U_s > 1- \tfrac{y}{\bn}) =0 \Big ).
\end{align*}
Let us show that we can manipulate any sum inside this probability by adding or subtracting $r_n$ summands, where $r_n$ is some integer sequence with $r_n=o(\bn)$. 
Indeed, for any fixed $x>0$ and sufficiently large $n$:
\[
\Prob\big (\textstyle \sum_{s=1}^{r_n} \ind (U_s > 1- \tfrac{x}{\bn}) =0 \big ) \geq 1- r_n \Prob(U_1 > 1- \tfrac{x}{\bn} ) = 1-\tfrac{x r_n}{\bn} \to 1, \qquad   n \to \infty.
\]
Now, by omitting the last $r_n$ summands of the first sum inside the probability defining $A_n(l,i)$, this sum becomes asymptotically independent of the remaining sums in the probability (at the cost of an additive $\alpha_{c_2}(r_n)$-error). The same can be done for the last sum and we  obtain
\begin{align*}
A_n({l,i}) &=
\Prob\Big  (\textstyle\sum_{s= \bn+1}^{ \lfloor (1+\xi ) \bn \rfloor  }\ind(U_s > 1- \tfrac{x}{\bn}) = l \Big ) \\
&\quad \times 
\Prob \Big ( \sum_{s =2\bn+1}^{ \lfloor (2+\xi ) \bn \rfloor}\ind(U_s > 1- \tfrac{y}{\bn}) =0 \Big ) \\
&\quad \times \Prob \Big (  \textstyle \sum_{ s= \lfloor (1+\xi ) \bn \rfloor+1}^{2\bn}\ind(U_s > 1- \tfrac{x}{\bn}) =i- l, \\
& \hspace{4cm} \sum_{ s = \lfloor (1+\xi ) \bn \rfloor+1}^{2\bn}\ind(U_s > 1- \tfrac{y}{\bn}) =0 \Big ) \\
&\quad+
 O(\alpha_{c_2}(r_n)) + O(r_n/\bn).
\end{align*}
This expression converges to $p^{(\xi x)}(l) p^{(\xi y)}(0) p_2^{((1-\xi)x,(1-\xi)y)}(i-l,0)$ by Theorem 4.1 in \cite{Rob09}. As a consequence,
\[
F_n(i,y) 
\to
\sum_{l=0}^i p^{(\xi x)}(l) p_2^{((1-\xi)x,(1-\xi)y)}(i-l,0) p^{(\xi y)}(0).
\]
In the case $y>x$ similar arguments imply that
\[
F_n(i,y)  \to p^{(\xi x)}(i) p^{(y)}(0) =  p^{(\xi x)}(i) e^{-\theta y}.
\]
Since both $\sum_{s \in I_2} \ind(U_s > 1- \tfrac{x}{\bn} )$ and $Z_{n (\lfloor (1+\xi ) \bn \rfloor+1)}^{\slb}$ are in $L_{2+\delta}(\Prob)$, weak convergence implies convergence of moments, whence
\begin{align*}
g_n(\xi)
& =
 \sum_{i=1}^\infty i \int_0^\infty \Prob \Big (\sum_ {s=\bn+1}^{2\bn} \ind(U_s > 1- \tfrac{x}{\bn}) = i , Z_{n( \lfloor (1+\xi ) \bn \rfloor+1)}^{\slb} \geq y  \Big ) \mathrm{d}y\\
& \to  \sum_{i=1}^\infty i \int_0^x \sum_{l=0}^i p^{(\xi x)}(l) p_2^{((1-\xi)x,(1-\xi)y)}(i-l,0) e^{-\theta \xi y}\, \diff y  \\
& \hspace{7cm} + \int_x^\infty  p^{(\xi x)}(i)e^{-\theta y} \,\diff y.
\end{align*}
Calculating the integral on the right-hand side explicitly yields \eqref{eq:limitmean}.
\end{proof}

\begin{lemma}\label{lemma:covZZ}
Suppose Conditions \ref{cond:cond}\eqref{item:rate2} and \eqref{item:moment} are met, then, as $n \to \infty$, 
\[
\Var(G_n^{\slb}) \to  \frac{2(\log(4)-1)}{\theta^2} .
\]
\end{lemma}

\begin{proof}[Proof of Lemma~\ref{lemma:covZZ}] 
As in  proof of Lemma \ref{lemma:covEZ} we will assume that the $Z_{nt}^{\slb}$ are measurable with respect to the $\Bc_{\cdot:\cdot}^\eps$-sigma fields.
Similar as in the beginning of the proof of Lemma \ref{lemma:covEZ}, one can show that
\begin{align*}
\Var (G_n^{\slb}) 
&= \frac{2}{\bn} \sum_{t=1}^{\bn} \Cov( Z_{n1}^{\slb}Z_{n,(1+t)}^{\slb})  + o(1)  \\
&= 
2 \int_0^1 h_n(\xi)\, \diff \xi  - 2\Exp[Z_{n1}^{\slb}]^2 + o(1) ,
\end{align*}
where $h_n:[0,1] \to \mathbb{R}$ is defined as
\[
h_n(\xi) = \sum_{t=1}^{\bn} \Exp[Z_{n1}^{\slb} Z_{n,(1+t)}^{\slb}] \ind\{ \xi \in [ \tfrac{t-1}{\bn},\tfrac{t}{\bn}) \} =\Exp[Z_{n1}^{\slb} Z_{n, (\lfloor \bn \xi \rfloor+1)}^{\slb}].
\]
Condition \ref{cond:cond}\eqref{item:moment} implies $\Exp[Z_{n1}^{\slb}]\to \theta^{-1}$. The limit of the integral over $h_n$ can deduced from pointwise convergence and the dominated convergence theorem. To see this, note that $\sup_n\Vert h_n \Vert_\infty \leq \sup_n \Exp[{Z_{n1}^{\slb}}^2] < \infty$, due to Condition \ref{cond:cond}\eqref{item:moment}. 
Regarding the pointwise convergence, suppose we have shown that, for any $\xi \in (0,1)$, there exists some random vector $(X^{(\xi)} , Y^{(\xi)})$ with dirtybution  function depending on $\xi$, such that
\begin{align} \label{eq:weak23}
(Z_{n1}^{\slb}, Z_{n,(\lfloor \bn \xi \rfloor+1)}^{\slb} ) \dto (X^{(\xi)} , Y^{(\xi)}).
\end{align}
In that case, $h_n(\xi) = \Exp[Z_{\scs n1}^{\slb} Z_{\scs n,(\lfloor \bn \xi \rfloor+1)}^{\slb}]$ converges to $\Exp[X^{(\xi)} Y ^{(\xi)}]$ by Condition~\ref{cond:cond}\eqref{item:moment}. Let us show \eqref{eq:weak23}. Fix $x,y \in \mathbb{R}^+$ and write
\begin{align*}
&\bar F_n(x,y) \\
=& \,\Prob (Z_{n1}^{\slb}>x , Z_{n,(\lfloor \bn \xi \rfloor+1)}^{\slb}>y) \\
=&\,
\Prob(N_{ 1:\lfloor \bn \xi \rfloor}< 1- \tfrac{x}{\bn}, N_{(\lfloor \bn \xi \rfloor +1):\bn} < 1 - \tfrac{x \vee y }{\bn}, N_{ (\bn+1):\lfloor \bn (\xi+1) \rfloor}< 1- \tfrac{y}{\bn}).
\end{align*}
Now, if $r_n$ is an integer sequence such that $r_n=o(\bn)$, then, for sufficiently large $n$,
\[
\Prob(N_{1:r_n} > 1- \tfrac{x}{\bn}) \leq \tfrac{x r_n}{\bn} \to 0 , \qquad n \to \infty,
\] 
which is why we can omit or add $r_n$ observations in the maximum without changing the limit of its distribution. Similar as in the proof of Lemma~\ref{lemma:covEZ} this  gives
\begin{multline*}
\bar F_n(x,y)
=
\Prob(N_{\scs 1:\lfloor \bn \xi \rfloor}< 1- \tfrac{x}{\bn}) 
\times
\Prob( N_{\scs (\lfloor \bn \xi \rfloor+1 ) :\bn} < 1 - \tfrac{x \vee y }{\bn}) \\
\times 
\Prob( N_{\scs (\bn+1): \lfloor \bn(\xi+1) \rfloor}< 1- \tfrac{y}{\bn})
+
O(\alpha_{c_2}(r_n))+O(\tfrac{r_n x\vee y}{\bn}),
\end{multline*}
which, by \eqref{eq:exp2}, converges to
\begin{align*}
 \bar F_\xi(x,y) &= \exp(-\theta \xi x )\exp(-\theta(1-\xi)(x \vee y)) \exp(- \theta \xi y) \\
 &=\exp\{-\theta(\xi (x \wedge y)+ x \vee y ) \}.
\end{align*}
This implies \eqref{eq:weak23}, with $(X^{(\xi)} , Y^{(\xi)})$ being defined by its joint survival function $\bar F_\xi:[0,\infty)^2 \to [0,1]$.
 Now, it is easy to see that
\[
\lim_{n \to \infty} h_n(\xi) = \Exp[X^{(\xi)} Y ^{(\xi)}] = \int _{\mathbb{R}^+} \int_{\mathbb{R}^+} \bar F_\xi(x,y) \mathrm{d}x \mathrm{d}y =\frac{2}{\theta^2(1+\xi)}.
\]
Finally, putting everything together, we obtain
\begin{align*}
\lim_{n \to \infty} \Var (G_n^{\slb}) &= 2 \int _0^1 \lim_{n \to \infty} h_n(\xi)  \mathrm{d} \xi - \frac{2}{\theta^2}= \frac{2}{\theta^2} \Big ( \int_0^1 \frac{2}{1+\xi} \mathrm{d} \xi -1 \Big) \\
&= \frac{2\{\log(4)- 1\}}{\theta^2}
\end{align*}
as asserted.
\end{proof}

\begin{lemma}\label{lem:convvar2}
Under the above conditions, $h_{\slb} = h$, where $h_{\slb}$ and $h$ are defined in Lemma \ref{lemma:covEZ} and Lemma \ref{lem:fidirob}, respectively.
\end{lemma}

\begin{proof}[Proof of Lemma~\ref{lem:convvar2}] 
By the definition of $p^{\scs (x)}$ and $p_2^{\scs (x,y)}$ in Lemma \ref{lemma:covEZ} and Lemma \ref{lem:fidirob} we obtain that
\begin{multline*}
 \sum_{l = 0}^i    p^{\scs (\xi x)}(l) p_2^{\scs((1-\xi)x,(1-\xi)y)}(i-l,0) 
=
 \Prob\Big (\sum_{j=1}^{\eta_1} \zeta_{j} + \sum_{j=1}^{\eta_2} \zeta_{j1}^{\scs (y/x)}=i  ,  \sum_{j=1}^{\eta_2} \zeta_{j2}^{\scs (y/x)}=0\Big ),
\end{multline*}
with independent random variables $\eta_1 \sim \text{Poisson}(\xi \theta x)$, $\eta_2 \sim \text{Poisson}((1-\xi) \theta x)$,
$\zeta_{i} \sim \pi , i\in\N,$ and $(\zeta_{i1}^{\scs(y/x)},\zeta_{i2}^{\scs(y/x)}) \sim \pi_2^{\scs(y/x)} , i\in\N$. For this reason, we can write 
\begin{align*}
\sum_{i=1}^\infty i &\, \sum_{l = 0}^i    p^{\scs (\xi x)}(l) p_2^{\scs ((1-\xi)x,(1-\xi)y)}(i-l,0)  \\
&=
\Exp \bigg [ \Big \{ \sum_{j=1}^{\eta_1} \zeta_{j} + \sum_{j=1}^{\eta_2} \zeta_{j1}^{\scs (y/x)} \Big \}\ind \Big( \sum_{j=1}^{\eta_2} \zeta_{j2}^{\scs (y/x)}=0 \Big) \bigg ]\\
&= \Exp \Big [\sum_{j=1}^{\eta_1} \zeta_{j} \Big ] \Prob  \Big( \sum_{j=1}^{\eta_2} \zeta_{j2}^{\scs (y/x)}=0 \Big) 
+
 \Exp \bigg [ \sum_{j=1}^{\eta_2} \zeta_{j1}^{\scs (y/x)} \ind \Big( \sum_{j=1}^{\eta_2} \zeta_{j2}^{\scs (y/x)}=0 \Big) \bigg ].
\end{align*}
By Wald's identity, we have $ \Exp \big [\sum_{j=1}^{\eta_1} \zeta_{j} \big ]= \xi x$. Independence of $\eta_2$ and $\zeta_{j2}^{\scs (y/x)},j \in \N$, further implies
\[
\Prob  \Big( \sum_{j=1}^{\eta_2} \zeta_{j2}^{\scs (y/x)}=0 \Big) = \sum_{k=0}^\infty \Prob (\zeta_{12}^{\scs(y/x)} =0)^k \Prob (\eta_2=k) = e^{-\theta(1-\xi)y},
\]
where we used that $\Prob (\zeta_{12}^{\scs(y/x)} =0) = 1-y/x$, see \eqref{eq:xi2}. Finally, \eqref{eq:condexp} implies that 
\[
 \Exp \bigg [ \sum_{j=1}^{\eta_2} \zeta_{j1}^{\scs (y/x)} \ind \Big( \sum_{j=1}^{\eta_2} \zeta_{j2}^{\scs (y/x)}=0 \Big) \bigg ]
=
\Exp \big[\zeta_{11}^{\scs(y/x)} \ind( \zeta_{12}^{\scs (y/x)}=0) \big] \theta (1- \xi) x e^{- (1-\xi)\theta y} .
\]
Altogether, we obtain 
\begin{multline*}
\sum_{i=1}^\infty i \sum_{l = 0}^i    p^{\scs (\xi x)}(l) p_2^{\scs ((1-\xi)x,(1-\xi)y)}(i-l,0) \\
=
\xi x e^{- \theta(1-\xi)y} + \Exp [\zeta_{11}^{\scs(y/x)} \ind( \zeta_{12}^{\scs (y/x)}=0) ] \theta (1- \xi) x e^{- (1-\xi)\theta y}.
\end{multline*}
Now, noting that $\sum_{i=1}^\infty i p^{\scs(\xi x)}(i)=  \Exp \big [\sum_{j=1}^{\eta_1} \zeta_{j} \big ]= \xi x$, we can rewrite $h_{\slb}$ as follows
{\small
\begin{align*}
h_{\slb}(x) 
&= 
 \frac2\theta \bigg [ \sum_{i=1}^\infty i \int_0^1 \bigg\{  \theta \int_0^x \sum_{l = 0}^i    p^{(\xi x)}(l) p_2^{((1-\xi)x,(1-\xi)y)}(i-l,0) e^{- \theta \xi y} \, \diff y +  \\ 
 & \hspace{2cm} p^{(\xi x)}(i) e^{-\theta x} \bigg \}\, \diff \xi  - x \bigg ]\\
&=  
2 \int_0^x \int_0^1 \xi x e^{- \theta y } \, \diff \xi \, \diff y + 2 \int_0^x \int_0^1  \Exp [\zeta_{11}^{\scs(y/x)} \ind( \zeta_{12}^{\scs (y/x)}=0) ] \theta (1- \xi) x e^{- \theta y} \, \diff \xi \, \diff y  \\
& \hspace{2cm} +
 \frac 2\theta \int_0^1 \xi x e^{-\theta x}  \, \diff \xi  - \frac{2x}{\theta}\\
&= 
\int_0^x x e^{-\theta y} \, \diff y + \int_0^x \Exp [\zeta_{11}^{\scs(y/x)} \ind( \zeta_{12}^{\scs (y/x)}=0) ] \theta x e^{- \theta y } \, \diff y + \frac x\theta e^{-\theta x }- \frac{2x}{\theta}\\
&=  
\int_0^1  \Exp [\zeta_{11}^{\scs(\sigma)} \ind( \zeta_{12}^{\scs (\sigma)}=0) ] \theta x^2 \exp^{-\theta\sigma x} \, \diff \sigma - \frac x\theta.
\end{align*}
}
From \eqref{eq:hdjb} we finally obtain that $h_{\slb}(x)=h(x )$.
\end{proof}

\section{Equivalence of estimators -- Proof of Theorem~3.1}
\label{sec:nor}

\begin{proof}[Proof of Theorem~\ref{theo:Northrop}] We will only give the proof for the disjoint blocks version of the theorem as the sliding blocks can be treated analogously. For notational reasons we will omit the upper index $\djb$. Define
\[
\hat T_n^{\Be} = \frac1{\kn} \sum_{i=1}^{\kn} \hat Z_{ni}, \quad T_n^{\Be} = \frac1{\kn}  \sum_{i=1}^{\kn} Z_{ni},  \quad
\hat T_n^{\No} = \frac1{\kn}  \sum_{i=1}^{\kn} \hat Y_{ni}, \quad T_n^{\No} = \frac1{\kn}  \sum_{i=1}^{\kn} Y_{ni}
\]
and note that
\begin{align*}
\sqrt{\kn} (\hat \theta_n^{\No}- \hat \theta_n^{\Be}) = \frac{1}{\hat T_n^{\Be} \hat T_n^{\No}} \times \sqrt{\kn}( \hat T_n^{\No} - \hat T_n^{\Be} ).
\end{align*}
The fraction on the right-hand side is $O_\Prob(1)$. Indeed, the elementary inequality $\log(x) \le x-1$ for $x>0$ implies that $\hat T_n^{\No} \ge \hat T_n^{\Be}$, and  $\hat T_n^{\Be}$ converges to $\theta^{-1}$ in probability by Theorem~\ref{theo:main}.
Now, we further decompose
\begin{align*}
\sqrt{\kn} ( \hat T_n^{\No} - \hat T_n^{\Be} )
&=
\sqrt{\kn} ( \hat T_n^{\No} - T_n^{\No} ) + \sqrt{\kn} ( T_n^{\No} - T_n^{\Be} ) + \sqrt{\kn} ( \hat T_n^{\Be} - T_n^{\Be} ) \\
&\equiv S_{n1} + S_{n2} + S_{n3} .
\end{align*}
By Lemma~\ref{lem:log2}, we immediately obtain $S_{n2}=o_\Prob(1)$. Furthermore, from Lemma~\ref{lem:log1} and \eqref{eq:long2}, we have, for any $\ell >0$,
\[
S_{n1} + S_{n3} = \int_0^{\max Z_{ni}} \frac{-x}{\bn -x} e_n(x) \mathrm{d}\hat H_{\kn} (x) + o_\Prob(1)= I_{n,\ell}+R_{n,\ell}+o_\Prob(1),
\]
with $I_{n,\ell}= \int_0^{\ell} \frac{-x}{\bn -x} e_n(x) \mathrm{d}\hat H_{\kn} (x)$ and $R_{n,\ell}=\int_\ell^{\max Z_{ni}} \frac{-x}{\bn -x} e_n(x) \mathrm{d}\hat H_{\kn} (x)$. Hence, $S_{n1} + S_{n3}$ is $o_\Prob(1)$ 
if we can show that, for any $\delta>0$, we have $\limsup_{ \ell \to \infty} \lim_{n \to \infty} \Prob(R_{n,\ell}> \delta)=0$ and that $I_{n,\ell}=o_\Prob(1)$, for any fixed $\ell>0$. The first part can be done by similar arguments as in the proof of Lemma~\ref{lem:boundsupp}. To see this, note that $\vert \frac{-x}{\bn-x}\vert \leq 1$ for $x \leq \bn/2$ and that $\Prob(\max Z_{ni}\leq \bn/2) \to 1$ as $n \to \infty$ by Condition~\ref{cond:cond}\eqref{item:blockdiv}. Furthermore, $I_{n,\ell}=o_\Prob(1)$ by Proposition 7.27 in \cite{Kos08}. \end{proof}

\begin{lemma}[Getting rid of the Logarithm I]\label{lem:log1}
Under Conditions~\ref{cond:cond}\eqref{item:rate2} and \eqref{item:blockdiv}, as $n\to\infty$,
\[
\frac{\bn}{\sqrt{\kn}} \sum_{i=1}^{\kn} (\log \hat N_{ni}^{\djb} - \log N_{ni}^{\djb}) =  - \tilde D^{\djb}_n + o_\Prob(1),
\] 
 and 
\[
\frac{\bn \sqrt{\kn}}{n-\bn+1} \sum_{t=1}^{n - \bn+1} (\log \hat N_{nt}^{\slb} - \log N_{nt}^{ \slb}) =  - \tilde D^{\slb}_n + o_\Prob(1),
\]
where
\[
\tilde D_n^{\djb}= 
\int_0^{\max Z_{ni}}  
\frac{\bn}{\bn-x} \tailp(x)\diff \hat H_{\kn}^{\djb}(x),
\quad
\tilde D_n^{\slb}= 
\int_0^{\max Z_{ni}^{\slb}}  
\frac{\bn}{\bn-x} \tailp(x)\diff \hat H_{\kn}^{\slb}(x).
\]
\end{lemma}

\begin{proof}
We will only give the proof for the disjoint blocks version of the theorem as the sliding blocks can be treated analogously. For notational reasons we will omit the upper index $\djb$.
By a Taylor expansion and a similar calculation as in \eqref{eq:long2}, we have
\begin{align*}
\frac{\bn}{\sqrt{\kn}} \sum_{i=1}^{\kn} (\log \hat N_{ni} - \log N_{ni} )  
=
- \tilde D_n  - \frac12 R_n,
\end{align*}
where, for some $s_i\in(0,1)$, 
\[
R_n = \frac{\bn}{\sqrt{\kn}} \sum_{i=1}^{\kn} \frac{(\hat N_{ni} - N_{ni} )^2}{\{N_{ni}^*(s_i)\}^2}  
\]
and where $N_{ni}^*(s) = N_{ni} + s_i (\hat N_{ni} - N_{ni})$.  Let $\eps \in (0,c_1)$. 
By Condition~\ref{cond:cond}\eqref{item:blockdiv}, we have $R_n = R_n \cdot \ind(\min_{i=1}^{\kn} N_{ni} > 1-\eps) + o_\Prob(1)$. 
Note that, by Condition~\ref{cond:cond}\eqref{item:rate2}, the sequence $(U_t\ind(U_t> 1-\eps))_t$ is $\alpha$-mixing with polynomial mixing rate and with stationary cdf $F_\eps$ satisfying $F_\eps(u)=u$ for $u \in (1-\eps,1)$. For this reason, its empirical process converges weakly in $\ell^\infty((1-\eps,1)), \Vert \cdot \Vert_\infty)$  and hence we obtain that
\begin{multline*}
\max_{i=1}^{\kn} \vert \hat N_{ni} - N_{ni} \vert \ind (N_{ni}> 1-\eps) \\
\leq \sup_{u \in (1-\eps,1)} \Big \vert \frac{1}{n} \sum_{i=1}^n \ind(U_i \ind(U_i > 1-\eps) \leq u ) - u \Big \vert
\end{multline*} 
is of the order $O_\Prob(n^{-1/2})$.
Thus, for sufficiently large $n$,
\[
\max_{i=1}^{\kn} \{ |N_{ni}^*(s_i)|^{-2}  \ind(N_{ni} > 1-\eps) \} \le \{ 1-\eps - O_\Prob(n^{-1/2}) \}^{-2} = O_\Prob(1)
\]
as $n\to\infty$.
As a consequence, 
\[
R_n =O_\Prob\{ (\bn \kn) / (\sqrt{\kn} n) \}  + o_\Prob(1) = O_\Prob(\kn^{-1/2}) + o_\Prob(1)=  o_\Prob(1)
\] 
as $n\to \infty$.
\end{proof}

\begin{lemma}[Getting rid of the Logarithm II]\label{lem:log2}
Under Conditions~\ref{cond:cond}\eqref{item:rate2}, \eqref{item:blockdiv} and \eqref{item:moment}, we have, as $n \to \infty$, 
\begin{align*}
\frac{1}{\sqrt{\kn}} \sum_{i=1}^{\kn} \{ - \bn \log (N_{ni}^{\djb}) - Z_{ni}^{\djb}\} &= o_\Prob(1), \\
\frac{\sqrt{\kn}}{n-\bn +1} \sum_{t=1}^{n-\bn+1} \{ - \bn \log (N_{nt}^{\slb}) - Z_{nt}^{\slb}\} &= o_\Prob(1).
\end{align*}
\end{lemma}

\begin{proof} 
We will only give the proof for the disjoint blocks version of the theorem as the sliding blocks can be treated analogously. For notational reasons we will omit the upper index $\djb$. By Condition~\ref{cond:cond}\eqref{item:blockdiv} and since we are only concerned with convergence in probability, it suffices to work on the event $\{\min_{i=1}^{\kn} N_{ni} > 1-\eps\}$, where $\eps>0$.  It then suffices to show convergence in $L_1$, and for that purpose note that
\begin{multline*}
\mathbb{E}\Big |\frac{1}{\sqrt{\kn}} \sum_{i=1}^{\kn} \{ - \bn \log (N_{ni}) - Z_{ni}\}  \ind ( \min_{i=1}^{\kn} N_{ni} > 1-\eps)\Big|  \\
\leq 
 \sqrt{\kn} \mathbb{E} [ \vert  - \bn \log (N_{ni}) - Z_{ni}\vert   \ind ( \min_{i=1}^{\kn} N_{ni} > 1-\eps) ].
\end{multline*}
By a Taylor expansion, we have
\begin{multline*}
\Big|- \bn \log (N_{n1}) - Z_{n1}\Big| \ind ( \min_{i=1}^{\kn} N_{ni} > 1-\eps)
 \\
 \le\frac12 \cdot \frac{1}{\bn N_{n1}^2}  \cdot Z_{n1}^2 \ind ( \min_{i=1}^{\kn} N_{ni} > 1-\eps)
\le
\frac{1}{2 b_n (1-\eps)^2}   Z_{n1}^2.
\end{multline*}
Hence, by Condition~\ref{cond:cond}\eqref{item:moment}, we immediately obtain
\[
\mathbb{E}\Big [\frac{1}{\sqrt{\kn}} \sum_{i=1}^{\kn} \{ - \bn \log (N_{ni}) - Z_{ni}\}  \ind ( \min_{i=1}^{\kn} N_{ni} > 1-\eps)\Big] =O(\sqrt{\kn}\bn^{-1})=o(1)
\] 
and the proof is finished.
\end{proof}

\section{Additional proofs} \label{sec:addproofs}

\begin{proof}[Proof of Proposition~\ref{prop:boot}]
Let
\[
\beta_{\eps}(\ell) =  \sup_{k \in \N} \beta(\Bc_{1:k}^\eps, \Bc_{k+\ell:\infty}^\eps) = \sup_{k\in\N} \frac12 \sup \sum_{i \in I} \sum_{j \in J} |\Prob(A_i\cap B_j)-\Prob(A_i)\Prob(B_j)|,
\]
where the last supremum is over all finite partitions $(A_i)_{i\in I} \subset \Bc_{1:k}^\eps$ and $(B_j)_{j\in J} \subset \Bc_{k+\ell:\infty}^\eps$ of $\Omega$. 
Decompose 
\[
\hat \sigma_{\djb}^2  = \frac{1}{\kn} \sum_{j=1}^{\kn} \hat B_{nj}^2 = A_{n1} + 2 A_{n2} + A_{n3},
\]
where 
\[
A_{n1} = \frac{1}{\kn} \sum_{j=1}^{\kn} \bar B_{nj}^2, 
\quad
A_{n2} =  \frac{1}{\kn} \sum_{j=1}^{\kn} (\hat B_{nj} - \bar B_{nj}) \bar B_{nj},
\quad
A_{n3} =  \frac{1}{\kn} \sum_{j=1}^{\kn} (\hat B_{nj} - \bar  B_{nj})^2,
\]
and where 
\[
\bar B_{nj }=
Z_{nj} - \theta^{-1}  +  \textstyle \sum_{s\in I_j}  \frac{1}{\kn} \sum_{i=1}^{\kn} \{  \ind(U_s > 1- \tfrac{Z_{ni}}{\bn}) - \tfrac{Z_{ni}}{\bn} \}, 
\]
By the Cauchy-Schwarz inequality, it suffices to show that $A_{n3}=o_\Prob(1)$ and that $A_{n1}=\sigma_{\djb}^2 + o_\Prob(1)$.

Let us first show that $A_{n3}=o_\Prob(1)$. Note that $U_s > 1-Z_{nj}/\bn$ iff $\hat U_s > 1- \hat Z_{nj}/\bn$, almost surely.  As a consequence, by a similar calculation as in \eqref{eq:long2}, we can write
\begin{align*}
\hat B_{nj}  - \bar B_{nj}
&= \hat Z_{nj} -  Z_{nj}  + \frac{1}\theta - \hat T_n +  \frac1{\kn} \sum_{i=1}^{\kn} (Z_{ni}-\hat Z_{ni}) \\
&= \frac{e_n(Z_{nj})}{\sqrt{\kn}} + \frac{1}{\theta} - \hat T_n - \frac{1}{\sqrt {\kn}} \sqrt{\kn}(\hat T_n - T_n)
=\frac{e_n(Z_{nj})}{\sqrt{\kn}}  + O_\Prob(\kn^{-1/2})
\end{align*}
almost surely, where the $O_\Prob$-term is uniformly in $j=1, \dots, n$.  We may further write
\[
e_n(Z_{nj}) = - \sqrt{n/\kn} \cdot \FF_n(1- Z_{nj}/\bn),
\]
where $\FF_n(u) = n^{-1/2} \sum_{s=1}^n \{\ind(U_s \le u) - u\}$ denotes the usual empirical process. By weak convergence of that process (a consequence of the assumption on beta-mixing) we can conclude that $\max_{j=1}^n | e_{n}(Z_{nj}) | = O_\Prob(\bn^{\scs 1/2})$. Hence,
\begin{align*}
A_{n3} &=  \frac{1}{\kn^2} \sum_{j=1}^{\kn} \big\{ e_n(Z_{nj})   + O_\Prob(1) \big\}^2 
= \Big\{ \frac{1}{\kn^2} \sum_{j=1}^{\kn} e_n^2(Z_{nj})  \Big\} + O_\Prob(\bn^{1/2}\kn^{-1}+\kn^{-1}) \\
& \le 
\frac{1}{\kn} \max_{j=1}^n | e_{n}(Z_{nj}) | \int_0^\infty |e_n(z)| \, \diff \hat H_{\kn}(z) + o_\Prob(1).
\end{align*}
Repeating arguments from the proof of Theorem~\ref{theo:main} (Wichura's theorem), it can be seen that the dominating term on the right-hand side of this display is of the order $O_\Prob(\sqrt{\bn}/\kn),$ which converges to $0$ by assumption.

It remains to be shown that $A_{n1}=\sigma_{\djb}^2 + o_\Prob(1)$. For that purpose, write $A_{n1}=C_{n1}+2 C_{n2}+ C_{n3}$, where
\begin{align*}
C_{n1} &= \frac{1}{\kn} \sum_{j=1}^{\kn} (Z_{nj}-\theta^{-1})^2 , \\
C_{n2} &= \frac{1}{\kn} \sum_{j=1}^{\kn} (Z_{nj}- \theta^{-1}) \Big\{ \textstyle \sum_{s \in I_j} \tfrac{1}{\kn}\sum_{i=1}^{\kn} \big\{ \ind(U_s>1-\tfrac{Z_{ni}}{\bn}) - \tfrac{Z_{ni}}{\bn}  \big\} \Big\}, \\
C_{n3} &= \frac{1}{\kn} \sum_{j=1}^{\kn} \Big\{ \textstyle \sum_{s \in I_j} \tfrac{1}{\kn}\sum_{i=1}^{\kn}\big\{ \ind(U_s>1-\tfrac{Z_{ni}}{\bn}) - \tfrac{Z_{ni}}{\bn} \big\}  \Big\}^2.
\end{align*}
From the proof of Lemma~\ref{lem:convvar} we know that $\sigma_{\djb}^2=\sigma_\infty^2$, where $\sigma_\infty^2$ is defined in \eqref{eq:siginf}. Therefore, it suffices to show  that
\begin{multline*}
C_{n1} \pto \theta^{-2}, \quad  
C_{n2} \pto \theta \int_0^\infty h(x) e^{-\theta x} \, \diff x, \quad \\
C_{n3} \pto \theta^2 \int_0^\infty\int_0^\infty r(x,y) e^{-\theta (x+y)} \, \diff x \diff y.
\end{multline*}
The first convergence  can be shown by considering expectations and variances:
first, $\Exp[C_{n1}] = \Exp[(Z_{n1}-\theta^{-1})^2] \to \theta^{-2}$ by Condition~\ref{cond:cond}\eqref{item:moment} and weak convergence of $Z_{n1}$. Second,
\begin{multline*}
\Var(C_{n1}) = \frac{1}{\kn} \Var \big\{ (Z_{n1}-\theta^{-1})^2 \big\} \\ 
+ \frac{1}{\kn}\sum_{\ell=1}^{\kn} \frac{\kn-\ell}{\kn} \Cov\{ (Z_{n1}-\theta^{-1})^2,(Z_{n,1+\ell}-\theta^{-1})^2  \}
\end{multline*}
which is of the order $O(\kn^{-1})$ by a standard inequality for covariances of strongly mixing time series and by finiteness of moments of $Z_{nj}$ of order larger  than 4.

Consider $C_{n2}$.  For integer $\ell\ge1$,  let
\[
C_{n2}(\ell) = \frac{1}{\kn^2} \sum_{j,i \in \{1, \dots, \kn\} \atop |j-i| \ge 2} \Big\{ (Z_{nj}-\theta^{-1})\textstyle   \sum_{s\in I_j} f(U_s, Z_{ni} ) \Big\} \ind(Z_{ni}\le \ell),
\]
where $f(u,z) = \ind(u>1-z/\bn) - z/\bn$.
Using similar arguments as in the proof of Lemma~\ref{lem:boundsupp} it can be shown that, for any $\delta>0$, $\limsup_{n\to\infty} \Prob(|C_{n2}(\ell) -C_{n2}| > \delta)$ converges to 0 for $\ell\to\infty$. Therefore, by Wichura's theorem (\citealp{Bil79}, Theorem 25.5), it is sufficient to show that
\[
C_{n2}(\ell) \to C_2(\ell) =  \theta \int_0^\ell h(x) e^{-\theta x} \, \diff x, \qquad n \to \infty,
\]
holds for any $\ell \in \N$. For that purpose, we will show that $\Exp[C_{n2}(\ell)] \to C_2(\ell)$ and that $\Var(C_{n2}(\ell)) \to 0$ as $n\to\infty$. 

Recall Berbee’s coupling Lemma (\citealp{Ber79}): if $X$ and $Y$ are two random variables in some Borel spaces $S_1$ and $S_2$, respectively, then there exists a random variable  $U$  independent of $(X,Y)$ and a measurable function $f$ such that $Y^*=f(X,Y,U)$ has  the same distribution as $Y$, is independent of $X$ and satisfies $\Prob(Y \ne Y^*) = \beta(\sigma(X), \sigma(Y))$. 
Apply this lemma with $X=(U_s)_{s\in I_j}$ and $Y=Z_{ni}$ (with $|i-j|\ge 2$) to construct a random variable $Z_{ni}^* \sim H_{\kn}$ ($H_{\kn}$ denoting the cdf of $Z_{n1}$) independent of $(U_s)_{s\in I_j}$ satisfying $\Prob(Z_{ni} \ne Z_{ni}^*) \le \beta(\bn)$. Write
\begin{align} \label{eq:cn2}
 &\, \Exp\big[ (Z_{nj}-\tfrac1\theta)\textstyle   \sum_{s\in I_j} f(U_s, Z_{ni} )  \ind(Z_{ni}\le \ell)\big]   \\
  = &\, 
\Exp\big[ (Z_{nj}-\tfrac1\theta)\textstyle   \sum_{s\in I_j} f(U_s, Z_{ni}^* ) \ind(Z_{ni}^*\le \ell)\big] \nonumber \\
 & \hspace{1.5cm} +
\Exp\Big[ (Z_{nj}-\tfrac1\theta)\textstyle   \sum_{s\in I_j} \big\{ f(U_s, Z_{ni} ) \ind(Z_{ni}\le \ell)\nonumber \\ 
& \hspace{3cm} -  f(U_s, Z_{ni}^* ) \ind(Z_{ni}^*\le \ell) \big\}  \ind(Z_{ni} \ne Z_{ni}^*) \Big] \nonumber
\end{align}
By H\"older's and Minkowski's inequality, the second expectation on the right-hand side of this display can be bounded in absolute value by
\[
\|Z_{nj}-\tfrac1\theta \|_{3}   \textstyle
\sum_{s\in I_j}\big\{  \| f(U_s, Z_{ni} ) \ind(Z_{ni}\le \ell)\|_3 +\|   f(U_s, Z_{ni}^* ) \ind(Z_{ni}^*\le \ell)\|_3 \big\} \beta(\bn)^{1/3} .
\]
This bound converges to $0$, since $\vert f(U_s,Z_{ni}) \vert \leq 1$ and  since the assumptions imply that $\limsup_{n\to\infty} \|Z_{n1}-\tfrac1\theta \|_{3} \le C$ and that $\bn \beta(\bn)^{1/3} = o(1)$.

As a consequence, rewriting the first summand on the right-hand side of \eqref{eq:cn2}, we obtain that
\[
\Exp[C_{n2}(\ell)] 
=\Exp [h_n(Z_{n1}^*) \ind(Z_{n1}^*\le \ell)]+ o(1),
\]
where $h_n(x) = \Exp\big[ (Z_{n1}-\theta^{-1}) \textstyle \sum_{s\in I_1} f(U_{s}, x) \big]$. By Condition~\ref{cond:cond}\eqref{item:momn} and \eqref{item:moment} $h_n(Z_{n1}^*)$ is uniformly integrable. Hence, to obtain that $\Exp[C_{n2}(\ell)] \to C_2(\ell)$ we only have to show that $h_n(Z_{n1}^*) \ind(Z_{n1}^*\le \ell)\dto h(Z) \ind(Z\le \ell)$ with $Z$ being exponentially distributed with parameter $\theta$. This in turn follows from the extended continuous mapping theorem, since $Z_{n1}^*\dto Z$ and $h_n(x_n)\ind(x_n \le \ell) \to h(x)\ind(x \le \ell)$ for any sequence $x_n \to x\neq \ell$. To see the latter, note that, for $x < \ell$ and $n$ large enough, Minkowski's inequality and Condition~\ref{cond:cond}\eqref{item:momn} and \eqref{item:moment} imply that
\[
\vert h_n(x_n) - h_n(x) \vert  = \big | \Exp\big[ (Z_{n1}-\theta^{-1})\{N_{\bn}^{\scs(x_n)} (E) - N_{\bn}^{\scs(x)} (E) \}\big] \big|\leq C \times \vert x_n-x \vert^{1/(2+\delta)}.
\]

Consider the variance of $C_{n2}(\ell)$. By the Cauchy-Schwarz inequality, up to negligible terms, it can be written as
\begin{multline} \label{eq:4}
\kn^{-4} \sum_{(i,i',j,j')\in J} \Cov \Big( (Z_{nj}-\theta^{-1}) \textstyle  \sum_{s \in I_j} f(U_s,Z_{ni})   \ind(Z_{ni} \le \ell),  \\
(Z_{nj'}-\theta^{-1})\textstyle \sum_{s' \in I_{j'}} f(U_{s'},Z_{ni'})   \ind(Z_{ni'} \le \ell) \Big)
\end{multline}
where $J$ denote the set of all $(i,i',j,j')\in \{1, \dots, \kn\}^4$ such that any two of the indexes are at distance larger than 2. We have to show that all covariances in this sum converge to $0$, uniformly in the indexes. 

First, consider the case where either  $i \vee j <i' \wedge j'$ or  $i'\vee j' < i \wedge j$. Recall Lemma 3.11 in \cite{DehPhi02}: for real-valued random variables $X,Y$ and real numbers $r,s,t>1$ such that $1/r+1/s+1/t=1$, we have
\begin{align} \label{eq:deh}
\big| \Exp[XY] - \Exp[X] \Exp[Y] \big| \le 10 \| X \|_r \| Y \|_s \alpha(\sigma(X), \sigma(Y) ) ^{1/t}.
\end{align}
Therefore, for some $\eps\in(0,\delta)$, the covariances inside the sum in \eqref{eq:4} are bounded by
\[
\| (Z_{nj}-\theta^{-1}) \textstyle  \sum_{s \in I_j} f(U_s,Z_{ni})   \ind(Z_{ni} \le \ell) \|_{2+\eps}^2 \{ \alpha_1(\bn) \}^{\eps/(2+\eps)},
\]
which can be seen to be $o(1)$ by Minkowski's inequality and the Cauchy-Schwarz inequality. 

The other cases are slightly more difficult. Consider the case $i<j'<j<i'$. Apply Berbee's coupling Lemma with $X=(U_s)_{s\in I_{j'}\cup I_{j} \cup I_{i'}}$ and $Y=(U_s)_{s\in U_i}$. Then the mixed moment inside the covariance can be written as
\begin{align*}
&\, \Exp\Big[  (Z_{nj}-\theta^{-1}) \textstyle  \sum_{s \in I_j} f(U_s,Z_{ni})   \ind(Z_{ni} \le \ell) \\
&\hspace{2.2cm} \times (Z_{nj'}-\theta^{-1})\textstyle \sum_{s' \in I_{j'}} f(U_{s'},Z_{ni'})   \ind(Z_{ni'} \le \ell) \Big] \\
 =  &\,
\Exp\Big[  (Z_{nj}-\theta^{-1}) \textstyle  \sum_{s \in I_j} f(U_s,Z_{ni}^*)\ind(Z_{ni}^* \le \ell)  \\ 
& \hspace{2.2cm} \times (Z_{nj'}-\theta^{-1})\textstyle \sum_{s' \in I_{j'}} f(U_{s'},Z_{ni'})   \ind(Z_{ni'} \le \ell) \Big] + o(1),
\end{align*}
where the remainder term has been handled by H\"older's and Minkowski's inequality just as in \eqref{eq:cn2}. A second application of Berbee's coupling Lemma (with $X=((U_s^*)_{s\in I_i} , (U_s)_{s\in I_{j'}\cup I_{j}})$ and $Y=(U_s)_{s\in I_{i'}}$) allows to rewrite the dominating term in the last display as
\begin{align*}
&\, \Exp\Big[  (Z_{nj}-\theta^{-1}) \textstyle  \sum_{s \in I_j} f(U_s,Z_{ni}^*)  \ind(Z_{ni}^* \le \ell) \\
& \hspace{3cm} \times (Z_{nj'}-\theta^{-1})\textstyle \sum_{s' \in I_{j'}} f(U_{s'},Z_{ni'}^*)   \ind(Z_{ni'}^* \le \ell) \Big] + o(1) \\
= &\, \Exp\Big[  (Z_{nj}-\theta^{-1}) \textstyle  \sum_{s \in I_j} f(U_s,Z_{ni}^*)  \ind(Z_{ni}^* \le \ell) \Big] \\
& \hspace{3cm} \times  \Exp[(Z_{nj'}-\theta^{-1})\textstyle \sum_{s' \in I_{j'}} f(U_{s'},Z_{ni'}^*)   \ind(Z_{ni'}^* \le \ell) \Big] + o(1),
\end{align*}
where the latter equality follows from \eqref{eq:deh}. Since 
\begin{multline*}
\Exp\Big[  (Z_{nj}-\theta^{-1}) \textstyle  \sum_{s \in I_j} f(U_s,Z_{ni}^*)  \ind(Z_{ni}^* \le \ell) \Big] \\
=
\Exp\Big[  (Z_{nj}-\theta^{-1}) \textstyle  \sum_{s \in I_j} f(U_s,Z_{ni})  \ind(Z_{ni} \le \ell) \Big] +o(1)
\end{multline*}
we finally obtain that
\begin{multline*}
\Cov\Big( (Z_{nj}-\theta^{-1}) \textstyle  \sum_{s \in I_j} f(U_s,Z_{ni}^*)   \ind(Z_{ni}^* \le \ell) , \\
(Z_{nj'}-\theta^{-1})\textstyle \sum_{s' \in I_{j'}} f(U_{s'},Z_{ni'}^*)   \ind(Z_{ni'}^* \le \ell) \Big) = o(1)
\end{multline*}

All other cases can be treated similarly by a successive application of Berbee's coupling Lemma. Also, $C_{n3}$ can be treated similarly. 
\end{proof}

 \begin{proof}[Proof of Lemma~\ref{lem:bias}]  We begin with the disjoint blocks estimator and write $(\hat T_n, T_n) = (\hat T_n^{\djb}, T_n^{\djb})$.
Recalling \eqref{eq:long2}, we can write
$
\kn \Exp[ \hat T_n - T_n]  = S_{n1} + S_{n2} + S_{n3} + S_{n4}, 
$
where 
\begin{align*}
S_{n1} &= \sum_{s =1}^{\bn} \Exp [ \ind (U_s > 1- \tfrac{Z_{n1}}{\bn}) - \tfrac{Z_{n1}}{\bn}]\\
S_{n2} &= \frac{\kn -1}{\kn} \sum_{s = 1 }^{\bn} \Exp [ \ind (U_s > 1- \tfrac{Z_{n2}}{\bn}) - \tfrac{Z_{n2}}{\bn}],\\
S_{n3} &= \frac{\kn - 1}{ \kn} \sum_{s=\bn+1}^{2 \bn}\Exp [ \ind (U_s > 1- \tfrac{Z_{n1}}{\bn}) - \tfrac{Z_{n1}}{\bn}]\\
S_{n4} &= \sum_{i=3}^{\kn } \frac{\kn - i+1}{\kn} \Big \{ \sum_{s \in I_1} \Exp[\ind(U_s > 1- \tfrac{Z_{ni}}{\bn}) - \tfrac{Z_{ni}}{\bn}]  \\
& \hspace{4cm}+ \sum_{s \in I_i}\Exp[\ind(U_s > 1- \tfrac{Z_{n1}}{\bn}) - \tfrac{Z_{n1}}{\bn}] \Big \}.
\end{align*}
Note that $S_{n1}= - \Exp [Z_{n1}] \to - \theta^{-1}$, as $n \to \infty$, by Condition \ref{cond:cond} \eqref{item:moment}. Hence, it remains to be shown that $S_{n2}$, $S_{n3}$ and $S_{n4}$ vanish as $n \to \infty$. 

Consider $S_{n2}$. Choose some integer $l\in \N$  and let $n$ be sufficiently large such that $ \bn>l $. Write $S_{n2} = (\kn-1)/\kn \{ S_{n2}^++S_{n2}^-\}$, where
\begin{align*}
S_{n2}^+ &= \sum_{s =1}^{\bn-l}  \Exp [ \ind(U_s > 1- \tfrac{Z_{n2}}{\bn}) - \tfrac{Z_{n2}}{\bn}], \\
S_{n2}^-&=\sum_{s=\bn - l +1}^{\bn} \Exp [\ind (U_s > 1- \tfrac{Z_{n2}}{\bn})- \tfrac{Z_{n2}}{\bn}].
\end{align*}
The absolute value of $S_{n2}^-$  can be bounded by
\[
\frac{l}{ \bn} \Exp[\vert Z_{n1} \vert] + l \; \Prob(\max_{s=1}^l U_s > \max_{s=l+1}^{l+\bn} U_s)
\]
which goes to $0$ as $n \to \infty$ for any fixed $l$ by Condition \ref{cond:cond} \eqref{item:moment} and similar reasons as in the proof of Lemma \ref{lem:fidirob}, see \eqref{eq:zn0}. 
For the treatment of $S_{n2}^+$ fix $q > 0$ such that $q < \lim_{n \to \infty} \Vert Z_{n1} \Vert_2 = {\sqrt 2}/{\theta}$. Then, for sufficiently large $n$, we can use the coupling construction leading to \eqref{eq:bradley} (with $X=U_s$ and $Y=Z_{n2}$) to find a random variable $Z_{n2}^*$ that has the same distribution as $Z_{n2}$, is in dependent of $U_s$ and satisfies
\[
\Prob(\vert Z_{n2} - Z_{n2}^* \vert > q ) \leq 18 ( \Vert Z_{n2} \Vert_2   / q)^{2/5} \alpha( \sigma(U_s) , \sigma(U_{n2}))^{4/5}.
\]
By a  monotonicity argument, we have
\begin{multline*}
\big \vert \Exp \big [ \big \{ \ind(U_s > 1- \tfrac{Z_{n2}}{\bn}) - \tfrac{Z_{n2}}{\bn}\big  \} \ind(\vert Z_{n2} - Z_{n2}^*\vert \leq q) \big ]\big \vert \\
\leq 
\big\vert  \Exp\big [\big \{ \ind(U_s > 1- \tfrac{Z_{n2}^* + q}{\bn}) - \tfrac{Z_{n2}^* +q }{\bn}\big\} \ind(\vert Z_{n2} - Z_{n2}^*\vert \leq q)\big ] \big \vert \\
+
\big \vert  \Exp \big [\big \{ \ind(Us > 1- \tfrac{Z_{n2}^*-q}{\bn}) - \tfrac{Z_{n2}^*-q}{\bn}\big \} \ind(\vert Z_{n2} - Z_{n2}^*\vert \leq  q) \big  ] \big \vert 
+
 \frac{2 q }{\bn}.
\end{multline*}
Furthermore, since $Z_{n2}^*$ is independent of $U_s$,
\begin{multline*}
\big\vert  \Exp\big [\big \{ \ind(U_s > 1- \tfrac{Z_{n2}^* \pm q}{\bn}) - \tfrac{Z_{n2}^* \pm q }{\bn}\big\} \ind(\vert Z_{n2} - Z_{n2}^*\vert \leq  q)\big ] \big \vert \\
=
 \big\vert  \Exp\big [\big \{ \ind(U_s > 1- \tfrac{Z_{n2}^* \pm q}{\bn}) - \tfrac{Z_{n2}^* \pm q }{\bn}\big\} \ind(\vert Z_{n2} - Z_{n2}^*\vert > q)\big ] \big \vert .
\end{multline*}
Combining everything we obtain 
\begin{align*}
|S_{n2}^+| &\le
\sum_{s =1}^{\bn-l}\big \vert \Exp \big [ \big \{ \ind(U_s > 1- \tfrac{Z_{n2}}{\bn}) - \tfrac{Z_{n2}}{\bn}\big  \} \ind(\vert Z_{n2} - Z_{n2}^*\vert \leq q) \big ]\big \vert \\
& \hspace{2cm} +
\sum_{s =1}^{\bn-l} \big \vert \Exp \big [ \big \{ \ind(U_s > 1- \tfrac{Z_{n2}}{\bn}) - \tfrac{Z_{n2}}{\bn}\big  \} \ind(\vert Z_{n2} - Z_{n2}^*\vert > q) \big ]\big \vert\\
&\le
 \frac{2 q (\bn -l)}{\bn} +54 (\Vert Z_{n2}\Vert_2/q)^{2/5} \sum_{s =l+1}^{\bn}   \alpha(s)^{4/5}.
\end{align*}
As a consequence, since $\alpha(s)\leq C s^{-\eta} \leq C s^{-3}$ by Condition \ref{cond:cond} \eqref{item:rate2},
\[
\limsup_{n\to \infty} \vert S_{n2} \vert \leq 2 q + 54 C (\sqrt 2/(\theta q))^{2/5} \sum_{s =l}^{\infty}   s^{-12/5}
\]
This bound in turn can be made arbitrarily small by first choosing $q$ sufficiently small and then choosing $l$ sufficiently large. Hence, $\lim_{n \to \infty} \vert S_{n2} \vert =0$. Along the same lines, we obtain that $\lim_{n \to \infty} \vert S_{n3} \vert =0$.

The term $S_{n4}$ can also be treated by a coupling construction. Here, we choose $q = q_n = \kn^{-1-\eps}$ for some $\eps \in (0,3/4)$. By similar arguments as before, we obtain that
\begin{align*}
\vert S_{n4} \vert 
&\le
 2 \sum_{i=3}^{\kn} \Big \{ 2 q_n + 54 (\Vert Z_{n1} \Vert_2/ q_n)^{2/5} \bn \alpha((i-2)\bn)^{4/5} \Big \}\\
&\le
 4 \kn^{- \eps} + 108\cdot \kn^{2/5(1+ \eps)} \bn ^{-7/5}\Vert Z_{n1} \Vert_2^{2/5}C  \sum_{i=3}^{\kn}(i-2)^{-12/5} \\
&= 
O(( \kn / \bn^2)^{2/5(1+\eps)}  \bn^{-3/5 + 4/5 \eps}) =o(1),
\end{align*}
by Condition \ref{cond:cond} \eqref{item:rate2} and by the choice of $\eps$. The proof for the disjoint blocks estimator is finished.

\textit{Sliding Blocks}.
By the definition of $\hat T_n^{\slb} $ and $T_n^{\slb}$ we can write
\[
\kn \Exp[\hat T_n^{\slb} - T_n^{\slb}]  = S_{n1}^{\slb} + S_{n2}^{\slb} + S_{n3} ^{\slb}+ S_{n4}^{\slb} + S_{n5}^{\slb} +o(1), 
\]
as $n \to \infty$, where
\begingroup
\allowdisplaybreaks 
\begin{align*}
S_{n1}^{\slb} &= \frac{1}{\bn} \sum_{s =1}^{\bn} \sum_{t=1}^{\bn} \Exp \Big[ \ind \Big(U_s > 1- \tfrac{Z_{nt}^{\slb}}{\bn}\Big) - \tfrac{Z_{nt}^{\slb}}{\bn}\Big]\\
S_{n2}^{\slb} &= \frac{1}{\bn} \frac{\kn -1}{\kn} \sum_{s = \bn+1 }^{2\bn}\sum_{t=1}^{\bn}  \Exp \Big[ \ind \Big(U_s > 1- \tfrac{Z_{nt}^{\slb}}{\bn}\Big) - \tfrac{Z_{nt}^{\slb}}{\bn}\Big]\\
S_{n3}^{\slb} &=\frac{1}{\bn} \frac{\kn - 2}{ \kn} \sum_{s=1}^{ \bn} \sum_{t=\bn+1}^{2\bn} \Exp \Big[ \ind\Big (U_s > 1- \tfrac{Z_{nt}^{\slb}}{\bn}\Big) - \tfrac{Z_{nt}^{\slb}}{\bn}\Big]\\
S_{n4}^{\slb}&=\frac{1}{\bn} \sum_{i = 3}^{\kn-1} \frac{\kn - i }{\kn } \sum_{s \in I_1} \sum_{t \in I_i} \Exp\Big [ \ind \Big(U_s > 1- \tfrac{Z_{nt}^{\slb}}{\bn}\Big) - \tfrac{Z_{nt}^{\slb}}{\bn}\Big]\\
S_{n5}^{\slb} &=\frac{1}{\bn} \sum_{i = 3}^{\kn} \frac{\kn - i + 1 }{\kn } \sum_{s \in I_i} \sum_{t \in I_1} \Exp\Big [ \ind \Big(U_s > 1- \tfrac{Z_{nt}^{\slb}}{\bn}\Big) - \tfrac{Z_{nt}^{\slb}}{\bn}\Big].
\end{align*}
\endgroup
$S_{n3}^{\slb}$ and $S_{n4}^{\slb}+S_{n5}^{\slb}$ are negligible by the same reasons as for the treatment of $S_{n2}$ and $S_{n4}$ above, respectively. Regarding $S_{n1}^{\slb}$, we can write
\begin{multline*}
 S_{n1}^{\slb}=\frac{1}{\bn} \sum_{s =1}^{\bn} \sum_{t=1}^{\bn} \Exp \Big [ \ind \Big (U_s > 1- \tfrac{Z_{nt}^{\slb}}{\bn} \Big) - \tfrac{Z_{nt}^{\slb}}{\bn}\Big ]\\
= \frac{1}{\bn} \sum_{t=1}^{\bn} \sum_{s=1}^{t-1} \Exp \Big[ \ind\Big(U_s > 1- \tfrac{Z_{nt}^{\slb}}{\bn}\Big) - \tfrac{Z_{nt}^{\slb}}{\bn}\Big] -  \frac{1}{\bn^2} \sum_{t=1}^{\bn} \sum_{s=t}^{\bn} \Exp [Z_{nt}^{\slb}].
\end{multline*}
 The first summand on the right-hand side vanishes by similar arguments as we used to show the negligibility of $S_{n2}$ above. Furthermore, the second sum on the right-hand side converges to $-\frac{1}{2 \theta}$ for $n \to \infty$, by Condition \ref{cond:cond}\eqref{item:moment}. Hence, $\lim_{n\to \infty} S_{n1}^{\slb} = - \frac{1}{2 \theta}$. Similarly,  $\lim_{n\to \infty} S_{n2}^{\slb} = - \frac{1}{2 \theta}$,  which finishes the proof.
\end{proof}

\begin{proof}[Proof of Lemma~\ref{lem:archbound}.]
A function $f$ is slowly varying with index $\alpha \in \R$, notationally $f\in RV_\alpha$, if $\lim_{t \to \infty} f(tx)/f(t) = x^\alpha$ for any $x>0$.
Recall the Potter bounds (\citealp{BinGolTeu87}, Theorem 1.5.6): if $f\in RV_\alpha$, then, for any $\delta_1, \delta_2>0$, there exists some constant $t_0=t_0(\delta_1, \delta_2)$ such that, for any $t$ and $x$ with $t\ge t_0, tx \ge t_0$:
\[
(1-\delta_1) x^\alpha \min(x^{\delta_2}, x^{-\delta_2}) \le \frac{f(tx)}{f(t)} \le (1+ \delta_1) x^\alpha \max(x^{\delta_2}, x^{-\delta_2}).
\] 
Let $U(z)=F^{\leftarrow}(1-1/z)=\{1/(1-F)\}^{\leftarrow}(z)$. Since $1-F(x) \sim cx^{-\kappa}$, the function $x\mapsto 1/(1-F(x))$ is regularly varying with index $\kappa$.  We obtain that $U\in RV_{1/\kappa}$ by, e.g., Proposition~0.8~(v) in \cite{Res87}, 

For non-negative integers $j>i$ define
\[
\textstyle \Pi_{i+1:j} = \prod_{k=i+1}^j A_k, \qquad Y_{i+1:j} = \sum_{k=i+1}^j \Pi_{k+1:j} B_k.
\]
Then $X_j = \Pi_{i+1:j} X_i + Y_{i+1:j}$ and $(\Pi_{i+1:j}, Y_{i+1:j})$ is independent of $X_i$. We obtain that
\begin{align*}
&\, \Prob(U_i > 1- y, U_j > 1-y)  \\
=&\,
\Prob\{ X_i > F^{\leftarrow}(1-y), \Pi_{i+1:j} X_i +Y_{i+1:j} >  F^{\leftarrow}(1-y) \} \\
\le&\, P_{n1} + P_{n2}
\end{align*}
where 
\begin{align*}
P_{n1} &= \Prob\{ X_i > F^{\leftarrow}(1-y), \Pi_{i+1:j} X_i >  F^{\leftarrow}(1-y) /2  \}, \\
P_{n2}& = 
\Prob\{ X_i > F^{\leftarrow}(1-y), Y_{i+1:j} >  F^{\leftarrow}(1-y)/2 \}.
\end{align*}
Consider $P_{n2}$. By independence of $Y_{i+1:j}$ and $X_i$, we get the bound
\begin{align*}
P_{n2} &\le \Prob\{ X_i > F^{\leftarrow}(1-y) \} \Prob\{ Y_{i+1:j} >  F^{\leftarrow}(1-y)/2 \} \\
&\le
y \Prob\{ X_j >  F^{\leftarrow}(1-y)/2 \} \\
&=
y [ 1- F\{ F^{\leftarrow}(1-y)/2\} ]  \\
&\le
2^{\kappa+2} y^2
\end{align*}
The last inequality follows from the Potter bounds applied to $1-F$ ($\delta_1=\delta_2=1$): we may choose $c_1$ sufficiently small  such that
\[
1- F\{ F^{\leftarrow}(1-y)/2\}  
\le
2 (1/2)^{-\kappa-1}  [ 1- F\{ F^{\leftarrow}(1-y)\}   ]  = 2^{\kappa+2} y  \qquad \forall \, y \in (0,c_1).
\]
Now consider $P_{n1}$. By Markov's inequality and a change of variable, for any $\xi\in(0,\kappa)$,
\begin{align*}
P_{n1} &= 
\int_{F^{\leftarrow}(1-y)}^\infty \Prob\left\{ \Pi_{i+1:j} u > F^{\leftarrow}(1-y)/2 \right\}\, F(\diff u) \\
&\le 
\int_{F^{\leftarrow}(1-y)}^\infty \Exp[ \Pi_{i+1:j}^\xi]  \left\{ \frac{U(1/y) )}{2u}\right\}^{-\xi} \, F(\diff u) \\
&= 
2^{\xi} \Exp[A_1^\xi]^{j-i}  \int_0^y  \left\{ \frac{U(1/v)}{U(1/y)}\right\}^{\xi} \, \diff v
\end{align*}
By the Potter bounds applied to $U\in RV_{1/\kappa}$, with $\delta_1=1$ and $\delta_2 \in (0,1/\xi - 1/\kappa)$,  we have, for all sufficiently large $t$ and for all $x\ge 1$,
\[
\frac{U(tx)}{U(t)} \le 2 x^{\tau}, \qquad \text{where } \tau=1/\kappa + \delta_2 < 1/\xi.
\]
With $t=1/y \ge 1/c_1$ and $x=y/v\ge 1$ we obtain, after decreasing $c_1$ if necessary,
\[
\int_0^y  \left\{ \frac{U(1/v)}{U(1/y)}\right\}^{\xi} \, \diff v 
\le
2^\xi \int_0^y (y/v)^{\xi \tau} \, \diff v
= \frac{2^\xi }{1-\tau\xi}  \cdot y
\]
As a consequence, $P_{n1} \le 4^\xi/(1-\tau\xi) \Exp[A_1^\xi]^{j-i} y$.

The derived bounds on $P_{n1}$ and $P_{n2}$ directly yield the bound 
\begin{align*}
&\hspace{-1.2cm} \Exp\Big\{\sum_{i=1}^{n} \ind(U_i > 1- y)  \Big\}^2  \\
=&\,
\sum_{i=1}^n \Prob(U_i>1-y) + 2\sum_{1\le i < j \le n}  \Prob(U_i >1-y, U_j>1-y) \\
\le&\, n y  + 2 n^2  \cdot 2^{\kappa+2} y^2 +  2 n \frac{4^\xi}{1-\tau\xi}\Big( \sum_{s=1}^\infty \Exp[A_1^\xi]^s\Big) y.
\end{align*}
The assertion follows from the fact that  $\Exp[A_1^\xi] < \Exp[A_1^\kappa]=1$ by condition (S).
\end{proof}

\section{Additional Simulation results}
\label{sec:addsim}

In this section, we present additional results of the simulation study (see also Section \ref{sec:sim}). Figures \ref{fig:mse2}, \ref{fig:mse3} and \ref{fig:mse4} depict the mean squared error $\Exp[(\hat \theta - \theta)^2]$ as a function of the block size parameter $b$ for the ARMAX, the squared ARCH and the Markovian Copula-model, respectively. The curves behave similar as for the ARCH-model (see Figure \ref{fig:mse1}). Additionally, in Figures \ref{fig:bias2}, \ref{fig:var2}, \ref{fig:bias3}, \ref{fig:var3}, \ref{fig:bias1}, \ref{fig:var1}, \ref{fig:bias4} and \ref{fig:var4} we depict the corresponding squared biases and variances for all of the four considered models as a function of the block size $b$. Finally, Figures \ref{fig:msevariance2} and \ref{fig:msevariance3} compare the performances of the estimator of the asymptotic variance for the estimators $\hat \theta_n^{\Be,\slb}$ and $\hat \theta_n^{\No,\slb}$ in the ARMAX and the squared ARCH model. We observe the same behavior as in the simulations for the ARCH model (see Figure \ref{fig:msevariance1}): the approximation for $\hat \theta_n^{\Be,\slb}$ is better than for $\hat \theta_n^{\No,\slb}$.


\setcounter{figure}{4}
\begin{figure}[h!]
\vspace{-.3cm}
\begin{center}
\includegraphics[width=0.46\textwidth]{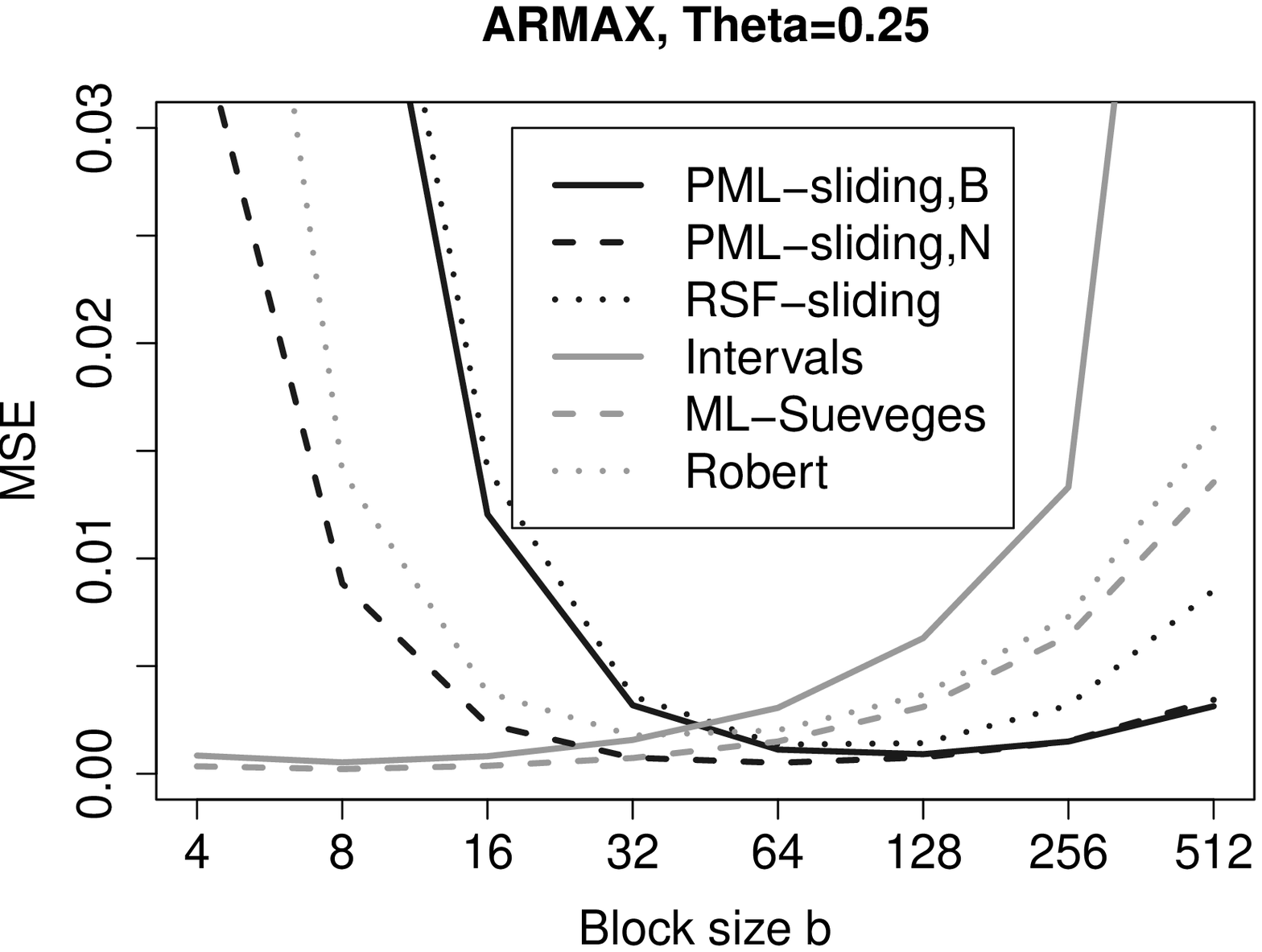}
\hspace{-.3cm}
\includegraphics[width=0.46\textwidth]{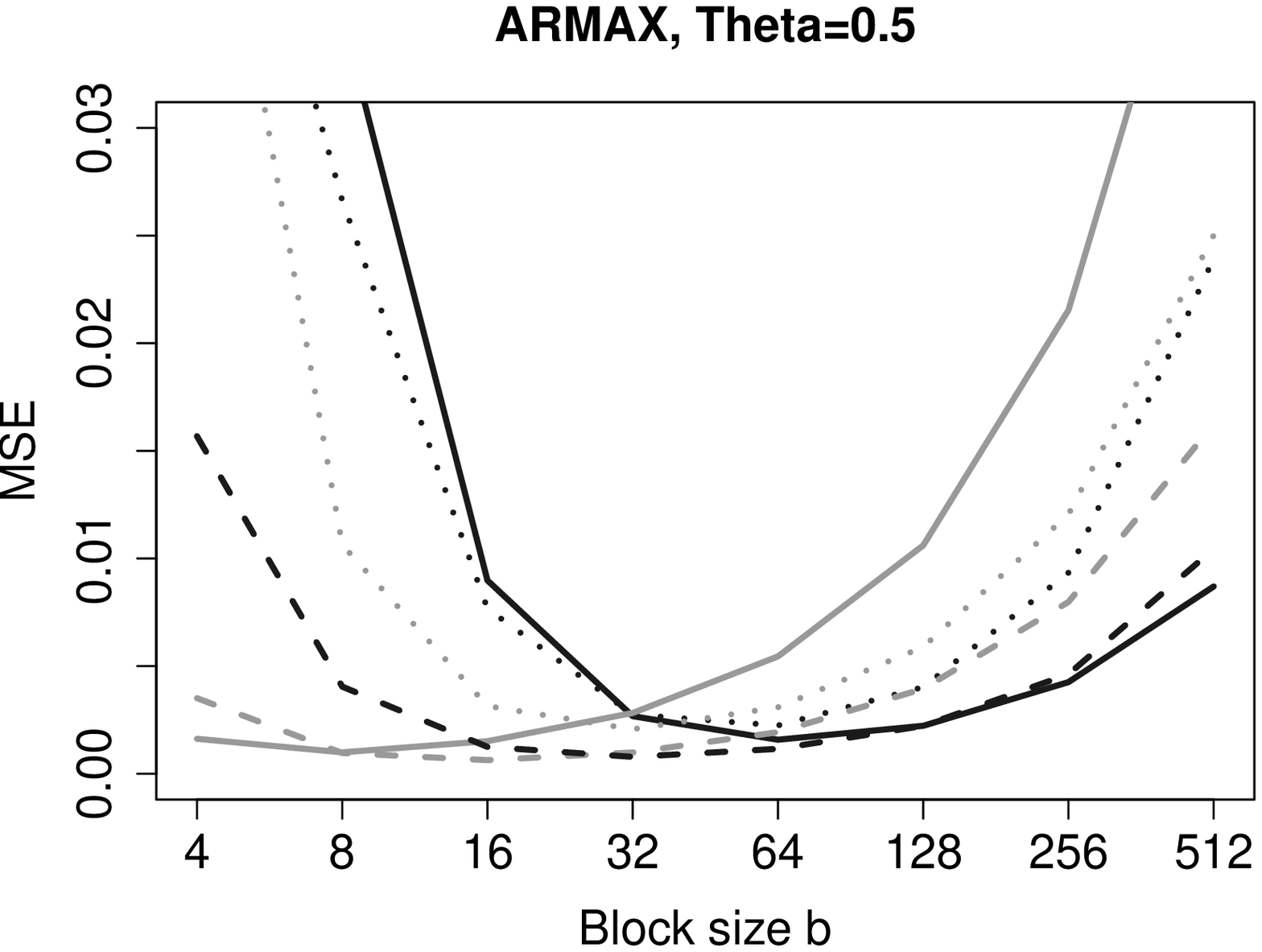}
\vspace{-.2cm}

\includegraphics[width=0.46\textwidth]{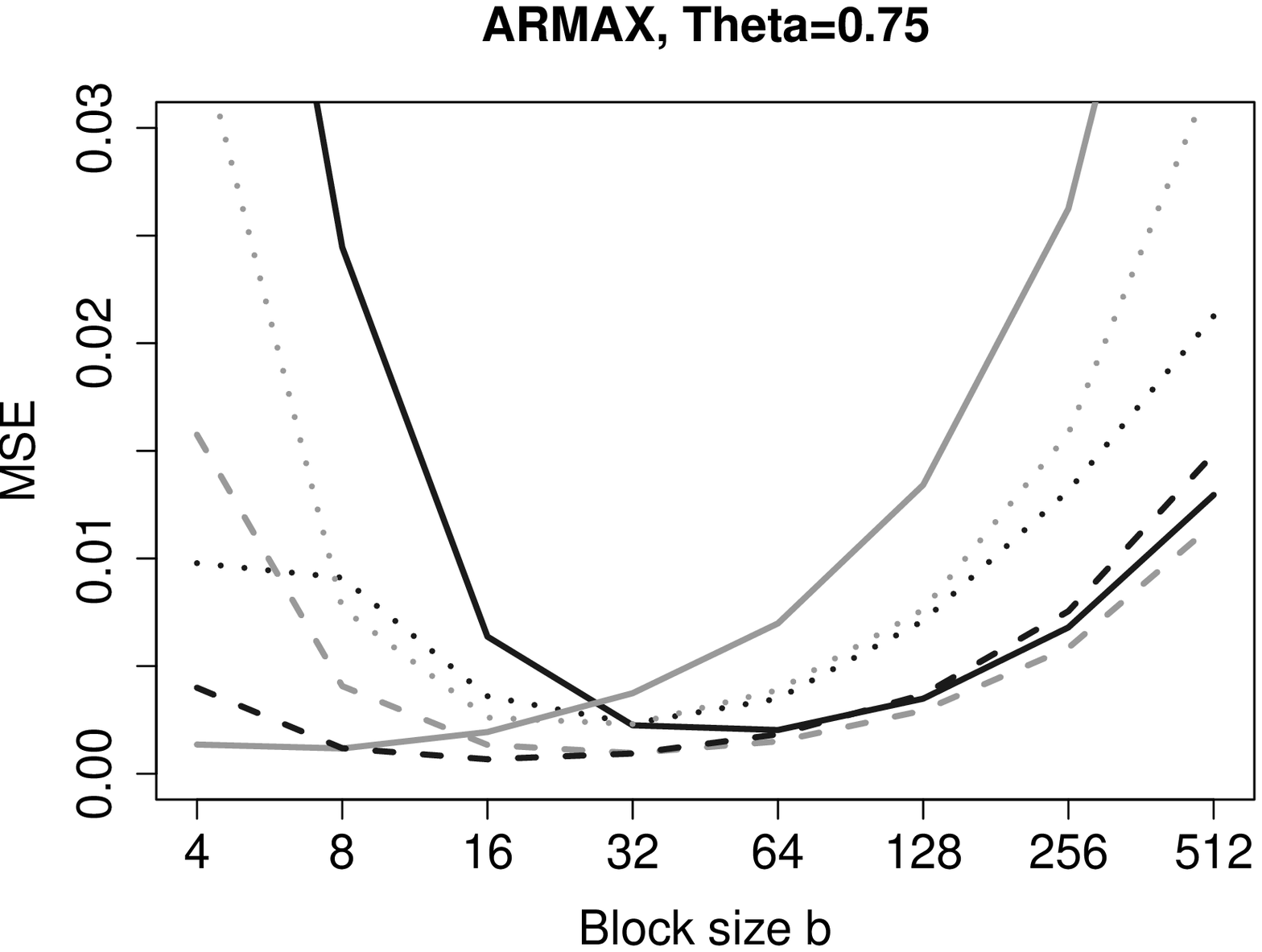}
\hspace{-.3cm}
\includegraphics[width=0.46\textwidth]{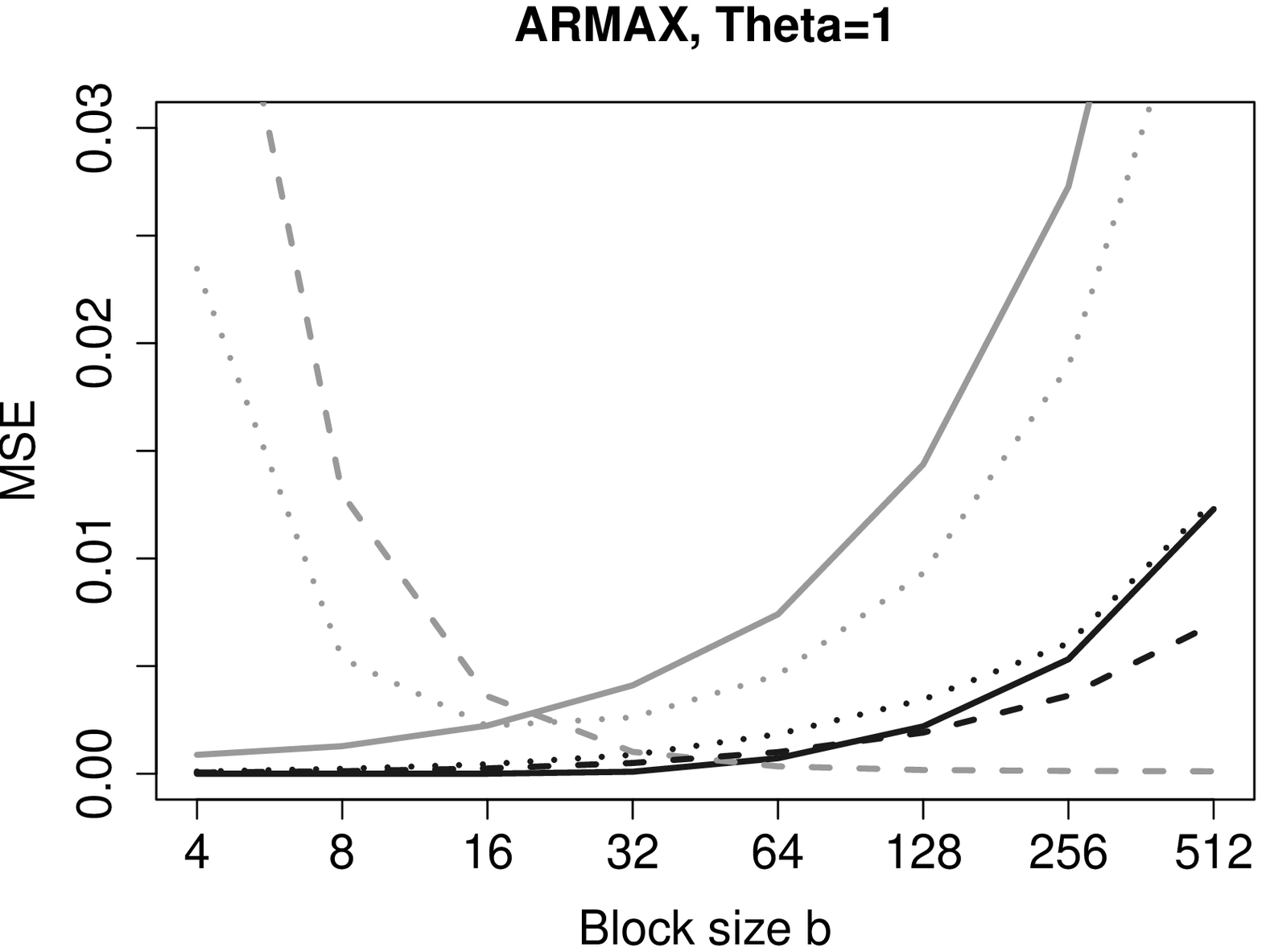}
\end{center}
\vspace{-.5cm}
\caption{\label{fig:mse2}  Mean squared error for the estimation of $\theta$ within the ARMAX-model for four values of $\theta\in\{0.25,0.5,0.75,1\}$. 
}

\vspace*{\floatsep}

\begin{center}
\includegraphics[width=0.43\textwidth]{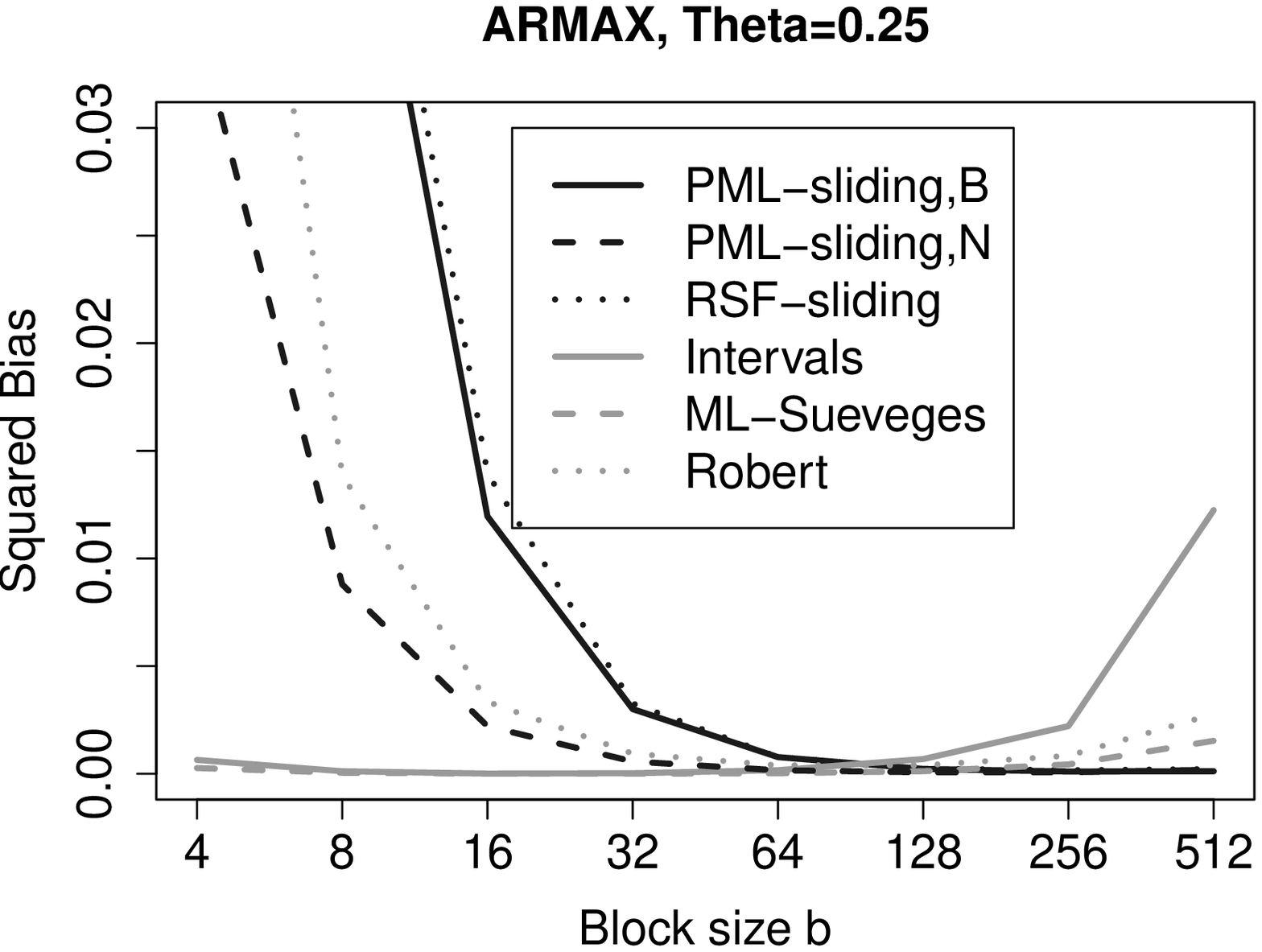}
\hspace{-.3cm}
\includegraphics[width=0.43\textwidth]{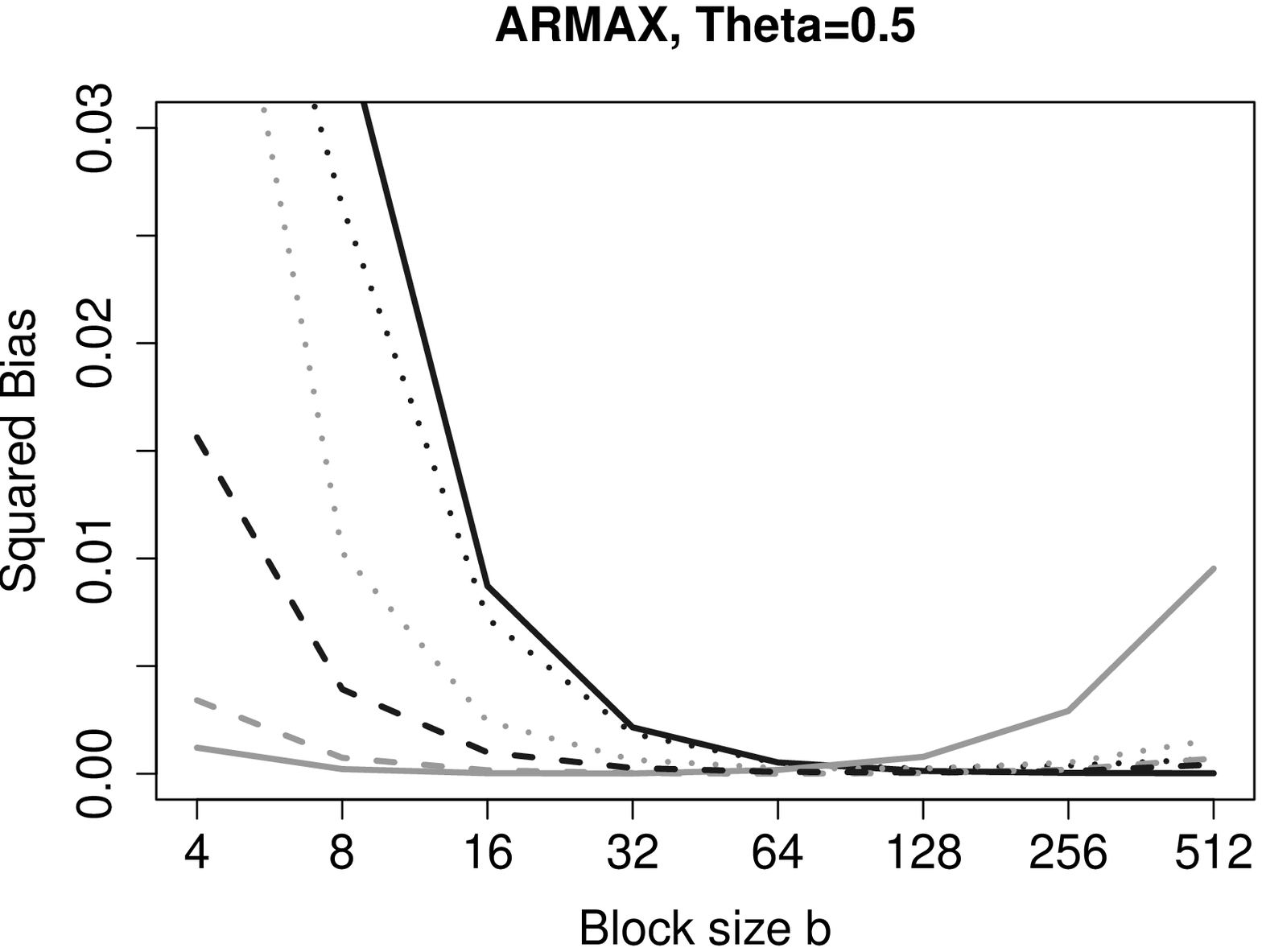}
\vspace{-.2cm}

\includegraphics[width=0.43\textwidth]{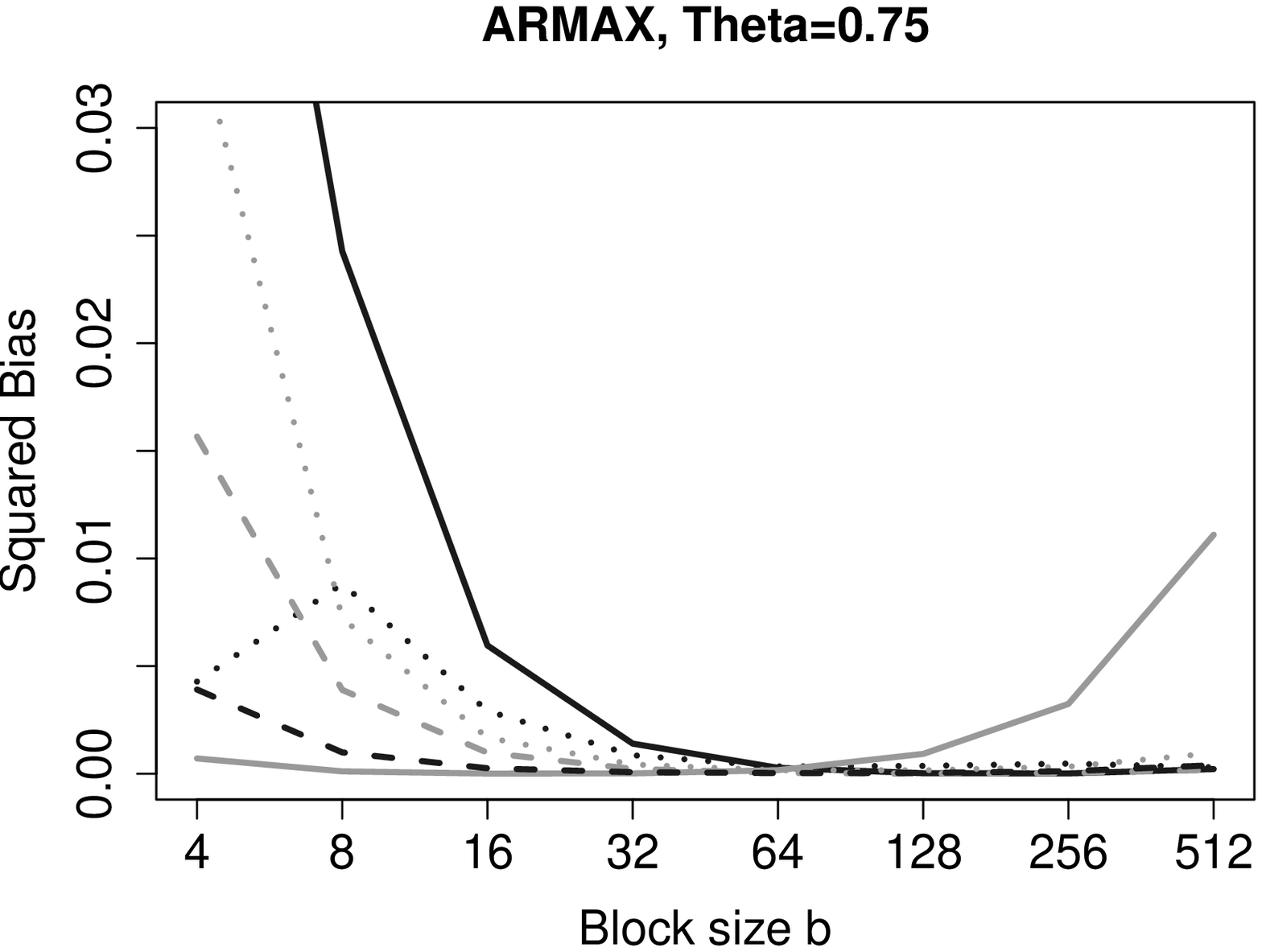}
\hspace{-.3cm}
\includegraphics[width=0.43\textwidth]{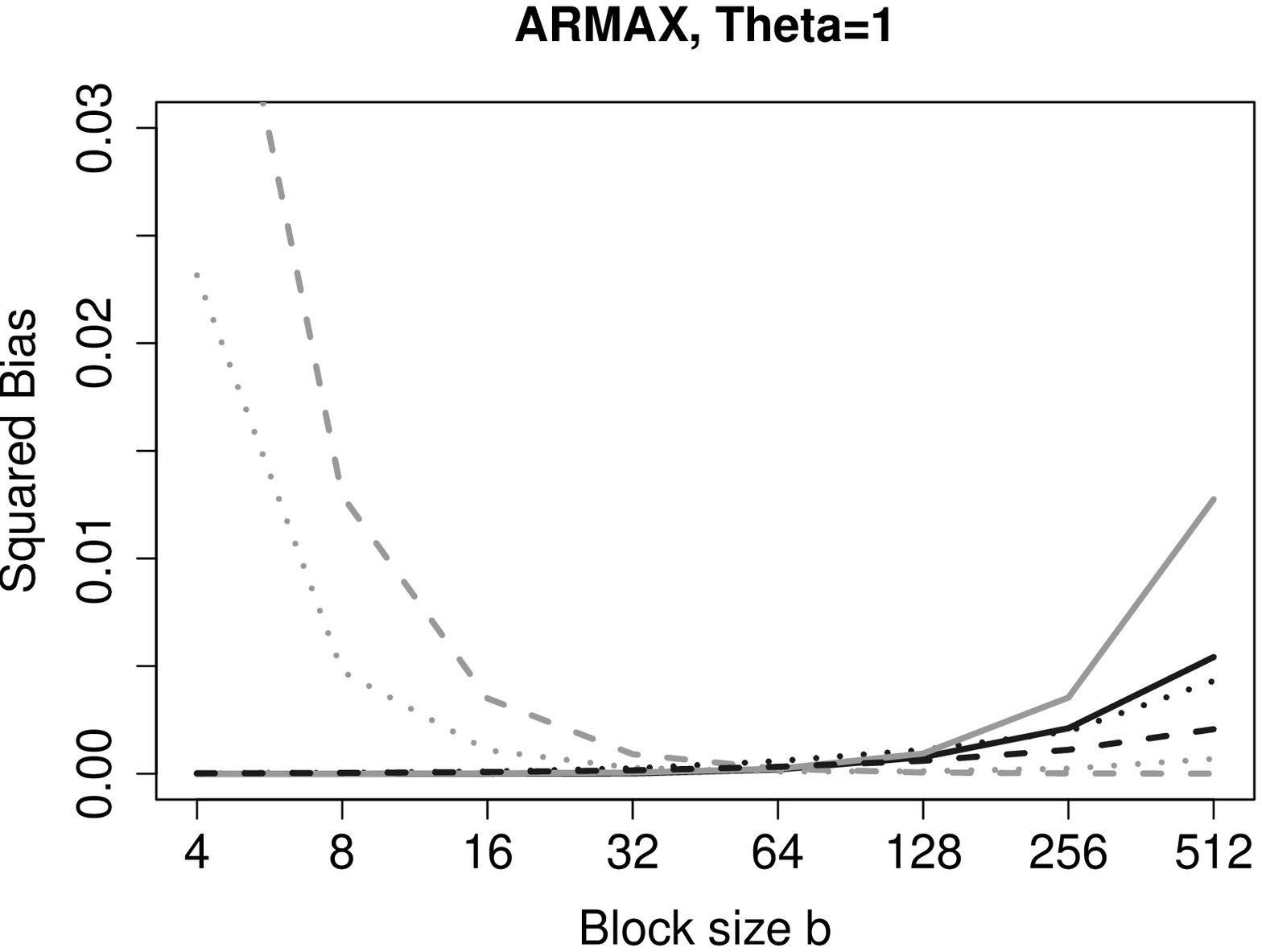}
\end{center}
\vspace{-.3cm}
\caption{\label{fig:bias2}  Squared Bias of the estimation of $\theta$ within the ARMAX-model for four values of $\theta\in\{0.25,0.5,0.75,1\}$. 
}
\end{figure}

\begin{figure}[t!]
\begin{center}
\includegraphics[width=0.43\textwidth]{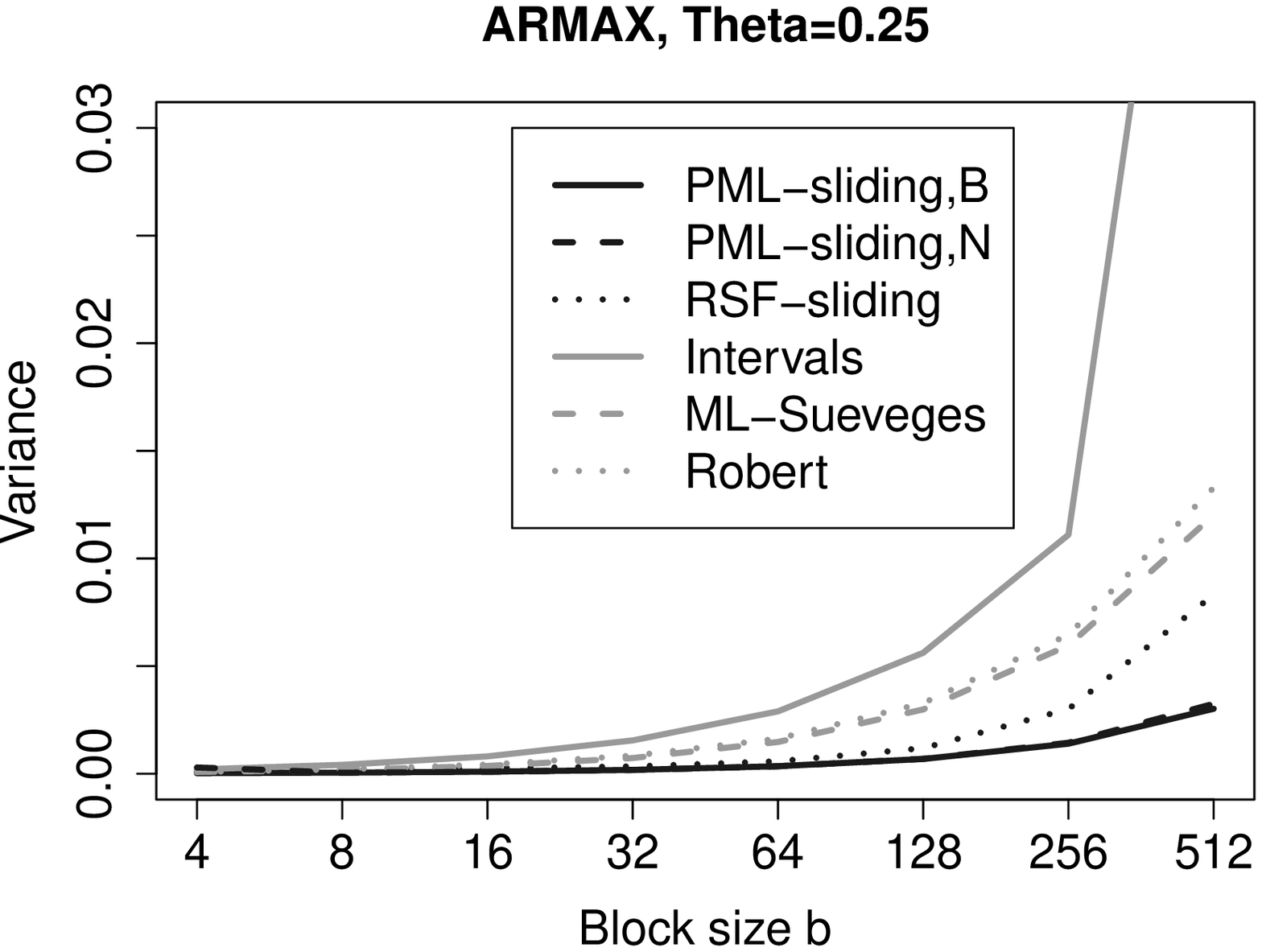}
\hspace{-.3cm}
\includegraphics[width=0.43\textwidth]{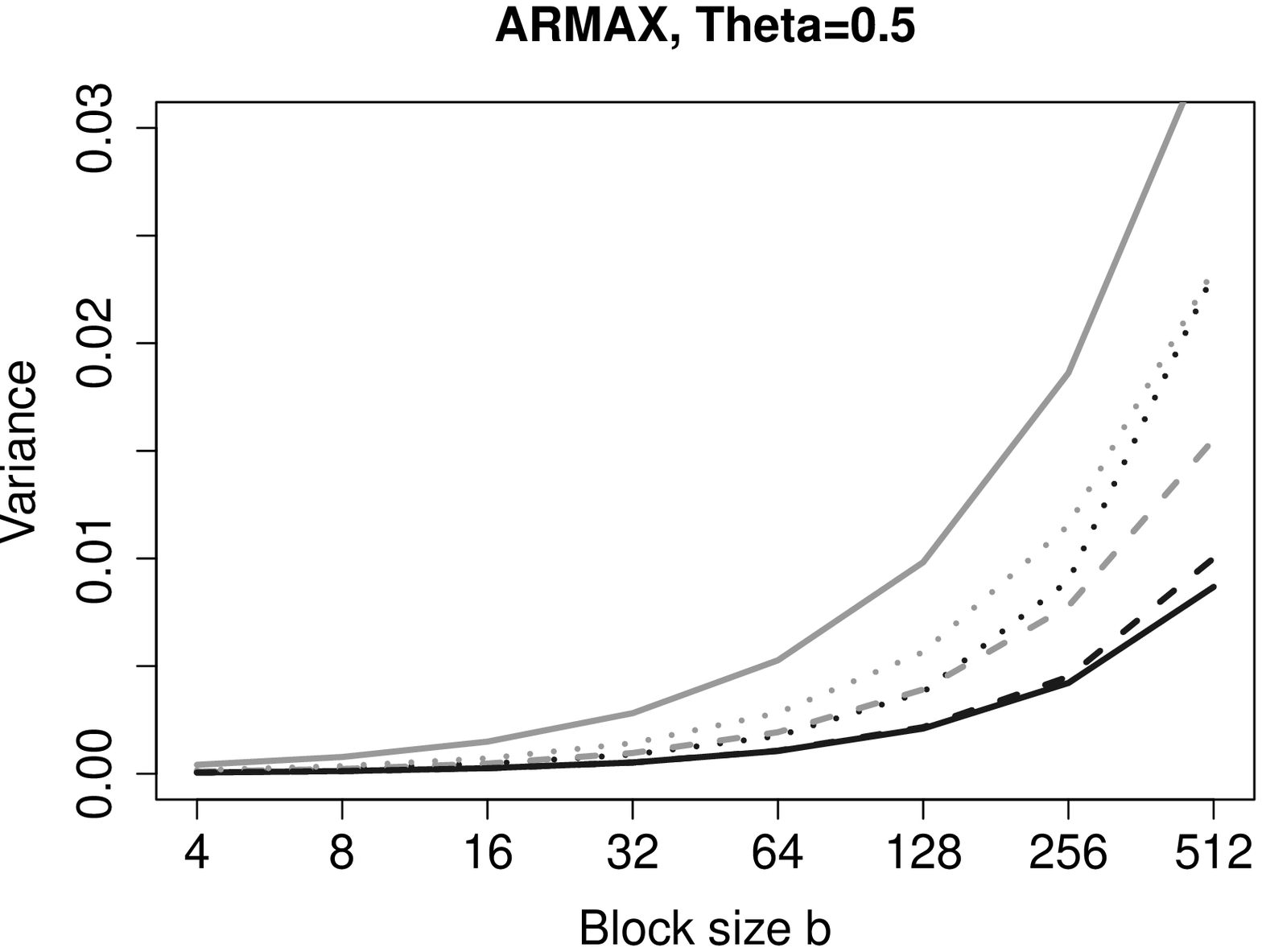}
\vspace{-.2cm}

\includegraphics[width=0.43\textwidth]{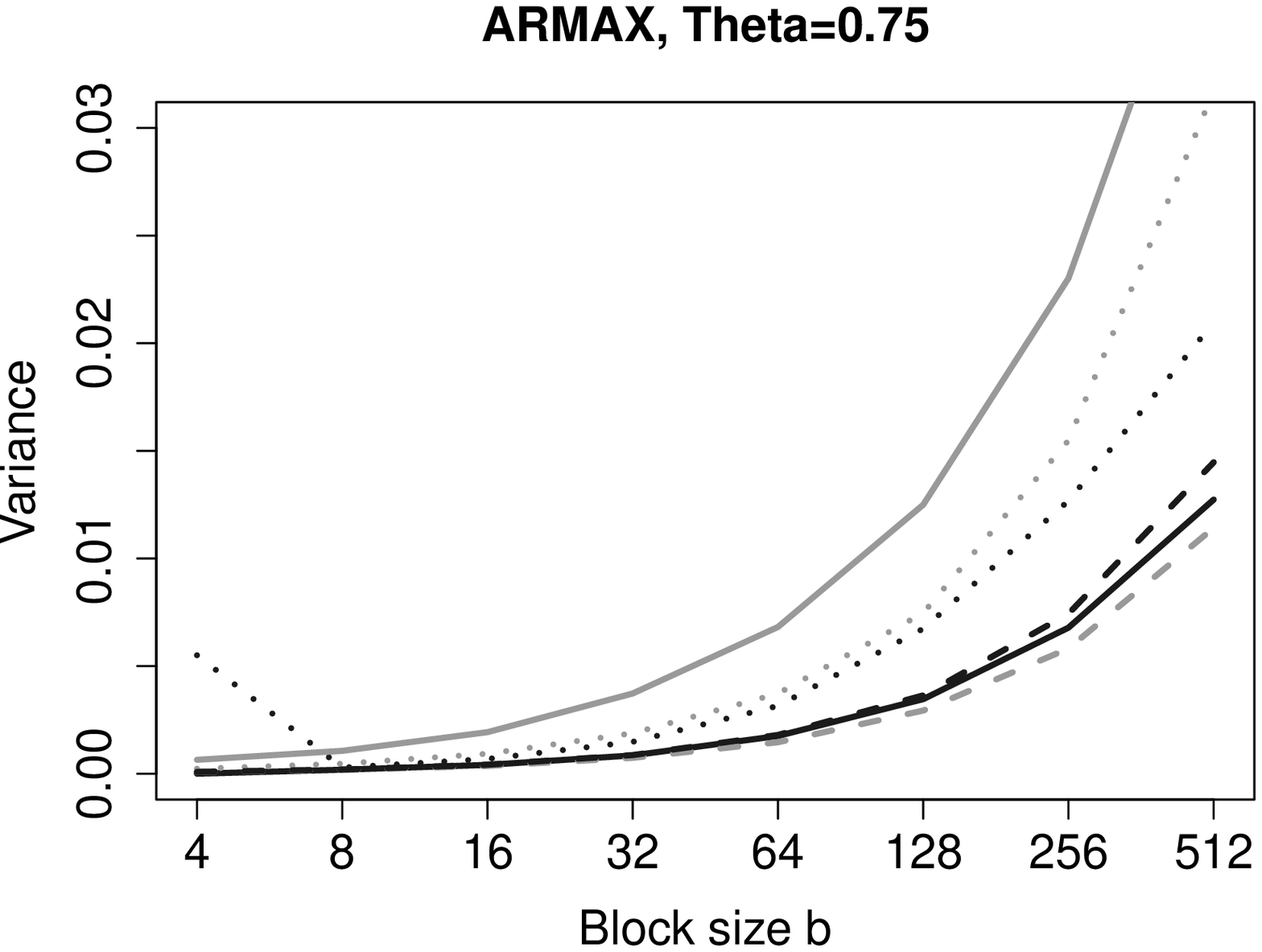}
\hspace{-.3cm}
\includegraphics[width=0.43\textwidth]{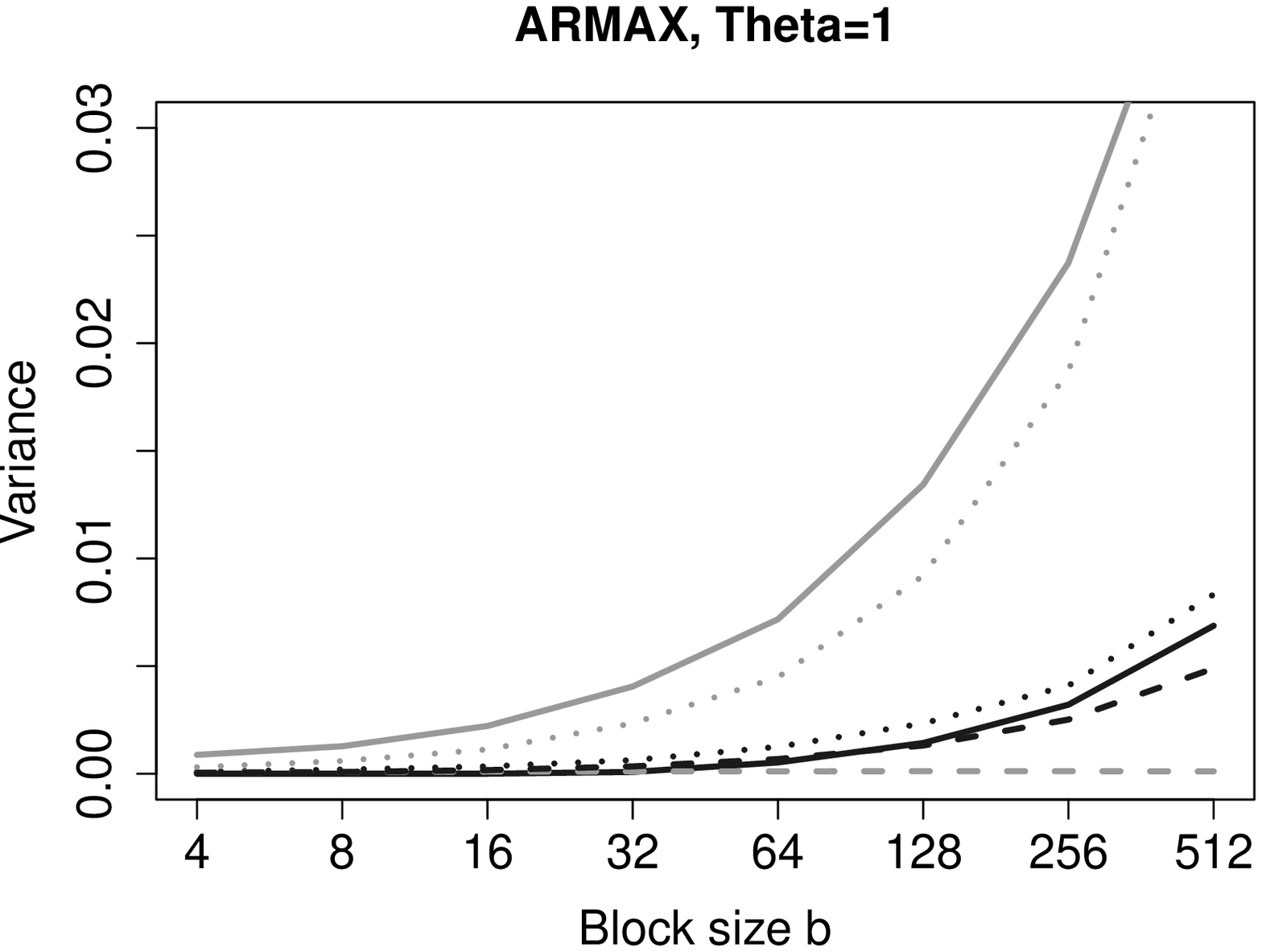}
\end{center}
\vspace{-.3cm}
\caption{\label{fig:var2}  Variance of the estimation of $\theta$ within the ARMAX-model for four values of $\theta\in\{0.25,0.5,0.75,1\}$. 
}
\vspace{-.1cm}
\vspace*{\floatsep}
%
%
\begin{center}
\includegraphics[width=0.43\textwidth]{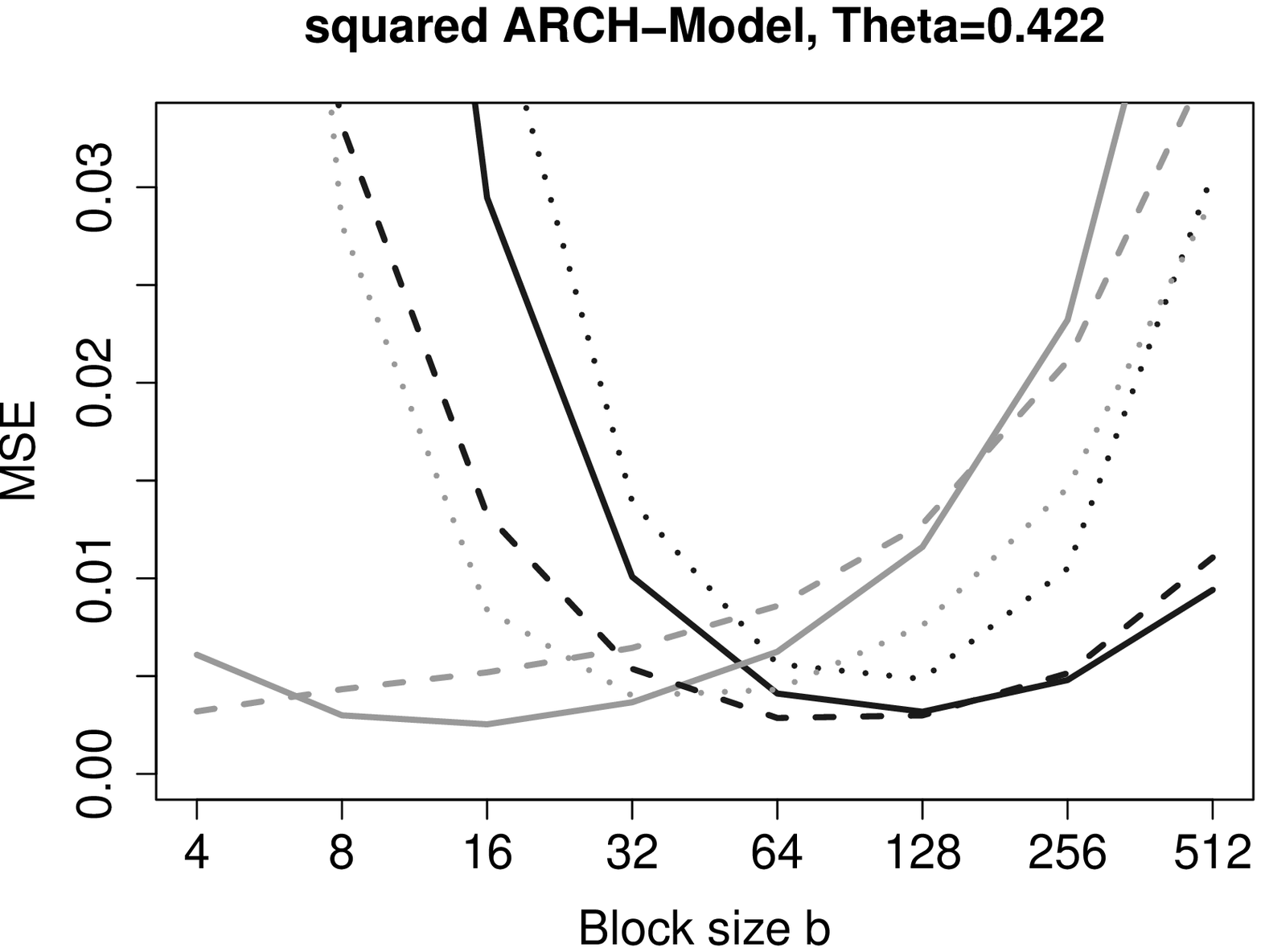}
\hspace{-.3cm}
\includegraphics[width=0.43\textwidth]{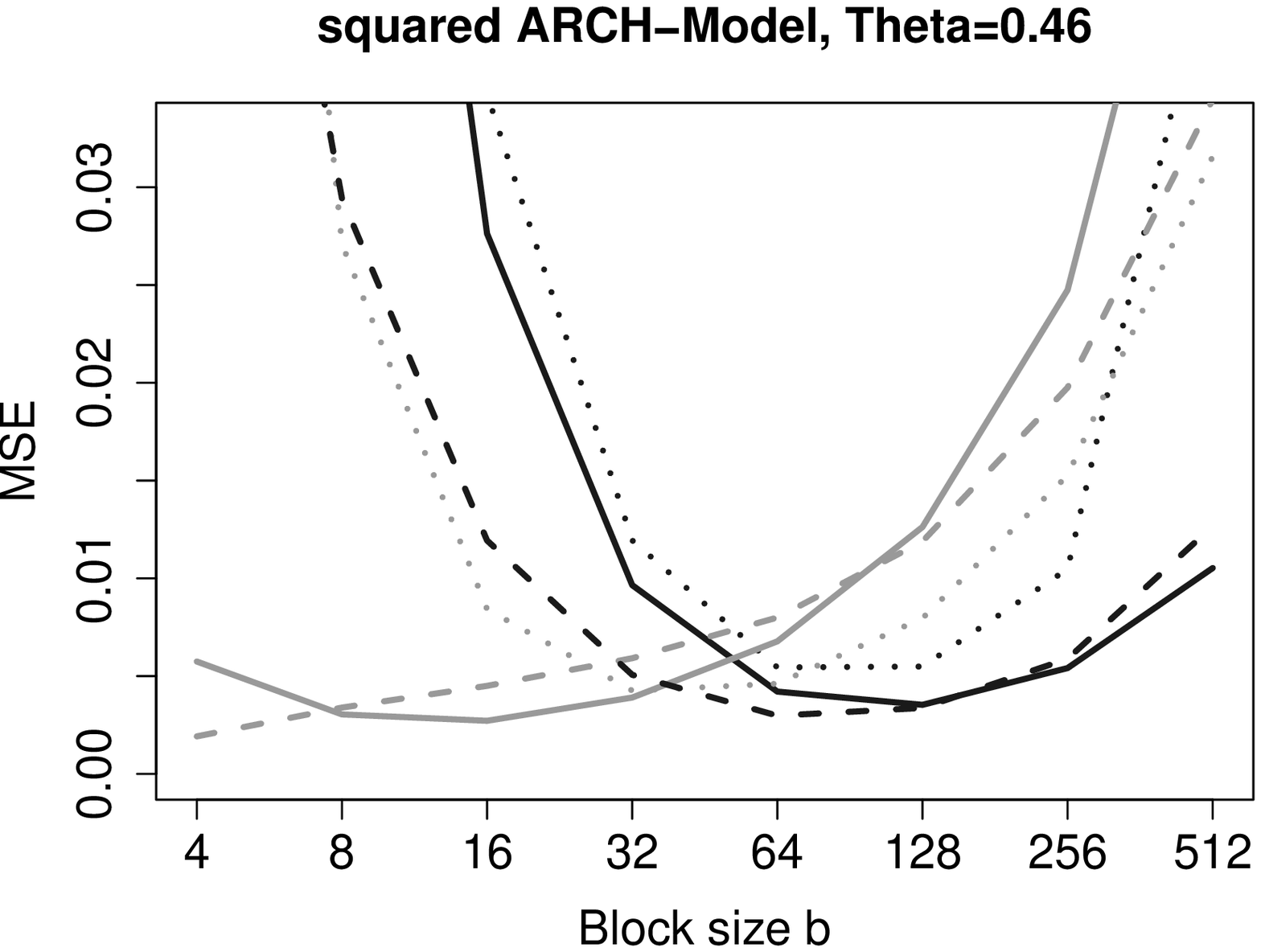}
\vspace{-.2cm}

\includegraphics[width=0.43\textwidth]{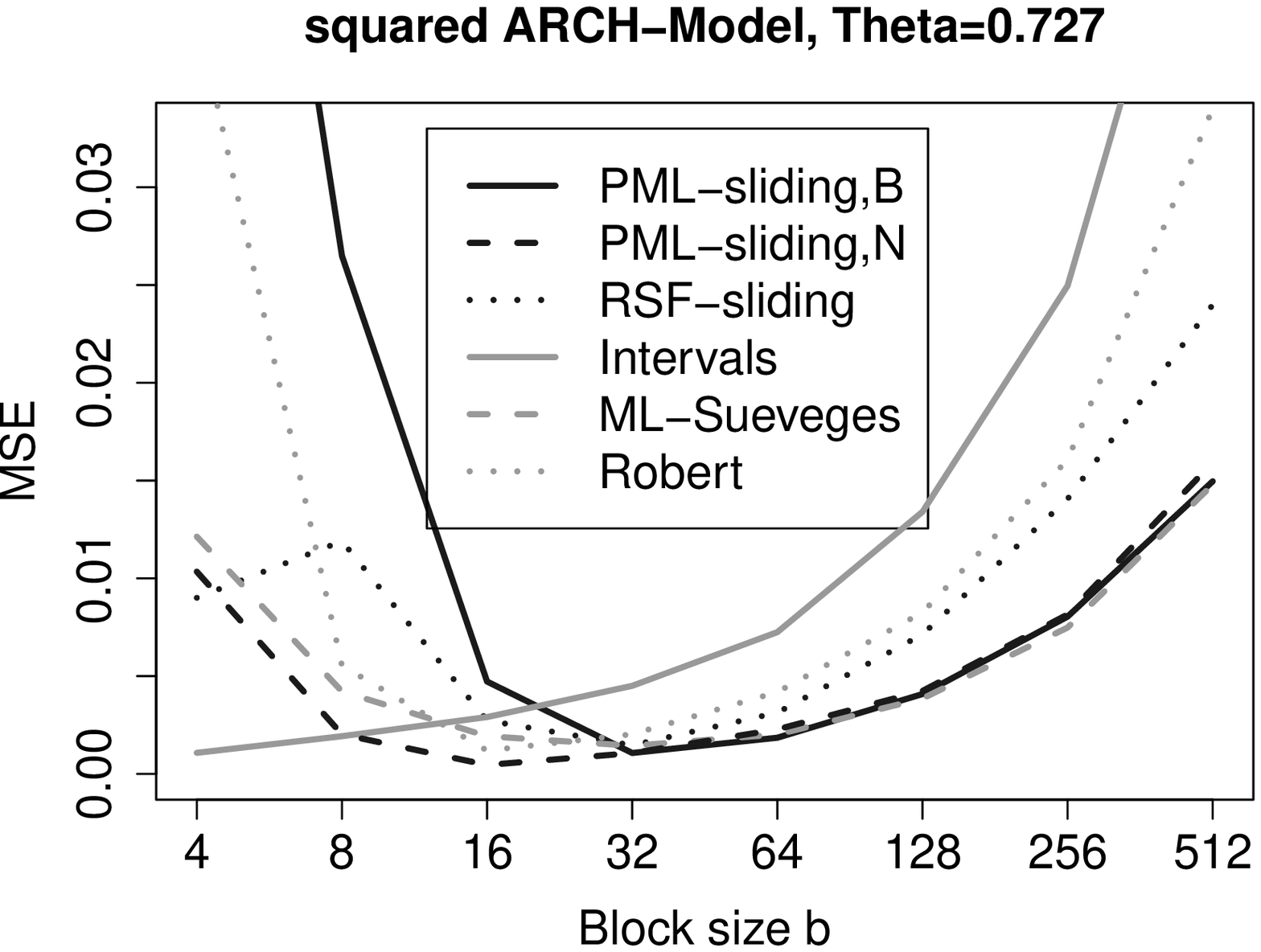}
\hspace{-.3cm}
\includegraphics[width=0.43\textwidth]{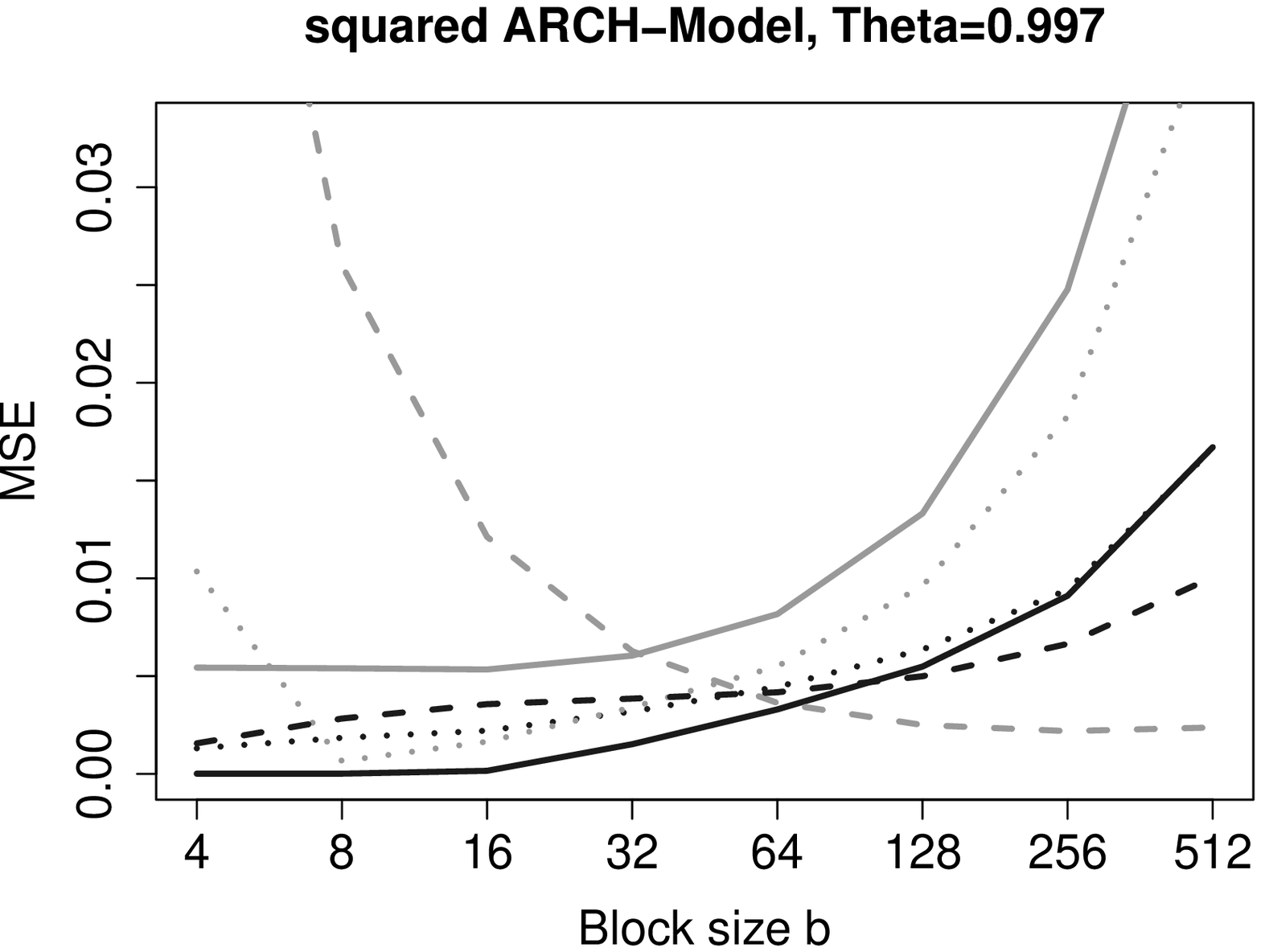}
\end{center}
\vspace{-.3cm}
\caption{\label{fig:mse3}  Mean squared error for the estimation of $\theta$ within the squared ARCH-model for four values of $\theta\in\{0.422, 0.460, 0.727, 0.997\}$. 
}
\end{figure}


\begin{figure}[p!]
\vspace{-.3cm}
\begin{center}
\includegraphics[width=0.43\textwidth]{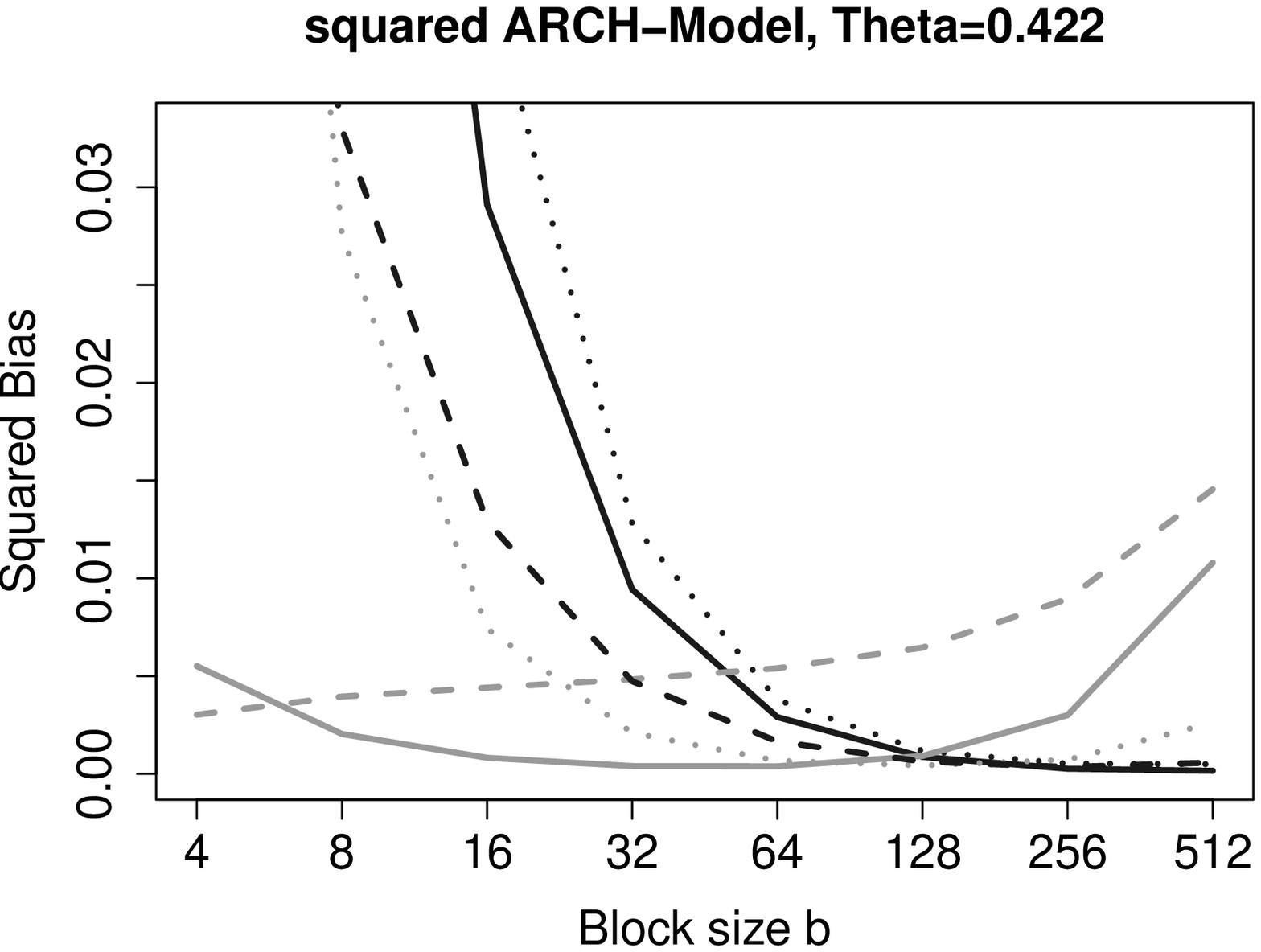}
\hspace{-.3cm}
\includegraphics[width=0.43\textwidth]{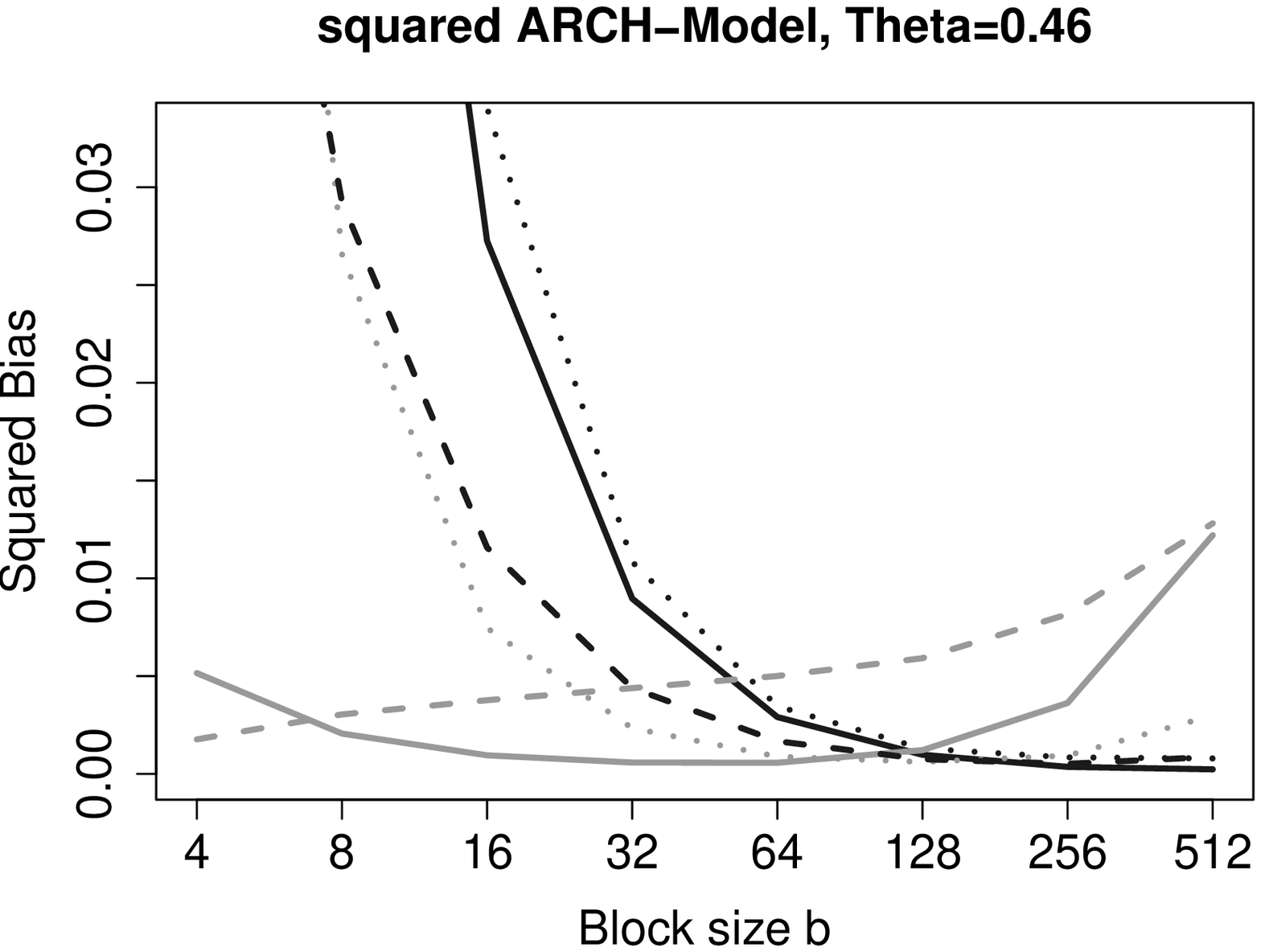}
\vspace{-.2cm}

\includegraphics[width=0.43\textwidth]{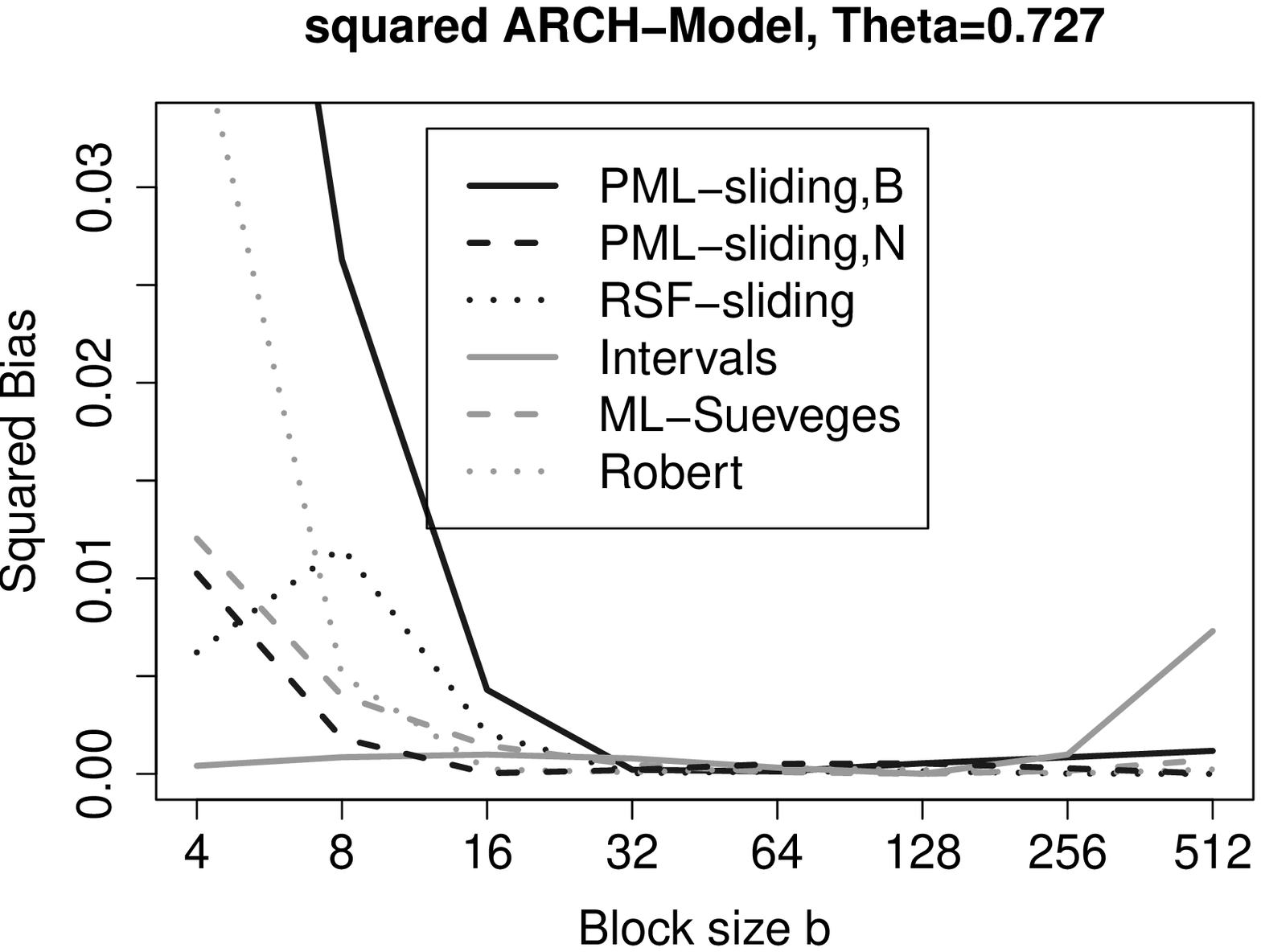}
\hspace{-.3cm}
\includegraphics[width=0.43\textwidth]{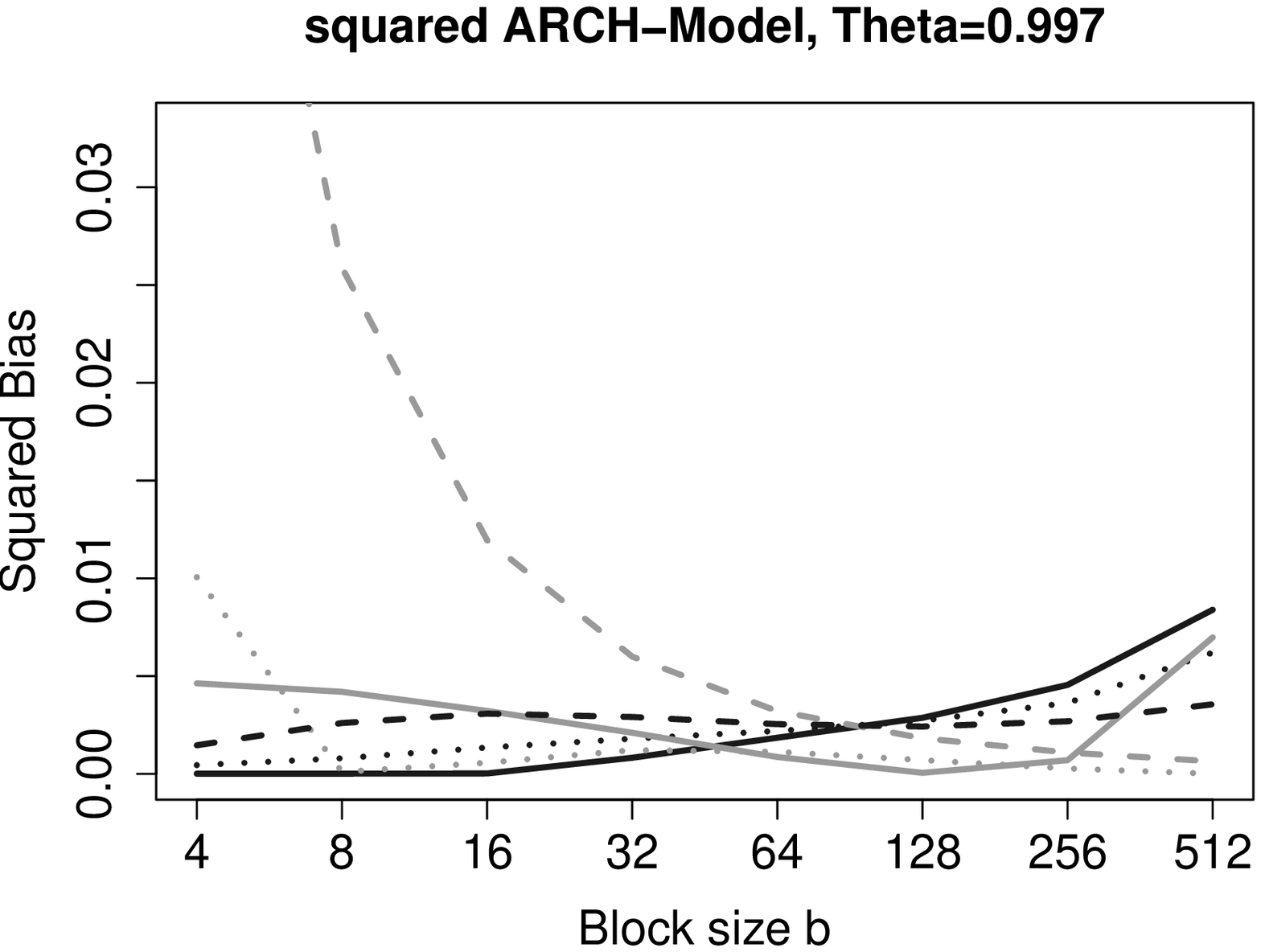}
\end{center}
\vspace{-.3cm}
\caption{\label{fig:bias3}  Bias of the estimation of $\theta$ within the squared ARCH-model for four values of $\theta\in\{0.422, 0.460, 0.727, 0.997\}$. 
}
\vspace*{\floatsep}
\begin{center}
\includegraphics[width=0.43\textwidth]{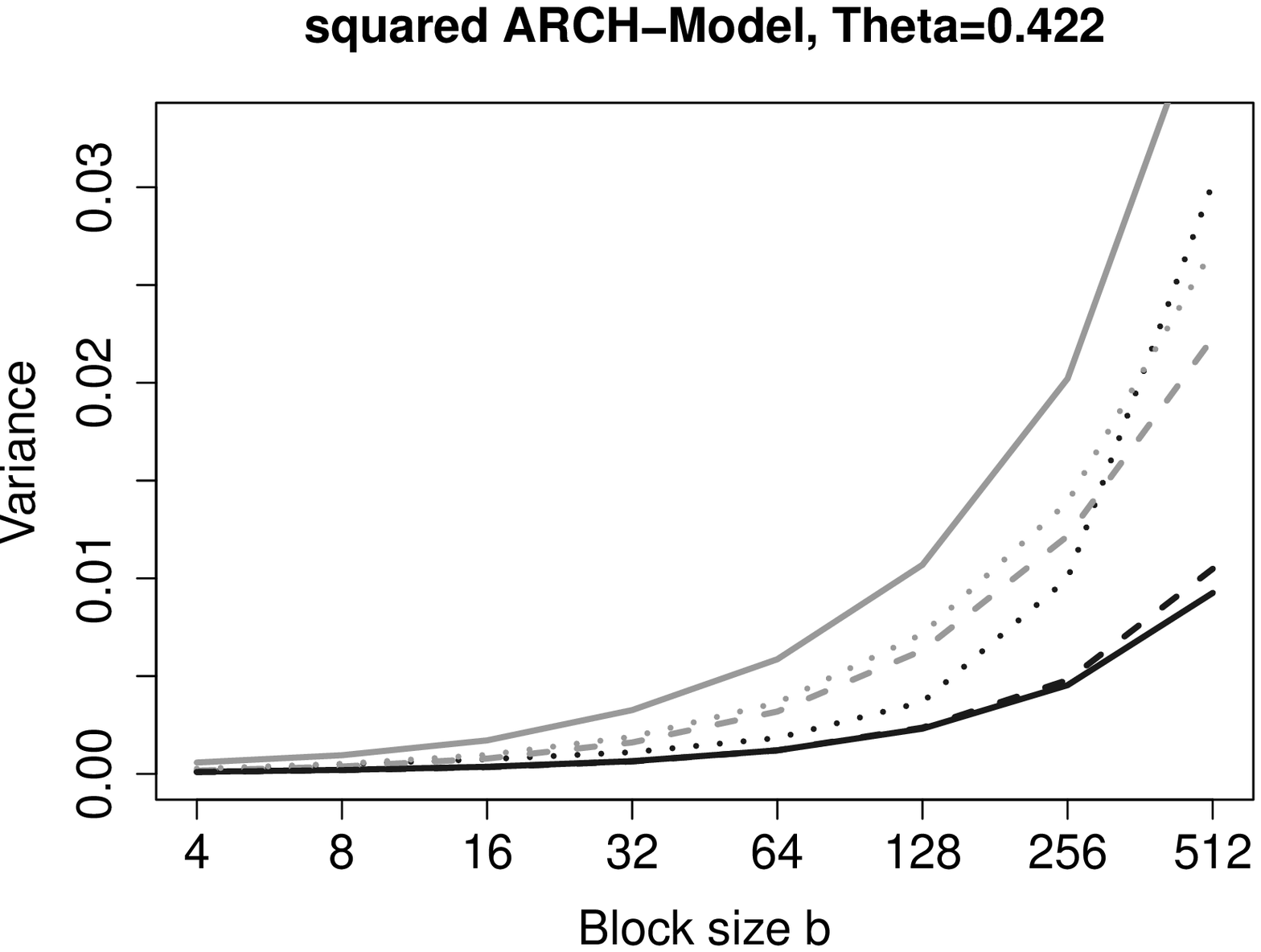}
\hspace{-.3cm}
\includegraphics[width=0.43\textwidth]{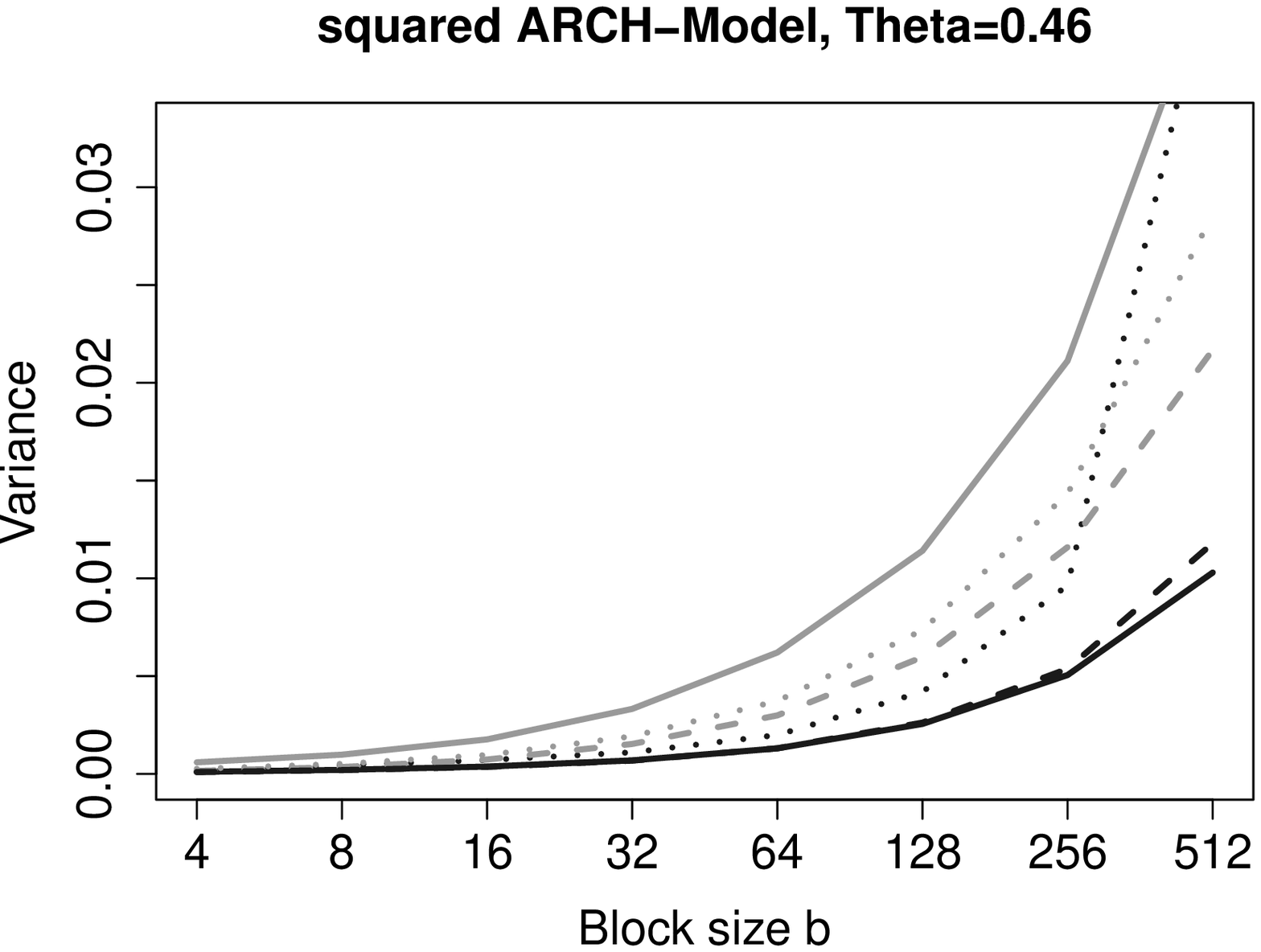}
\vspace{-.2cm}

\includegraphics[width=0.43\textwidth]{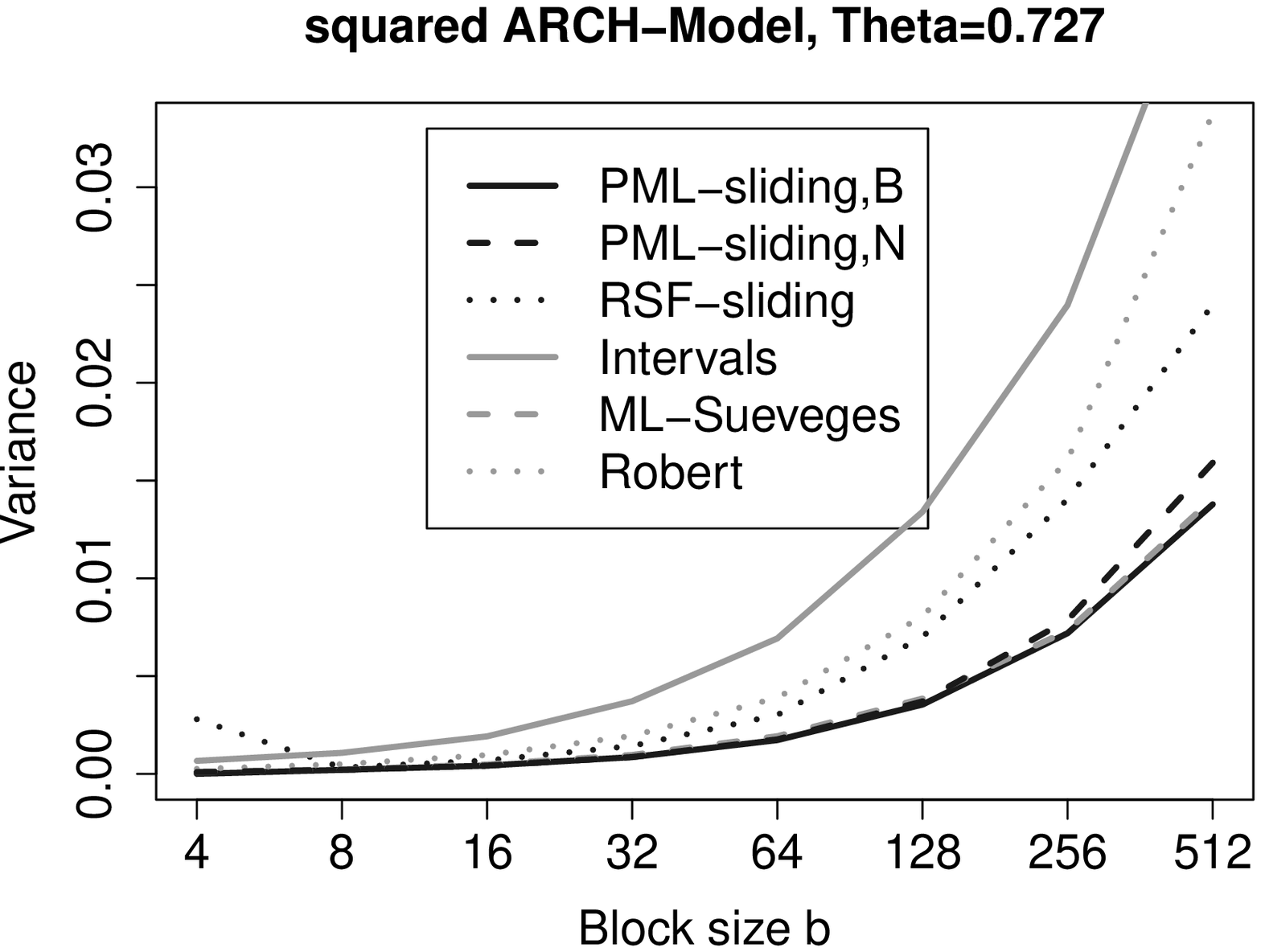}
\hspace{-.3cm}
\includegraphics[width=0.43\textwidth]{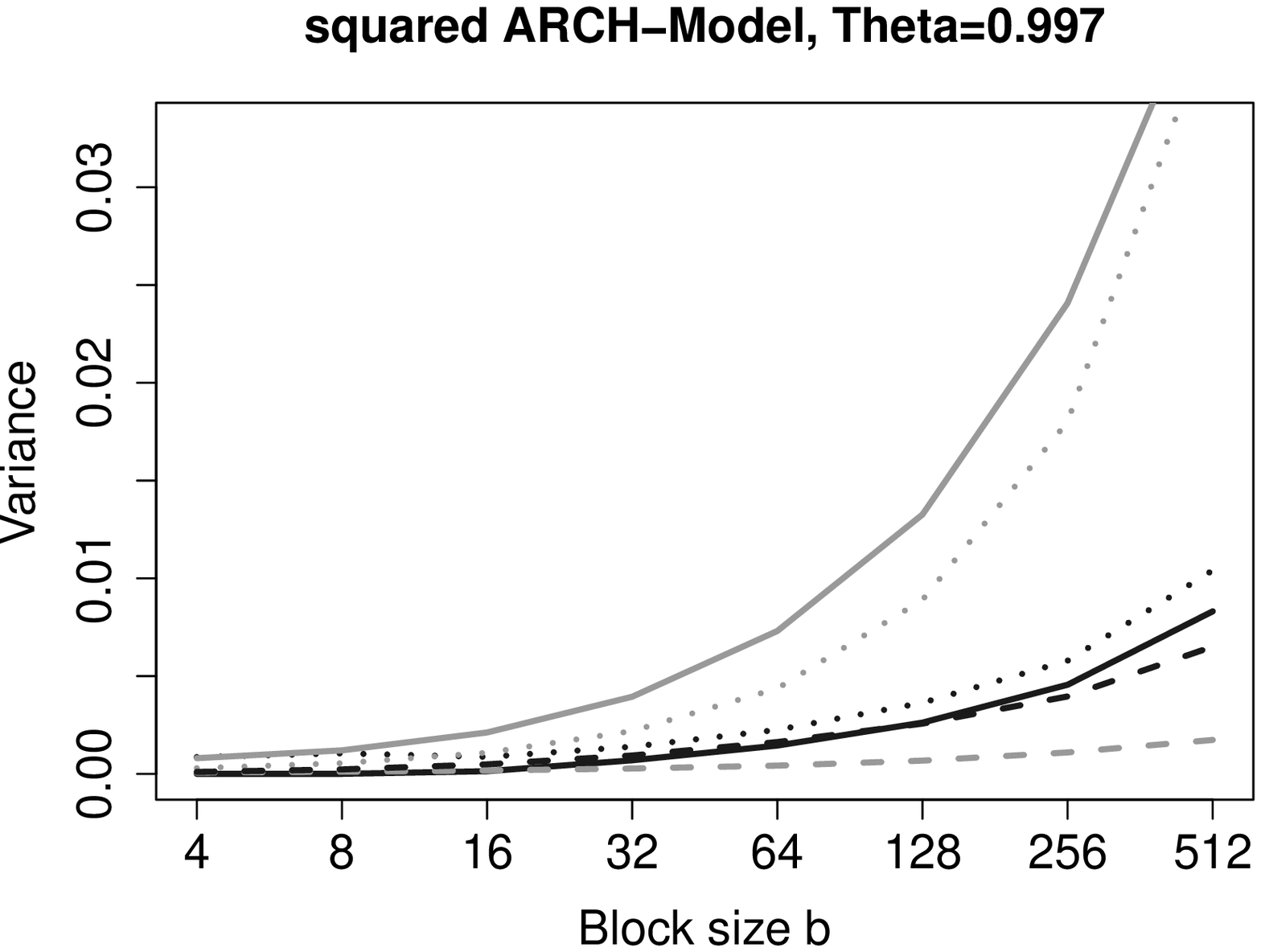}
\end{center}
\vspace{-.3cm}
\caption{\label{fig:var3}  Variance of the estimation of $\theta$ within the squared ARCH-model for four values of $\theta\in\{0.422, 0.460, 0.727, 0.997\}$. 
}

\vspace{-.1cm}
\end{figure}

\begin{figure}[p!]
\vspace{-.3cm}
\begin{center}
\includegraphics[width=0.43\textwidth]{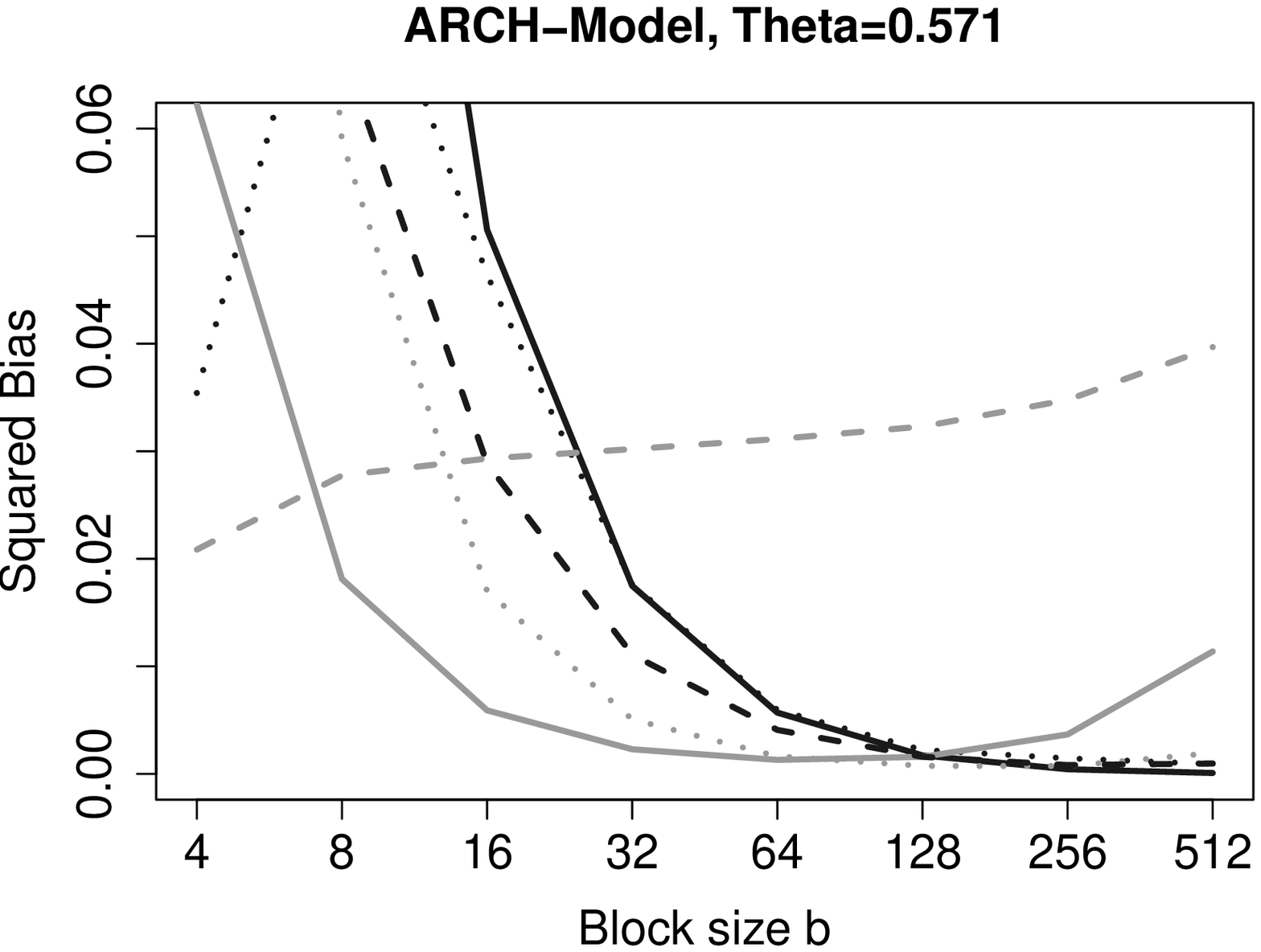}
\hspace{-.3cm}
\includegraphics[width=0.43\textwidth]{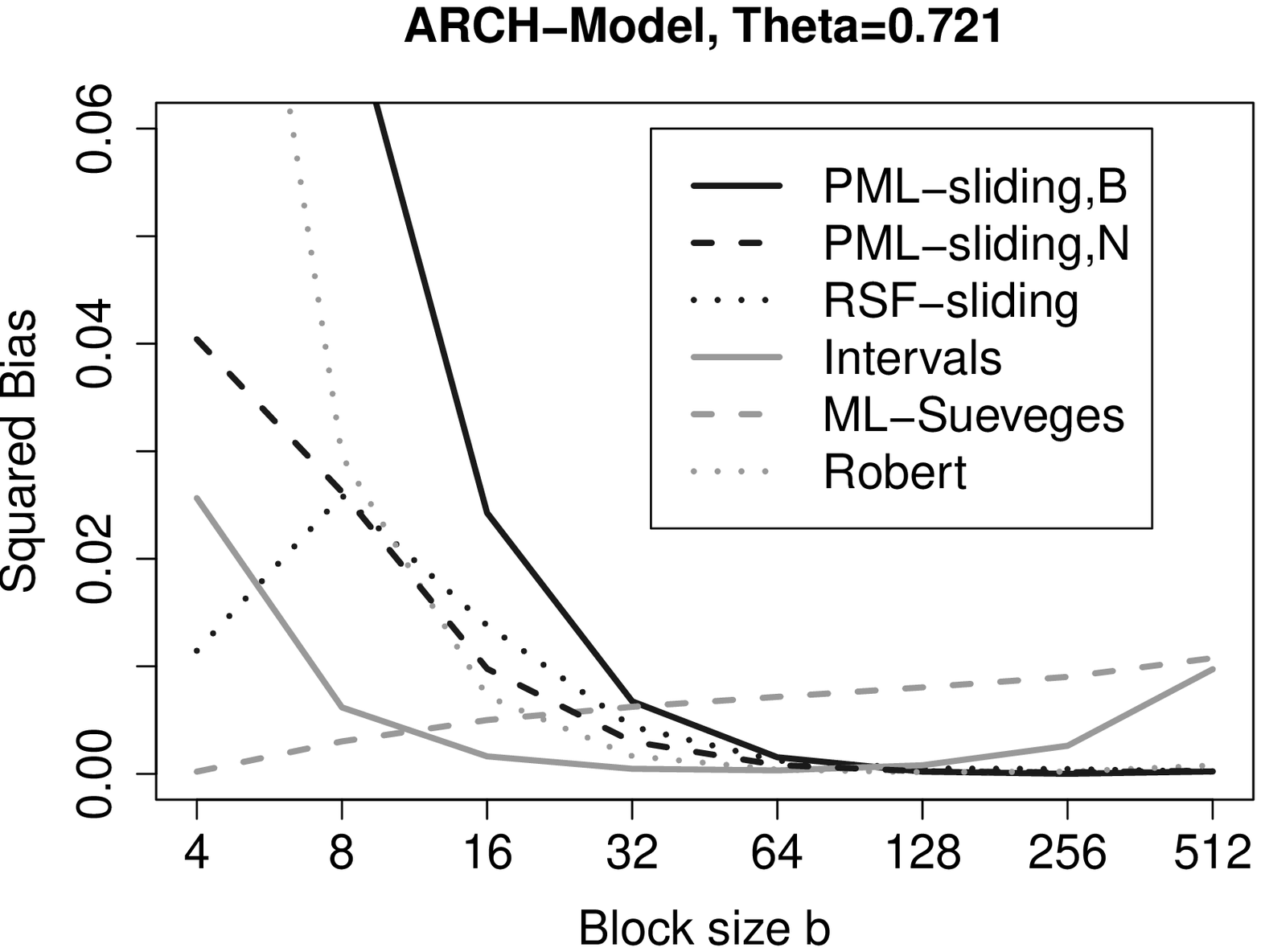}
\vspace{-.2cm}

\includegraphics[width=0.43\textwidth]{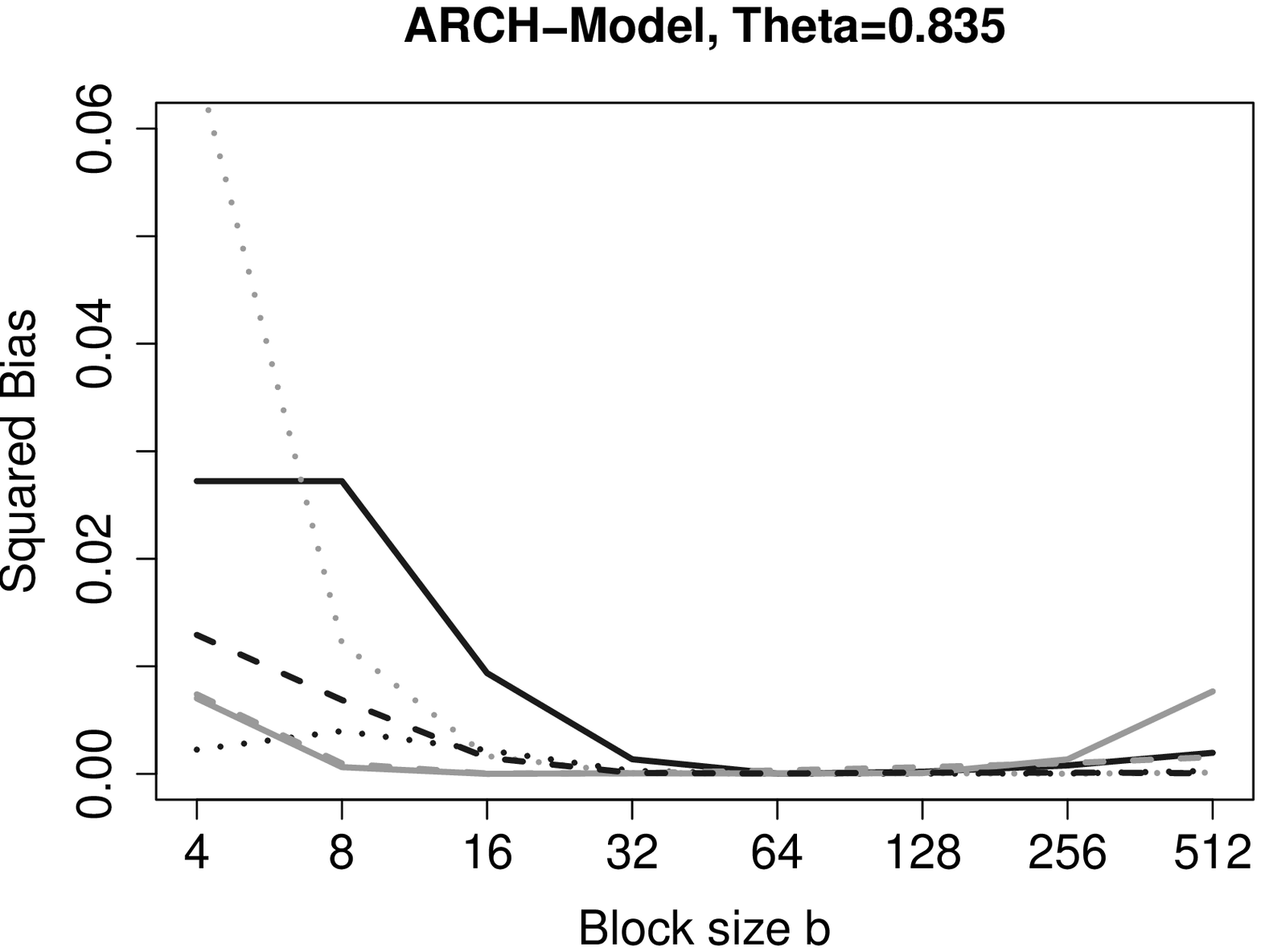}
\hspace{-.3cm}
\includegraphics[width=0.43\textwidth]{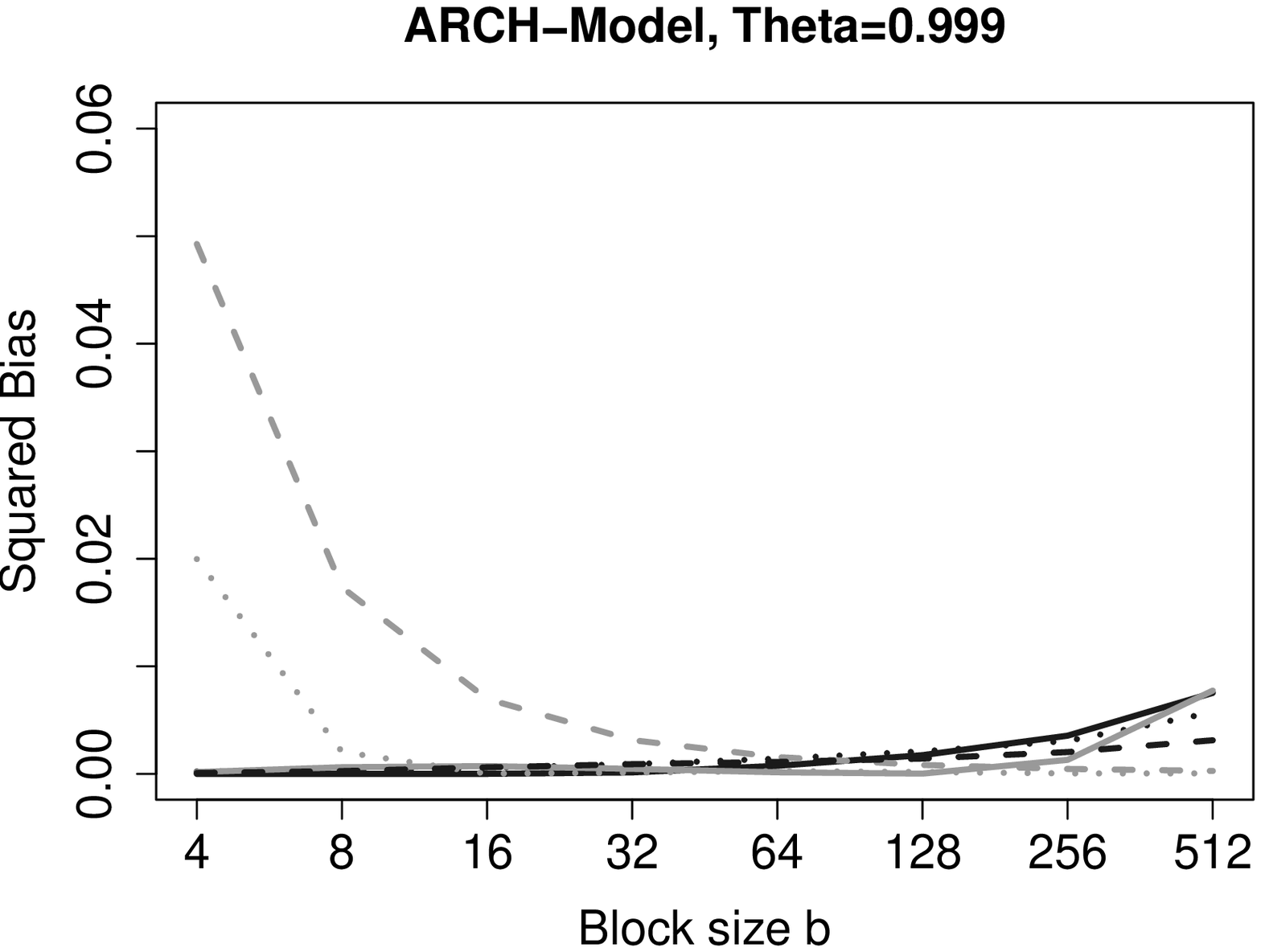}
\end{center}
\vspace{-.3cm}
\caption{\label{fig:bias1}  Bias of the estimation of $\theta$ within the ARCH-model for four values of $\theta\in\{0.571, 0.721, 0.835, 0.999\}$. 
}
\vspace*{\floatsep}
\begin{center}
\includegraphics[width=0.43\textwidth]{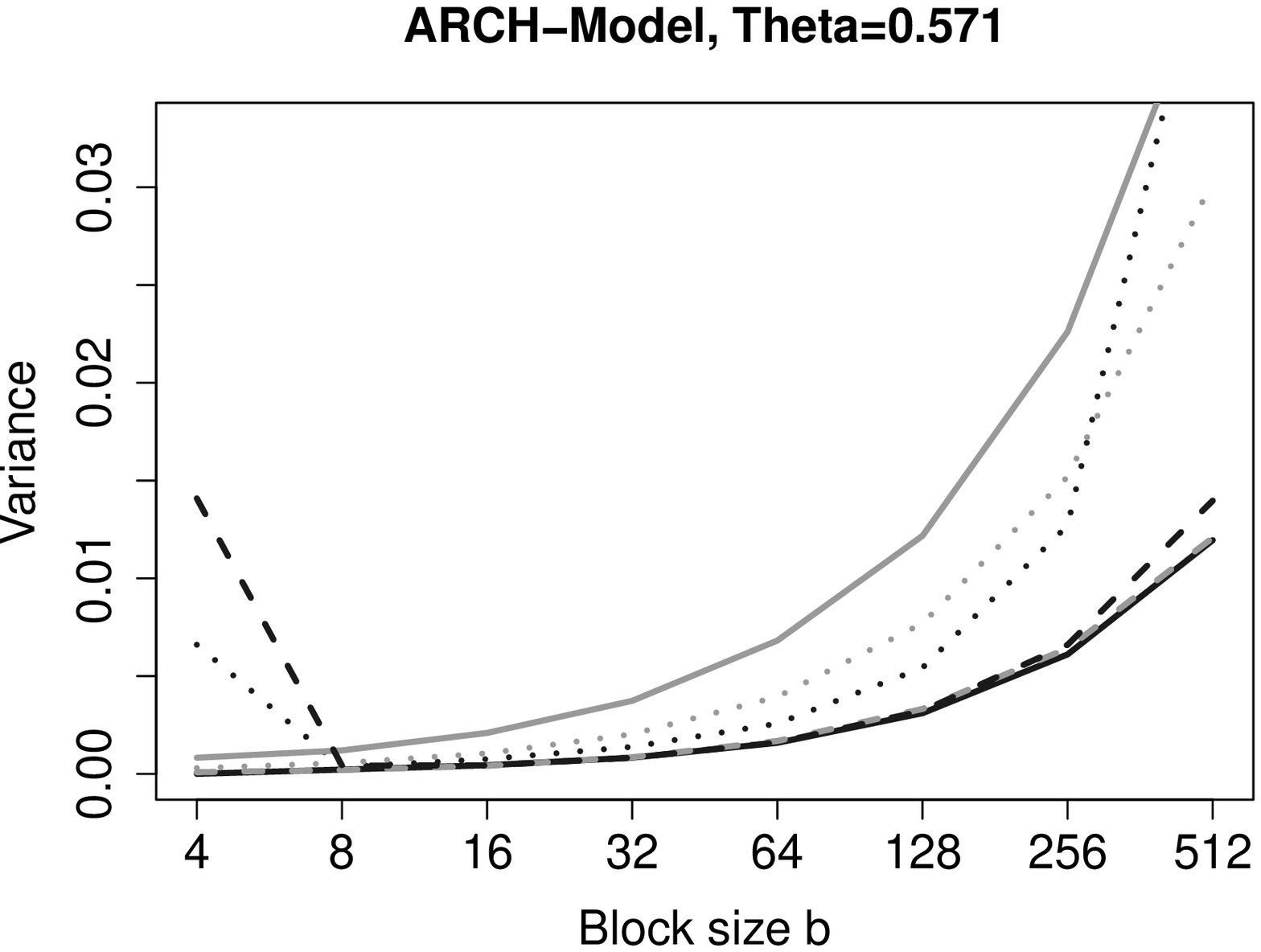}
\hspace{-.3cm}
\includegraphics[width=0.43\textwidth]{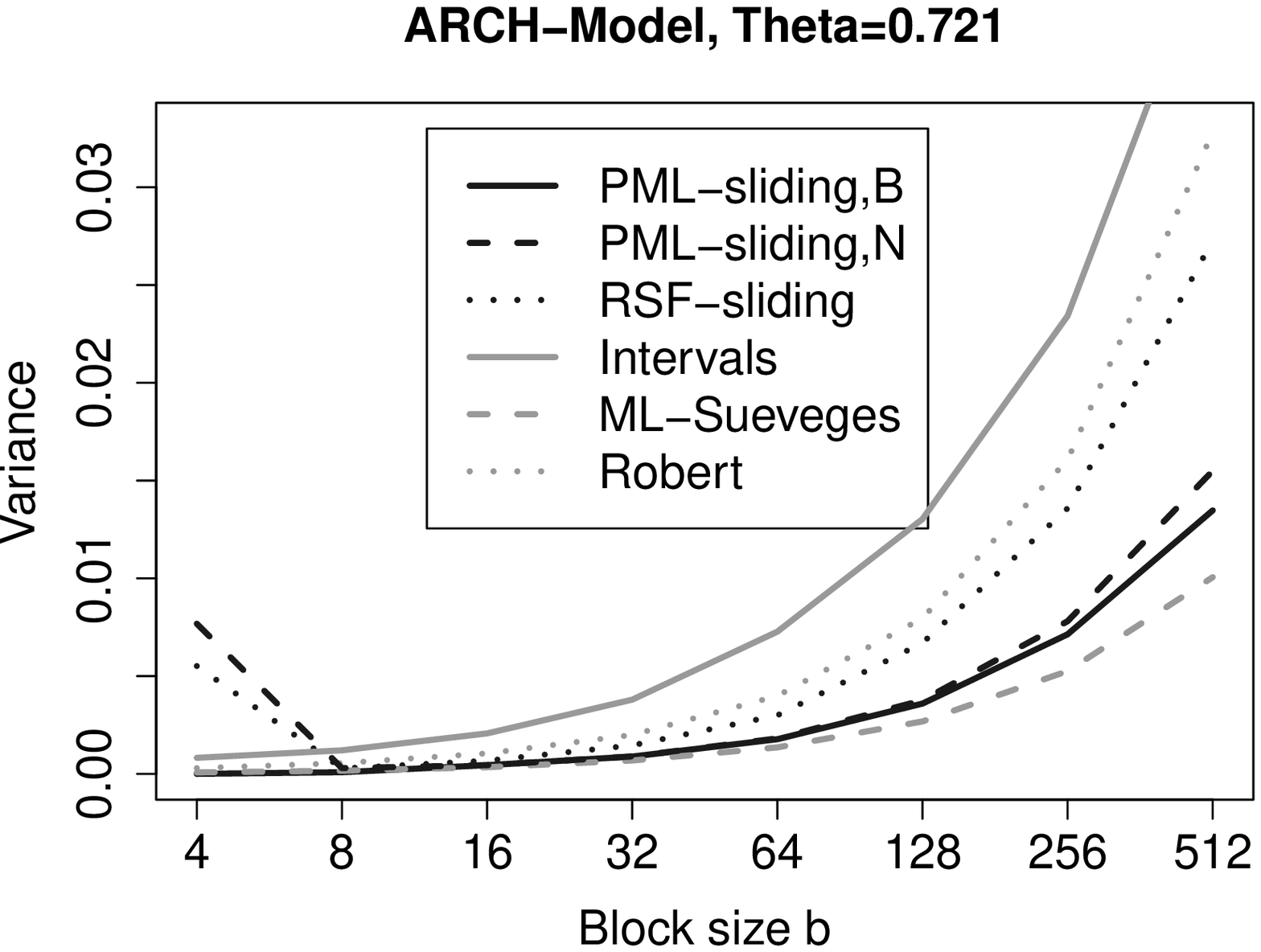}
\vspace{-.2cm}

\includegraphics[width=0.43\textwidth]{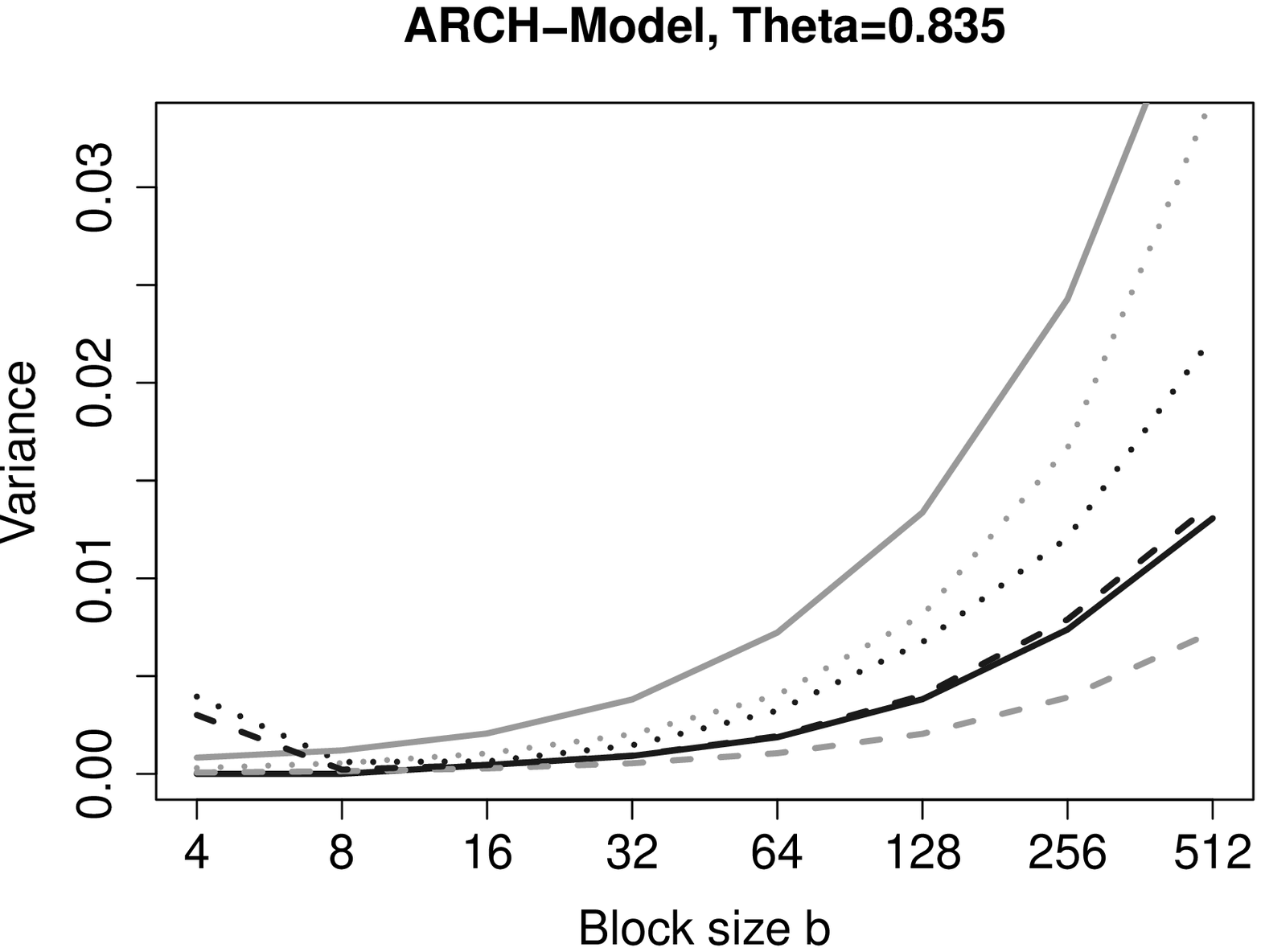}
\hspace{-.3cm}
\includegraphics[width=0.43\textwidth]{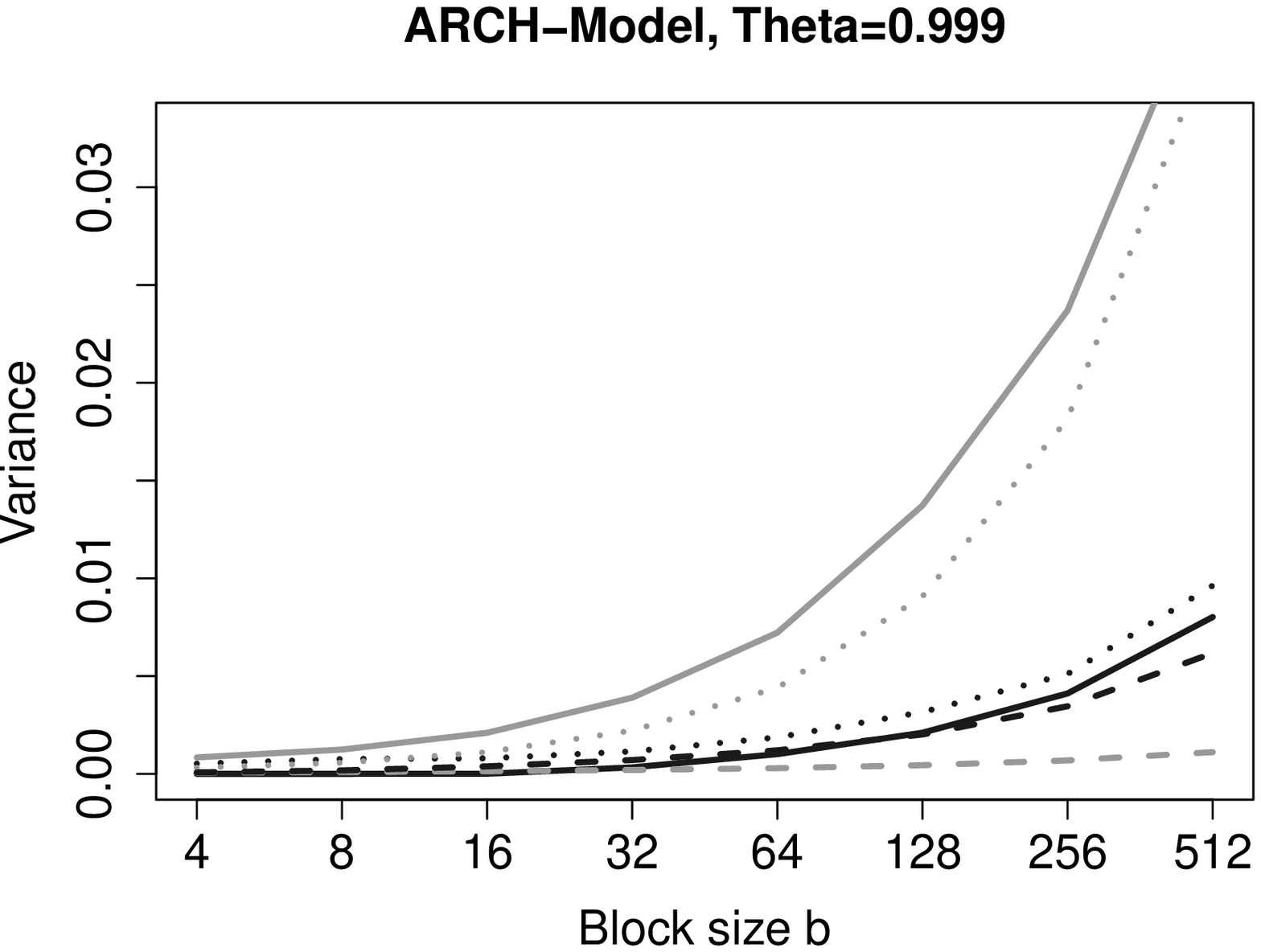}
\end{center}
\vspace{-.3cm}
\caption{\label{fig:var1}  Variance of the estimation of $\theta$ within the ARCH-model for four values of $\theta\in\{0.571, 0.721, 0.835, 0.999\}$. 
}

\vspace{-.1cm}
\end{figure}

\begin{figure}[p!]
\begin{center}
\includegraphics[width=0.43\textwidth]{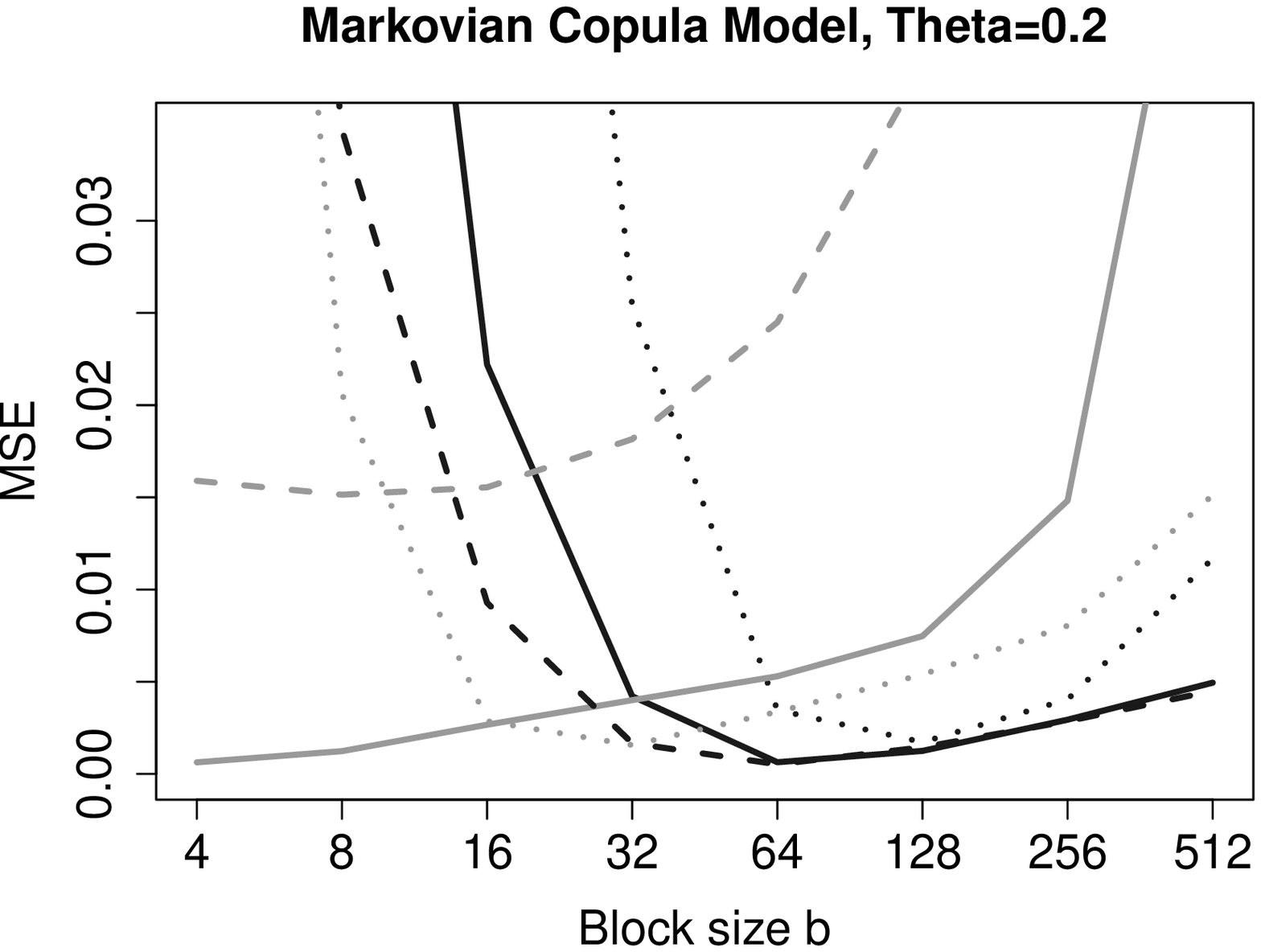}
\hspace{-.3cm}
\includegraphics[width=0.43\textwidth]{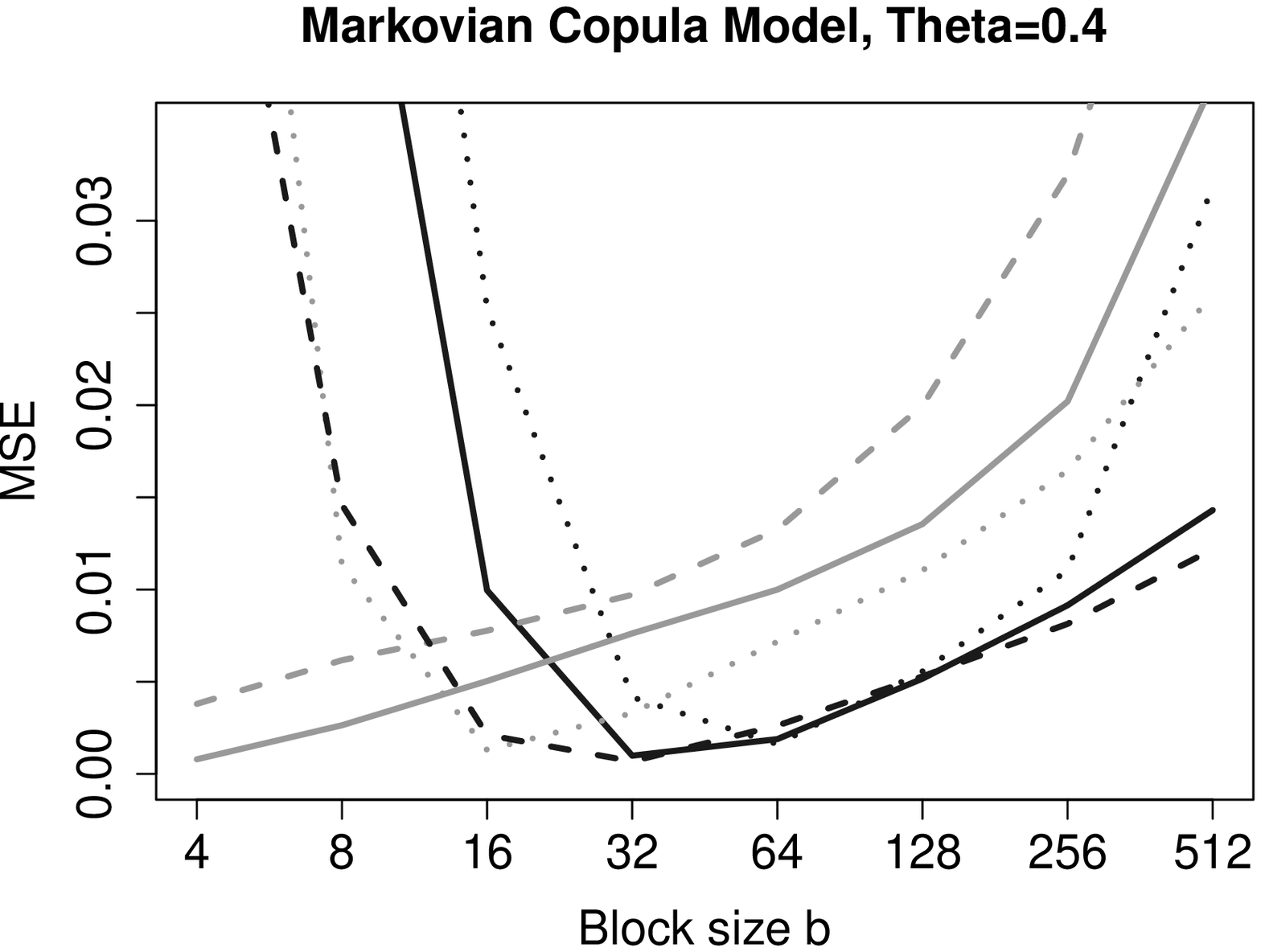}
\vspace{-.2cm}

\includegraphics[width=0.43\textwidth]{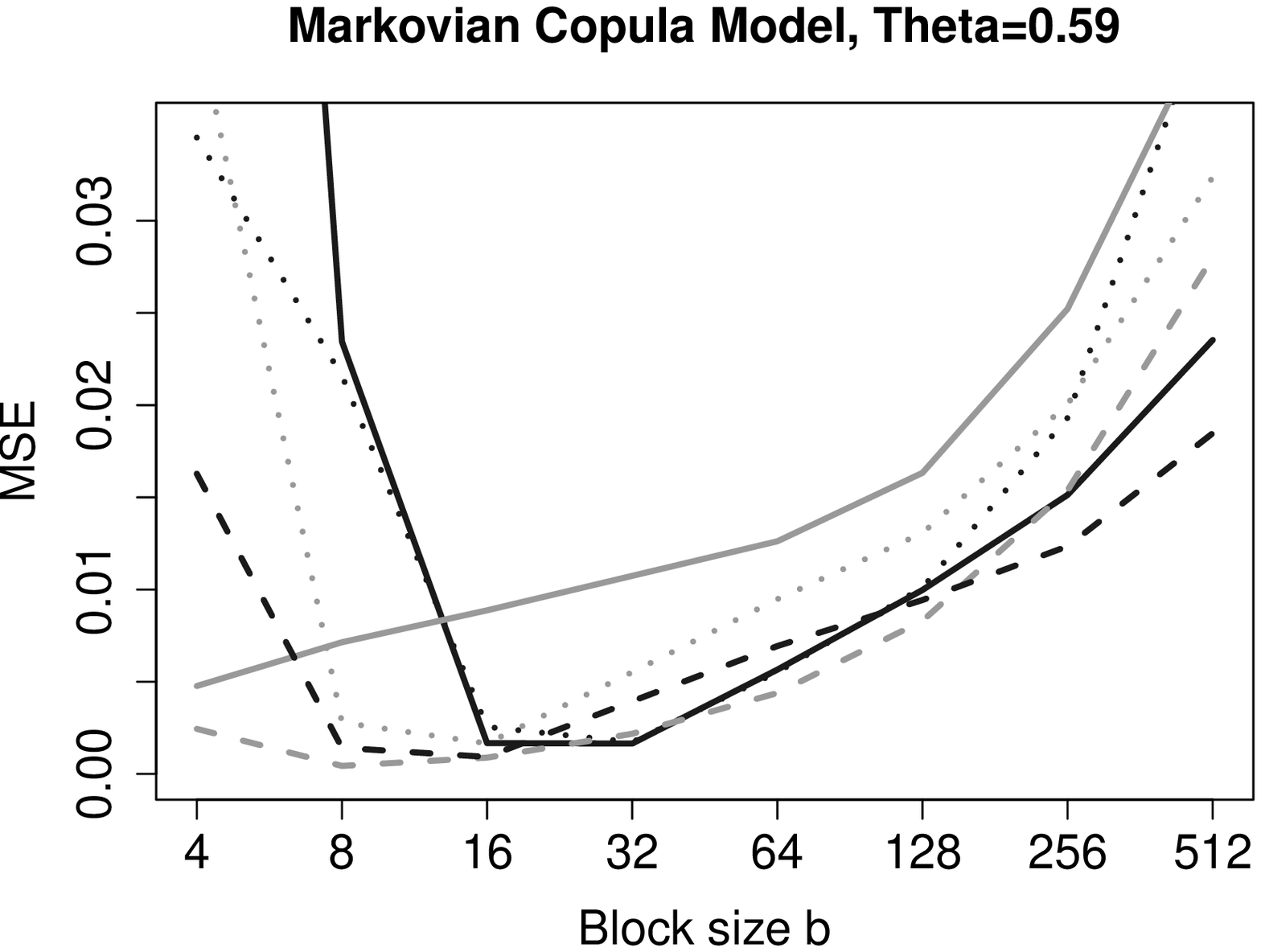}
\hspace{-.3cm}
\includegraphics[width=0.43\textwidth]{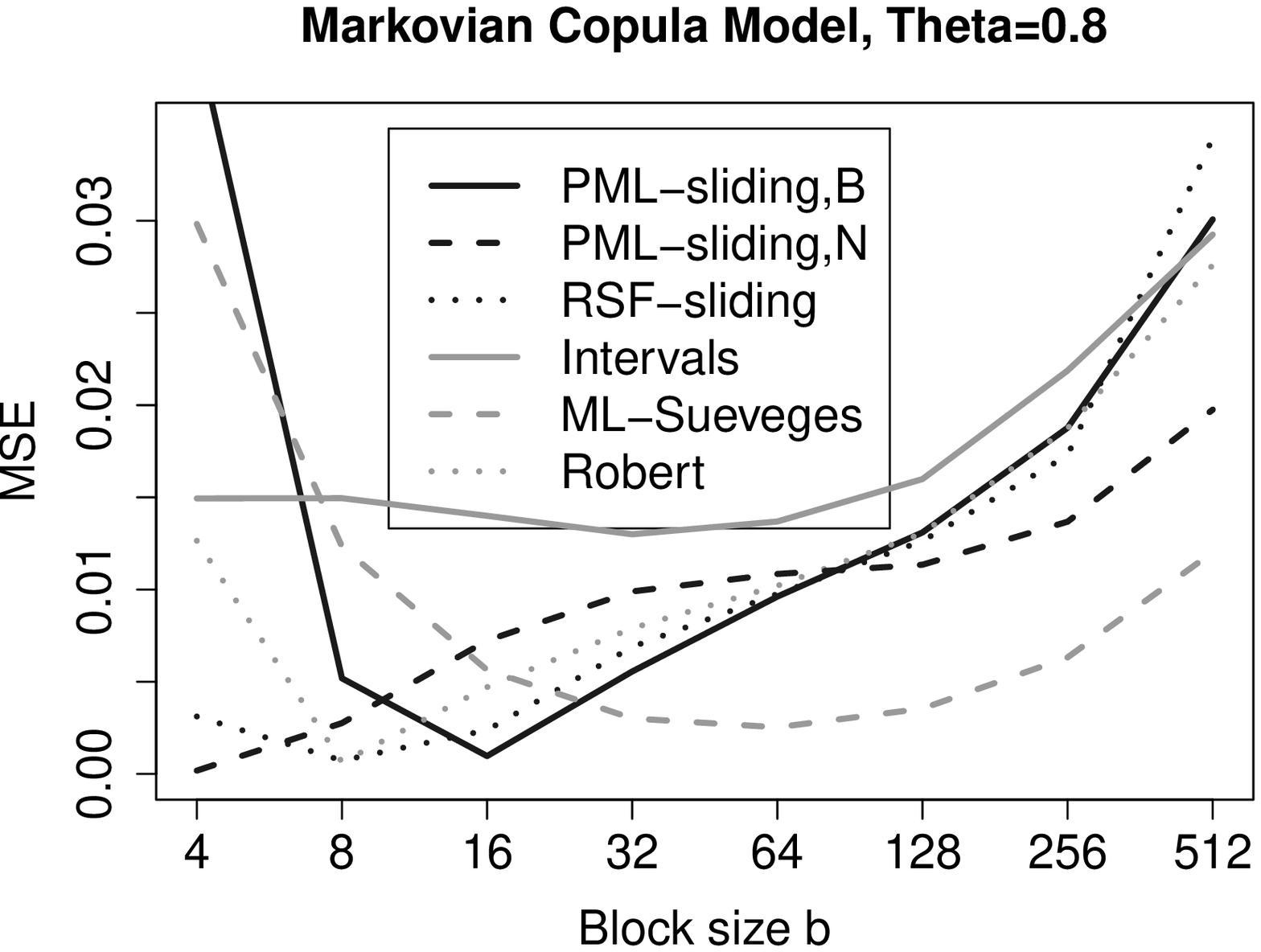}
\end{center}
\vspace{-.3cm}
\caption{\label{fig:mse4}  Mean squared error for the estimation of $\theta$ within the Markovian copula model for four values of $\theta\in\{0.2, 0.4, 0.6, 0.8\}$. 
}
\vspace*{\floatsep}
\begin{center}
\includegraphics[width=0.43\textwidth]{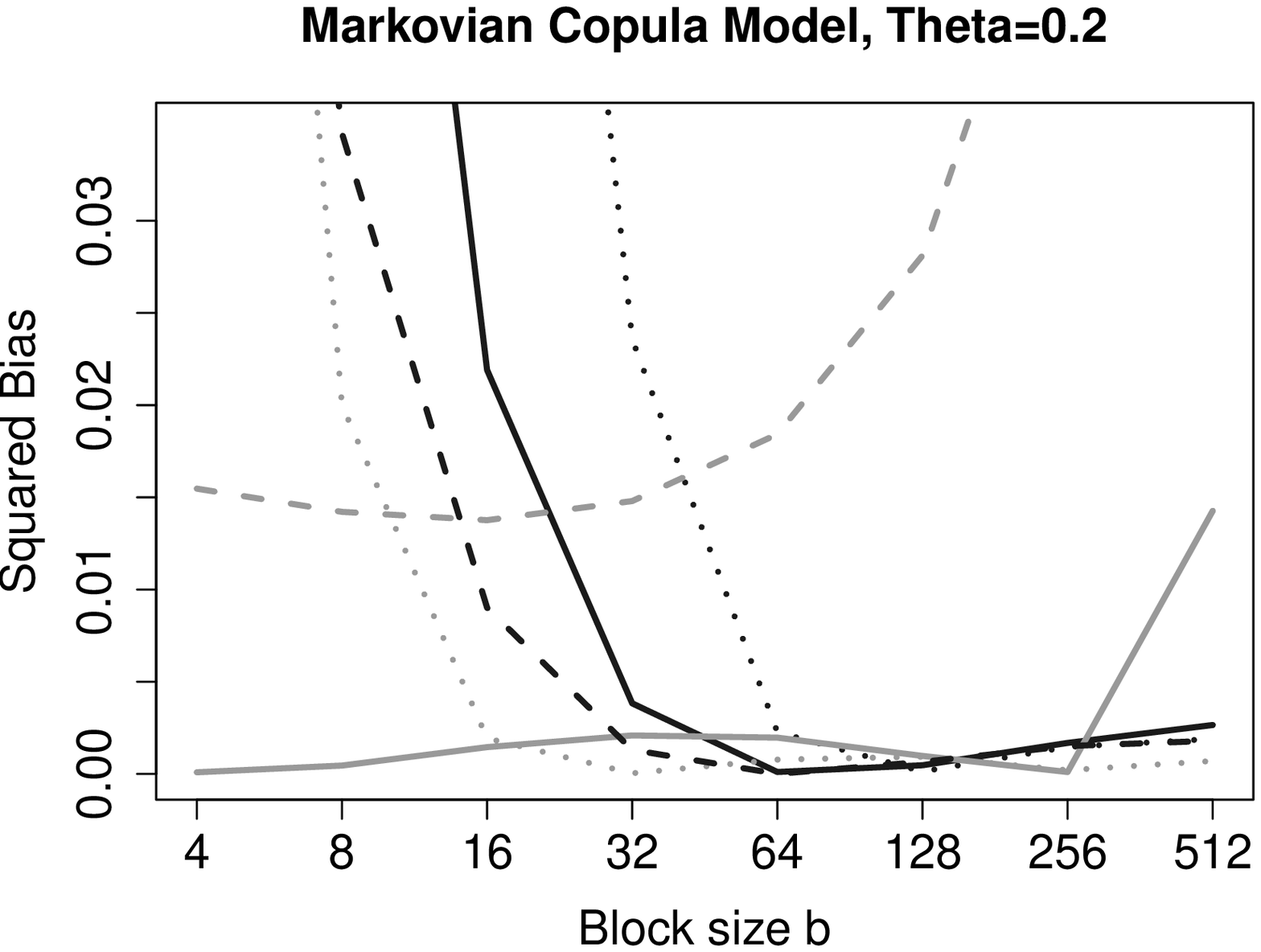}
\hspace{-.3cm}
\includegraphics[width=0.43\textwidth]{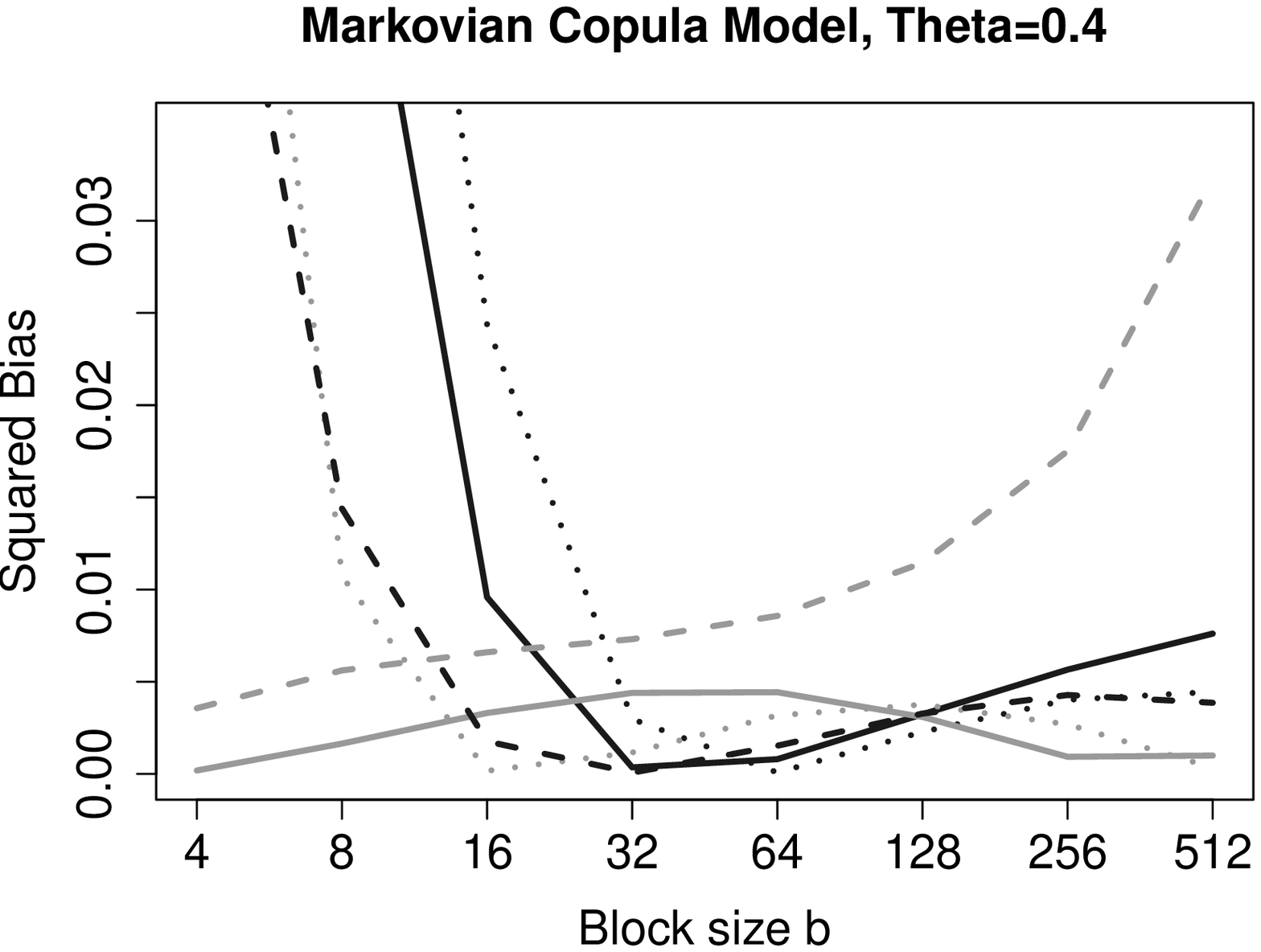}
\vspace{-.2cm}

\includegraphics[width=0.43\textwidth]{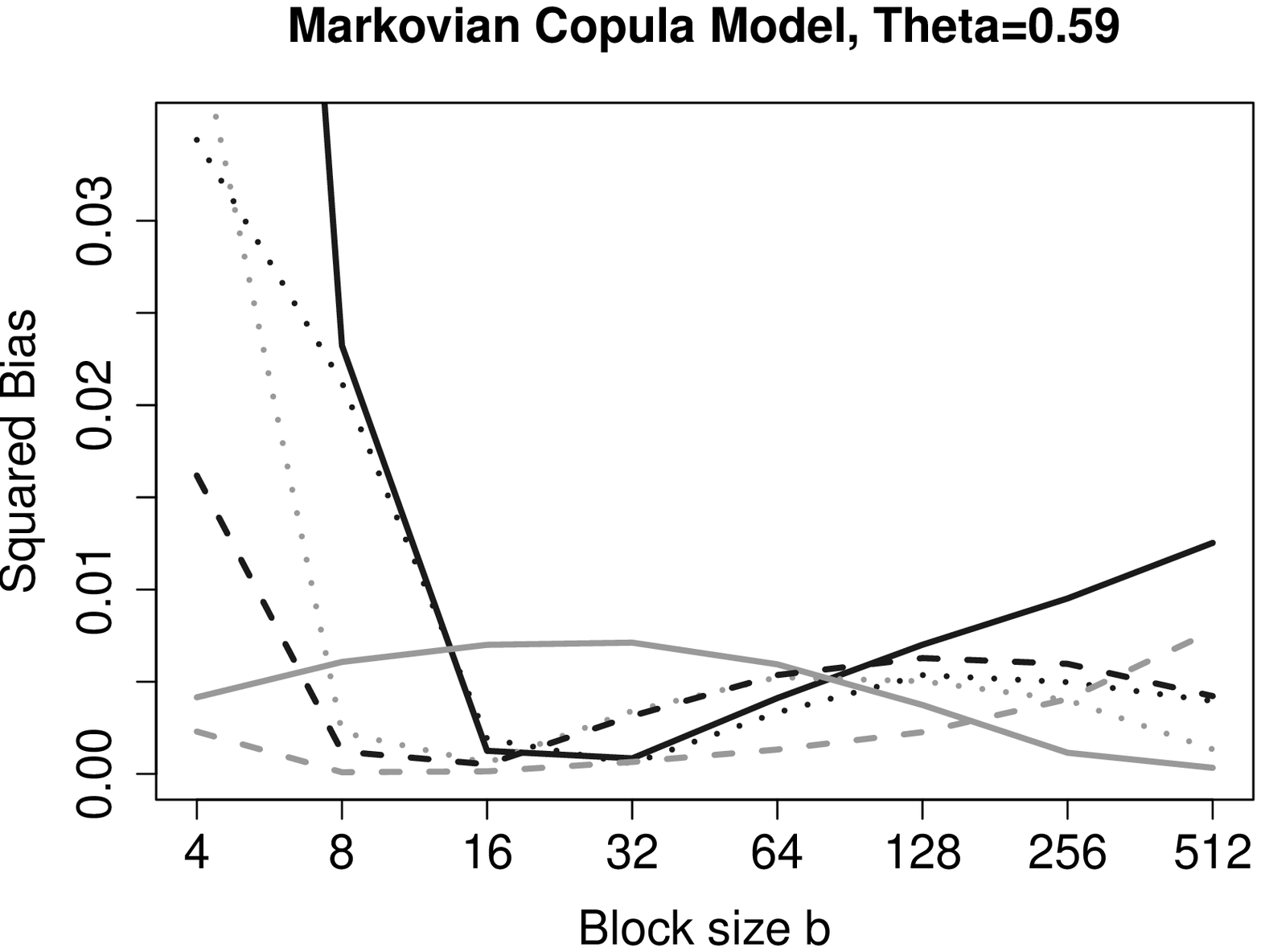}
\hspace{-.3cm}
\includegraphics[width=0.43\textwidth]{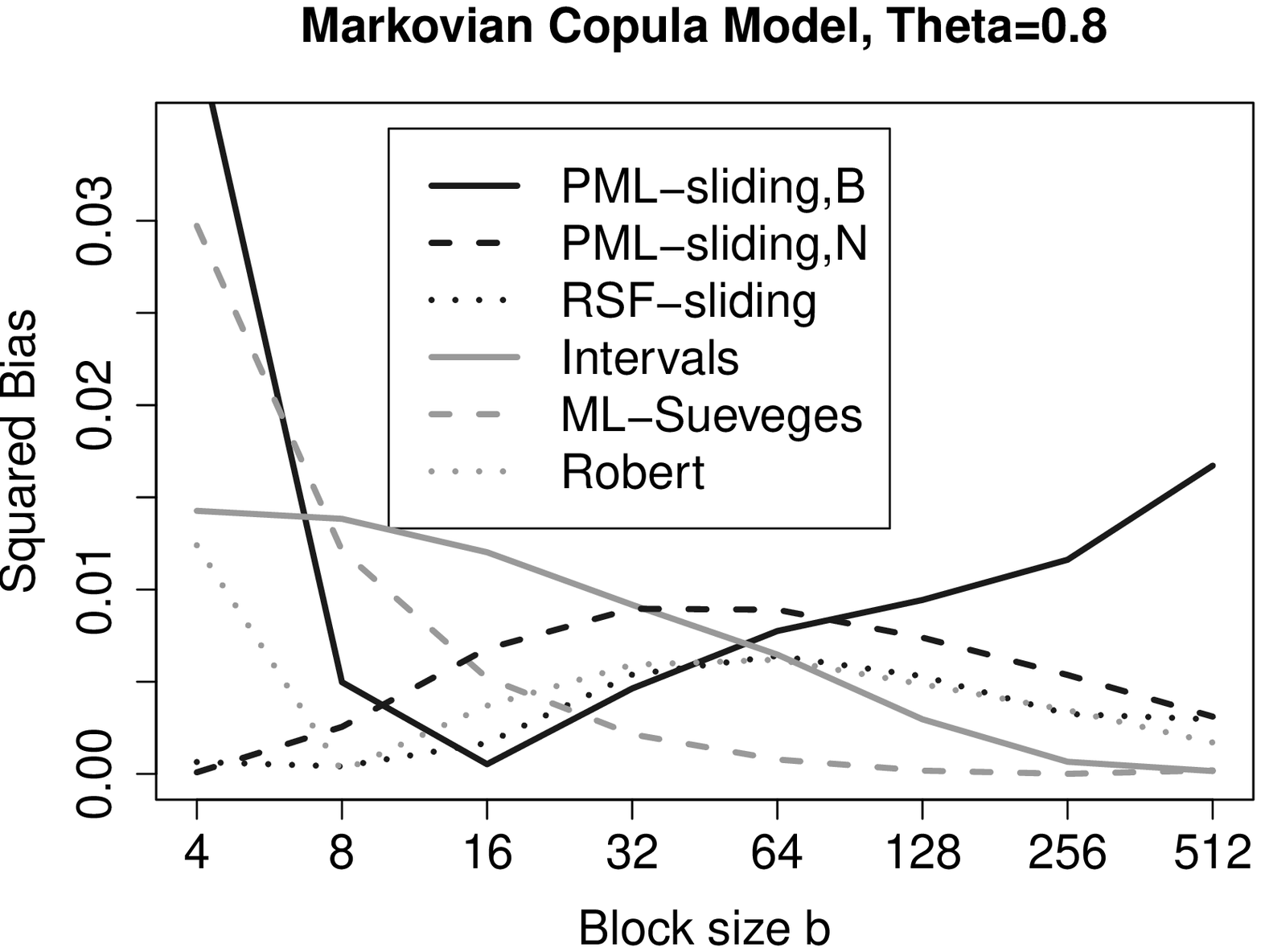}
\end{center}
\vspace{-.3cm}
\caption{\label{fig:bias4}  Bias of the estimation of $\theta$ within theMarkovian copula model for four values of $\theta\in\{0.2, 0.4, 0.6, 0.8\}$. 
}

\end{figure}

\begin{figure}[p!]

\begin{center}
\includegraphics[width=0.43\textwidth]{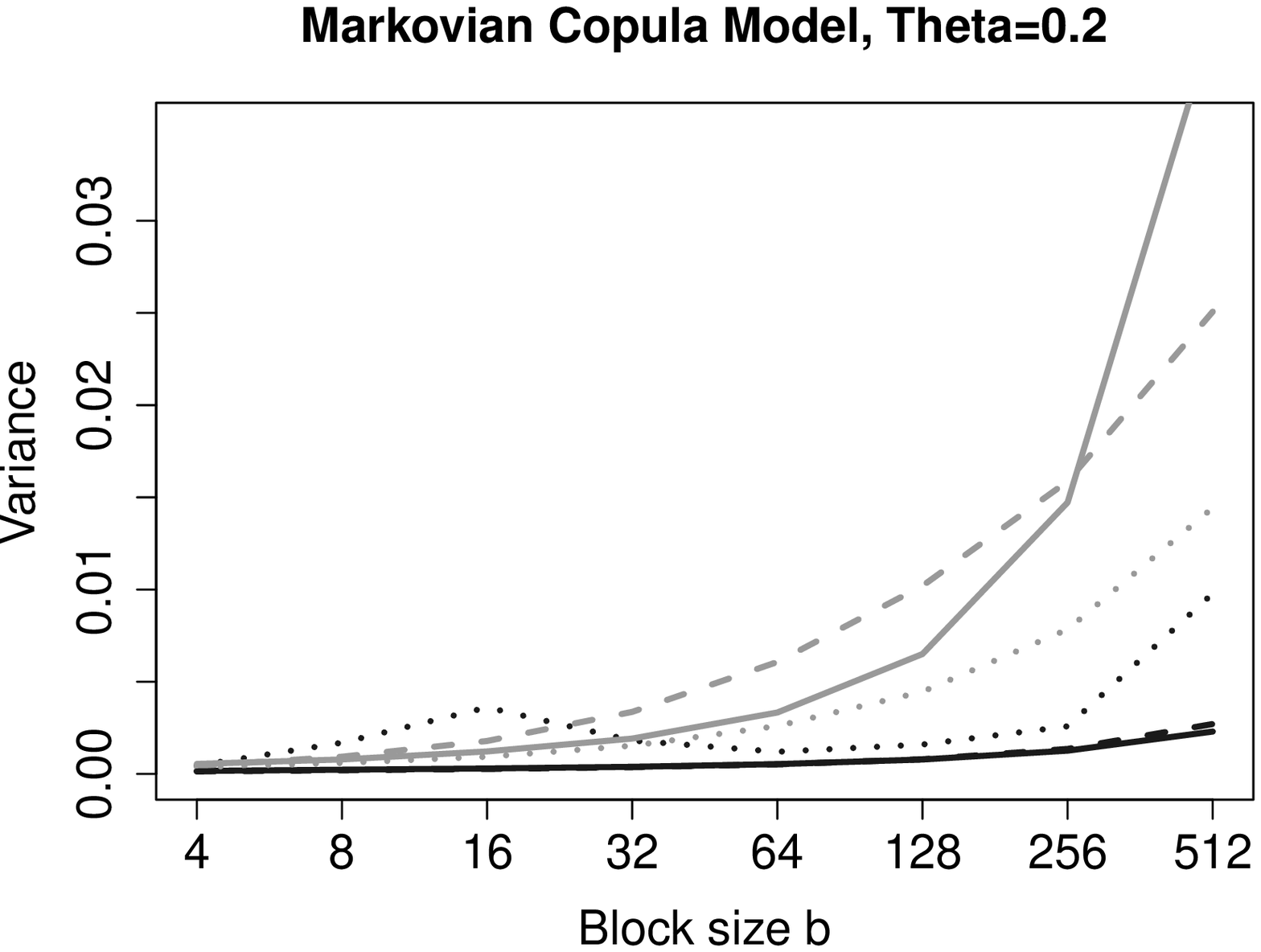}
\hspace{-.3cm}
\includegraphics[width=0.43\textwidth]{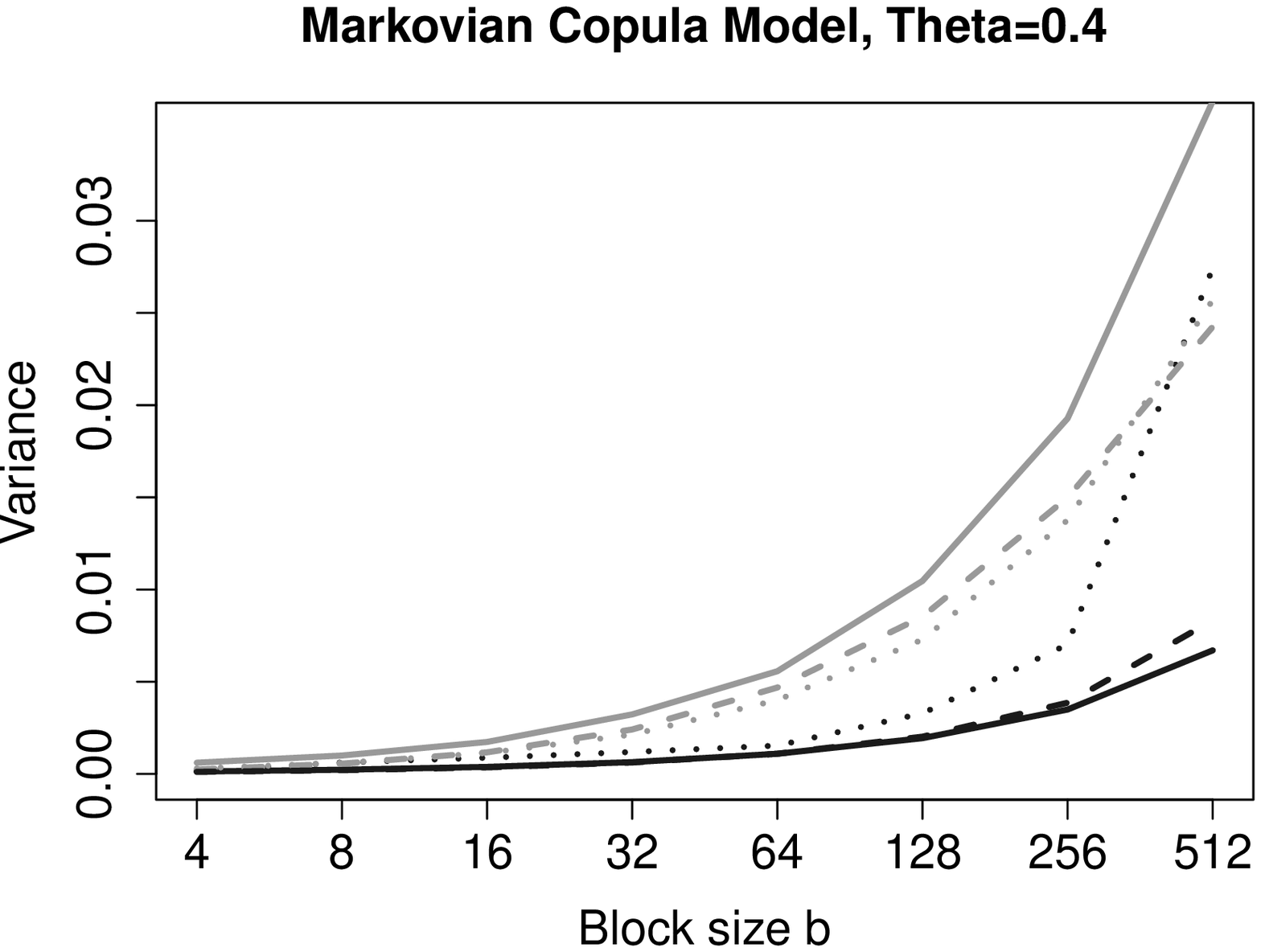}
\vspace{-.2cm}

\includegraphics[width=0.43\textwidth]{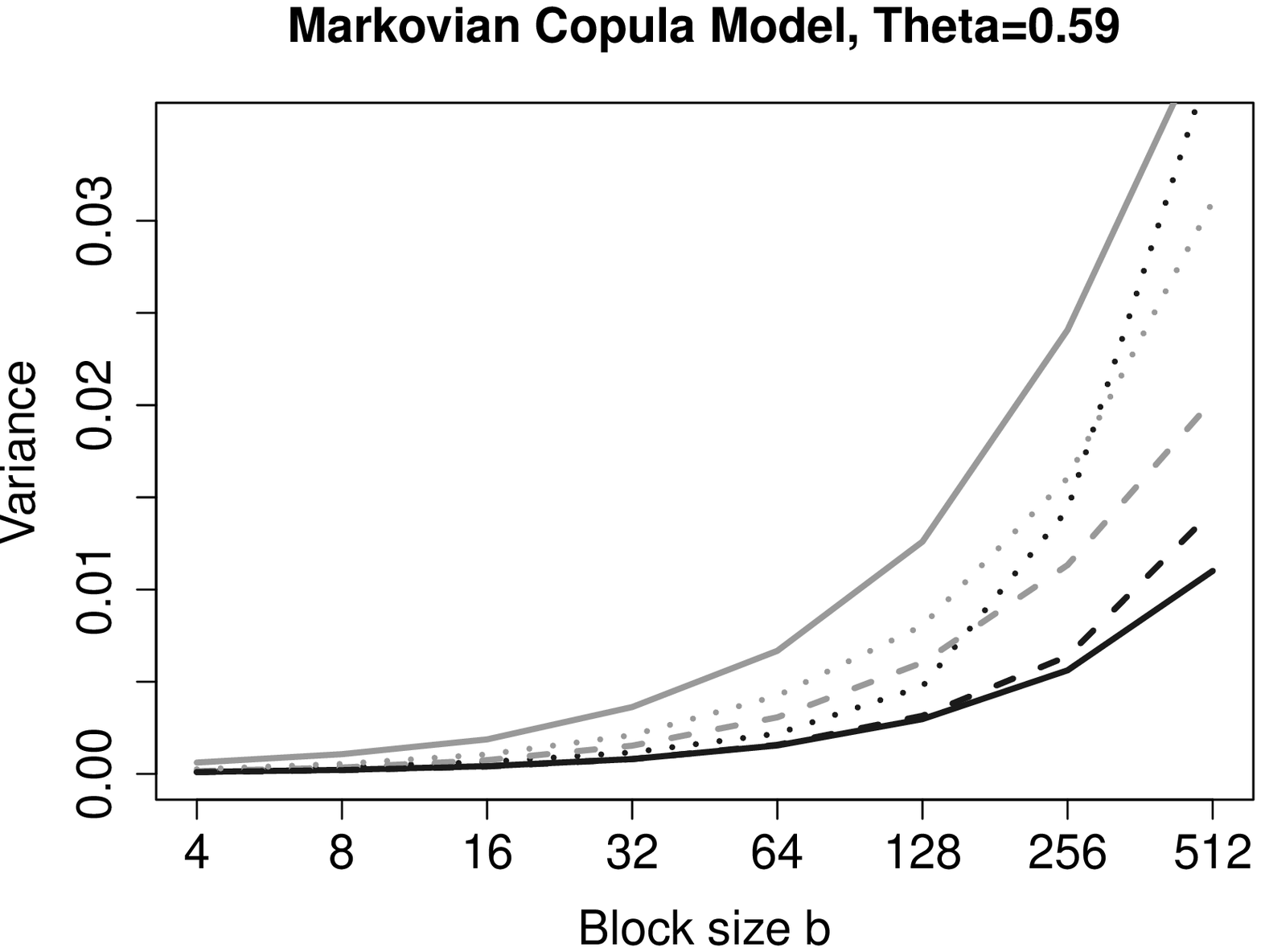}
\hspace{-.3cm}
\includegraphics[width=0.43\textwidth]{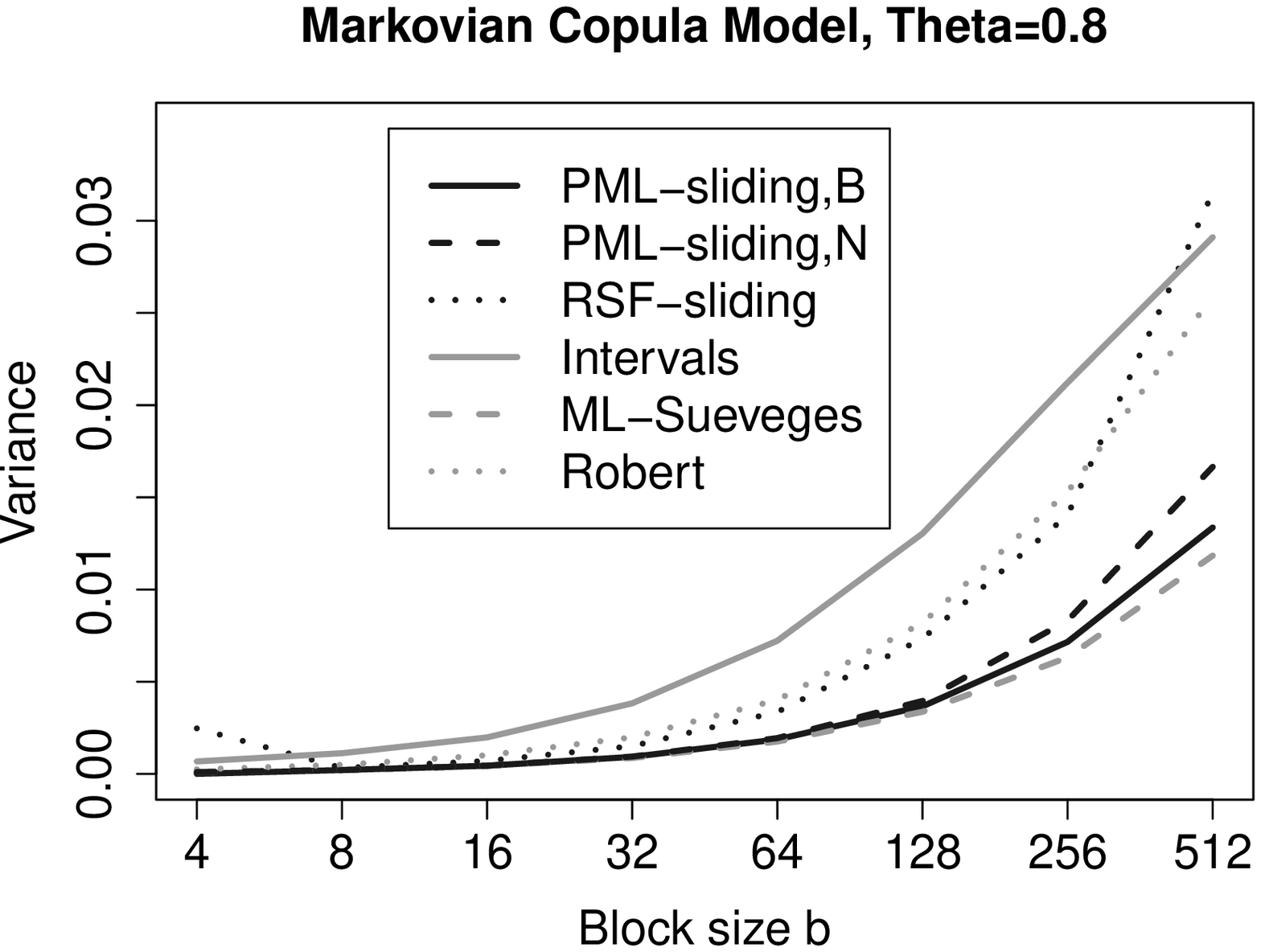}
\end{center}
\vspace{-.3cm}
\caption{\label{fig:var4}  Variance of the estimation of $\theta$ within the Markovian copula model for four values of $\theta\in\{0.2, 0.4, 0.6, 0.8\}$. 
}
\vspace*{\floatsep}
\begin{center}
\includegraphics[width=0.46\textwidth]{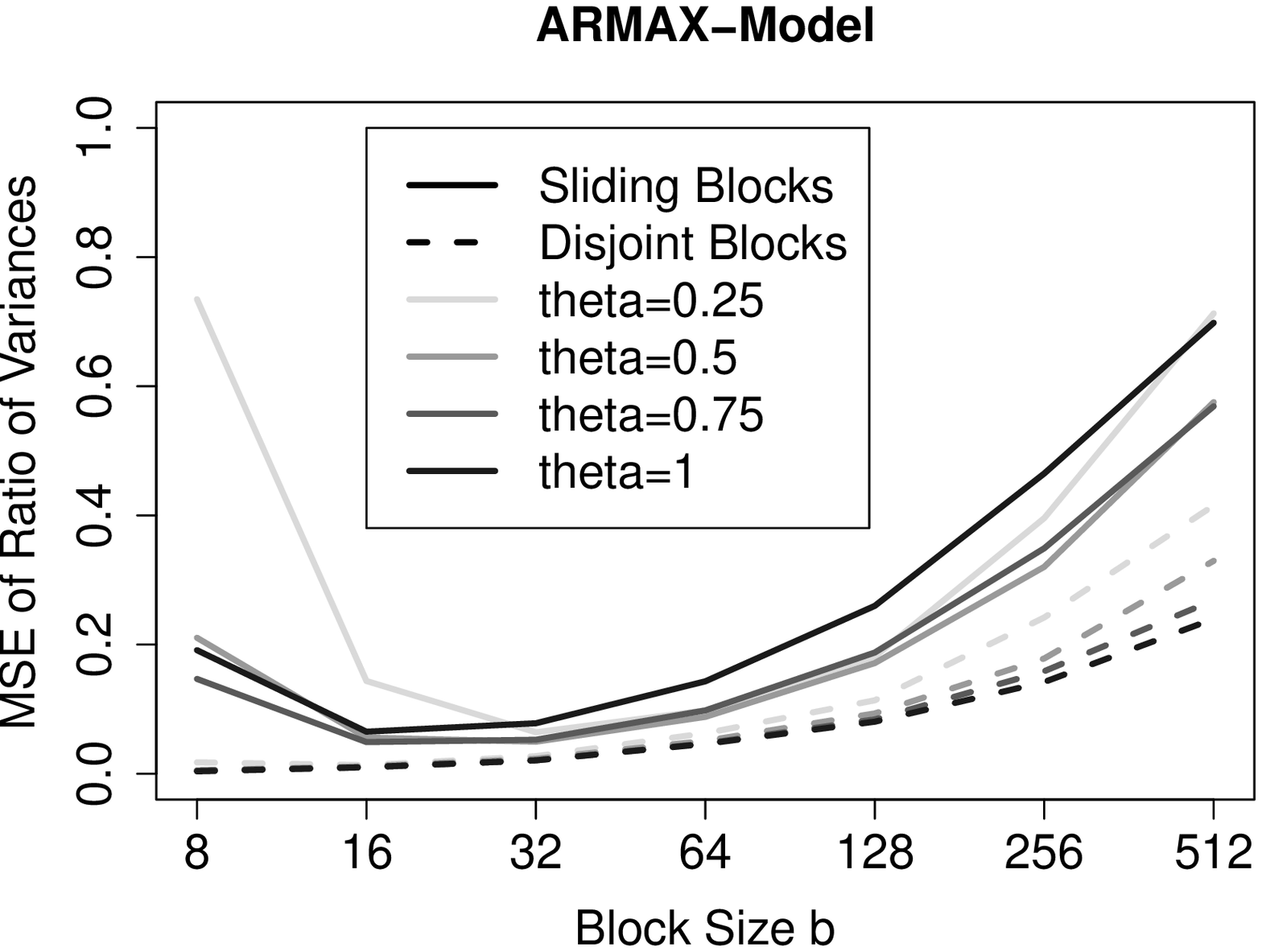}
\hspace{-.3cm}
\includegraphics[width=0.46\textwidth]{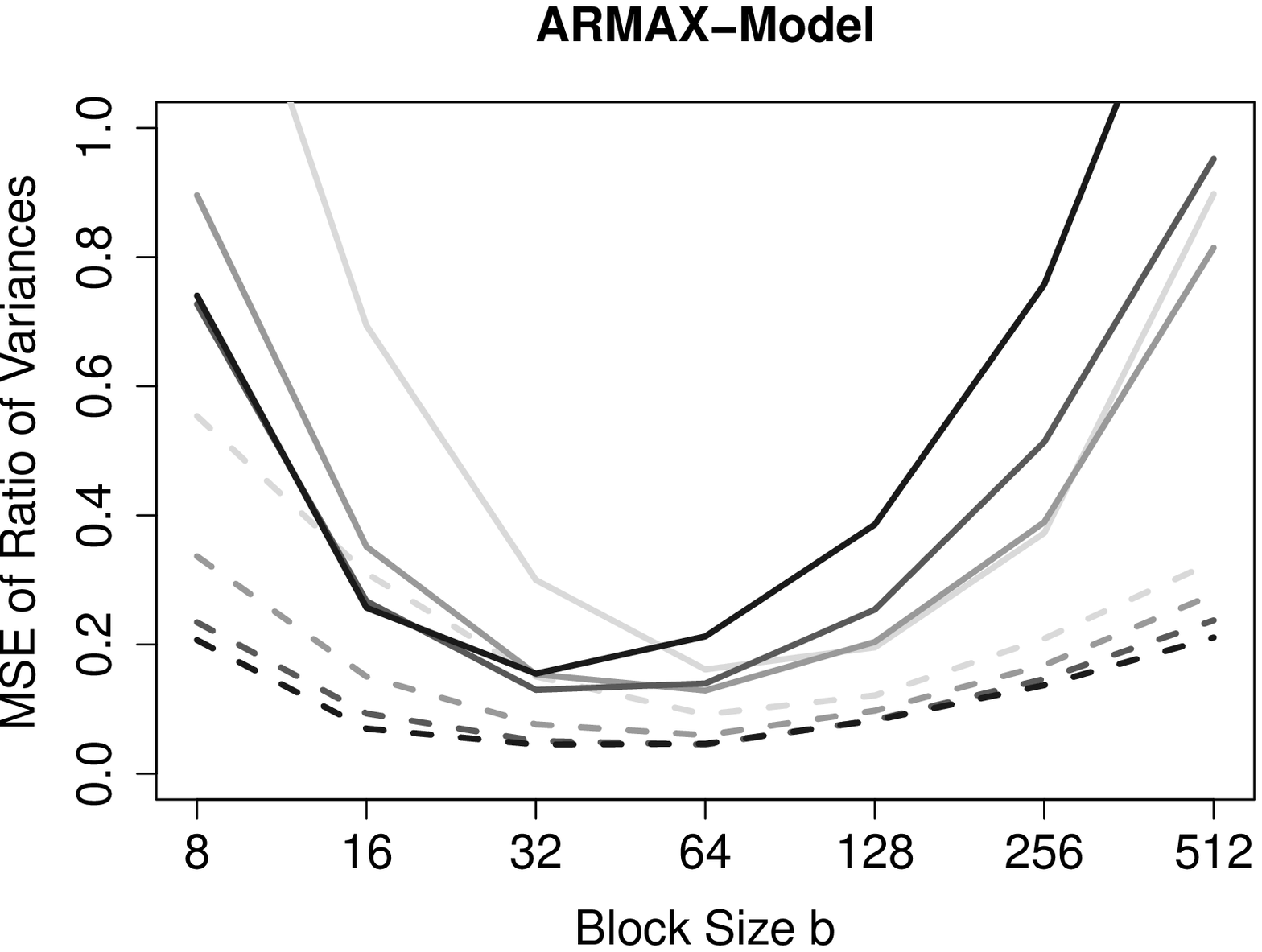}
\vspace{-.5cm}

\includegraphics[width=0.46\textwidth]{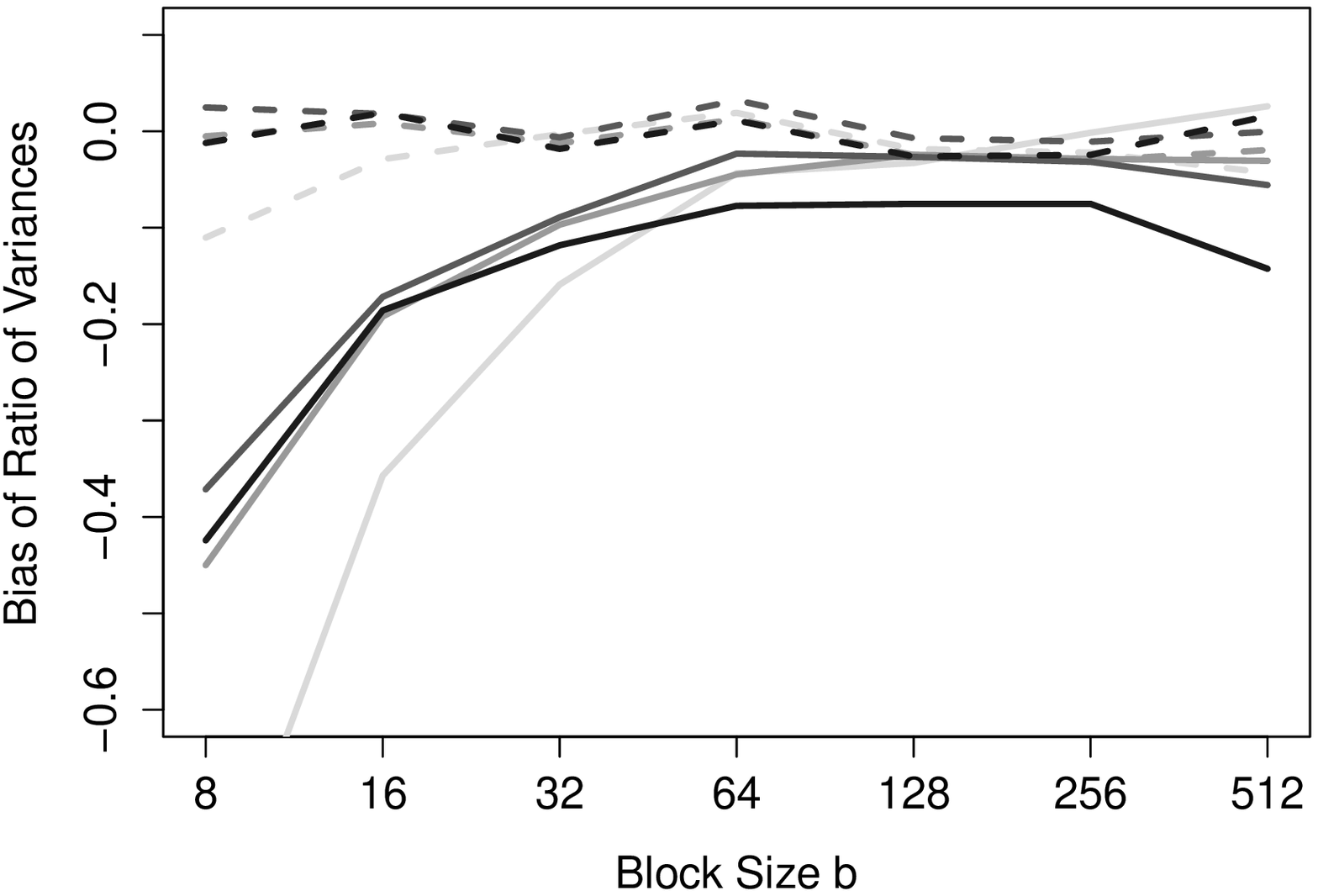}
\hspace{-.3cm}
\includegraphics[width=0.46\textwidth]{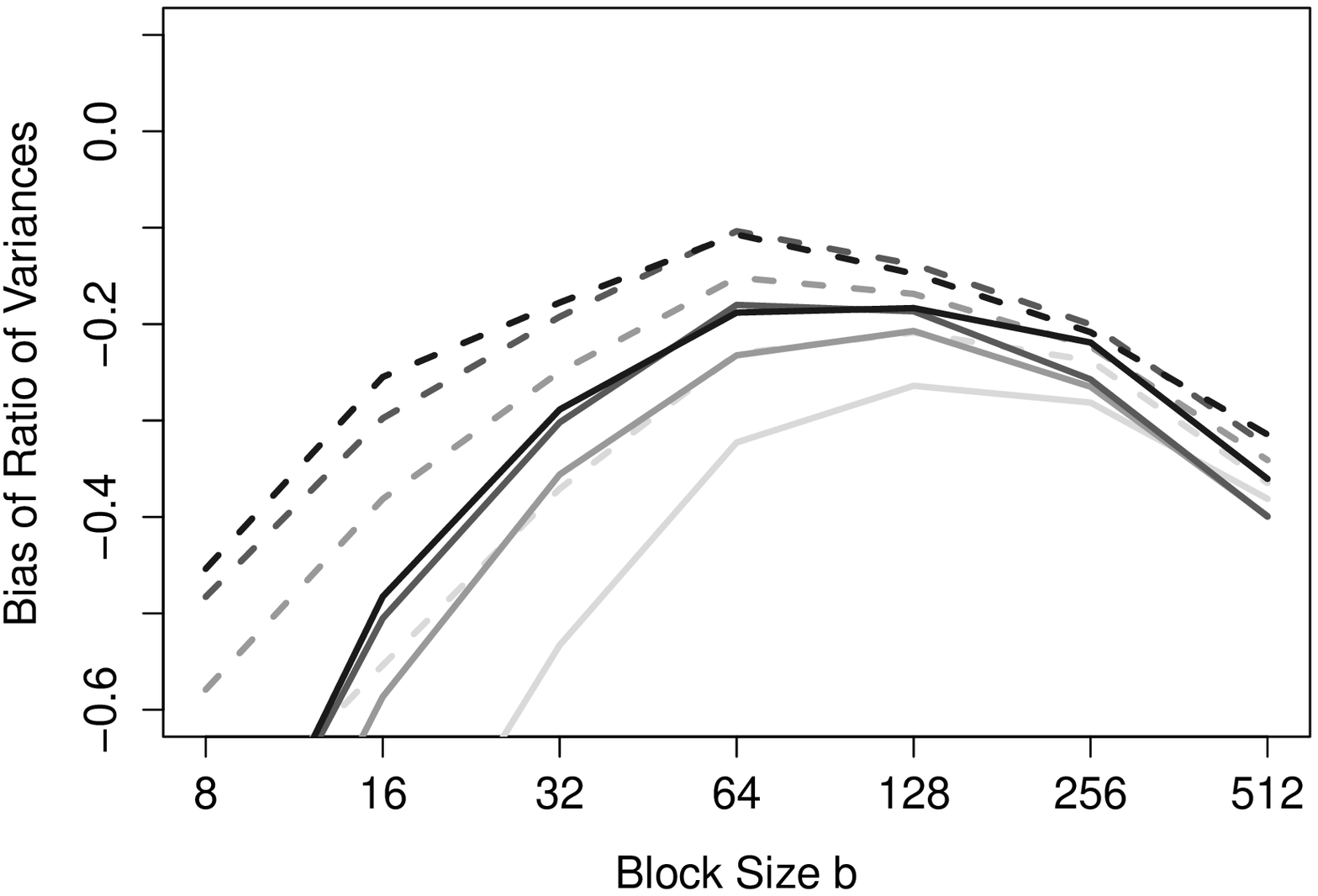}
\end{center}
\vspace{-.5cm} 
\caption{\label{fig:msevariance2}  Mean squared error $\Exp[(\hat \tau^2/ \Var (\hat \theta_n) - 1 )^2]$ and bias  $\Exp[\hat \tau^2/ \Var (\hat \theta_n) ]- 1 $ within the ARMAX-model for the unconstrained estimators  $\hat \theta_n^{\Be}$ (left) and $\hat \theta_n^{\No}$ (right).
}
\end{figure}

\begin{figure}[t!]

\begin{center}
\includegraphics[width=0.43\textwidth]{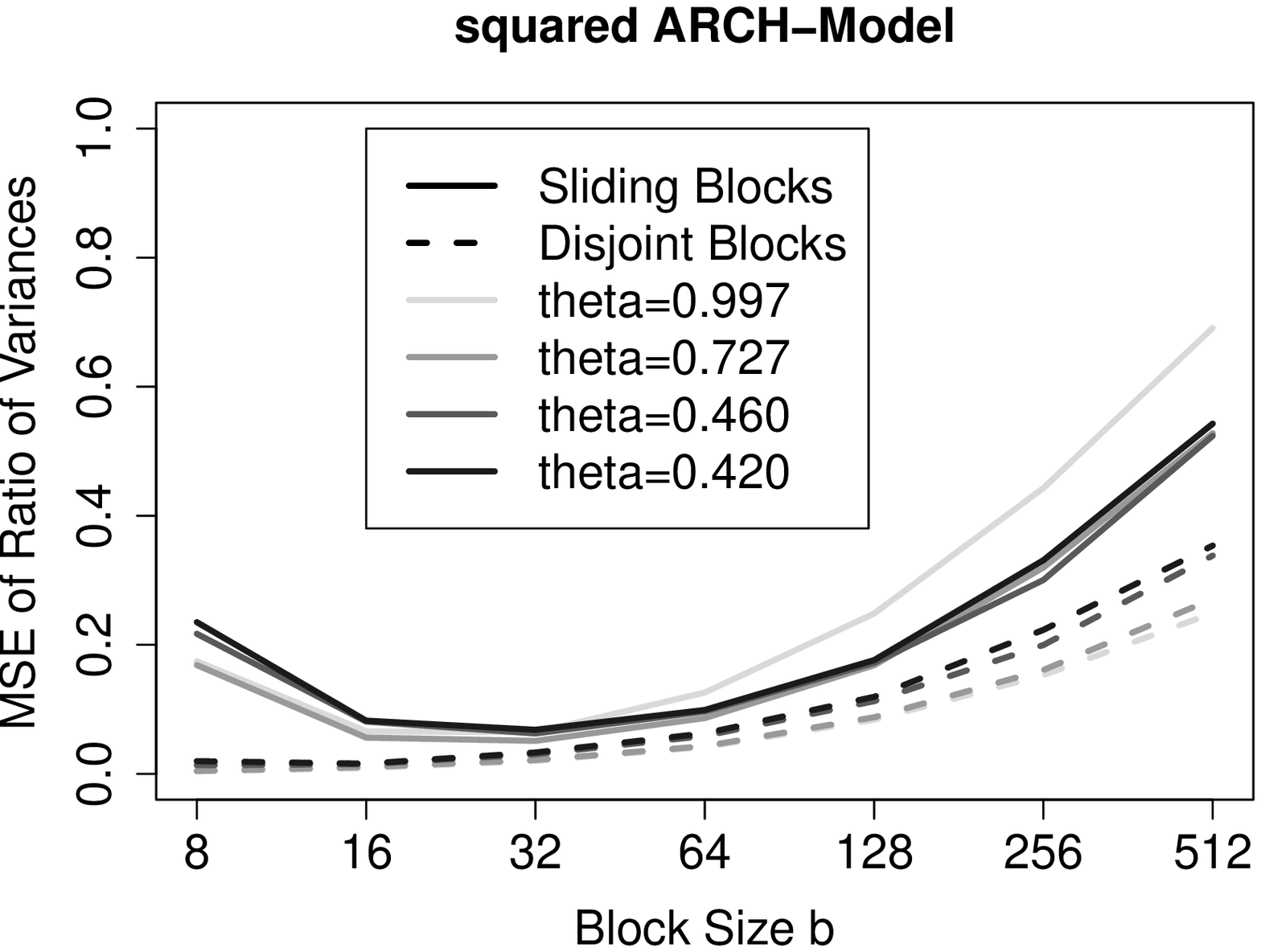}
\hspace{-.3cm}
\includegraphics[width=0.43\textwidth]{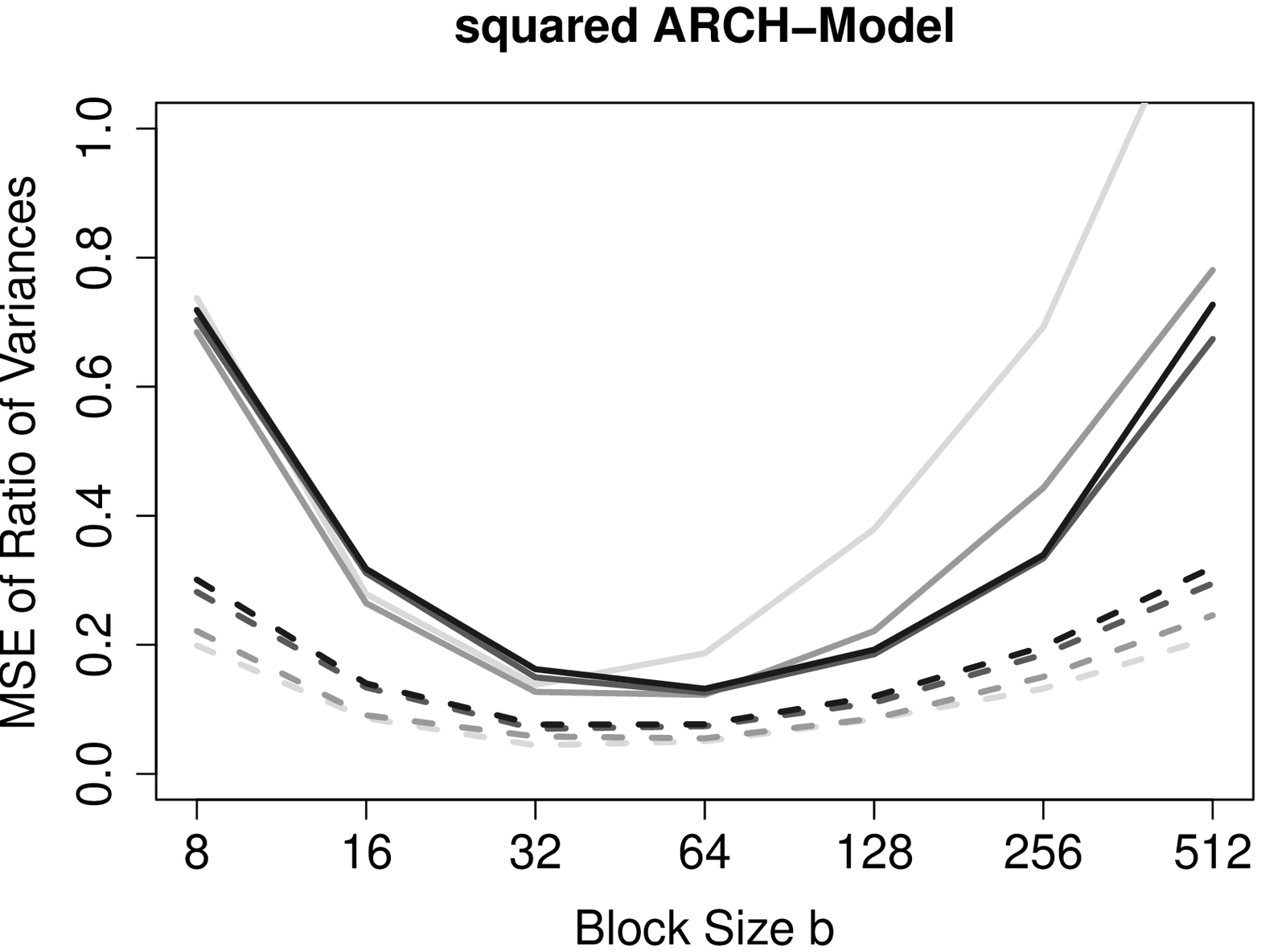}
\vspace{-.2cm}

\includegraphics[width=0.43\textwidth]{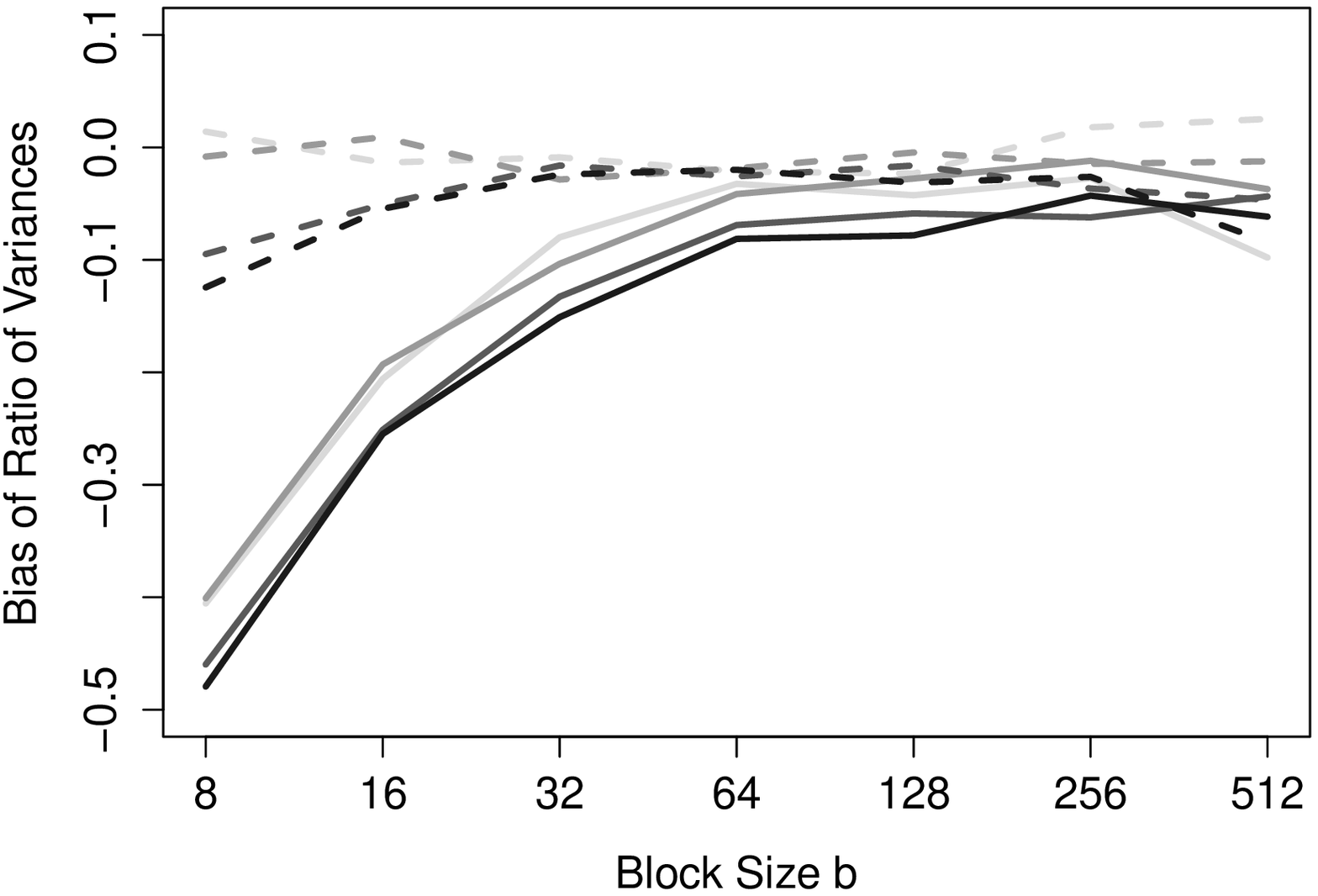}
\hspace{-.3cm}
\includegraphics[width=0.43\textwidth]{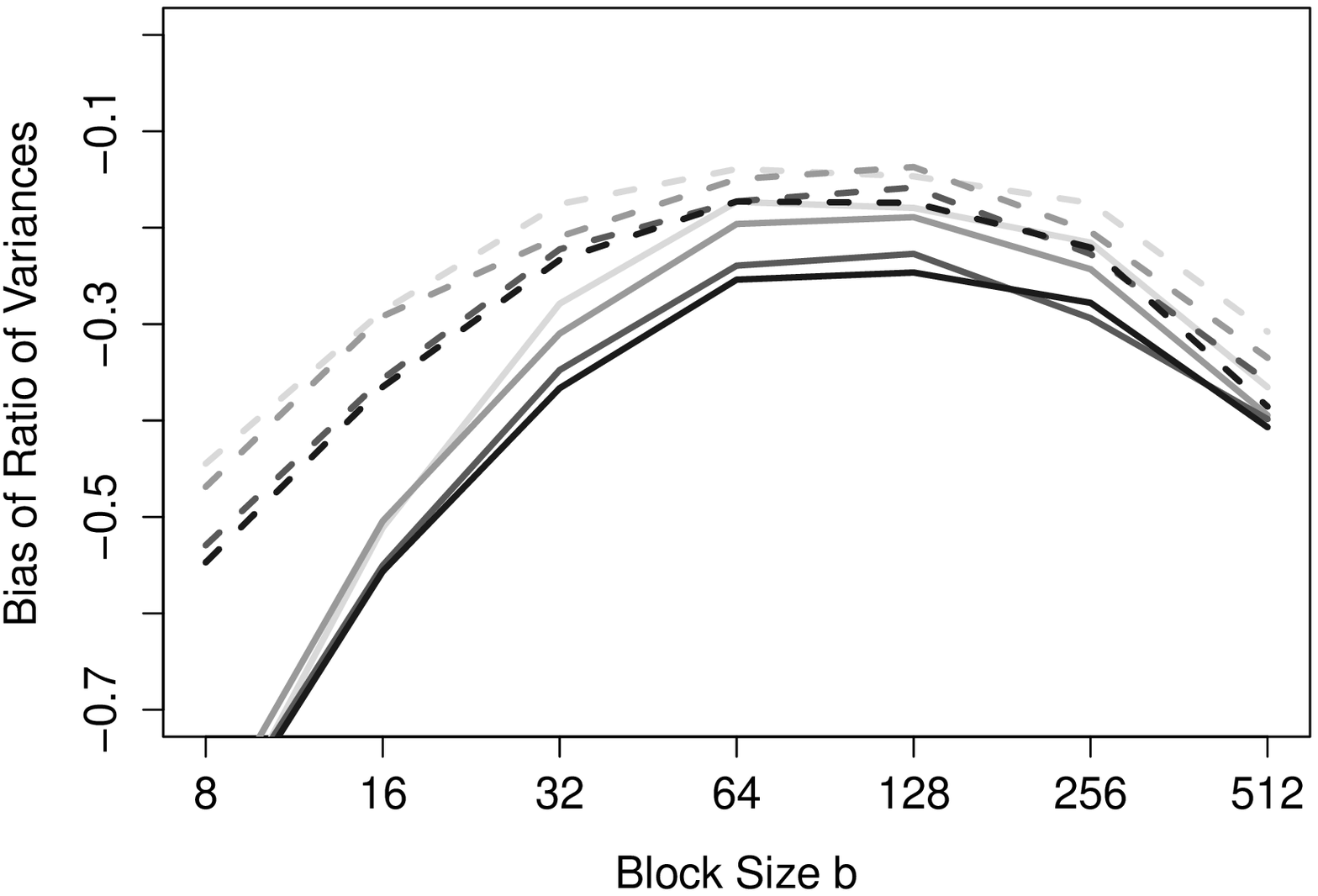}
\end{center}
\vspace{-.3cm}
\caption{\label{fig:msevariance3}  Mean squared error $\Exp[(\hat \tau^2/ \Var (\hat \theta_n) - 1 )^2]$ and bias  $\Exp[\hat \tau^2/ \Var (\hat \theta_n) ]- 1 $ within the squared ARCH-model for the PML-estimator (left) and Northrop's estimator (right).
}
\end{figure}

\bibliographystyle{chicago}
\bibliography{biblio}

\end{document}